\documentclass[11pt]{article}

\usepackage{etoolbox}
\makeatletter
\robustify\@latex@warning@no@line
\makeatother
\usepackage{authblk}

\usepackage[utf8]{inputenc}
\usepackage[british]{babel}
\usepackage[a4paper,twoside,top=1in, bottom=1.1in, left=1in, right=1in]{geometry} 
\usepackage{graphicx}
\usepackage{csquotes}
\usepackage{hyperref} 
\hypersetup{
	colorlinks = true,
	linkcolor = {blue},
	urlcolor = {red}
}

\usepackage{amsmath}
\usepackage{amsthm}
\usepackage{float}
\usepackage{amssymb}
\usepackage{amsfonts, bbold}
\usepackage{esint} 
\usepackage{mathrsfs} 
\usepackage{color}
\usepackage{array}
\usepackage{hhline}
\usepackage{enumitem} 
\usepackage{imakeidx} 
\usepackage{xparse} 
\usepackage{mathtools}
\usepackage{ifthen} 
\usepackage{forloop} 
\usepackage{xstring}
\usepackage{blindtext} 

\usepackage[nameinlink, capitalize]{cleveref} 

\newcounter{Results}[section] 

\makeatletter
\newcommand{\newreptheorem}[2]{\newtheorem*{rep@#1}{\rep@title}\newenvironment{rep#1}[1]{\def\rep@title{#2 \ref*{##1}}\begin{rep@#1}}{\end{rep@#1}}}
\makeatother
\newtheorem{theorem}{Theorem}[section]

\newtheorem{lemma}[theorem]{Lemma}
\newtheorem{definition}[theorem]{Definition}

\newtheorem{remark}[theorem]{Remark}
\newtheorem{rmk}[theorem]{Remark}
\newtheorem{proposition}[theorem]{Proposition}

\newtheorem{corollary}[theorem]{Corollary}

\newreptheorem{theorem}{Theorem}
\newreptheorem{lemma}{Lemma}

\numberwithin{equation}{section}

\crefname{equation}{Equation}{Equations}
\crefname{aligned}{Equation}{Equations}
\crefname{theorem}{Theorem}{Theorems}
\crefname{lemma}{Lemma}{Lemmas}
\crefname{proposition}{Proposition}{Propositions}
\crefname{corollary}{Corollary}{Corollaries}
\crefname{exercise}{Exercise}{Exercises}
\crefname{remark}{Remark}{Remark}
\crefname{definition}{Definition}{Definitions}
\crefname{section}{Section}{Sections}
\crefname{chapter}{Capitolo}{Capitoli}

\newcommand{\N}{\ensuremath{\mathbb N}}
\newcommand{\Z}{\ensuremath{\mathbb Z}}
\newcommand{\R}{\ensuremath{\mathbb R}}

\newcommand{\Fcal}{\ensuremath{\mathcal{F}}}

\newcommand{\hausd}{\mathcal H} 

\newcommand\eps{\ensuremath{\epsilon}}

\let\div\undefined
\DeclareMathOperator{\div}{div}

\DeclareMathOperator*{\argmin}{argmin} 

\ifdefined\virgola
\renewcommand{\virgola}{\ensuremath{\, \text{, }}}
\else
\newcommand{\virgola}{\ensuremath{\, \text{, }}}
\fi

\usepackage{amscd}

\usepackage{emptypage}

\usepackage{titlesec}

\titleformat{\chapter}[display]
{\normalfont\centering}{\MakeUppercase{\chaptertitlename\ \thechapter }}{0pt}{\Large\bfseries\uppercase}[\vspace{2ex}\titlerule]

\titleformat{\section}
{\normalfont\bfseries\large}
{\thesection.}{0.5em}{}

\newcommand{\Scal}[2]{\ensuremath{\langle #1 , #2 \rangle}} 

\newcommand{\sgn}{\ensuremath{\mathrm{sgn}}}%

\newlength\tindent
\usepackage{tikz}
\usetikzlibrary{patterns,decorations,backgrounds,patterns,decorations.pathreplacing}
\usepackage{subcaption}
\usetikzlibrary{external}

\usetikzlibrary{positioning,backgrounds}

\newcommand{\insieme}[1]{\left\{ #1 \right\}}
\def\per{\ensuremath{\mathrm{Per}}}
\def\loc{\ensuremath{\mathrm{loc}}}
\def\peroned{\ensuremath{\mathrm{Per}^{\text{1D}}}}
\def\d {\ensuremath{\,\mathrm {d} }}
\def\dx{\ensuremath{\,\mathrm {d}x}}
\def\dz{\ensuremath{\,\mathrm {d}z}}
\def\ds{\ensuremath{\,\mathrm {d}s}}
\def\du{\ensuremath{\,\mathrm {d}u}}
\def\dv{\ensuremath{\,\mathrm {d}v}}
\def\dt{\ensuremath{\,\mathrm {d}t}}
\def\dy{\ensuremath{\,\mathrm {d}y}}
\def\dw{\ensuremath{\mathrm{d}w}}

\def\dist{\ensuremath{\mathrm{dist}}}

\usepackage{xspace}

\makeatletter
\DeclareRobustCommand{\etc}{%
	\@ifnextchar{.}%
	{etc}%
	{etc.\@\xspace}%
}
\makeatother

\DeclareRobustCommand{\eg}{e.g.\@\xspace}
\DeclareRobustCommand{\ie}{i.e.\@\xspace}
\DeclareRobustCommand{\withoutloss}{w.l.o.g.\@\xspace}
\DeclareRobustCommand{\withoutLoss}{w.l.o.g.\@\xspace}

\DeclareRobustCommand{\WithoutLoss}{W.l.o.g.\@\xspace}
\DeclareRobustCommand{\lhs}{l.h.s.\@\xspace}
\DeclareRobustCommand{\rhs}{r.h.s.\@\xspace}

\DeclareRobustCommand{\st}{s.t.\@\xspace}
\DeclareRobustCommand{\ae}{a.e.\@\xspace}

\DeclareUnicodeCharacter{0308}{\"i}

\def\BiBFile{bibliography/allbib.bib}

\ifx\bblForArxiv\undefined{}
\usepackage[backend=biber,doi=false,url=false]{biblatex}
\newcommand{\printbibliographyMio}{\printbibliography}
\else
\newcommand{\printbibliographyMio}{\bibliography{\BiBFile}}
\fi

\def\d {\ensuremath{\,\mathrm {d} }}
\def\dx{\ensuremath{\,\mathrm {d}x}}
\def\dz{\ensuremath{\,\mathrm {d}z}}
\def\ds{\ensuremath{\,\mathrm {d}s}}
\def\du{\ensuremath{\,\mathrm {d}u}}
\def\dv{\ensuremath{\,\mathrm {d}v}}
\def\dt{\ensuremath{\,\mathrm {d}t}}
\def\dy{\ensuremath{\,\mathrm {d}y}}
\def\dw{\ensuremath{\mathrm{d}w}}

\def\S{\mathbb{S}}
\def\hd{\mathcal{H}^{d-1}}

\def\r{\mathbf{r}}
\usepackage{authblk}
\author[1]{Sara Daneri\thanks{sara.daneri@gssi.it}}
\author[1,2]{Eris Runa\thanks{eris.runa@gssi.it}}
\affil[1]{Gran Sasso Science Institute, L'Aquila, Italy}
\affil[2]{Courant Institute, NYU, USA}

\title{A rigorous approach to pattern formation for isotropic isoperimetric problems with competing  nonlocal interactions}
\date{}

\newcommand{\etautheta}{\ensuremath{e_{\tau, \delta, \theta}}}
\newcommand{\etau}{\ensuremath{e_{\tau, \delta}}}

\makeatletter
\tikzset{hatch distance/.store in=\hatchdistance,hatch distance=5pt,hatch thickness/.store in=\hatchthickness,hatch thickness=5pt}

\pgfdeclarepatternformonly[\hatchdistance,\hatchthickness]{north east hatch}
{\pgfqpoint{-\hatchthickness}{-\hatchthickness}}
{\pgfqpoint{\hatchdistance+\hatchthickness}{\hatchdistance+\hatchthickness}}
{\pgfpoint{\hatchdistance}{\hatchdistance}}%
{
	\pgfsetcolor{\tikz@pattern@color}
	\pgfsetlinewidth{\hatchthickness}
	\pgfpathmoveto{\pgfqpoint{-\hatchthickness}{-\hatchthickness}}
	\pgfpathlineto{\pgfqpoint{\hatchdistance+\hatchthickness}{\hatchdistance+\hatchthickness}}
	\pgfusepath{stroke}
}
\pgfdeclarepatternformonly[\hatchdistance,\hatchthickness]{north west hatch}
{\pgfqpoint{-\hatchthickness}{-\hatchthickness}}
{\pgfqpoint{\hatchdistance+\hatchthickness}{\hatchdistance+\hatchthickness}}
{\pgfpoint{\hatchdistance}{\hatchdistance}}%
{
	\pgfsetcolor{\tikz@pattern@color}
	\pgfsetlinewidth{\hatchthickness}
	\pgfpathmoveto{\pgfqpoint{\hatchdistance+\hatchthickness}{-\hatchthickness}}
	\pgfpathlineto{\pgfqpoint{-\hatchthickness}{\hatchdistance+\hatchthickness}}
	\pgfusepath{stroke}
}
\makeatother

\begin{document}
	\maketitle
	\vspace{-5mm}
	\begin{center}
		\textit{Dedicated to Pierangelo Marcati on the occasion of his 71\textsuperscript{st} birthday.
		}
	\end{center}
	\vspace{3mm}
	
	\begin{abstract}
		We introduce a rigorous approach to the study of the symmetry breaking and pattern formation phenomenon for  isotropic functionals with local/nonlocal interactions in competition.
		
		We consider a general class of nonlocal variational problems in dimension $d\geq 1$,  in which an isotropic  surface term favouring pure phases competes with an isotropic  nonlocal term with power law kernel favouring alternation between different phases.
		
		Close to the critical regime in which the two terms are of the same order, we give a rigorous proof of the conjectured structure of global minimizers, in the shape of domains with flat boundary (\eg stripes or lamellae).
		
		The natural framework in which our approach is set and developed is the one of calculus of variations and geometric measure theory.
		
		Among others,  we identify a nonlocal curvature-type quantity which is controlled by the energy functional and whose finiteness implies flatness for sufficiently regular boundaries.
		
		The power of decay of the considered kernels at infinity is $p\geq d+3$, and it is related to pattern formation in synthetic antiferromagnets.
		
	\end{abstract}
	

	\section{Introduction}\label{sec:intro}
	We consider the following class of functionals, in general dimension $d\geq1$: for $J>0$, $\Omega\subset\R^d$ open and bounded and $E\subset\R^d$, let
	\begin{equation}
		\label{eq:tildeFunctional}
		\tilde{\Fcal}_{J,p,d}(E,\Omega) =\frac{1}{|\Omega|}\Biggl[ J \per(E; \Omega) - \int_{\Omega}\int_{\R^d} \Big| \chi_E(x+\zeta) - \chi_E(x) \Big |  K(\zeta)   \dx \d\zeta\Biggr],
	\end{equation}
	where $\per(\cdot ,\Omega)$ is the classical isotropic perimeter functional relative to $\Omega$ (measuring for regular sets the surface measure of their  boundary inside $\Omega$)
	and  $K$ is  an isotropic  integrable  kernel with $p$-power law decay at infinity (for precise assumptions see~\eqref{eq:kass1}--\eqref{eq:kass4}).

	From the physical point of view, it is since long well-established  that such a type of energy competition (between short range interactions favouring pure phases and long range interactions favouring alternation between different phases) is at the base of spontaneous  pattern formation in nature. In particular, experiments and simulations in different physical systems suggest that in a suitable regime in which the competition is active (modulated in our case by the value of the constant $J$) the continuous symmetry of the functionals is broken, namely minimizers have fewer symmetries than the original functionals, and moreover they are organized,   up to boundary effects, into periodic or nearly periodic structures
	(see \eg~\cite{sa, bates1999block, thomas1988periodic, hubert1969stray, stradner2004equilibrium, andelman2009modulated, hubert2008magnetic,muratov2002theory} and references therein). To avoid boundary effects, one typically considers $\Omega=[0,L)^d$, $L>0$, and $[0,L)^d$-periodic sets $E\subset\R^d$.

	A long standing problem in the mathematical community is to understand rigorously the mechanisms that are at the base of energy-driven pattern formation. While in dimension $d=1$ periodicity of minimizers  is now  understood  under convexity (see \eg~\cite{alberti2001new,muller1993singular,chen2005periodicity,ren2003energy,giuliani2012striped}) or reflection positivity assumptions on the kernel (see \eg~\cite{giuliani2006ising,giuliani2009periodic}), in more than one space dimension the additional phenomenon of symmetry breaking makes the problem significantly more challenging.
	
	In general, suitably tuning the parameters modulating  the competition between the short range and the long range term, one expects to observe a variety of different patterns.
	
	A very common pattern is that given by small droplets centered at the vertices of a periodic lattice.
	This picture is observed for example  in a regime in which the surface term is dominant w.r.t. the nonlocal term and a volume constraint is imposed.
	The problem of periodic droplets formation, and the related study of the shape of minimizers in the  Gamow's liquid drop model, has received a lot of attention and has  been widely investigated, giving deep and  interesting results in low volume fraction regimes  (see \eg~\cite{choksi2010small, choksi2011small, knupfer2013isoperimetric, muratov2014isoperimetric, julin2014isoperimetric, bonacini2014local,cicalese13spadaro, lu2014nonexistence, frank2016nonexistence, knupfer2016low, goldman2013gamma, goldman2014gamma, muratov2010droplet, choksi2017old} and references therein).
	However, the exact shape of droplets and their arrangement in periodic structures still remains a major open problem in its full generality.
	
	In this paper we are interested in studying the emergence of another type of pattern ubiquitous in nature, namely the so-called  stripes/lamellae.
	By stripes we mean phases separated by flat and disjoint interfaces orthogonal  to a given direction.
	Such patterns are  the first one which are expected to emerge from uniform phases once symmetry is broken and are observed for example in a regime in which the surface term and the nonlocal term are of the same order.
	
	In this regime, in more than one space dimension, symmetry breaking and pattern formation have been recently proved in anisotropic settings. The anisotropy of the model was either due to the choice of a discrete lattice as a domain (in~\cite{gs_cmp}) or, in a continuous setting, by the choice of interactions with a discrete symmetry group (in~\cite{dr_arma,gr}).
	In the  discrete domain given by a square lattice,  for  kernels with power $p>2d$, striped pattern formation was first proved in~\cite{gs_cmp}.
	On  continuous domains  but  for anisotropic interactions (perimeter and kernel) symmetric with respect to the group of coordinate permutations, symmetry breaking was proved in~\cite{gr}, for powers $p>2d$.
	For the same class of interactions and domains, in~\cite{dr_arma} the authors proved exact striped pattern formation for exponents $p\geq d+2$ (see also~\cite{DRannSNS} and \cite{ker} for an extension to $p\geq d+2-\eps$, with $0<\eps\ll1$).
	Similar results were then obtained for screened Coulomb kernels~\cite{dr_siam}, for diffuse interface versions of the model~\cite{dkr_jfa, dr_therm} (proving exact one-dimensionality of minimizers in general dimension), and including a volume constraint~\cite{dr_vol}.
	
	Isotropic interactions are of particular interest given that, in the large majority of physical literature, models exhibiting spontaneous pattern formation have isotropic interactions.
	However, up to now, there are only partial results in this setting in the literature:
	for Coulomb  kernels in dimension $d\geq2$ we mention the seminal paper of Alberti, Choksi and Otto~\cite{ACO} in which they prove that minimizers  satisfy a uniform distribution of energy; in dimension $d=2$ exact pattern formation was proved on thin domains $\Omega=(0,\varepsilon)\times(0,1)$, with $\varepsilon\ll1$~(see~\cite{morini_sternberg}).
	
	More recently, Muratov and Simon~\cite{MS} proved that in dimension $d=2$ and  for kernels decaying at infinity  with power $p=d+3$, when considering the $\Gamma$-limit as the radius of the regularization of the kernel at the origin tends to $0$ in the critical regime, generalized minimizers exist and cannot be disks.
	Such a decay (in dimension $d=2$) is physically related to the energy of two identical, but oppositely oriented dipolar patches lying in parallel planes separated by a given positive distance.
	Such a model is, for example, relevant to synthetic antiferromagnets (see, \eg~\cite{moser2002magnetic}).
	Heuristically, on large scales the kernel decays faster than the dipolar kernel, thus
	reducing the long-range repulsion in the far field.
	
	\vskip 0.2 cm
	In this paper, we consider fully isotropic interactions and kernels with power law decay at infinity  $p\geq d+3$ in general dimension~$d$.
	
	If the kernel decays at infinity like a power $p>d+1$, it is well known that there is a critical constant $J_c>0$ such that for $J>J_c$, $\tilde\Fcal_{J,p,d}\geq 0$ and it is minimized by the trivial sets $\emptyset,\R^d$ (see~\cite{giuliani2006ising}), while for $J<J_c$ the trivial sets are not minimizers. In Theorem \ref{thm:jc} we show that the critical constant  $J_c$ is given by
	\begin{equation}
		\label{eq:Jc}
		J_c=\int_{\R^d}|\zeta_1|K(\zeta)\d\zeta,
	\end{equation}
	where $\zeta_1=\Scal{\zeta}{e_1}$.
	Such a constant is the counterpart of the one found in \cite{giuliani2011checkerboards} for  the discrete setting. 
	Symmetry breaking and striped pattern formation is conjectured for $J<J_c$, $|J-J_c|\ll1$.
	
	Our main result (see Theorem~\ref{thm:main} below) consists in proving such a conjecture. To state it precisely, we need to introduce some further notation and to suitably rescale the functional (without changing the structure of the minimizers).
	Set $\tau:=J_c-J>0$.
	Minimizing $\tilde{\Fcal}_{{J}_c-\tau,p,d}$ in the class of periodic unions of stripes, one can see that  for $0<\tau\ll1$ the stripes with optimal energy have width and distance of order  $\tau^{-1/p-d-1}$ and energy of order $\tau^{(p-d)/(p-d-1)}$.
	Therefore it is natural to  rescale the spatial variables and the functional  so that the optimal width and distance for unions of stripes is $O(1)$ and  the energy is $O(1)$ for $0<\tau\ll1$. Then, setting $\tau^{-1/(p-d-1)}\zeta'=\zeta$, $\tau^{-1/(p-d-1)}x'=x$ and $\tilde{ \Fcal}_{J_c-\tau,p,d}(E,[0,L)^d)=\tau^{(p-d)/(p-d-1)}\Fcal_{\tau,p,d}(E\tau^{1/(p-d-1)},[0,L\tau^{1/(p-d-1)})^d)$, and rescaling the kernel $K$ into $K_\tau$ satisfying assumptions~\eqref{eq:kass1}--\eqref{eq:kass4}, one ends up considering the functional
	\begin{equation}\label{eq:ftauintro}
		\Fcal_{\tau,p,d}(E,[0,L)^d)=\frac{1}{L^d}\Biggl[J_\tau\per(E,[0,L)^d)-\int_{[0,L)^d}\int_{\R^d}|\chi_E(x+\zeta)-\chi_E(x)| K_\tau(\zeta)\dx\d\zeta\Biggr],
	\end{equation}
	where
	\begin{equation}\label{eq:ktau}
		\frac{1}{C\max\{\tau^{1/(p-d-1)},\|\zeta\|\}^p}\leq K_\tau(\zeta) \leq \frac{C}{\max\{\tau^{1/(p-d-1)},\|\zeta\|\}^p}
	\end{equation}
	and
	\begin{equation}\label{eq:jtau}
		J_\tau=\int_{\{\|\zeta\|\leq 1\}}|\zeta_\theta|K_\tau(\zeta)\d\zeta\quad \text{ for any $\theta\in\S^{d-1}$, $\zeta_\theta=\Scal{\zeta}{\theta}$, $\zeta=\zeta_\theta\theta+\zeta_\theta^\perp$.}
	\end{equation}
	
	Our main result is the following
	
	\begin{theorem}
		\label{thm:main}
		Let $d\geq 1$, $p\geq d+3$,  $L>0$. Then, there exists $\hat\tau>0$ such that for  every $0<\tau < \hat\tau$ the $[0,L)^d$-periodic  minimizers $E_\tau$ of $\Fcal_{\tau,p,d}(\cdot, [0,L)^d)$ are, up to a rigid motion, of the form
		\begin{equation}
			E_\tau=\widehat E_\tau\times \R^{d-1},\quad\widehat E_\tau=\underset{k\in\N}{\bigcup}(2kh^*_{L}, (2k+1)h^*_{L}),
		\end{equation}
		for some $h^*_{L}>0$ such that $2kh^*_{L}=L$ for some $k\in\N$.
	\end{theorem}
	
	Moreover,  for a general bounded domain $\Omega\subset\R^d$, we prove that every sequence $E_\tau\subset\R^d$ of sets of equibounded energies in $\Omega$, \ie satisfying $\sup_{\tau}\Fcal^{\mathrm{loc}}_{\tau, p, d}(E_\tau,\Omega)<+\infty$ (where $\Fcal^{\mathrm{loc}}_{\tau, p, d}$ is a meaningful restriction of $\Fcal_{\tau,p,d}$ to arbitrary domains $\Omega$ avoiding boundary effects, see \eqref{eq:floc}) converges in $L^1(\Omega)$ to a set $E_0$ with flat boundary. More precisely, one has the following:
	
	\begin{theorem}
		\label{thm:omega}
		Let $d\geq1$, $p\geq d+3$. Let $\Omega\subset\R^d$ be a bounded open set and $E_\tau\subset\R^d$ a sequence of sets such that
		\begin{equation}\label{eq:suptau}
			\sup_\tau\Fcal^{\mathrm{loc}}_{\tau, p, d}(E_\tau,\Omega)<+\infty.
		\end{equation}
		Then, up to subsequences, $E_\tau\cap\Omega$ converge in $L^1(\Omega)$ to $E_0\cap\Omega$ such that $\per(E_0,\Omega)<+\infty$ and $\partial E_0\cap\Omega$ is given by a disjoint union $\bigcup_{i=1,\dots,N}H_i\cap\Omega$ where each $H_i$ is an affine hyperplane in $\R^d$.
		
		If $\Omega=[0,L)^d$ and one restricts to $[0,L)^d$-periodic sets (thus, $\Fcal^{\mathrm{loc}}_{\tau,p,d}(E,[0,L)^d)=\Fcal_{\tau,p,d}(E,[0,L)^d)$), then the hyperplanes $\{H_i\}$ above are all parallel.
	\end{theorem}
	
	Notice that in the above theorem we do not need to impose any periodic boundary conditions to obtain flatness of the boundary of sets of finite energy in the critical regime. In particular, inside the domain $\Omega$ the minimizers of $\Fcal^{\mathrm{loc}}_{\tau, p, d}$ for $0<\tau\ll1$ are close to sets with flat boundary in $\Omega$, thus showing symmetry breaking.
	The $[0,L)^d$-periodicity in this case ensures only that the hyperplanes $H_i$ are all parallel, and thus $E_0$ is given by a union of stripes.

	To prove continuous symmetry breaking (namely, breaking of rotational symmetry) requires major steps forward with respect to discrete symmetry breaking (namely, breaking symmetry with respect to coordinate permutations). In both cases, the goal is to show that the normal to the boundary of the minimizers is locally constant. However, while in the anisotropic setting there is a discrete hence disconnected set of preferred directions $\{e_1,\dots,e_d\}$,
	in the isotropic setting every direction $\nu\in\S^{d-1}$ is equally favoured.
	Moreover, one of the difficulties in the isotropic setting is to control both small curvature deviations on large scales and large curvature deviations on small scales. In order to do so, we first identify an integral geometric formulation (see Proposition~\ref{prop:intgeom}) for the functional~\eqref{eq:ftauintro} which in the critical regime $\tau=0$ allows to control a nonlocal type of curvature of the boundary of sets of finite energy (see Proposition~\ref{prop:brezis}). Such a control  implies flatness of the boundaries of Lipschitz sets of finite energy (see Lemma~\ref{lemma:d+3}). Hence, to show flatness we proceed by showing regularity of the boundary of  sets of finite energy in the  critical regime (see Theorem~\ref{thm:regularity} and for more details on the analysis in the critical regime, see Section~\ref{sec:rigidity}). Once the symmetry is broken in the limit $\tau\to0$, for $0<\tau\ll1$ we proceed via a $d$-dimensional optimization argument.

	\subsection{Plan of the paper} In Section~\ref{sec:prel} we introduce some preliminary notation, classical geometric measure theory facts, and the main assumptions on the nonlocal kernel.  The proof of Theorem~\ref{thm:main} is organized as follows: finding a suitable decomposition of the functional bounding quantities with geometrical meaning; proving a rigidity estimate that shows that minimizers are close to stripes as $\tau\to0$; showing stability estimates for $0<\tau\ll1$. In Section~\ref{sec:int} we introduce an integral geometric formulation for the functional~\eqref{eq:ftauintro} which will be crucial in our analysis. Section~\ref{sec:1D} contains a series of one dimensional energy bounds and one dimensional optimization procedures which will be applied in the following sections to the one dimensional slices of the  minimizing sets.
	Section~\ref{sec:rigidity} contains the proof of Theorem~\ref{thm:omega}, \ie the main rigidity estimates in the critical regime and $\Gamma$-convergence results as $\tau\to0$ in the periodic setting. To guide the reader, in Section~\ref{sec:rigidity} we give an outline of the main steps of the proof. A part of the regularity estimates needed in Section~\ref{sec:rigidity}  is reported  in the \hyperref[sec:appendix]{Appendix}. Finally, in Section~\ref{sec:positive_tau} we complete the proof of Theorem~\ref{thm:main} by showing exact striped pattern formation for $0<\tau\ll1$.
	
	\section*{Acknowledgements}
	The second author is supported by the European Union's Horizon 2020 research and innovation programme under the Marie Skłodowska-Curie Grant Agreement No 101063588.
	All authors are members of the GNAMPA group in INDAM\@.
	The authors thank Guido De Philippis and Camillo De Lellis for the hospitality respectively at Courant Institute and Institute for Advanced Study.

	\section{Notation and preliminaries}\label{sec:prel}
	We denote by $\S^{d-1}$ the unit sphere in $\R^d$ and by $B_r(x)$ the Euclidean ball of radius $r>0$. We let  $\Scal{x}{y}$ be   the scalar product between $x,y\in\R^d$ and by $\|x\|$ the Euclidean norm of $x\in\R^d$.
	We let $\hausd^{d-1}(E)$ be the $(d-1)$-dimensional Hausdorff measure of a set $E\subset\R^d$ and by $|E|$ its Lebesgue measure.  When restricting to  $k$-dimensional affine subspaces of $\R^d$ and it is clear from the context, we will use the same notation $|\cdot|$ to denote the $k$-dimensional Hausdorff measure.
	We denote by $\omega_d$ the Lebesgue measure of the unit ball in $\R^d$.
	Given a Radon measure $\mu$ on $\R^d$, we denote by $|\mu|$ its total variation and by $\mathrm{spt}\mu$ its support.
	Given a bounded set $A\subset\R^d$ and a point $x\in\R^d$, we define $\mathrm{dist}(x,A)=\inf\{\|y-x\|:\,y\in A\}$.
	Moreover, given $\rho>0$, we define $(A)_\rho=\{z\in\R^d:\,\mathrm{dist}(z,A)<\rho\}$.
	We denote by $\#A$ the cardinality of the set $A$ and by $\bar A$ its closure. For sets $A,B\subset\R^d$, we let $A\Delta B=(A\setminus B)\cup(B\setminus A)$ be their symmetric difference.

	For every $\theta\in\S^{d-1}$, we let
	\begin{equation}
		\label{eq:notation1}
		\theta^{\perp} := \insieme{x\in\R^{d}:\ \Scal{x}{\theta} = 0}
	\end{equation}
	and given $\Omega\subset\R^d$ we denote by $(\Omega)_\theta^\perp$ the projection of $\Omega$ on the $(d-1)$-dimensional plane $\theta^\perp$.
	Given a vector $x\in\R^d$ we can decompose it as a sum of orthogonal vectors $x=x_\theta^\perp+x_\theta\theta$, where $\theta\in\S^{d-1}$, $x_\theta^\perp\in\theta^\perp$ and $x_\theta\in\R$.
	We define the one dimensional slice of $E$ in direction $\theta$ with reference point $x^\perp_\theta\in \theta^\perp$  as
	\begin{equation}
		E_{x^\perp_{\theta}} := \insieme{z\in \R^d:\ z\in E\cap  (x^\perp_{\theta} + \theta\R) }.
	\end{equation}
	With a slight abuse of notation we will also identify the points in $ E_{x^\perp_{\theta}}$ with the points $s\in\R$ such that $x^\perp_{\theta} + s\theta\in E$.

	Given $k=1,\dots,d-1$, we denote the Grassmanian of $k$-dimensional planes in $\R^d$, by $G(k,\R^d)$.
	We denote the elements of $G(k,\R^d)$ by $\pi_k$ and by $\pi_k^\perp$ the $(d-k)$-dimensional subspace orthogonal to $\pi_k$. We also decompose points $x\in\R^d$ as  $x=x_{\pi_k}+x_{\pi_k}^\perp$, where  $x_{\pi_k}\in\pi_k$ and $x_{\pi_k}^\perp\in\pi_k^\perp$ and denote by $ E_{x^\perp_{\pi_k}}$ the $k$-dimensional slice of $E$ with the affine $k$-dimensional plane parallel to $\pi_k$  and passing through the point $x^\perp_{\pi_k}$.
	We let $\S^{k-1}_{\pi_k}=\S^{d-1}\cap\pi_k$ and for $\theta\in\S^{k-1}_{\pi_k}$ we denote the points in $(\pi_k)_\theta^\perp$ by $x_{(\pi_k)_\theta^\perp}$.
	
	We denote by $\mu_{k,d}$ the invariant  finite Radon measure on $G(k,\R^d)$ such that for any measurable function $f:\S^{d-1}\to\R^+$
	\begin{equation}\label{eq:haar}
		\int_{G(k,\R^d)} \int_{\S^{k-1}_{\pi_k}}f(\theta)\d\theta\d\mu_{k,d}(\pi_k)=\int_{\S^{d-1}}f(\theta)\d\theta.
	\end{equation}

	For a set of locally finite perimeter $E\subset\R^d$, we denote by $D\chi_E$ the associated Radon measure, by $\partial ^*E$ the reduced boundary and  by $\nu_E$ the measure theoretic exterior normal at points of $\partial^*E$.
	We let $\partial E$ be  the topological boundary of $E$, where we consider  a representative $E$ such that
	\begin{equation}
		\mathrm{spt}D\chi_E=\{x\in\R^d:\,0<|E\cap B_r(x)|<\omega_d r^d,\quad\forall\,r>0\}=\partial E.\label{eq:topbdry}
	\end{equation}
	In general, $\partial^*E\subset\partial E$ and $\overline{\partial^*E}=\partial E$.

	We notice also that, by the blow-up properties of the reduced boundary, for every $z\in\partial^* E$ and for every $r>0$ it holds
	\begin{align}
		\Bigl|\{s\in [0,r]:\, z+s\theta\in E\}\Bigr|&>0,\quad\text{for \ae $\theta$ \st $\langle\nu_E(z)\cdot\theta\rangle<0$}\notag\\
		\Bigl|\{s\in [0,r]:\, z+s\theta\in \R^d\setminus E\}\Bigr|&>0,\quad\text{for \ae $\theta$ s.t. $\langle\nu_E(z)\cdot\theta\rangle>0$}\label{eq:stimeblowup}.
	\end{align}
	
	For $x\in\partial^*E$ we denote by $H_{\nu_E(x)}(x)$ the affine halfspace $H_{\nu_E(x)}(x):=\{y\in\R^d:\,\Scal{y-x}{\nu_E(x)}<0\}$.
	
	Given $\eta>0$, $\nu\in\S^{d-1}$, $x\in\R^d$, we define the cone with vertex at $x$, base plane $\nu^\perp$ and  opening $\eta$ as
	\[
	K_\eta(x,\nu)=\{y\in\R^d: \,|y_\nu-x_\nu|<\eta\}.
	\]
	Using the slicing properties of sets of finite perimeter, for any set $E$ of locally finite perimeter, for a.e. $\theta\in\S^{d-1}$, for a.e. $x_\theta^\perp\in\theta^\perp$, the set $E_{x_\theta^\perp}$ is a one dimensional set of locally finite perimeter. Hence, on any compact interval $[a,b]\subset\R$, $\#(\partial^*E_{x_\theta^\perp}\cap[a,b])<+\infty$.
	
	Given a point $x_\theta\in\partial^*E_{x_\theta^\perp}\cap\Omega_{x_\theta^\perp}$, we denote by $x_\theta^+$ and $x_\theta^-$ the points in $\partial^*E_{x_\theta^\perp}$ which are closest to $x_\theta$ and for which respectively $x_\theta^+-x_\theta>0$, $x_\theta^--x_\theta<0$. For later use, we introduce also the following notation
	\begin{equation}
		\r_\theta(x):=|x_\theta-x_{\theta}^+|.
	\end{equation}
	
	We will denote by $\peroned(A)$ the one dimensional perimeter of a set $A\subset\R$, and by $\peroned(A,\Omega)$ the one dimensional perimeter of $A$ relative to $\Omega\subset\R$. Moreover, with a slight abuse of notation, given $B\subset\R^d$ such that  $B = \insieme{\overline{x} + s \theta:\, s\in A}$  for some $\overline{x} \in \R^d$,  $\theta\in \S^{d-1}$ and $A\subset\R$, we define $\peroned(B) := \peroned(A)$. The same applies to the perimeter of a set $B$ as above relative to a set $\Omega\subset\R^d$.
	
	We recall (see \eg~\cite{afp,Maggi}) the following classical slicing formula: given $\Omega\subset\R^d$ and a measurable function $g:\Omega\times\S^{d-1}\to\R^+$ such that $g(\cdot, \theta)$ is $\mathcal H^{d-1}\llcorner(\partial^*E\cap \Omega)$-summable, one has that
	
	\begin{equation}
		\label{eq:slicing}
		\int_{\partial^* E\cap \Omega}|\langle \nu_E(x),\theta\rangle|g(x,\theta)\d\mathcal H^{d-1}(x)=\int_{\theta^\perp}\sum_{s\in \partial^* E_{x_\theta^\perp}\cap \Omega_{x_\theta^\perp}}g(x_\theta^\perp+s\theta, \theta)\dx_\theta^\perp.
	\end{equation}
	Moreover,
	\begin{equation}
		\label{eq:slicing_sign}
		\int_{\partial^* E\cap \Omega}\langle \nu_E(x),\theta\rangle g(x,\theta)\d\mathcal H^{d-1}(x)=\int_{\theta^\perp}\sum_{s\in\partial^* E_{x_\theta^\perp}\cap \Omega_{x_\theta^\perp}}\mathrm{sign}\bigl(\Scal{\nu_E(x_\theta^\perp+s\theta)}{\theta}\bigr)\, g(x_\theta^\perp+s\theta, \theta)\dx_\theta^\perp.
	\end{equation}
	
	Given a set of locally finite perimeter $E\subset\R^d$, and a point $x\in\R^d$, one  defines the (spherical) excess of $E$ in the ball of radius $r$ centered at $x$ as
	
	\begin{equation}\label{eq:exc}
		Exc(E,x,r)=\frac{1}{r^{d-1}}\Bigl[|D\chi_E|(B_r(x))-|D\chi_E(B_r(x))|\Bigr].
	\end{equation}
	
	One has that for all $x\in\partial^*E$,  $Exc(E,x,r)\to0$ as $r\to0$.  The converse does not necessarily hold. However, we will use in the proof the following sufficient condition, guaranteeing that $\partial E=\partial^*E$ (see~\cite{Tam}). Whenever  for  $x\in\partial E$ there exist $R, C_1,C_2,\alpha>0$ such that for all $0<r<R$, the following holds:
	\begin{align}
		{\per(E, B_r(x))}&\geq C_1r^{d-1}\label{eq:e1}\\
		Exc(E,x,r)&\leq C_2r^\alpha,\label{eq:e2}
	\end{align}
	then $x\in\partial^*E$. In particular, if~\eqref{eq:e1} and~\eqref{eq:e2} hold for every $x\in\partial E$, then $\partial E=\partial^*E$.
	
	Given a sequence of sets $\{E_n\}_{n\in\N}\subset\R^d$, we say that $E_n\to E$ in $L^1(\Omega)$ whenever $|(E_n\Delta E)\cap\Omega|\to0$ as $n\to\infty$, or equivalently if $\|\chi_{E_n}-\chi_E\|_{L^1(\Omega)}\to0$ as $n\to\infty$.

	For simplicity of exposition we fix the kernel $K$ in~\eqref{eq:tildeFunctional} to be
	\begin{equation}
		\label{eq:specific_kernel}
		K(\zeta) := \frac{1}{\max(1, \|\zeta\|^{p})}
	\end{equation}
	and thus the rescaled kernels (see Section~\ref{sec:intro}) are
	\begin{equation*}
		K_{\tau}(\zeta) := \frac{1}{\max(\tau^{1/(p-d-1)}, \|\zeta\|)^{p}}.
	\end{equation*}
	However, the specific form~\eqref{eq:specific_kernel} is not necessary for Theorem~\ref{thm:main} to hold.
	
	More in general,  we will need the following properties, which are in particular satisfied by~\eqref{eq:specific_kernel} (for a proof of the last property  see Section~\ref{sec:1D}):
	
	\begin{align}
		&\exists\,C:\quad\frac{1}{C(\|\zeta\|+\tau^{1/(p-d-1)})^p}\leq K_\tau(\zeta)\leq C\frac{1}{(\|\zeta\|+\tau^{1/(p-d-1)})^p},\quad p\geq d+3,\label{eq:kass1}\\
		&K_\tau(\zeta) \text{ converges monotonically increasing for $\tau\downarrow 0$ to $\frac{1}{\|\zeta\|^p}$ }, \label{eq:kass2}\\
		&K_{\tau}\text{ is symmetric under rotations.} \label{eq:kass3}
	\end{align}
	
	Moreover, we need to assume an additional property on $K_\tau$. Define
	\begin{equation*}
		\hat{K}_{\tau}(t):= \int_{\R^{d-1}} K_{1}(t e_{1} + t_1^\perp) \dt_1^\perp , \quad t\in\R,
	\end{equation*}
	and the one dimensional  functional
	\begin{equation}
		\label{eq:1dFunc}
		\Fcal^{\mathrm{1D}}_{\tau, p,d}(E,[0,L)) := \frac{1}{L} \bigg(
		J_\tau \per^{\mathrm{1D}}(E; [0,L))
		-
		\int_{[0, L)} \int_{\R} |\chi_{E}(s+t) - \chi_{E}(s)| \hat{K}_{\tau}(t) \ds\dt.
		\bigg)
	\end{equation}
	where $E\subset \R$  is an $L$-periodic set and $J_\tau$ is defined in~\eqref{eq:jtau}. We assume that for every $L>0$ there exists $\bar \tau>0$ such that for every $0<\tau<\bar \tau$  and  for every $E_\tau \in \argmin \Fcal^{\mathrm{1D}}_{\tau,p,d}(\cdot,[0,L))$ there exists $h^*_{L}$ such that up to translations it holds
	\begin{equation}
		\label{eq:kass4}
		E_\tau = \bigcup_{k} (2kh^*_{L}, (2k+1)h^*_{L}).
	\end{equation}

	Thanks to condition~\eqref{eq:kass3}, we will also use the notation $K_\tau(t):=K_\tau(z)$ for $z\in\partial B_{|t|}(0)$, $t\in\R$.

	For simplicity of notation, for $t\in\R$ we will also denote $\overline K_{\tau}(t) := |t|^{d-1}K_{\tau}(t) $.

	Notice that for any $[0,L)^d$-periodic set which, up to a rigid motion, is of the form $E=\widehat E\times\R^{d-1}$, $\widehat E\subset\R$, then
	\[
	\Fcal_{\tau,p,d}(E,[0,L)^d)=\Fcal^{\mathrm{1D}}_{\tau,p,d}(\widehat E,[0,L)).
	\]
	
	In the following we will often use this notation: given $a,b\in \R^+$
	\begin{align*}
		a\lesssim b,\quad a\gtrsim b,
	\end{align*}
	to denote that there exists a geometric constant $C>0$ depending only on $p,d$ such that respectively it holds
	\begin{equation*}
		a\leq Cb,\quad a\geq Cb.
	\end{equation*}

	\section{An integral geometric formulation}\label{sec:int}
	
	In this section we provide an integral geometric formulation for the functional $\Fcal_{\tau,p,d}$ introduced in~\eqref{eq:ftauintro}, which will be fundamental in our analysis.
	
	We recall the definition of the functional $\Fcal_{\tau,p,d}$ relative to $[0,L)^d\subset\R^d$
	\begin{equation}\label{eq:ftau2}
		\Fcal_{\tau,p,d}(E,[0,L)^d)=\frac{1}{L^d}\Big(J_\tau\per(E,[0,L)^d)-\int_{[0,L)^d}\int_{\R^d}\bigl|\chi_E(x+\zeta)-\chi_E(x)\bigr| K_\tau(\zeta)\dx\d\zeta\Big)
	\end{equation}
	
	Our main result is the following
	
	\begin{proposition}\label{prop:intgeom} Let $\Fcal_{\tau,p,d}$ be the functional defined in~\eqref{eq:ftau2}. One has that
		\begin{align}
			\Fcal_{\tau,p,d}(E, [0,L)^d)=\frac{1}{L^d}\int_{\S^{d-1}}\int_{\theta^\perp}\sum_{s\in\partial^* E_{x_\theta^\perp}\cap [0,L)^d_{x_\theta^\perp}}r_{\tau}(E_{x_\theta^\perp},s)\dx_\theta^\perp\d\theta,\label{eq:forintgeom}
		\end{align}
		where we set
		\begin{align}
			r_{\tau}(E_{x_\theta^\perp},s)&:=\int_{-1}^{1}|\rho|\overline K_\tau(\rho)\d\rho-\int_{s^-}^s\int_0^{+\infty}\bigl|\chi_{E_{x_\theta^\perp}}(u)-\chi_{E_{x_\theta^\perp}}(u+\rho)\bigr|\overline K_\tau(\rho)\d\rho\du\notag\\
			&-\int_{s}^{s^+}\int_{-\infty}^0\bigl|\chi_{E_{x_\theta^\perp}}(u)-\chi_{E_{x_\theta^\perp}}(u+\rho)\bigr|\overline K_\tau(\rho)\d\rho\du.\label{eq:rme}
		\end{align}
		Moreover, for any open and bounded set $\Omega$
		\begin{equation}
			\per(E; \Omega)  = \frac{1}{C_{1,d}}\int_{\S^{d-1}} \int_{\pi^{\perp}_{\theta}} \per^{\mathrm{1D}}(E_{x^\perp_{\theta}}, \Omega_{x^\perp_{\theta}})\dx_{\theta}^{\perp}\d\theta,\label{eq:igper}
		\end{equation}
		where
		\begin{equation}
			C_{1,d} := \int_{\S^{d-1}} |\Scal{\theta}{e_1}| \d\theta.
			\label{eq:Cd}
		\end{equation}
		
	\end{proposition}

	\begin{proof}
		Let us first consider the local term of $\Fcal_{\tau,p,d}$ on a generic bounded and open set $\Omega$. By the rotational invariance of the kernel, polar change of coordinates and Fubini Theorem, one has that
		\begin{align}
			J_\tau\per(E;\Omega)&=\int_{\partial ^*E\cap \Omega}\int_{\{\|\zeta\|\leq 1\}}|\zeta_\theta|K_\tau(\zeta)\d\zeta\d\mathcal{H}^{d-1}(x)\notag\\
			&=\int_{\partial ^*E\cap\Omega}\int_{\{\|\zeta\|\leq1\}}|\langle\zeta,\nu_E(x)\rangle |K_\tau(\zeta)\d\zeta\d\mathcal H^{d-1}(x)\notag\\
			&=\frac12\int_{\partial^* E\cap \Omega}\int_{\S^{d-1}}\int_{[-1,1]}|\rho|^d|\langle \theta,\nu_E(x)\rangle| K_\tau(\rho)\d\rho\d\theta \d\mathcal{H}^{d-1}(x)\notag\\
			&=\frac12\int_{\S^{d-1}}\int_{\partial ^*E\cap \Omega}|\langle \theta,\nu_E(x)\rangle|\int_{[-1,1]}|\rho|^d K_\tau(\rho)\d\rho\d\mathcal{H}^{d-1}(x)\d\theta\notag\\
			&=\frac12\int_{\S^{d-1}}\int_{\theta^\perp}\sum_{s\in\partial^* E_{x_\theta^\perp}\cap\Omega_{x_\theta^\perp}}\int_{-1}^{1}|\rho|\overline K_\tau(\rho)\d\rho\dx_\theta^\perp\d\theta,\label{eq:lintgeom}
		\end{align}
		where in the last equality we used the classical slicing formula for the perimeter in direction $\theta$ given in~\eqref{eq:slicing}.

		Let us now consider the nonlocal term. By polar change of coordinates, Fubini Theorem and~\eqref{eq:slicing}, one obtains
		
		\begin{align}
			\int_{[0,L)^d}\int_{\R^d}& \bigl| \chi_E(x+\zeta) - \chi_E(x) \bigr |  K_\tau(\zeta)   \dx \d\zeta=\notag\\
			&=\frac{1}{2}\int_{\S^{d-1}}\int_{-\infty}^{+\infty}\int_{[0,L)^d}\bigl|\chi_E(x)-\chi_E(x+\rho\theta)\bigr||\rho|^{d-1}K_\tau(\rho)\dx\d\rho\d\theta\notag\\
			&=\frac12\int_{\S^{d-1}}\int_{\theta^\perp}\int_{[0,L)^d_{x_\theta^\perp}}\int_{-\infty}^{+\infty}\bigl|\chi_{E_{x_\theta^\perp}}(u)-\chi_{E_{x_\theta^\perp}}(u+\rho)\bigr|\overline K_\tau(\rho)\d\rho\du\dx_\theta^\perp\d\theta\notag\\
			&=\frac12\int_{\S^{d-1}}\int_{\theta^\perp}\sum_{s\in\partial ^*E_{x_\theta^\perp}\cap [0,L)^d_{x_\theta^\perp}}\Bigl[\int_{s^-}^s\int_0^{+\infty}\bigl|\chi_{E_{x_\theta^\perp}}(u)-\chi_{E_{x_\theta^\perp}}(u+\rho)\bigr|\overline K_\tau(\rho)\d\rho\du\notag\\
			&+\int_{s}^{s^+}\int_{-\infty}^0\bigl|\chi_{E_{x_\theta^\perp}}(u)-\chi_{E_{x_\theta^\perp}}(u+\rho)\bigr|\overline K_\tau(\rho)\d\rho\du\Bigr]\dx_\theta^\perp\d\theta.\label{eq:nlintgeom}
		\end{align}
		
		Putting together~\eqref{eq:lintgeom} for $\Omega=[0,L)^d$ and~\eqref{eq:nlintgeom}, one obtains~\eqref{eq:forintgeom}.
		
		To show~\eqref{eq:igper}, which is a classical formula, one can  use  Fubini Theorem and the slicing formula~\eqref{eq:slicing} as follows
		\begin{align*}
			\per(E; \Omega)  &= \int_{\partial^* E\cap \Omega}\|\nu_E(x)\|\d\mathcal H^{d-1}(x)\\
			&=\frac{1}{C_{1,d}}\int_{\partial^* E\cap \Omega}\int_{\S^{d-1}}|\Scal{\nu_E(x)}{\theta}|\d\theta\d\mathcal H^{d-1}(x)\\
			&=\frac{1}{C_{1,d}}\int_{\S^{d-1}}\int_{\partial^* E\cap \Omega}|\Scal{\nu_E(x)}{\theta}|\d\mathcal H^{d-1}(x)\d\theta\\
			&=\frac{1}{C_{1,d}}\int_{\S^{d-1}} \int_{{\theta}^\perp} \per^{\mathrm{1D}}(E_{x^\perp_{\theta}}, \Omega_{x^\perp_{\theta}})\dx_{\theta}^{\perp}\d\theta.
		\end{align*}
		
	\end{proof}
	
	\begin{remark}\label{rmk:slicingJ} Let now $\tilde{\Fcal}_{J,p,d}$ be the functional defined in \eqref{eq:tildeFunctional} for $\Omega=[0,L)^d$ and $J>J_c$. Then, 
		\begin{equation}
			\tilde{\Fcal}_{J,p,d}(E,[0,L)^d)=\frac{(J-J_c)}{L^d} \per(E, [0,L)^d)+	\tilde{\Fcal}_{J_c,p,d}(E,[0,L)^d).
		\end{equation}
		By the same argument used in Proposition \ref{prop:intgeom}, where now $K_\tau$ is replaced by $K$ and $J_\tau$ is replaced by $J_c$ as in \eqref{eq:Jc}, one has the following  formula 
		\begin{align}
			\tilde{\Fcal}_{J,p,d}(E,[0,L)^d)=\frac{(J-J_c)}{L^d} \per(E, [0,L)^d)+\frac{1}{L^d}\int_{\S^{d-1}}\int_{\theta^\perp}\sum_{s\in\partial^*E_{x_\theta^\perp}\cap[0,L)^d_{x_\theta^\perp}}\tilde r(E_{x_\theta^\perp},s)\dx_\theta^\perp\d\theta,\label{eq:intgeomJ}
		\end{align}
		where we set
		\begin{align}
			\tilde{r}(E_{x_\theta^\perp},s)&:=\int_{-\infty}^{+\infty}|\rho|\overline K(\rho)\d\rho-\int_{s^-}^s\int_0^{+\infty}\bigl|\chi_{E_{x_\theta^\perp}}(u)-\chi_{E_{x_\theta^\perp}}(u+\rho)\bigr|\overline K(\rho)\d\rho\du\notag\\
			&-\int_{s}^{s^+}\int_{-\infty}^0\bigl|\chi_{E_{x_\theta^\perp}}(u)-\chi_{E_{x_\theta^\perp}}(u+\rho)\bigr|\overline K(\rho)\d\rho\du.\label{eq:rtilde}
		\end{align}
		
	\end{remark}
	
	\begin{remark}\label{rmk:floc}
		Given an open and bounded set $\Omega\subset\R^d$, by using the integral geometric formula \eqref{eq:forintgeom} one can define the following restriction of  $\Fcal_{\tau,p,d}$ to $\Omega$:
		
		\begin{align}\label{eq:floc}
			\Fcal^{\mathrm{loc}}_{\tau,p,d}(E,\Omega)=\frac{1}{|\Omega|}\int_{\S^{d-1}}\int_{\theta^\perp}\sum_{s\in\partial^* E_{x_\theta^\perp}\cap \Omega_{x_\theta^\perp}}r_{\tau}(E_{x_\theta^\perp},s)\dx_\theta^\perp\d\theta,
		\end{align} 
		
		Notice that, for $s\in\Omega_{x_\theta^\perp}\cap \partial^*E_{x_\theta^\perp}$, the nonlocal part of  $r_\tau(E_{x_\theta^\perp},s)$ measures interactions of $u\in(s^-,s)$ with $u+\rho$, $\rho\in(0,+\infty)$ and interactions of  $u\in(s,s^+)$ with $u+\rho$, $\rho\in(-\infty,0)$. Here $s^-$ and $s^+$ are not necessarily contained in $\Omega_{x_\theta^\perp}$.  Notice also that simply restricting $\Fcal_{\tau,p,d}$ to $\Omega$ by replacing $[0,L)^d$ with $\Omega$ in the definition of $\Fcal_{\tau,p,d}$ given in \eqref{eq:ftauintro} would possibly create sequences of sets with unbounded energy from below as $\tau\to0$ simply due to boundary effects (for example, if $\partial E\subset \R^d\setminus \Omega$). 
		
		Defining also the one-dimensional functionals
		\begin{equation}\label{eq:floc2}
			\bigl|\Omega_{x_\theta^\perp}\bigr|\overline F^{\mathrm{1D}}_{\tau,p,d}(E_{x_\theta^\perp}, \Omega_{x_\theta^\perp})=\sum_{s\in\partial E_{x_\theta^\perp}\cap \Omega_{x_\theta^\perp}}r_{\tau}(E_{x_\theta^\perp},s),
		\end{equation}
		formula~\eqref{eq:floc} can be rewritten as
		\begin{equation}\label{eq:fintgeom}
			\Fcal^{\mathrm{loc}}_{\tau,p,d}(E,\Omega)=\frac{1}{|\Omega|}\int_{\S^{d-1}}\int_{\theta^\perp}      |\Omega_{x_\theta^\perp}|\overline  F^{\mathrm{1D}}_{\tau,p,d}(E_{x_\theta^\perp},\Omega_{x_\theta^\perp} )\dx_\theta^\perp\d\theta.
		\end{equation}
		For every $x=x_\theta^\perp+x_\theta\theta\in\partial^* E$, let us also set for later convenience, with a slight abuse of notation,
		\begin{equation*}
			r_{\tau,\theta}(E,x):=r_{\tau}(E_{x_\theta^\perp},x_\theta).
		\end{equation*}
		Thus,
		\begin{equation*}
			\Fcal^{\mathrm{loc}}_{\tau,p,d}(E, \Omega)=\frac{1}{2|\Omega|}\int_{\S^{d-1}}\int_{\partial^* E\cap \Omega}|\langle \nu_E(x),\theta\rangle|r_{\tau,\theta}(E,x)\d\hd(x)\d\theta.
		\end{equation*}
		Notice that
		\[
		\overline F^{\mathrm{1D}}_{\tau,p,d}(E,[a,b))\neq\Fcal^{\mathrm{1D}}_{\tau,p,d}(E, [a,b)),
		\]
		where $\Fcal^{\mathrm{1D}}_{\tau,p,d}$ was defined in~\eqref{eq:1dFunc}, since the two kernels are different.
	\end{remark}
	
	\section{One dimensional estimates}\label{sec:1D}
	
	In this section we derive one dimensional estimates, namely estimates which depend only on one dimensional slices $E_{x_\theta^\perp}$ of $E$. We assume, also when not explicitly stated, that $p>d+1$.
	
	Recall the formula for $r_\tau$ given in~\eqref{eq:rme}
	\begin{align*}
		r_{\tau}(E_{x_\theta^\perp}, s)&:=\int_{-1}^1|\rho|\overline  K_\tau(\rho)\d\rho-\int_{s^-}^s\int_0^{+\infty}\bigl|\chi_{E_{x_\theta^\perp}}(u)-\chi_{E_{x_\theta^\perp}}(u+\rho)\bigr|\overline K_\tau(\rho)\d\rho\du\notag\\
		&-\int_{s}^{s^+}\int_{-\infty}^0\bigl|\chi_{E_{x_\theta^\perp}}(u)-\chi_{E_{x_\theta^\perp}}(u+\rho)\bigr|\overline  K_\tau(\rho)\d\rho\du.
	\end{align*}
	
	The following proposition contains one dimensional estimates analogous to those proved along the coordinate directions for the anisotropic functionals considered in~\cite{gr,dr_arma}.
	Also in this setting the proof is very similar, though we report it here for completeness and consistency of notation.
	
	\begin{proposition}\label{prop:1dbound}
		There exist $\gamma_0, \gamma_1>0$ depending only on  $p,d$ such that
		\begin{align}\label{eq:stima1d}
			r_\tau(E_{x_\theta^\perp},s)&\geq -\gamma_0+\gamma_1\min\bigl\{|s-s^-|^{-(p-d-1)},\tau^{-1}\bigr\}+\gamma_1\min\bigl\{|s-s^+|^{-(p-d-1)},\tau^{-1}\bigr\}.
		\end{align}
		In particular, whenever $p>d+1$  there exist $\tau_0,\eta_0>0$ such that whenever $\min\{|s-s^-|,|s-s^+|\}<\eta_0$ and $\tau<\tau_0$, then {$ r_{\tau}(E_{x_\theta^\perp},s)\geq \frac{\gamma_1}{2}\min\bigl\{|s-s^-|^{-(p-d-1)},|s-s^+|^{-(p-d-1)}, \tau^{-1}\bigr\}$.}
	\end{proposition}
	
	\begin{proof}
		By the estimates
		\begin{align}
			\int_{s^{-}}^{s} \bigl|\chi_{E_{x_\theta^\perp}}(u+\rho) - \chi_{E_{x_\theta^\perp}}(u)\bigr| \du& \leq \min\{\rho, s- s^{-}\},\quad\forall\,\rho>0,\label{eq:nl1}\\
			\int_{s}^{s^+} \bigl|\chi_{E_{x_\theta^\perp}}(u+\rho) - \chi_{E_{x_\theta^\perp}}(u)\bigr| \du &\leq \min\{-\rho, s^{+} - s\}, \quad\forall\,\rho\leq0,\label{eq:nl2}
		\end{align}
		(see~\cite{gr, dr_arma}) one finds that
		
		\begin{align*}
			\int_{0}^{1} \rho \overline{K}_{\tau}(\rho) \d\rho &- \int_{s^{-}}^{s}\int_{0}^{+\infty }\bigl|\chi_{E_{x_\theta^\perp}}(u + \rho) - \chi_{E_{x_\theta^\perp}}(u)\bigr| \overline{K}_{\tau}(\rho) \d\rho\du \notag\\&\geq \int_{0}^{1}\rho \overline{K}_\tau(\rho)\d\rho - \int_{0}^{+\infty} \min\{|s-s^{-}|, \rho\} \overline{K}_{\tau}(\rho)\d\rho\notag\\
			& = \int_{|s-s^{-}|}^{1} \bigl(\rho-|s-s^{-}|\bigr) \overline{K}_{\tau}(\rho)\d\rho - \int_{1}^{+\infty } |s-s^{-}| \overline{K}_{\tau}(\rho)\d\rho
		\end{align*}
		and analogously
		\begin{align*}
			\int_{-1}^{0} |\rho| \overline{K}_{\tau}(\rho) \d\rho &- \int_{s}^{s^+}\int_{-\infty}^{0 }\bigl|\chi_{E_{x_\theta^\perp}}(u + \rho) - \chi_{E_{x_\theta^\perp}}(u)\bigr| \overline {K}_{\tau}(\rho) \d\rho\du \\&\geq\int_{|s-s^{+}|}^1 \bigl(\rho-|s-s^{+}|\bigr) \overline{K}_{\tau}(\rho)\d\rho - \int_{1}^{+\infty } |s-s^{+}| \overline{K}_{\tau}(\rho)\d\rho.
		\end{align*}
		
		In particular, by property~\eqref{eq:kass1} and $p>d+1$, implying  in particular the uniform integrability of  $\rho\overline K_\tau(\rho)$ on the interval $[1,+\infty)$, there exist $\gamma_0, \gamma_1>0$ depending only on  $p,d$ such that
		\begin{align*}
			r_\tau(E_{x_\theta^\perp},s)&\geq\int_{|s-s^{-}|}^{+\infty} \bigl(\rho-|s-s^{-}|\bigr) \overline{K}_{\tau}(\rho)\d\rho \notag\\
			&+\int_{|s-s^{+}|}^{+\infty} (\rho-|s-s^{+}|) \overline{K}_{\tau}(\rho)\d\rho - 2\int_{1}^{+\infty } \rho\overline {K}_{\tau}(\rho)\d\rho\notag\\
			&\geq -\gamma_0+\gamma_1\min\bigl\{|s-s^-|^{-(p-d-1)},\tau^{-1}\bigr\}+\gamma_1\min\bigl\{|s-s^+|^{-(p-d-1)},\tau^{-1}\bigr\}.
		\end{align*}
		
		The last statement in the proposition follows immediately from the fact that $p>d+1$.

	\end{proof}
	
	A first consequence of the estimates used in the proof of Proposition \ref{prop:1dbound} is the following 
	
	\begin{theorem}\label{thm:jc}
		Let $\tilde{\Fcal}_{J,p,d}$ be the functional defined in \eqref{eq:tildeFunctional}. For $J\geq J_c$ and $\Omega=[0,L)^d$, $\tilde{\Fcal}_{J,p,d}\geq0$ and $\tilde{\Fcal}_{J,p,d}(E,[0,L)^d)=0$ if and only if $E=\emptyset$ or $E=\R^d$. 
	\end{theorem}
	
	\begin{proof}
		Recall the formula \eqref{eq:intgeomJ}. By the estimates \eqref{eq:nl1} and \eqref{eq:nl2}, one has that 
		\begin{align}
			\tilde{r}(E_{x_\theta^\perp},s)&\geq \int_{0}^{+\infty}\rho\overline{K}(\rho)\d\rho-\int_0^{+\infty}\min\{|s-s^-|,\rho\}\overline{K}(\rho)\d\rho\notag\\
			&+\int_{-\infty}^{0}|\rho|\overline{K}(\rho)\d\rho-\int_{-\infty}^{0}\min\{|s-s^+|,\rho\}\overline{K}(\rho)\d\rho\notag\\
			&\geq 0,
		\end{align}
		thus implying that $\tilde{\Fcal}_{J,p,d}(\cdot,[0,L)^d)\geq0$ for all $J\geq J_c$. 
		Clearly, whenever $E=\emptyset$ or $E=\R^d$, then  $\tilde{\Fcal}_{J,p,d}(E,[0,L)^d)=0$. The converse holds as well, since $\tilde{\Fcal}_{J,p,d}(E,[0,L)^d)=0$ implies by the above that $\per(E,[0,L)^d)=0$. 
	\end{proof}

	As an immediate consequence of Proposition \ref{prop:1dbound}  we show as in~\cite{gr,dr_arma} that the equiboundedness of the function $r_\tau$ on a family of sets of locally finite perimeter in $\R$ and on the set of their boundary points which are contained in a fixed  open bounded set implies compactness of the sets in the $L^1$ topology and convergence to a set of locally finite perimeter.
	Moreover, a stronger type of convergence holds, namely the sets $\partial^*E_\tau$ converge in the Hausdorff topology.
	This will be a fundamental ingredient of the proof of the $\Gamma$-convergence {Theorem~\ref{thm:gammaconv}}.
	The lemma below is the analogue of~\cite[Lemma 7.5]{dr_arma} and we report the proof here for completeness.
	In the proof of {Theorem~\ref{thm:gammaconv}} it will be applied to the one dimensional slices of $E$ in direction $\theta\in\S^{d-1}$.
	\begin{lemma}[Compactness]\label{lemma:comp1d}
		Let $p>d+1$, let $\{E_\tau\}_{\tau>0}\subset\R$ be a family of sets of locally finite perimeter and let $I\subset \R$ be a bounded  open interval. If
		\begin{equation}\limsup_{\tau\to0}\sum_{s\in\partial^* E_{\tau}\cap I} r_{\tau}(E_{\tau},s)<+\infty,\label{eq:emls}
		\end{equation}
		then there exists $E_0\subset\R$ of finite perimeter in $I$, such that, up to subsequences, $E_{\tau}\to E_0$ in $L^1( I)$.  Moreover, if  $\{s_1^0,\dots,s_{m(0)}^0\}=\partial^* E_0\cap I$, then 
		\begin{equation}\label{eq:1dlowbound}
			\liminf_{\tau\downarrow 0}\sum_{s\in\partial^* E_\tau\cap I} r_\tau(E_\tau,s)\geq\sum_{i=1}^{m(0)}\Big(-\gamma_0+\gamma_1|s_{i}^0-(s_{i}^0)^+|^{-(p-d-1)}+\gamma_1|s_{i}^0-(s_{i}^0)^-|^{-(p-d-1)}\Big),
		\end{equation}
		where $\gamma_0$, $\gamma_1$ are constants  as in~\eqref{eq:stima1d}.
	\end{lemma}
	
	\begin{remark}
		Before diving into the proof of Lemma \ref{lemma:comp1d}, let us be more precise on the meaning of $(s_1^0)^-$ and $(s_{m(0)}^0)^+$ in \eqref{eq:1dlowbound}. A more precise statement of the above lemma, that we give here in order to avoid too heavy notation, is the following.   Let $p>d+1$, let $\{E_\tau\}_{\tau>0}\subset\R$ be a family of sets of locally finite perimeter and let $I\subset \R$ be a bounded  open interval such that 
		\begin{equation}\limsup_{\tau\to0}\sum_{s\in\partial^* E_{\tau}\cap I} r_{\tau}(E_{\tau},s)<+\infty.\label{eq:emls2}
		\end{equation}
		Let $\{s^\tau_{1},\ldots,s^\tau_{m({\tau})}\} = \partial^* E_{\tau}\cap I$ and $s^\tau_0=(s^\tau_{1})^-,\,s^\tau_{m(\tau)+1}=(s^\tau_{m(\tau)})^+\in\partial^*E_\tau\setminus I$. Let $\tilde E_\tau\subset\R$ of locally finite perimeter such that $\partial^*\tilde E_\tau=\{s^\tau_0,\dots,s^\tau_{m(\tau)+1}\}$. Then, $\tilde E_\tau\to\tilde E_0$ in $L^1_{\mathrm{loc}}(\R)$, where $\tilde E_0$ is a set of locally finite perimeter with $\partial^*E_0=\{s^0_0,\dots,s^0_{m(0)+1}\}$ and 
		\begin{equation}\label{eq:1dlowbound2}
			\liminf_{\tau\downarrow 0}\sum_{s\in\partial^* E_\tau\cap I} r_\tau(E_\tau,s)\geq\sum_{i=0}^{m(0)}\Big(-\gamma_0+\gamma_1|s_{i}^0-s_{i+1}^0|^{-(p-d-1)}\Big).
		\end{equation}
	\end{remark}
	
	\begin{proof}
		
		Let us denote by  $\{s^\tau_{1},\ldots,s^\tau_{m({\tau})}\} = \partial^* E_{\tau}\cap I$.
		We will also denote by
		\begin{equation*}
			s^{\tau}_0 = \sup\bigl\{ s\in \partial^* E_\tau: s < s^\tau_1\bigr\} \qquad\text{and}\qquad
			s^\tau_{m(\tau)+1} = \inf\bigl\{ s\in \partial^* E_\tau: s > s^\tau_{m(\tau)}\bigr\} .
		\end{equation*}
		Denote by $A$ the \rhs of~\eqref{eq:emls}.
		From~\eqref{eq:stima1d}, one has that for $i=1,\dots,m(\tau)$ it holds $r_{\tau}(E_\tau,s^{\tau}_{i}) \geq -\gamma_0  + \gamma_1 \min\bigl\{|s^{\tau}_{i} - s^{\tau}_{i+1} |^{-(p-d-1)},\tau^{-1}\bigr\}  + \gamma_1 \min\bigl\{|s^{\tau}_{i} - s^{\tau}_{i-1} |^{-(p-d-1)},\tau^{-1}\bigr\}$.
		Thus, by the last statement of Proposition~\ref{prop:1dbound}, there exist $\tilde\eta$ and $\tilde{\tau}> 0$ such that for every $\tau <\tilde{\tau}$, whenever
		\begin{equation*}
			\begin{split}
				\min_{i\in \{1,\ldots,{m(\tau)}\}}\min\{|s^{\tau}_{i+1}- s^{\tau}_{i}|,|s^{\tau}_{i-1}- s^{\tau}_{i}| \}< \tilde \eta
			\end{split}
		\end{equation*}
		then
		\begin{equation*}
			\sum_{s\in \partial^* E_{\tau}\cap I}r_{\tau}(E_{\tau},s) \geq A.
		\end{equation*}
		Hence, assume there exists a subsequence $\tau_{k}$ such that $|s^{\tau_k}_{i+1}- s^{\tau_{k}}_{i}| >\tilde  \eta$  and $|s^{\tau_k}_{i-1}- s^{\tau_{k}}_{i}| >\tilde  \eta$ for all $i=1,\dots,m({\tau_k})$.
		Up to  relabelling, let us assume that it holds true  for the whole sequence of $E_\tau$.
		Let $\tilde E_\tau\subset\R$ of locally finite perimeter such that $\partial^*\tilde E_\tau=\{s^\tau_0,\dots,s^\tau_{m(\tau)+1}\}$. Then, by the above $\tilde E_\tau\to\tilde E_0$ in $L^1_{\mathrm{loc}}(\R)$, where $\tilde E_0$ is a set of locally finite perimeter with $\partial^*E_0=\{s^0_0,\dots,s^0_{m(0)+1}\}$. 
		Moreover, since $\min_{i} | s^{\tau}_{i+1} - s^\tau_{i} | > \tilde\eta$ and  $\min_{i} | s^{\tau}_{i-1} - s^\tau_{i} | > \tilde\eta$, the convergence $\tilde E_{\tau}\to \tilde E_{0}$ in  $L^1_{\mathrm{loc}}(\R)$ can be upgraded to the Hausdorff convergence of the reduced boundaries, namely one has that there exists a $\tilde{\tau}$ such that for $\tau<\tilde{\tau}$, it holds $\#\partial^* \tilde E_{\tau} = \#\partial^* \tilde E_{0}=m(0)+2$ and $s^{\tau}_{i} \to s^0_{i}$.
		
		Then because of the convergence of the boundaries, we have that
		\begin{align*}
			\liminf_{\tau\downarrow 0}\sum_{s\in \partial ^*E_{\tau}\cap I} r_{\tau}(E_{\tau},s) &\geq \liminf_{\tau\downarrow 0}\sum_{j=0}^{m(0)} \big( - \gamma_0+ \gamma_1\min\bigl\{|s^{\tau}_{i} - s^\tau_{i+1}|^{-(p-d-1)},\tau^{-1}\bigr\} \big)\notag
			\\ & \geq\sum_{j=0}^{m(0)} \big( - \gamma_0 + \gamma_1|s^{0}_{i} - s^0_{i+1}|^{-(p-d-1)}\big).  		\label{eq:gstr13}
		\end{align*}
		
	\end{proof}

	From the above one  can indeed show the following quantitative perimeter bounds.
	
	\begin{corollary}\label{cor:fboundsp}
		Let $p>d+1$. There exists $0<\tau_1\ll1$ such that the following holds.
		\begin{itemize}
			\item[(i)] Let $I=(a,b)\subset \R$ and let $E\subset\R$ be a set of locally finite perimeter such that
			\[
			\sum_{s\in\partial^*E\cap I}{r}_\tau(E,s)\leq A<+\infty,\quad\text{ for some $0<\tau<\tau_1$}.
			\] Then, there exists a constant $\eta_1>0$  such that
			\begin{equation}\label{eq:4.9}
				\per(E,I)\leq (1+\gamma_0)\frac{b-a}{\eta_1}+A,
			\end{equation}
			where $\gamma_0$ is as in the one-dimensional estimate~\eqref{eq:stima1d}.
			
			\item[(ii)] Let $E \subset \R^{d}$ be a set of locally finite perimeter.
			Assume that
			\begin{equation*}
				\Fcal^{\mathrm{loc}}_{\tau,p,d}(E, \Omega)<+\infty, \quad \text{ for some $0<\tau<\tau_1$}.
			\end{equation*}
			Then,
			\begin{equation}
				\per(E,\Omega) \lesssim |\Omega| \Bigl(1+\Fcal^{\mathrm{loc}}_{\tau,p,d}(E, \Omega)\Bigr).
			\end{equation}
		\end{itemize}
		
	\end{corollary}
	
	\begin{proof}
		$(i)$.  From Proposition~\ref{prop:1dbound}, there exist $\eta_1,\tau_1>0$ such that, whenever $s\in\partial^*E\cap I$ satisfies $\min\bigl\{|s-s^-|, |s-s^+|\bigr\}<\eta_1$ and $0<\tau<\tau_1$, then $r_\tau(E,s)>1$. This implies that there are at most $\frac{b-a}{\eta_1}$ points $s\in\partial^* E\cap I$ where $r_\tau(E,s)<1$. Given that, by Proposition~\ref{prop:1dbound}, $ r_\tau(E,\cdot)\geq -\gamma_0$ for some universal constant $\gamma_0$ independent of $\tau$, it follows immediately that
		\begin{equation*}
			\sum_{s\in\partial^*E\cap I}r_\tau(E,s)\geq -\gamma_0\frac{b-a}{\eta_1}+\max\Bigl\{\per(E,I)-\frac{b-a}{\eta_1},0\Bigr\}.
		\end{equation*}
		Hence, as claimed,~\eqref{eq:4.9} holds.
		
		$(ii)$. For the proof of the second statement, we use~\eqref{eq:igper}, $(i)$ and the integral geometric formulation~\eqref{eq:forintgeom}. Indeed,  one has that for $\tau<\tau_1$
		\begin{align*}
			\per(E,\Omega)&=\frac{1}{C_{1,d}}\int_{\S^{d-1}}  \int_{\theta^\perp}\per^{\mathrm{1D}}(E_{x_\theta^\perp},\Omega_{x_\theta^\perp})\dx_\theta^\perp\d\theta \notag\\
			&\leq \frac{1}{C_{1,d}}\int_{\S^{d-1}} \int_{\theta^\perp}\frac{(1+\gamma_0)}{\eta_1}\Bigl[\bigl|\Omega_{x_\theta^\perp}\bigr|+ \sum_{s\in\partial ^*E_{x_\theta^\perp}\cap \Omega_{x_\theta^\perp}} r_{\tau}(E_{x_\theta^\perp}, s)\Bigr]\dx_\theta^\perp\d\theta\notag\\
			&\lesssim |\Omega|\Bigl(1+\Fcal^{\mathrm{loc}}_{\tau,p,d}(E, \Omega)\Bigr).
		\end{align*}

	\end{proof}
	
	Let us now consider the rescaled kernel
	\begin{equation*}
		K_{\tau}(\zeta) := \frac{1}{\max(\tau^{p/(p-d-1)},\|\zeta\|^{p})}.
	\end{equation*}

	In the following we are concerned with the periodicity of optimal stripes for the corresponding functional $\Fcal_{\tau,p,d}$ introduced in~\eqref{eq:ftauintro}. This is a simple extension of the periodicity of one dimensional minimizers for functionals with reflection positive kernels, given in~\cite{giuliani2006ising,giuliani2009periodic}.

	\begin{definition}
		Given $\theta \in \S^{d-1}$, let us denote by $\mathscr{S}_{\theta,L}$ the family of $[0,L)^d$-periodic sets $E$ composed of finitely many stripes such that $\nu_{E}(x) \in \insieme{\pm \theta}$, for all $x\in\partial E$.
		We will say that $E$ is a simple periodic set if up to rotations and translations there exists $\widehat E \subset \R$, $h>0$  such that $\widehat E = \bigcup_{k\in\Z}[2kh, (2k+1)h)$
		and $E = \widehat E \times \R^{d-1}$.
	\end{definition}
	
	The following periodicity result holds.
	
	\begin{proposition}
		\label{prop:almost_reflection_positivity}
		For every $L>0$ and every $\theta\in\S^{d-1}$ such that $\mathscr{S}_{\theta,L}\neq\emptyset$, let us consider $E \in \argmin_{ \mathscr{S}_{\theta,L}}\Fcal_{\tau, p,d}(\cdot,[0,L)^d)$. Then $E$ is a simple periodic set.
	\end{proposition}
	
	\begin{proof}
		\WithoutLoss, we can assume that $\theta = e_{d}$.
		Thus, there exists $\widehat E \subset  \R$ $(0, L)$-periodic, such that $E =  \R^{d-1}\times \widehat E$  and $\widehat E \cap[0,L)= \bigcup_{k\in\Z}(x_{2k}, x_{2k+1})$.
		Denote by
		\begin{equation*}
			\widehat{K}_{\tau}(t) := \int_{\R^{d-1}} K_\tau(\zeta^{\perp}_{d}+t e_{d}) \d\zeta^{\perp}_{d}.
		\end{equation*}
		For any $\alpha,\beta>0$, let
		\begin{align}
			\widehat J_{\tau,\alpha}&=\int_{\{|\zeta_d|\leq \alpha\}}|\zeta_d|K_\tau(\zeta)\d\zeta=\int_{-\alpha}^{\alpha}|t|\widehat K_\tau(t)\dt\notag\\
			J_{\tau,\beta}&=\int_{\{\|\zeta\|\leq \beta\}}|\zeta_d|K_\tau(\zeta)\d\zeta.
		\end{align}
		In particular, $J_{\tau,1}=J_\tau$, where $J_\tau$ was defined in~\eqref{eq:jtau} and $\widehat J_{\tau,\alpha}>J_{\tau,\alpha}$ for every $\alpha>0$.
		Moreover, by the integrability of the kernel and the fact that $K_\tau(\zeta)=\|\zeta\|^{-p}$ whenever $\tau^{1/(p-d-1)}<\|\zeta\|$, it is not difficult to see that there exists $\gamma>0$ such that $J_{\tau,1}=\widehat J_{\tau, 1-\gamma}$ for all $\tau$ sufficiently small.
		
		When restricting the functional to $\mathscr{S}_{e_{d},L}$, we have that
		\begin{equation*}
			\Fcal_{\tau, p,d}(E,[0,L)^d) =  \Fcal^{\mathrm{1D}}_{\tau,p,d}(\widehat E,[0,L))
		\end{equation*}
		with
		\begin{equation*}
			\begin{split}
				L\Fcal^{\mathrm{1D}}_{\tau,p,d}(\widehat E,[0,L)) := \int_{-1+\gamma}^{1-\gamma} |t| \widehat{K}_{\tau}(t) \per^{\mathrm{1D}}(\widehat E; [0,L)) &- \int_{0}^{L}\int_{0}^{+\infty } |\chi_{\widehat E}(s + t)-\chi_{\widehat E}(s)|\widehat{K}_{\tau}(t) \dt\ds\\ &-
				\int_{0}^{L}\int_{-\infty }^{0} |\chi_{\widehat E}(s + t)-\chi_{\widehat E}(s)|\widehat{K}_{\tau}(t) \dt\ds.
			\end{split}
		\end{equation*}
		It is well-known that if $\widehat{K}_{\tau}$ is reflection positive (\ie it is the Laplace transform of a nonnegative function), any $\widehat E$ as above is simple periodic (see~\cite{giuliani2006ising,giuliani2009periodic}).
		Given that for $\tau=0$ the kernel $\widehat K_{\tau}$ is not integrable at the origin, it is convenient to write the functional as
		\begin{align*}
			L\Fcal^{\mathrm{1D}}_{\tau,p,d}(\widehat E,[0,L))& := \int_{-1+\gamma}^{1-\gamma} \bigg(
			|t| \per(\widehat E; [0,L))
			- \int_{0}^{L}|\chi_{\widehat E}(s + t)-\chi_{\widehat E}(s)| \ds
			\bigg)\widehat{K}_{\tau}(t) \dt\\
			&-
			\int_{0}^{L}\int_{\R \setminus (-1+\gamma,1-\gamma) } |\chi_{\widehat E}(s + t)-\chi_{\widehat E}(s)|\widehat{K}_{\tau}(t) \dt\ds
		\end{align*}

		If $E\in\mathscr{S}_{e_d,L}$ is a minimizer, applying  Proposition~\ref{prop:1dbound} to the analogous functional in which $\pm1\mp\gamma$ is substituted by $\pm1$, there exist $\tau_{0}, \eta_{0} >0$ such that whenever $0<\tau<\tau_{0}$, we have that
		\begin{equation}\label{eq:xmin}
			\min\{|x_{i}-x_{j}|:\ i\neq j, \,x_i,x_j\in\partial^* \widehat E\} > \eta_{0}.
		\end{equation}
		Moreover,
		\begin{equation*}
			\begin{split}
				\int_{0}^{L}\int_{0}^{+\infty} |\chi_{\widehat E}(s+t) - \chi_{\widehat E}(s)| \widehat K_{\tau}(t)\dt\ds  =&
				\sum_{k}\int_{x_{k}}^{x_{k+1}} \int_{x_{k+1}}^{+\infty} \widehat{K}_{\tau}(s-t)\ds\dt\\
				&-
				\sum_{k}\sum_{l \geq 1}\int_{x_{k}}^{x_{k+1}} \int_{x_{k+2l}}^{x_{k+2l +1}} \widehat{K}_{\tau}(s-t)\ds\dt.
			\end{split}
		\end{equation*}
		
		The first term on the \rhs of the above formula can be computed explicitly as
		\begin{equation*}
			\int_{x_{k}}^{x_{k+1}} \int_{x_{k+1}}^{+\infty} \widehat{K}_{\tau}(s-t)\ds\dt =\int_{0}^{+\infty }\min\bigl\{|x_{k} - x_{k+1}|, t\bigr\} \widehat K_{\tau}(t)\dt.
		\end{equation*}
		
		Given  that~\eqref{eq:xmin} holds and for $0<\tau < \eta_{0}<1$, it holds $\widehat K_{\tau}(t) = \widehat K_{0}(t)$ whenever $t > \eta_{0}$, we have that

		\begin{equation*}
			\begin{split}
				\int_{0}^{1-\gamma} t \widehat K_{\tau} - \int_{x_{k}}^{x_{k+1}} \int_{x_{k+1}}^{+\infty} \widehat{K}_{\tau}(s-t)\ds\dt &=
				\int_{0}^{1-\gamma}  \big(t - \min\bigl\{|x_{k} - x_{k+1}|, t\bigr\}\big)\widehat K_{\tau}(t)\dt\\
				&= \int_{0}^{1-\gamma}  \big(t - \min\bigl\{|x_{k} - x_{k+1}|, t\bigr\}\big) \widehat K_{0}(t)\dt\\
			\end{split}
		\end{equation*}
		and that
		\begin{equation*}
			\int_{x_{k}}^{x_{k+1}} \int_{x_{k+2l}}^{x_{k+2l +1}} \widehat{K}_{\tau}(s-t)\ds\dt =
			\int_{x_{k}}^{x_{k+1}} \int_{x_{k+2l}}^{x_{k+2l +1}} \widehat{K}_{0}(s-t)\ds\dt.
		\end{equation*}

		Thus, whenever $\widehat E$  is such that $\partial^* \widehat E = \bigcup_{k\in\Z}\{x_{k}\}$ with $\min\bigl\{|x_{i}-x_{j}|:\,i\neq j,\, x_i,x_j\in\partial ^*\widehat E\bigr\} > \eta_{0}$ we have that $\Fcal^{\mathrm{1D}}_{\tau,p,d}(E,[0,L)) = \Fcal^{\mathrm{1D}}_{0,p,d}(E,[0,L))$.
		To conclude, it is sufficient to notice that for $\tau =0$ the kernel $\widehat K_0$ is reflection positive, thus the only minimizers $\widehat E$ are simple periodic sets.
		
	\end{proof}
	
	By Proposition~\ref{prop:almost_reflection_positivity} one can define, for every $L>0$ and for $\tau$ sufficiently small,
	\begin{equation}\label{eq:hL}
		h^*_L=\min_{E \text{ $L$-periodic}}\Fcal_\tau^{\mathrm{1D}}(E, [0,L))
	\end{equation}
	and
	\begin{equation}
		h^*=\min_{L>0}\min_{E \text{ $L$-periodic}}\Fcal_\tau^{\mathrm{1D}}(E, [0,L)).
	\end{equation}

	The following one dimensional Lemma will be used in the proof of Theorem~\ref{thm:main} in Section~\ref{sec:positive_tau}.
	
	\begin{lemma}
		\label{lemma:1d_replacement}
		Let $\eta_0,\tau_0$ be as in Proposition~\ref{prop:1dbound}. There exists $M_0>0$ such that for all $0<\tau<\tau_0$, $h>0$, $E \subset \R$ set of locally finite perimeter,  $A_{0} \subset  \partial^* E \cap [0,h)$ and $A_{1} = \{s \in \partial^* E\cap [0,h):\,\min\{|s-s^-|, |s-s^+|\}<\eta_0\}$,
		whenever $F$ is a set of finite perimeter  such that $\partial^* F \cap[0,h)= \partial^* E\cap [0,h) \setminus (A_{0}\cup A_{1})$, then
		\begin{equation}\label{eq:feineq}
			\sum_{s \in \partial^* F \cap [0,h)} r_{\tau}(F, s) \leq \sum_{s \in \partial^* E \cap [0,h)} r_{\tau}(E, s) + M_0 \# (A_0\cup A_1).
		\end{equation}
	\end{lemma}
	
	\begin{proof}
		By Proposition~\ref{prop:1dbound}, there exist $\tau_0>0$,  $\eta_{0}>0$ such that whenever $r_{\tau}(E, s) <0$ it holds $\min\bigl\{|s-s^{+}|, |s-s^{-}|\bigr\} > \eta_{0}$.
		We say that a set $C \subset \partial^* E$ is connected if there are  no points $s \in A_{0}\cup  A_{1}$  and $x, y \in C$ such that $x< s<y$.
		Let $C_{1}, C_{2}, \ldots, C_{N}$ be the connected components of $\partial^* F$.
		It is immediate to notice that $\partial ^*F = \bigcup_{i} C_{i}$.
		Consider $C_{1}:=\{s_{1}, \ldots, s_m \}$ and denote by $s_{m+1} = \inf\{s \in \partial ^*E: s > s_{m}\}$ and by $s_{0} :=\sup \{s \in \partial ^*E: s < s_{1}\}$.
		From the above we have that $|s_{0} - s_{1}|, |s_{m}- s_{m+1}| > \eta_{0}$.
		Moreover, for any $x,y \in (s_{0}, s_{m+1})$ we have that $|\chi_{E}(x) - \chi_{E}(y)| = |\chi_{F}(x) - \chi_{F}(y)|$.
		From the definition of $r_{\tau}$ we have that
		\begin{equation*}
			\Big|\sum_{s \in C_{1}} r_{\tau}(E, s) - r_{\tau} (F, s) \Big|\leq 2\int_{s_{0}}^{s_{m}}\int_{s_{m+1}}^{+\infty } K_{\tau}(x-y) \dx\dy
			+ 2\int_{s_{1}}^{s_{m+1}}\int_{-\infty }^{s_{0}} K_{\tau}(x-y)\dx\dy.
		\end{equation*}
		Because of the integrability of $K_{\tau}$, we have that the \rhs of the above is bounded by a constant $\bar{C}$ depending only  on $\eta_{0}$.
		Thus, by the fact that $r_\tau(E,s)\geq -\gamma_0$ for all $s\in\partial^*E$ (see Proposition~\ref{prop:1dbound}) and since $N\leq \#(A_0\cup A_1)$,  we have that
		\begin{align*}
			\sum_{s \in \partial ^*E\cap[0,h)}r_{\tau}(E,s) -
			\sum_{s \in \partial^* F\cap[0,h)}r_{\tau}(F,s)
			&\geq
			\sum_{i} \Biggl(\sum_{s \in C_{i}}r_{\tau}(E,s) -
			\sum_{s \in C_{i}}r_{\tau}(F,s) \Biggr)+ \sum_{s \in A_{0} \cup  A_{1}} r_{\tau}(E, s) \\
			&\geq -\bar CN -\gamma_0\#\big(A_{0}\cup A_{1}\big)\notag\\
			&\geq -(\bar C+\gamma_0)\#\big(A_{0}\cup A_{1}\big).
		\end{align*}
		Finally, to obtain  the desired claim~\eqref{eq:feineq}, it is sufficient to take $M_0 \geq \bar C+\gamma_0$.
	\end{proof}
	
	The next lemma is similar to~\cite[Lemma 7.7]{dr_arma}. It says that given a set of locally finite perimeter $F\subset \R$, and an interval $I\subset \R$ such that the minimal distance between points in $\partial ^*F\cap I$  is bounded from below, then the contribution  to the energy $\sum_{s\in \partial^* F\cap I}r_{\tau}(F,s)$, for $\tau$ sufficiently small, is comparable to the periodic case up to a constant $C$ depending only on the dimension and on the minimal distance between points in  $\partial^* F\cap I$.
	Among periodic sets $F$, the energy contribution on an interval $I$ is bigger than or equal to the contribution of periodic stripes with the same number of boundary points of $F$ in $I$.
	\begin{lemma}
		\label{lemma:1D-optimization}
		Let $\eta_0,\tau_0$ be as in Proposition~\ref{prop:1dbound}. There exists a constant $C=C(\eta_0)$ such that for all $0<\tau<\tau_0$ the following holds.
		Let $F$ of locally finite perimeter, $I\subset\R$ bounded open interval, $\{k_1,\dots,k_m\}=\partial^*F\cap I$. Assume that
		\begin{align}
			\label{eq:kij}
			&\inf_{i,j\in\{1,\dots,m\}}|k_i-k_j|>\eta_0,\\
			&\mathrm{dist}(k_1,\partial^*F\setminus I)>\eta_0,\label{eq:bark0}\\
			&\mathrm{dist}(k_m,\partial^*F\setminus I)>\eta_0,\label{eq:barkm}\\
		\end{align}
		and let $k_0,k_{m+1},k_{m+2}\in\R\setminus I$ such that $k_0<k_1<\dots<k_{m}<k_{m+1}<k_{m+2}$ and
		\begin{equation}
			\inf_{i,j\in\{0,\dots,m+2\}}|k_i-k_j|>\eta_0.
		\end{equation}
		Then, let $\tilde I=[k_0,k_{\mathrm{max}})$, with
		\begin{equation}
			k_{\mathrm{max}}=\left\{\begin{aligned}
				&k_{m+1} && &\text{if $m$ is odd}\\
				&k_{m+2} && &\text{if $m$ is even}
			\end{aligned}\right.
		\end{equation}
		and $\tilde F$ the $|\tilde I|$-periodic set of locally finite perimeter such that
		\begin{align}
			\tilde F\cap I&=F\cap I\\
			\partial^*\tilde F\cap \overline{\tilde I}&=\{k_0,k_1,\dots,k_{\mathrm{max}-1},k_{\mathrm{max}}\}.
		\end{align}
		Then,
		\begin{equation}\label{eq:stimafper}
			\sum_{s\in\partial^*F\cap I}r_{\tau}(F,s)\geq\sum_{s\in\partial^*\tilde F\cap \tilde I}r_{\tau}(\tilde F,s)-C.
		\end{equation}
		Moreover,
		\begin{equation}\label{eq:stimaSper}
			\sum_{s\in\partial^*\tilde F\cap\tilde I}r_\tau(\tilde F,s)\geq|\tilde I|\overline F^{\mathrm{1D}}_{\tau,p,d}(\tilde S,\tilde I),
		\end{equation}
		where $\tilde S$ is the simple periodic set of period $2\frac{|\tilde I|}{\#\partial^*\tilde F\cap\overline{\tilde I}}$.
	\end{lemma}

	\begin{proof}
		Let us denote by $k_1< \ldots< k_m $ the  points of $\partial^*F\cap I$, and let $k_0,k_{m+1},k_{m+2}$, $\tilde I$, $\tilde F$ be as in the statement of the lemma.
		
		To prove~\eqref{eq:stimaSper}, observe that the boundary points of $\tilde F$ are, by assumption, at mutual distance larger that $\eta_0>2\tau^{1/(p-d-1)}$. Thus,
		$K_\tau (z)=\|z\|^{-p}$ behaves like a reflection positive functional (similarly to the proof of Proposition~\ref{prop:almost_reflection_positivity} ). Therefore, as shown in~\cite{giuliani2009modulated} for reflection positive functionals,~\eqref{eq:stimaSper} holds.

		To show~\eqref{eq:stimafper}, notice that the symmetric difference between $F$ and $\tilde F$ satisfies
		
		\begin{equation*}
			F \Delta \tilde F  \subset (- \infty, k_1 - \eta_0) \cup (k_m+\eta_0, + \infty).
		\end{equation*}
		Let us assume \withoutLoss that $k_{\mathrm{max}}=k_{m+2}$ and denote by
		\[
		\bar k_0=\sup\{s\in\partial^*F\setminus I: \, s<k_1\}, \quad 	\bar k_{m+1}=\inf\{s\in\partial^*F\setminus I:\,s>k_m\}.
		\]
		By assumptions~\eqref{eq:bark0} and~\eqref{eq:barkm}, $|\bar k_0-k_1|>\eta_0$, $|\bar k_{m+1}-k_m|>\eta_0$.
		Using the fact that $r_\tau(E,s)<C/6$ whenever $\min\{|s-s^-|,|s-s^+|\}>\eta_0$ (see Proposition~\ref{prop:1dbound}), we have that
		\begin{align}
			\Bigl|\sum_{i=1}^mr_\tau(F,k_i)-\sum_{i=0}^{m+2}r_\tau(\tilde F, k_i)\Bigl|&\leq \frac{C}{2}	+\Bigl|\sum_{i=1}^mr_\tau(F,k_i)-\sum_{i=1}^{m}r_\tau(\tilde F, k_i)\Bigl|\notag\\
			&=\frac{C}{2}+A+B,
		\end{align}
		where
		\begin{equation*}
			\begin{split}
				A =  \sum_{i=0}^{m-1}&  \int_{k_{i}}^{k_{i+1}}\int_{0}^{+\infty} (s - |\chi_{F}(s+u) - \chi_{F}(u)|)\widehat{K}_{\tau}(s) \ds \du
				\\  & - \sum_{i=0}^{m-1} \int_{k_{i}}^{k_{i+1}}\int_{0}^{+\infty} (s - |\chi_{\tilde F}(s+u) - \chi_{\tilde F}(u)|)\widehat{K}_\tau(s) \ds \du\\
				B =  \sum_{i=1}^{m} & \int_{k_{i}}^{k_{i+1}}\int_{-\infty}^{0} (s - |\chi_{F}(s+u) - \chi_{F}(u)|)\widehat{K}_\tau(s) \ds \du
				\\ &- \sum_{i=1}^{m} \int_{k_{i}}^{k_{i+1}}\int_{-\infty}^{0} (s - |\chi_{\tilde F}(s+u) - \chi_{\tilde F}(u)|)\widehat{K}_\tau (s)\ds \du.
			\end{split}
		\end{equation*}
		Thus, by using the integrability of $\widehat K_\tau$, we have that
		\begin{equation*}
			\begin{split}
				|A| \leq \int_{k_{0}} ^{k_{m}} \int_{0}^{+\infty} \chi_{F \Delta \tilde F}(u +s  ) \widehat{K}_{\tau}(s) \ds \du \leq\int_{k_0}^{k_m}\int_{k_{m} + \hat\eta}^{\infty} \widehat{K}_{\tau}(u-v) \dv\du \leq  \frac{C}{4},
			\end{split}
		\end{equation*}
		where $C$ is a constant depending only on $\eta_0$.  Similarly, $|B| \leq C/4$.

		Thus, we have that
		\begin{equation*}
			\Big|\sum_{i=1}^m r_\tau(F,k_i) - \sum_{i=0}^{m+2} r_\tau(\tilde F,k_i)\Big| \leq C
		\end{equation*}
		and the lemma is proved.
		
	\end{proof}

	\section{A rigidity result}\label{sec:rigidity}
	
	The main goal  of this section is to prove Theorem~\ref{thm:omega}.
	In particular, Theorem~\ref{thm:omega} will follow from Theorem~\ref{thm:rigidity} below and the one dimensional estimates of Section~\ref{sec:1D}. In order to state Theorem~\ref{thm:omega}, let us define the following functional:

	\begin{align}
		\overline \Fcal_{0,p,d}(E,\Omega):=\int_{\S^{d-1}}\int_{\partial^*E\cap \Omega}\frac{|\Scal{\nu_E(x)}{\theta}|}{\r_\theta(x)^{p-d-1}} \d\mathcal H^{d-1}(x)\d\theta, \label{eq:f0p}
	\end{align}
	where
	\begin{equation}\label{eq:f0pmin}
		\r_\theta(x)=|x_\theta-x_\theta^+|.
	\end{equation}

	We recall that $x=x_\theta^\perp+\theta x_\theta$, where $x_\theta^\perp\in\theta^\perp$ and $x_\theta\in\R$. The point $x_\theta^+\in\R$ is such that  $x_\theta^\perp+\theta x_\theta^+\in\partial^*E$ is the closest point to $x$ in $\partial^*E_{x_\theta^\perp}$ with $x_\theta^+-x_\theta>0$.
	
	Notice that, as in Proposition~\ref{prop:intgeom}, one can show the following integral geometric formulation for the functional $\overline \Fcal_{0,p,d}$:
	\begin{align}\label{eq:f0intgeom}
		\overline \Fcal_{0,p,d}(E,\Omega)=\int_{\S^{d-1}}\int_{\theta^\perp}\sum_{s\in\partial^*E_{x_\theta^\perp}\cap \Omega_{x_\theta^\perp}}\frac{1}{|s-s^+|^{p-d-1}}\dx_\theta^\perp\d\theta.
	\end{align}

	\begin{theorem}\label{thm:rigidity}
		Let $\Omega \subset  \R^{d}$ be an open bounded domain and let $E\subset\R^d$ be a set such that $\overline\Fcal_{0,p,d}(E,\Omega)<+\infty$ for some  $p\geq d+3$. Then $\partial^{*} E\cap \Omega$ is given by a disjoint union $\underset{i=1,\dots,N}{\cup} H_i\cap \Omega$, where each $H_i$ is an affine hyperplane in $\R^d$ and $H_{i}\cap H_{j}\cap \Omega = \emptyset$.

		In particular, when $\Omega = [0,L)^{d}$ and $E$ is a $[0,L)^{d}$-periodic set, then up to a rigid motion the set $E$ is  of the form
		\begin{equation}\label{eq:Estripes}
			E=\widehat E\times\R^{d-1},\quad \widehat E\subset\R,\quad\widehat E\cap [0,L)=\bigcup_{i=1}^{N_0}(s_i,t_i), \quad s_i<t_i<s_{i+1}<t_{i+1}.
		\end{equation}
	\end{theorem}
	
	Before diving into the proof of Theorem~\ref{thm:rigidity}, let us show how it leads to the proof of Theorem~\ref{thm:omega}.
	
	\begin{proof}[Proof of Theorem~\ref{thm:omega}: ]
		Let $\Omega\subset\R^d$ open and bounded and $E_\tau\subset\R^d$ such that \eqref{eq:suptau} holds.

		By  the integral geometric formula 
		\begin{equation*}
			\Fcal^{\mathrm{loc}}_{\tau, p, d}(E_{\tau},\Omega) = \frac{1}{|\Omega|}\int_{\S^{d-1}}\int_{\theta^{\perp}} \overline F_{\tau, p, d}^{\mathrm{1D}} ((E_{\tau})_{x^{\perp}_{\theta}}, \Omega_{x^{\perp}_{\theta}})\dx_{\theta}^\perp\d\theta,
		\end{equation*}
		the condition \eqref{eq:suptau} implies that for a.e. $(\theta,x_\theta^\perp)$ 
		\begin{equation*}
			\sup_\tau \overline F_{\tau, p, d}^{\mathrm{1D}} ((E_{\tau})_{x^{\perp}_{\theta}}, \Omega_{x^{\perp}_{\theta}})<+\infty. 
		\end{equation*}
		Moreover, using Fatou's Lemma, we have that 
		\begin{equation*}
			\sup_\tau \Fcal^{\mathrm{loc}}_{\tau, p, d}(E_{\tau},\Omega) \geq\liminf_{\tau\to 0} \Fcal^{\mathrm{loc}}_{\tau, p, d}(E_{\tau},\Omega) \geq \frac{1}{|\Omega|}\int_{\S^{d-1}}\int_{\theta^{\perp}} \liminf_{\tau\to 0} \overline F_{\tau, p, d}^{\mathrm{1D}} ((E_{\tau})_{x^{\perp}_{\theta}}, \Omega_{x^{\perp}_{\theta}})\dx_{\theta}^\perp\d\theta.
		\end{equation*}
		By the definition of $\overline F_{\tau, p, d}^{\mathrm{1D}} ((E_{\tau})_{x^{\perp}_{\theta}}, \Omega_{x^{\perp}_{\theta}})$ given in \eqref{eq:floc2} and Lemma \ref{lemma:comp1d},
		one has then that for a.e. $(\theta, x_\theta^\perp)$ there exists $(E_0)_{x_\theta^\perp}$ of finite perimeter in $\Omega_{x_\theta^\perp}$ such that, up to subsequences, $(E_{\tau})_{x^{\perp}_{\theta}}\to(E_0)_{x_\theta^\perp}$ in $L^1(\Omega_{x_\theta^\perp})$ and with respect to the Hausdorff convergence. Moreover, by  \eqref{eq:1dlowbound}, Corollary \ref{cor:fboundsp},  \eqref{eq:f0intgeom} and the above,
		\begin{equation}\label{eq:5.5}
			\liminf_{\tau\to 0} \Fcal^{\mathrm{loc}}_{\tau, p, d}(E_{\tau},\Omega)\gtrsim -\frac{1}{|\Omega|}\per(E_0,\Omega)+\frac{1}{|\Omega|}	\overline \Fcal_{0,p,d}(E_0,\Omega).
		\end{equation}
		
		Thus we can apply Theorem \ref{thm:rigidity} to $E_0$ and conclude the proof of the Theorem for a general domain $\Omega$.
		
		In the periodic setting with $\Omega=[0,L)^d$, one has a local control over the perimeter of the sets $E_\tau$ and $E_0$ in the whole $\R^d$. In particular, the above Hausdorff convergence of the slices of $E_\tau$ to the slices of $E_0$ is intended in the whole $\R^d$ and \eqref{eq:5.5} can be replaced by
		\[
		\Fcal_{\tau,p,d}(E_\tau,[0,L)^d)\gtrsim-\frac{1}{L^d}\per(E_0,[0,L)^d)+\frac{1}{L^d}\overline \Fcal_{0,p,d}(E,[0,L)^d).
		\]
		Thus, by the second part of Theorem~\ref{thm:rigidity} we can conclude the proof of the theorem in the periodic setting.
		
	\end{proof}
	
	An immediate consequence of the above proof is the precise identification of the $L^1$ $\Gamma$-limit of $\Fcal_{\tau,p,d}(\cdot, [0,L)^d)$ as $\tau\to0$ in the periodic setting. In order to give a precise statement, define $K_0(\zeta)=\|\zeta\|^{-p}$, $\overline K_0(\rho)=\rho^{d-1}K_0(\rho)$, and to avoid the problem of non-integrability of $K_0$ at the origin,
	\begin{align*}
		r_{0}(E_{x_\theta^\perp},s)&:=\int_{0}^{1}\Bigl(|\rho|-\int_{s^-}^s\bigl|\chi_{E_{x_\theta^\perp}}(u)-\chi_{E_{x_\theta^\perp}}(u+\rho)\du\bigr|\Bigr)\overline K_0(\rho)\d\rho\notag\\
		&+\int_{-1}^{0}\Bigl(|\rho|-\int_{s}^{s^+}\bigl|\chi_{E_{x_\theta^\perp}}(u)-\chi_{E_{x_\theta^\perp}}(u+\rho)\du\bigr|\Bigr)\overline K_0(\rho)\d\rho\notag\\
		&-\int_{s}^{s^+}\int_{-\infty}^{-1}\bigl|\chi_{E_{x_\theta^\perp}}(u)-\chi_{E_{x_\theta^\perp}}(u+\rho)\bigr|\overline K_0(\rho)\d\rho\du\notag\\
		&-\int_{s}^{s^+}\int_{-\infty}^{-1}\bigl|\chi_{E_{x_\theta^\perp}}(u)-\chi_{E_{x_\theta^\perp}}(u+\rho)\bigr|\overline K_0(\rho)\d\rho\du.
	\end{align*}
	Finally, let
	\begin{align}
		\Fcal_{0,p,d}(E, [0,L)^d)&:=\frac{1}{L^d}\int_{\S^{d-1}}\int_{\theta^\perp}\sum_{s\in\partial^* E_{x_\theta^\perp}\cap [0,L)^d_{x_\theta^\perp}}r_{0}(E_{x_\theta^\perp},s)\dx_\theta^\perp\d\theta\notag\\
		&=\frac{1}{L^d}\int_{\S^{d-1}}\int_{\theta^\perp}\overline F_{0,p,d}^{\mathrm{1D}}(E_{x_\theta^\perp}, [0,L)^d_{x_\theta^\perp})\dx_\theta^\perp\d\theta.\label{eq:forintgeom0}
	\end{align}
	One has that
	\begin{align*}
		L^d	\Fcal_{0,p,d}(E, [0,L)^d)&\gtrsim \int_{\S^{d-1}}\int_{\theta^\perp}\sum_{s\in\partial^* E_{x_\theta^\perp}\cap[0,L)^d_{x_\theta^\perp}} \Biggl[-1+\frac{1}{|s-s^+|^{p-d-1}}\Biggr]\dx_\theta^\perp\d\theta\notag\\
		&\gtrsim -\per (E,[0,L)^d)+L^d\overline\Fcal_{0,p,d}(E, [0,L)^d),
	\end{align*}
	where $\overline\Fcal_{0,p,d}$ was defined in~\eqref{eq:f0p}.

	One has the following
	
	\begin{theorem}
		\label{thm:gammaconv}
		Let $p>d+1$. The functionals $\Fcal_{\tau, p, d}(\cdot, [0,L)^d)$ $\Gamma-$converge as $\tau\to0$ with respect to the $L^1-$convergence to the functional  $\Fcal_{0, p, d}(\cdot, [0,L)^d)$ defined in~\eqref{eq:forintgeom0}.
	\end{theorem}
	
	\begin{proof}
		The proof of the $\Gamma$-liminf inequality follows as in the proof of Theorem~\ref{thm:omega}, by slicing and the one-dimensional estimates of Section~\ref{sec:1D}.
		
		For the $\Gamma$-limsup inequality, for any set of $\Fcal_{0,p,d}$ finite energy $E$ it is sufficient to take $E_{\tau} = E$ for all $\tau>0$ and notice that $r_\tau(E,s)\to r_0(E,s)$, thus $\Fcal_{\tau,p,d}(E,[0,L)^d)\to\Fcal_{0,p,d}(E,[0,L)^d)$ as $\tau\to0$.
	\end{proof}
	
	Let us now move to the proof of Theorem~\ref{thm:rigidity}.

	The proof of Theorem~\ref{thm:rigidity} will follow from a series of preliminary lemmas and propositions of independent interest.
	
	In order to prove such a rigidity estimate, we develop a novel strategy. Let us now give a brief outline of the main steps of the proof, in order to facilitate the reader.

	\begin{itemize}
		\item [(i)] First, we show that whenever $p\geq 2d$ the functional $\overline\Fcal_{0,p,d}$ controls a nonlocal generalized version of the curvature of the relative boundary of the set $E$. More precisely, we show the following
		
		\begin{proposition}\label{prop:brezis}
			Let $d\geq2$, $p\geq 2d$, $M>0$.  Then, there exists $R_0>0, C(M)>0$ such that the following holds: for all $\Omega\subset\R^d$ bounded open sets and {for every $E\subset \R^d$ of locally finite perimeter} such that $\overline\Fcal_{0,p,d}(E,\Omega)\leq M$, for all $0<r<R_0$ and for all $z\in\Omega$ such that $\mathrm{dist}(z,\partial \Omega)>2r$, it holds
			\begin{equation}\label{eq:brezis}
				\int_{\partial^*E\cap B_r(z)}\int_{\partial ^*E\cap B_r(z)}\frac{\|\nu_E(x)-\nu_E(y)\|^2}{\|x-y\|^{p-2}}\d\mathcal H^{d-1}(x)\d\mathcal H^{d-1}(y)\lesssim C(M)\overline \Fcal_{0,p,d}(E,B_r(z)).
			\end{equation}
		\end{proposition}
		
		\item [(ii)] We prove the following regularity result.
		\begin{theorem}\label{thm:regularity}
			Let $d\geq 2$, $p>2d$, $\ell>0$, $M>0$. There exists $\bar R>0$ such that the following holds. Let $\Omega\subset\R^d$ open and bounded and  let $E\subset \R^d$ of locally finite perimeter such that $\overline\Fcal_{0,p,d}(E,\Omega)\leq M$. Then, for all $0<r<\bar R$, on $\Omega_{r }=\{z\in\Omega:\,\mathrm{dist}(z,\partial \Omega)>r\}$ the set $E$ enjoys the following regularity properties:
			\begin{enumerate}
				\item\label{item1_thm:regularity} $\partial^*E\cap \Omega_{r}=\partial E\cap\Omega_{r}$;
				\item\label{item2_thm:regularity} For every $x\in\partial E\cap \Omega_{r}$ the set $\partial E\cap B_r(x)$ is given by the graph of an $\ell$-Lipschitz function defined on a connected open subset of a  $(d-1)$-dimensional affine subspace of $\R^d$.
			\end{enumerate}
		\end{theorem}
		
		\item [(iii)] We show that a uniform bound (with respect to $r\ll1$ and $x$) such as
		\begin{equation}\label{eq:brezisbis}
			\sup_{x\in\Omega_r, 0<r\leq\bar R}\int_{\partial^*E\cap B_r(z)}\int_{\partial ^*E\cap B_r(z)}\frac{\|\nu_E(x)-\nu_E(y)\|^2}{\|x-y\|^{p-2}}\leq C
		\end{equation}
		(as in~\eqref{eq:brezis}) together with regularity conditions on $\partial E\cap \Omega$ as properties~\ref{item1_thm:regularity} and ~\ref{item2_thm:regularity}  in Theorem~\ref{thm:regularity} imply when $p\geq d+3$ that $\partial E\cap \Omega$ is given by the disjoint union of the intersections of finitely many affine  hyperplanes with $\Omega$  (the fact that they are parallel follows if the domain is a cube $[0,L)^d$ and we impose $[0,L)^d$-periodicity).
		This result is the content of Lemma~\ref{lemma:d+3}.
		The proof of such Lemma uses a result of  Bourgain, Brezis  and Mironescu~\cite{bbm,brezis} implying that if $\partial E\cap\Omega_r$ can be locally uniformly parametrized by Lipschitz maps defined on $\R^{d-1}$ and~\eqref{eq:brezisbis} holds, then $\nu_E$ is uniformly locally constant on $\partial E\cap\Omega_r$ and then $\partial E\cap\Omega_r$ is flat.
		
		\item[(iv)] 
		
		Given Lemma~\ref{lemma:d+3}, in order to show that the boundary of sets of  $\overline\Fcal_{0,p,d}$-bounded energy is flat we would like then to combine Proposition~\ref{prop:brezis}, Theorem~\ref{thm:regularity} and Lemma~\ref{lemma:d+3}.
		However, notice that while for the validity of Lemma~\ref{lemma:d+3} it is sufficient (and indeed necessary, as a simple computation with a $C^2$ boundary would show) that $p\geq d+3$, for the validity of Propositions~\ref{prop:brezis} and~\ref{thm:regularity} we need to assume that $p>2d$,  which is a stronger condition as soon as $d\geq3$.
		
		Thus, the proof of Theorem~\ref{thm:rigidity} will be given first for $d=2$ (where $2d<d+3$) and then extended via a slicing argument with affine two-dimensional planes to the case of general dimension (see Proposition~\ref{prop:nusliced} and the proof of Theorem~\ref{thm:rigidity}).
		More precisely, we will apply a slicing formula of the functional $\overline\Fcal_{0,p,d}$ with respect to  the Grassmanian of $2$-planes in $\R^d$ (see Proposition~\ref{prop:slicing}), thus showing that \ae two-dimensional slice of a set of finite energy $E$ has flat boundary inside $\Omega$, which in turn will imply that $\partial E\cap \Omega$ is flat in $\R^d$.

	\end{itemize}
	In the next Sections~\ref{subs:5.1}--\ref{subs:5.4}, we proceed to the proof of (i)--(iv).

	\subsection{(i) A nonlocal curvature bound}\label{subs:5.1}
	
	For simplicity of notation, define
	\begin{equation}
		e(x):=\int_{\S^{d-1}}\frac{|\langle \nu_E(x),\theta\rangle|}{\r_\theta(x)^{p-d-1}}\d\theta, \qquad x\in\partial^*E.
	\end{equation}
	In particular, \[
	\overline\Fcal_{0,p,d}(E,\Omega)=\int_{\partial^*E\cap\Omega}e(x)\d\mathcal H^{d-1}(x).
	\]

	\begin{lemma} \label{lemma:epointwisenu}
		Let $E\subset\R^d$ be a set of locally finite perimeter.
		The following holds
		\begin{align}
			e(x)+e(y)\gtrsim\frac{\|\nu_E(x)-\nu_E(y)\|^2}{\|x-y\|^{p-d-1}}, \quad\forall\,x,y\in\partial^*E. \label{eq:exy}
		\end{align}
	\end{lemma}
	
	\begin{proof}
		For every $x=x_\theta^\perp+x_\theta\theta\in\partial^* E$, $y\in\partial^* E$, define the sets
		\begin{align*}
			\Omega(x,\|y-x\|)=\{\theta\in\S^{d-1}:\,\r_\theta(x)<2\|x-y\|\},\quad 	\Omega(y, \|y-x\|)=\{\theta\in\S^{d-1}:\,\r_\theta(y)<2\|x-y\|\}.
		\end{align*}
		\WithoutLoss, we can assume that $\max\bigl\{|\Omega(x,\|y-x\|)|,|\Omega(y,\|y-x\|)|\bigr\}\leq d\omega_d/4$.
		Indeed, if $|\Omega(x,\|y-x\|)|>d\omega_d/4$ one has that
		\begin{align*}
			e(x)&\gtrsim\int_{\Omega(x,\|y-x\|)\cap\{\theta:\,|\langle\nu_E(x),\theta\rangle|\geq1/10\}}\frac{|\langle\nu_E(x),\theta\rangle|}{\|x-y\|^{p-d-1}}\d\theta\notag\\
			&\gtrsim\frac{\|\nu_E(x)-\nu\|^2}{\|x-y\|^{p-d-1}},\quad\forall\,\nu\in\S^{d-1},
		\end{align*}
		where we used the fact that $\|\nu_E(x)-\nu\|\leq\sqrt{2}$ and $\Omega(x,\|y-x\|)\cap\{\theta:\,|\langle\nu_E(x),\theta\rangle|\geq1/10\}\gtrsim 1$ whenever $|\Omega(x,\|y-x\|)|>d\omega_d/4$.
		In particular, for points $x,y$ such that $\max\bigl\{|\Omega(x,\|y-x\|)|, |\Omega(y,\|y-x\|)|\bigr\}>d\omega_d/4$,~\eqref{eq:exy} is proved.

		Assuming $|\Omega(x,\|y-x\|)|\leq d\omega_d/4$, one has that
		\begin{align}
			e(x)&\gtrsim \int_{\Omega(x,\|y-x\|)\cap \{\theta:\,|\langle\nu_E(x),\theta\rangle|\geq|\Omega(x,\|y-x\|)|/(2d\omega_d)\} }\frac{|\langle\nu_E(x),\theta\rangle|}{\|x-y\|^{p-d-1}}\d\theta\notag\\
			&\gtrsim \frac{|\Omega(x,\|y-x\|)|^2}{\|x-y\|^{p-d-1}}.\label{eq:eomega}
		\end{align}
		The rest of the proof is thus devoted to show that
		\begin{align}\label{eq:bgtrsimnu}
			|\Omega(x,\|y-x\|)|+|\Omega(y,\|y-x\|)|\gtrsim\|\nu_E(x)-\nu_E(y)\|.
		\end{align}
		To this aim, define the cones
		\begin{align*}
			C^+(x,2\|y-x\|)&:=\{x+s\theta:\,\langle\nu_E(x),\theta\rangle>0,\,s\in(0,2\|y-x\|)\},\\
			C^-(x,2\|y-x\|)&:=\{x+s\theta:\,\langle\nu_E(x),\theta\rangle<0,\,s\in(0,2\|y-x\|)\},\\
			C^+(y,2\|y-x\|)&:=\{y+s\theta:\,\langle\nu_E(y),\theta\rangle>0,\,s\in(0,2\|y-x\|)\},\\
			C^-(y,2\|y-x\|)&:=\{y+s\theta:\,\langle\nu_E(y),\theta\rangle<0,\,s\in(0,2\|y-x\|)\}.
		\end{align*}
		Notice that, by simple geometric considerations (see Figure~\ref{fig:4}),
		{\begin{equation}
				|C^+(x,\|y-x\|)\cap C^-(y,\|y-x\|)|+|C^-(x,\|y-x\|)\cap C^+(y,\|y-x\|)| \geq \bar C\frac12d\omega_d\|\nu_E(x)-\nu_E(y)\|\|x-y\|^d,\label{eq:cmax}
		\end{equation}}
		for some dimensional constant $\bar C>0$.
		\begin{figure}
			\centering
			\begin{tikzpicture}[scale=3]
				\tikzset{arrow/.style={-latex}}
				\clip (-1.88,-1) rectangle (1.88,1);
				\begin{scope}[shift={(-1.0, 0.19323323583816937)},rotate=-22.466734847371534,opacity=0.7]
					\clip (-2,-2) rectangle (4,2);
					\draw[fill=black!50,opacity=0.5] (-100,-100) -- (-100,0) -- (100,0) -- (100,-100) --cycle;
					\draw[line width=1pt] (-100,-100) -- (-100,0) -- (100,0) -- (100,-100) --cycle;
					\draw[line width=1pt,arrow] (0,0) -- (0,.5);
					\draw[fill=black] (0,0) circle[radius=.4pt];
					\draw (0.1,0.1) node {$x$};
					\draw (0.1,0.65) node {$\nu_{E}(x)$};
				\end{scope}
				\begin{scope}[shift={(1, 0.19323323583816937)},rotate=22.466734847371534,opacity=0.7]
					\clip (-4,-2) rectangle (3,1);
					\draw[fill=black!50,opacity=0.5] (-100,-100) -- (-100,0) -- (100,0) -- (100,-100) --cycle;
					\draw[line width=1pt] (-100,-100) -- (-100,0) -- (100,0) -- (100,-100) --cycle;
					\draw[fill=black] (0,0) circle[radius=.4pt];
					\draw (-0.1,0.1) node {$y$};
					\draw[line width=1pt,arrow] (0,0) -- (0,.5);
					\draw (0.1,0.6) node {$\nu_{E}(y)$};
				\end{scope}
				\draw[fill=black] (0,-0.223) circle[radius=.4pt];
				\draw[line width=1pt] (-2,0.79908) -- (-1.98662,0.78841) -- (-1.97324,0.77783) -- (-1.95987,0.76735) -- (-1.94649,0.75697) -- (-1.93311,0.74668) -- (-1.91973,0.7365) -- (-1.90635,0.72641) -- (-1.89298,0.71642) -- (-1.8796,0.70651) -- (-1.86622,0.69669) -- (-1.85284,0.68694) -- (-1.83946,0.67727) -- (-1.82609,0.66767) -- (-1.81271,0.65812) -- (-1.79933,0.64863) -- (-1.78595,0.63919) -- (-1.77258,0.62979) -- (-1.7592,0.62043) -- (-1.74582,0.6111) -- (-1.73244,0.6018) -- (-1.71906,0.59252) -- (-1.70569,0.58327) -- (-1.69231,0.57405) -- (-1.67893,0.56485) -- (-1.66555,0.55568) -- (-1.65217,0.54655) -- (-1.6388,0.53746) -- (-1.62542,0.52841) -- (-1.61204,0.51941) -- (-1.59866,0.51048) -- (-1.58528,0.50162) -- (-1.57191,0.49284) -- (-1.55853,0.48416) -- (-1.54515,0.47558) -- (-1.53177,0.46711) -- (-1.51839,0.45877) -- (-1.50502,0.45056) -- (-1.49164,0.44248) -- (-1.47826,0.43455) -- (-1.46488,0.42676) -- (-1.45151,0.41912) -- (-1.43813,0.41163) -- (-1.42475,0.40429) -- (-1.41137,0.39708) -- (-1.39799,0.39001) -- (-1.38462,0.38305) -- (-1.37124,0.37621) -- (-1.35786,0.36946) -- (-1.34448,0.3628) -- (-1.3311,0.3562) -- (-1.31773,0.34964) -- (-1.30435,0.34312) -- (-1.29097,0.33661) -- (-1.27759,0.3301) -- (-1.26421,0.32357) -- (-1.25084,0.31702) -- (-1.23746,0.31042) -- (-1.22408,0.30377) -- (-1.2107,0.29707) -- (-1.19732,0.29031) -- (-1.18395,0.28351) -- (-1.17057,0.27665) -- (-1.15719,0.26976) -- (-1.14381,0.26285) -- (-1.13043,0.25593) -- (-1.11706,0.24903) -- (-1.10368,0.24217) -- (-1.0903,0.23537) -- (-1.07692,0.22866) -- (-1.06355,0.22207) -- (-1.05017,0.21563) -- (-1.03679,0.20936) -- (-1.02341,0.20329) -- (-1.01003,0.19745) -- (-0.99666,0.19186) -- (-0.98328,0.18652) -- (-0.9699,0.18146) -- (-0.95652,0.17668) -- (-0.94314,0.17218) -- (-0.92977,0.16795) -- (-0.91639,0.16398) -- (-0.90301,0.16025) -- (-0.88963,0.15675) -- (-0.87625,0.15343) -- (-0.86288,0.15026) -- (-0.8495,0.14721) -- (-0.83612,0.14423) -- (-0.82274,0.14127) -- (-0.80936,0.13829) -- (-0.79599,0.13524) -- (-0.78261,0.13208) -- (-0.76923,0.12877) -- (-0.75585,0.12526) -- (-0.74247,0.12153) -- (-0.7291,0.11755) -- (-0.71572,0.11331) -- (-0.70234,0.10879) -- (-0.68896,0.10401) -- (-0.67559,0.09897) -- (-0.66221,0.0937) -- (-0.64883,0.08823) -- (-0.63545,0.0826) -- (-0.62207,0.07687) -- (-0.6087,0.07109) -- (-0.59532,0.06534) -- (-0.58194,0.05969) -- (-0.56856,0.05421) -- (-0.55518,0.04898) -- (-0.54181,0.04407) -- (-0.52843,0.03957) -- (-0.51505,0.03553) -- (-0.50167,0.03201) -- (-0.48829,0.02906) -- (-0.47492,0.02672) -- (-0.46154,0.02501) -- (-0.44816,0.02394) -- (-0.43478,0.0235) -- (-0.4214,0.02366) -- (-0.40803,0.02438) -- (-0.39465,0.0256) -- (-0.38127,0.02724) -- (-0.36789,0.02921) -- (-0.35452,0.0314) -- (-0.34114,0.03371) -- (-0.32776,0.03601) -- (-0.31438,0.03817) -- (-0.301,0.04007) -- (-0.28763,0.04159) -- (-0.27425,0.0426) -- (-0.26087,0.04301) -- (-0.24749,0.04271) -- (-0.23411,0.04165) -- (-0.22074,0.03975) -- (-0.20736,0.037) -- (-0.19398,0.03338) -- (-0.1806,0.02893) -- (-0.16722,0.0237) -- (-0.15385,0.01777) -- (-0.14047,0.01124) -- (-0.12709,0.00425) -- (-0.11371,-0.00304) -- (-0.10033,-0.01046) -- (-0.08696,-0.01782) -- (-0.07358,-0.0249) -- (-0.0602,-0.03151) -- (-0.04682,-0.03744) -- (-0.03344,-0.04247) -- (-0.02007,-0.04643) -- (-0.00669,-0.04915) -- (0.00669,-0.04915) -- (0.02007,-0.04643) -- (0.03344,-0.04247) -- (0.04682,-0.03744) -- (0.0602,-0.03151) -- (0.07358,-0.0249) -- (0.08696,-0.01782) -- (0.10033,-0.01046) -- (0.11371,-0.00304) -- (0.12709,0.00425) -- (0.14047,0.01124) -- (0.15385,0.01777) -- (0.16722,0.0237) -- (0.1806,0.02893) -- (0.19398,0.03338) -- (0.20736,0.037) -- (0.22074,0.03975) -- (0.23411,0.04165) -- (0.24749,0.04271) -- (0.26087,0.04301) -- (0.27425,0.0426) -- (0.28763,0.04159) -- (0.301,0.04007) -- (0.31438,0.03817) -- (0.32776,0.03601) -- (0.34114,0.03371) -- (0.35452,0.0314) -- (0.36789,0.02921) -- (0.38127,0.02724) -- (0.39465,0.0256) -- (0.40803,0.02438) -- (0.4214,0.02366) -- (0.43478,0.0235) -- (0.44816,0.02394) -- (0.46154,0.02501) -- (0.47492,0.02672) -- (0.48829,0.02906) -- (0.50167,0.03201) -- (0.51505,0.03553) -- (0.52843,0.03957) -- (0.54181,0.04407) -- (0.55518,0.04898) -- (0.56856,0.05421) -- (0.58194,0.05969) -- (0.59532,0.06534) -- (0.6087,0.07109) -- (0.62207,0.07687) -- (0.63545,0.0826) -- (0.64883,0.08823) -- (0.66221,0.0937) -- (0.67559,0.09897) -- (0.68896,0.10401) -- (0.70234,0.10879) -- (0.71572,0.11331) -- (0.7291,0.11755) -- (0.74247,0.12153) -- (0.75585,0.12526) -- (0.76923,0.12877) -- (0.78261,0.13208) -- (0.79599,0.13524) -- (0.80936,0.13829) -- (0.82274,0.14127) -- (0.83612,0.14423) -- (0.8495,0.14721) -- (0.86288,0.15026) -- (0.87625,0.15343) -- (0.88963,0.15675) -- (0.90301,0.16025) -- (0.91639,0.16398) -- (0.92977,0.16795) -- (0.94314,0.17218) -- (0.95652,0.17668) -- (0.9699,0.18146) -- (0.98328,0.18652) -- (0.99666,0.19186) -- (1.01003,0.19745) -- (1.02341,0.20329) -- (1.03679,0.20936) -- (1.05017,0.21563) -- (1.06355,0.22207) -- (1.07692,0.22866) -- (1.0903,0.23537) -- (1.10368,0.24217) -- (1.11706,0.24903) -- (1.13043,0.25593) -- (1.14381,0.26285) -- (1.15719,0.26976) -- (1.17057,0.27665) -- (1.18395,0.28351) -- (1.19732,0.29031) -- (1.2107,0.29707) -- (1.22408,0.30377) -- (1.23746,0.31042) -- (1.25084,0.31702) -- (1.26421,0.32357) -- (1.27759,0.3301) -- (1.29097,0.33661) -- (1.30435,0.34312) -- (1.31773,0.34964) -- (1.3311,0.3562) -- (1.34448,0.3628) -- (1.35786,0.36946) -- (1.37124,0.37621) -- (1.38462,0.38305) -- (1.39799,0.39001) -- (1.41137,0.39708) -- (1.42475,0.40429) -- (1.43813,0.41163) -- (1.45151,0.41912) -- (1.46488,0.42676) -- (1.47826,0.43455) -- (1.49164,0.44248) -- (1.50502,0.45056) -- (1.51839,0.45877) -- (1.53177,0.46711) -- (1.54515,0.47558) -- (1.55853,0.48416) -- (1.57191,0.49284) -- (1.58528,0.50162) -- (1.59866,0.51048) -- (1.61204,0.51941) -- (1.62542,0.52841) -- (1.6388,0.53746) -- (1.65217,0.54655) -- (1.66555,0.55568) -- (1.67893,0.56485) -- (1.69231,0.57405) -- (1.70569,0.58327) -- (1.71906,0.59252) -- (1.73244,0.6018) -- (1.74582,0.6111) -- (1.7592,0.62043) -- (1.77258,0.62979) -- (1.78595,0.63919) -- (1.79933,0.64863) -- (1.81271,0.65812) -- (1.82609,0.66767) -- (1.83946,0.67727) -- (1.85284,0.68694) -- (1.86622,0.69669) -- (1.8796,0.70651) -- (1.89298,0.71642) -- (1.90635,0.72641) -- (1.91973,0.7365) -- (1.93311,0.74668) -- (1.94649,0.75697) -- (1.95987,0.76735) -- (1.97324,0.77783) -- (1.98662,0.78841) -- (2,0.79908);
			\end{tikzpicture}
			\caption{The light grey regions correspond to the sets $ C^+(x,\|y-x\|)\cap C^-(y,\|y-x\|)$, $C^-(x,\|y-x\|)\cap C^+(y,\|y-x\|)$.}
			\label{fig:4}
		\end{figure}
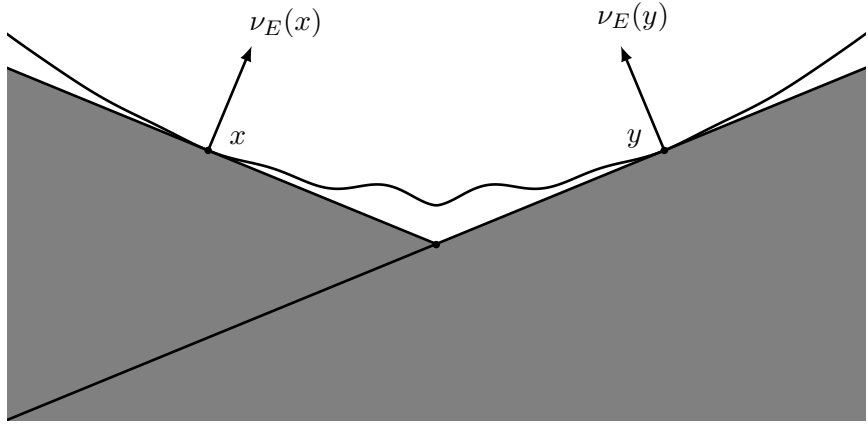
		
		We now claim that~\eqref{eq:bgtrsimnu} holds in the form (see Figure~\ref{fig:5})
		\begin{equation}
			|\Omega(x,\|y-x\|)|+|\Omega(y,\|y-x\|)|\geq \frac{\bar C}{20}d\omega_d\|\nu_E(x)-\nu_E(y)\|.\label{eq:bgtrsimnu2}
		\end{equation}
		Indeed, assume that~\eqref{eq:bgtrsimnu2} does not hold for some $x,y\in\partial^*E$
		and let
		\begin{align*}\delta(x,y)&:=\frac{\bar C}{20}\|\nu_E(x)-\nu_E(y)\|, \\
			A(x, \delta(x,y))&=\{\theta\in\S^{d-1}:\,|\langle\nu_E(x),\theta\rangle|<\delta(x,y)\},\\
			A(y, \delta(x,y))&=\{\theta\in\S^{d-1}:\,|\langle\nu_E(y),\theta\rangle|<\delta(x,y)\}.
		\end{align*}
		
		By~\eqref{eq:stimeblowup} and the fact that 	$|\Omega(x,\|y-x\|)|+|\Omega(y,\|y-x\|)|\leq \delta(x,y)d\omega_d$, we obtain
		\begin{align*}
			\bigl|\{\theta:\,\langle\nu_E(x),\theta\rangle<0 \text{ and }x_\theta+[0,2\|y-x\|]\theta\in E_{x_\theta^\perp}\}\bigr|&\geq\bigl|\{\theta:\,\langle\nu_E(x),\theta\rangle<0\}\setminus (\Omega(x,\|y-x\|)\cup A(x,\delta(x,y)))\bigr|\notag\\
			&\geq\bigl(1/2-2\delta(x,y)\bigr)d\omega_d\\
			\bigl |\{\theta:\,\langle\nu_E(y),\theta\rangle<0 \text{ and }y_\theta+[0,2\|y-x\|]\theta\in E_{y_\theta^\perp}\}\bigr|&\geq\bigl|\{\theta:\,\langle\nu_E(y),\theta\rangle<0\}\setminus (\Omega(y,\|y-x\|)\cup A(y,\delta(x,y)))\bigr|\notag\\
			&\geq\bigl(1/2-2\delta(x,y)\bigr)d\omega_d\\
			\bigl  |\{\theta:\,\langle\nu_E(x),\theta\rangle>0 \text{ and }x_\theta+[0,2\|y-x\|]\theta\in E^c_{x_\theta^\perp}\}\bigr|&\geq\bigl|\{\theta:\,\langle\nu_E(x),\theta\rangle>0\}\setminus (\Omega(x,|y-x|)\cup A(x,\delta(x,y)) )\bigr|\notag\\
			&\geq\bigl(1/2-2\delta(x,y)\bigr)d\omega_d\\
			\bigl  |\{\theta:\,\langle\nu_E(y),\theta\rangle>0 \text{ and }y_\theta+[0,2\|y-x\|]\theta\in E^c_{y_\theta^\perp}\}\bigr|&\geq\bigl|\{\theta:\,\langle\nu_E(y),\theta\rangle>0\}\setminus (\Omega(y,\|y-x\|)\cup A(y,\delta(x,y)))\bigr|\notag\\
			&\geq\bigl(1/2-2\delta(x,y)\bigr)d\omega_d     \end{align*}
		which respectively imply
		\begin{align*}
			|C^-(x,2\|x-y\|)\cap E|&\geq (1/2-2\delta(x,y))\omega_d\|x-y\|^d,\\
			|C^+(x,2\|x-y\|)\cap(\R^d\setminus E)|&\geq (1/2-2\delta(x,y))\omega_d\|x-y\|^d,\\
			|C^-(y,2\|x-y\|)\cap E|&\geq (1/2-2\delta(x,y))\omega_d\|x-y\|^d,\\
			|C^+(y,2\|x-y\|)\cap (\R^d\setminus E)|&\geq (1/2-2\delta(x,y))\omega_d\|x-y\|^d.
		\end{align*}
		By the above, one obtains
		\begin{align*}
			|C^+(x,2\|y-x\|)\cap C^-(y,2\|y-x\|)|&\leq 	|C^+(x,2\|y-x\|)\cap C^-(y,2\|y-x\|)\cap E|\notag\\
			&+	|C^+(x,2\|y-x\|)\cap C^-(y,2\|y-x\|)\cap (\R^d\setminus E)|\notag\\
			&\leq |C^+(x,2\|y-x\|)\cap E|\notag\\
			&+ |C^-(y,2\|y-x\|)\cap (\R^d\setminus E)|\notag\\
			&\leq 4\delta(x,y)\omega_d\|x-y\|^d,
		\end{align*}
		and analogously
		\begin{equation*}
			|C^-(x,2\|y-x\|)\cap C^+(y,2\|y-x\|)|\leq 4\delta(x,y)\omega_d\|x-y\|^d.
		\end{equation*}
		This, given the definition of $\delta(x,y)$, contradicts~\eqref{eq:cmax}, thus proving~\eqref{eq:bgtrsimnu2}.
		
		\begin{figure}
			\centering
			\begin{tikzpicture}[scale=2]
				\clip (-3,-1.2) rectangle (3,2.2);
				\tikzset{arrow/.style={-latex}}
				\draw[line width=1pt,fill=black!10!white,smooth] (-3,3) -- (-2.92929,2.86025) -- (-2.85859,2.72384) -- (-2.78788,2.59076) -- (-2.71717,2.46101) -- (-2.64646,2.33459) -- (-2.57576,2.21151) -- (-2.50505,2.09176) -- (-2.43434,1.97534) -- (-2.36364,1.86226) -- (-2.29293,1.75251) -- (-2.22222,1.64609) -- (-2.15152,1.54301) -- (-2.08081,1.44325) -- (-2.0101,1.34684) -- (-1.93939,1.25375) -- (-1.86869,1.164) -- (-1.79798,1.07758) -- (-1.72727,0.99449) -- (-1.65657,0.91474) -- (-1.58586,0.83832) -- (-1.51515,0.76523) -- (-1.44444,0.69547) -- (-1.37374,0.62905) -- (-1.30303,0.56596) -- (-1.23232,0.50621) -- (-1.16162,0.44978) -- (-1.09091,0.39669) -- (-1.0202,0.34694) -- (-0.94949,0.30051) -- (-0.87879,0.25742) -- (-0.80808,0.21766) -- (-0.73737,0.18124) -- (-0.66667,0.14815) -- (-0.59596,0.11839) -- (-0.52525,0.09196) -- (-0.45455,0.06887) -- (-0.38384,0.04911) -- (-0.31313,0.03268) -- (-0.24242,0.01959) -- (-0.17172,0.00983) -- (-0.10101,0.0034) -- (-0.0303,0.00031) -- (0.0404,0.00054) -- (0.11111,0.00412) -- (0.18182,0.01102) -- (0.25253,0.02126) -- (0.32323,0.03483) -- (0.39394,0.05173) -- (0.46465,0.07197) -- (0.53535,0.09553) -- (0.60606,0.12244) -- (0.67677,0.15267) -- (0.74747,0.18624) -- (0.81818,0.22314) -- (0.88889,0.26337) -- (0.9596,0.30694) -- (1.0303,0.35384) -- (1.10101,0.40407) -- (1.17172,0.45764) -- (1.24242,0.51454) -- (1.31313,0.57477) -- (1.38384,0.63834) -- (1.45455,0.70523) -- (1.52525,0.77547) -- (1.59596,0.84903) -- (1.66667,0.92593) -- (1.73737,1.00616) -- (1.80808,1.08972) -- (1.87879,1.17661) -- (1.94949,1.26684) -- (2.0202,1.36041) -- (2.09091,1.4573) -- (2.16162,1.55753) -- (2.23232,1.66109) -- (2.30303,1.76798) -- (2.37374,1.87821) -- (2.44444,1.99177) -- (2.51515,2.10866) -- (2.58586,2.22889) -- (2.65657,2.35245) -- (2.72727,2.47934) -- (2.79798,2.60956) -- (2.86869,2.74312) -- (2.93939,2.88001) -- (3.0101,3.02024) -- (3.08081,3.16379) -- (3.15152,3.31068) -- (3.22222,3.46091) -- (3.29293,3.61446) -- (3.36364,3.77135) -- (3.43434,3.93157) -- (3.50505,4.09513) -- (3.57576,4.26201) -- (3.64646,4.43223) -- (3.71717,4.60579) -- (3.78788,4.78268) -- (3.85859,4.96289) -- (3.92929,5.14645) -- (4,5.33333);
				\begin{scope}[shift={(1.0, 0.3333333333333333)},rotate=213.69006752596877]
					\clip (-1,-0.2) rectangle (1,0.5);
					\draw[line width=1pt,pattern=north west hatch,hatch distance=6pt, hatch thickness=.5pt,opacity=0.5] (-2,-2) -- (-2,0) -- (2,0) -- (2,-2);
					\draw[line width=1pt] (-2,-2) -- (-2,0) -- (2,0) -- (2,-2);
				\end{scope}
				\begin{scope}[shift={(1.0, 0.3333333333333333)},rotate=303.69006752596874]
					\draw[line width=1pt,arrow] (0,0) -- (1,0);
				\end{scope}
				\begin{scope}[shift={(-1.0, 0.3333333333333333)},rotate=146.30993247403123]
					\clip (-1.01,-0.2) rectangle (1.01,0.5);
					\draw[line width=1pt,pattern=north east hatch,hatch distance=6pt, hatch thickness=.5pt,opacity=0.5] (-2,-2) -- (-2,0) -- (2,0) -- (2,-2);
					\draw[line width=1pt] (-2,-2) -- (-2,0) -- (2,0) -- (2,-2);
				\end{scope}
				\begin{scope}[shift={(-1.0, 0.3333333333333333)},rotate=236.30993247403123]
					\draw[line width=1pt,arrow] (0,0) -- (1,0);
				\end{scope}
				\draw[color=red,arrow] (1,0.33333) -- (-2.5,2.08333);
				\draw[color=red,arrow] (1,0.33333) -- (-2.11429,1.49007);
				\draw[color=red,arrow] (1,0.33333) -- (-1.72857,0.99599);
				\draw[color=red,arrow] (1,0.33333) -- (-1.34286,0.60109);
				\draw[color=red,arrow] (1,0.33333) -- (-0.95714,0.30537);
				\draw[color=red,arrow] (1,0.33333) -- (-0.57143,0.10884);
				\draw[color=red,arrow] (1,0.33333) -- (-0.18571,0.0115);
				\draw[color=red,arrow] (1,0.33333) -- (0.2,0.01333);
				\draw (0.95,0.53333) node {$x$};
				\draw (-0.95,0.03333) node {$y$};
				\draw (-0.4,1.33333) node[color=red] {$\theta$};
				\draw (1.5,-0.7) node {$\nu_{E}(x)$};
				\draw (-1.5,-0.7) node {$\nu_{E}(y)$};
				\draw[fill,color=blue!70!black] (1,0.33333) circle[radius=0.05];
				\draw[fill,color=blue!70!black] (-1,0.33333) circle[radius=0.05];
			\end{tikzpicture}
			\caption{By the blow-up properties~\eqref{eq:stimeblowup}, a difference between $\nu_E(x)$ and $\nu_E(y)$ controls from below  $\Omega(x,\|y-x\|)$ and $\Omega(y,\|y-x\|)$.}
			\label{fig:5}
		\end{figure}
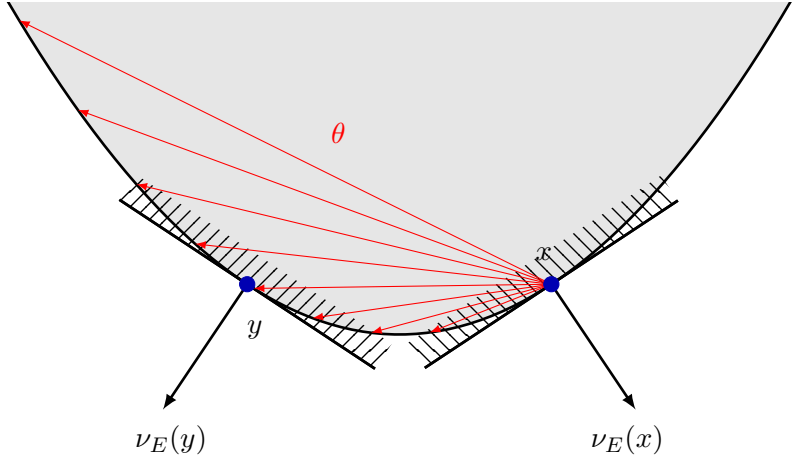
	\end{proof}

	Before stating the next lemma, we give the following definitions: for all $k\in\N\cup\{0\}$ let
	\begin{align}
		\tilde\Omega_k(x)&:=\bigl\{\theta\in\S^{d-1}:\,\r_\theta(x)\in(2^{-(k+1)}, 2^{-k}]\bigr\},\label{eq:tildeomegak}\\
		\Omega_k(x)&:=\bigl\{\theta\in\S^{d-1}:\,\r_\theta(x)\leq 2^{-k}\bigr\}.\label{eq:omegak}
	\end{align}
	Notice that
	\begin{equation}\label{eq:too}
		\Omega_k=\cup_{j\geq k}\tilde\Omega_j,
	\end{equation}
	where the union is disjoint.

	\begin{lemma}\label{lemma:est_omega_ex}
		For all $E\subset\R^d$ sets of locally finite perimeter and for all $x\in\partial^*E$,
		\begin{align}
			e(x)&\gtrsim\sum_{k\geq0}\frac{|\tilde \Omega_k(x)|^2}{2^{-k(p-d-1)}},\label{eq:etildeo}\\
			e(x)&\gtrsim\sum_{k\geq0}\frac{|\Omega_k(x)|^2}{2^{-k(p-d-1)}}.\label{eq:eo}
		\end{align}
	\end{lemma}
	
	\begin{proof} {The proof of~\eqref{eq:etildeo} follows from the same argument used to prove~\eqref{eq:eomega} in  Lemma~\ref{lemma:epointwisenu}.}
		Indeed, since the $\tilde\Omega_k(x)$ are disjoint,
		\[
		e(x)\gtrsim \sum_{k\geq 0}\int_{\tilde\Omega_k(x)\cap\bigl\{|\Scal{\nu_E(x)}{\theta}|\geq\min\{ 1/10,|\tilde \Omega_k|/2\}\bigr\}}\frac{|\Scal{\nu_E(x)}{\theta}|}{2^{-k(p-d-1)}}\d\theta\gtrsim \sum_{k\geq0}\frac{|\tilde \Omega_k(x)|^2}{2^{-k(p-d-1)}}.
		\]
		
		To prove~\eqref{eq:eo},  we use~\eqref{eq:etildeo}, the disjoint union~\eqref{eq:too} and the following fact: there  exists a constant $C=C(p,d)>0$ such that
		\begin{equation}\label{eq:estoo}
			\sum_{k\geq0}\bigl(\sum_{j\geq k}|\tilde\Omega_j(x)|\bigr)^{2}2^{k(p-d-1)}\leq C\sum_{k\geq0}|\tilde\Omega_k(x)|^{2}2^{k(p-d-1)}.
		\end{equation}
		
		To show~\eqref{eq:estoo}, we set for simplicity of notation $a_k:=|\tilde\Omega_k(x)|2^{k\frac{(p-d-1)}{2}}$.
		With this notation,~\eqref{eq:estoo} reads as
		\begin{equation}
			\label{eq:estooak}
			\sum_{k\geq0}\bigl(\sum_{j\geq k}a_j2^{(k-j)\frac{(p-d-1)}{2}}\bigr)^2\leq C\sum_{k\geq0}a_k^2.
		\end{equation}
		Using the fact that $\sum_{j\geq k}2^{(k-j)\frac{(p-d-1)}{2}}=\frac{1}{1-2^{-\frac{(p-d-1)}{2}}}$ and Jensen's inequality, one has that
		\begin{align}
			\sum_{k\geq0}\bigl(\sum_{j\geq k}a_j2^{(k-j)\frac{(p-d-1)}{2}}\bigr)^2&\leq \frac{1}{1-2^{-\frac{(p-d-1)}{2}}}\sum_{k\geq 0}\sum_{j\geq k}a_{j}^{2} 2^{(k-j)\frac{(p-d-1)}{2}}\notag\\
			&=\frac{1}{1-2^{-\frac{(p-d-1)}{2}}}\sum_{j\geq 0}\sum_{k=0}^{j}a_{j}^{2} 2^{(k-j)\frac{(p-d-1)}{2}}\notag\\
			&=\Biggl(\frac{1}{1-2^{-\frac{(p-d-1)}{2}}}\Biggr)^2\sum_{j\geq 0}a_j^2\Bigl(1-2^{-(j+1)\frac{(p-d-1)}{2}}\Bigr)\notag\\
			&\leq\Biggl(\frac{1}{1-2^{-\frac{(p-d-1)}{2}}}\Biggr)^2\sum_{j\geq 0}a_j^2,\notag
		\end{align}
		corresponding to~\eqref{eq:estooak}.
	\end{proof}
	
	In the next lemma we show an upper bound for the perimeter of sets of finite energy  inside a ball, which is uniform with respect to the centres of the balls and with respect to the family of sets of equibounded energy.
	
	\begin{lemma} \label{lemma:upper_bound_perimeter}Let $d\geq2$, $p\geq 2d$, $M>0$.
		Then,  there exist $R_0>0$, $C_0(M)>0$ such that the following holds: for all $\Omega\subset\R^d$ bounded open sets,  $E\subset \R^d$ of locally finite perimeter such that $\overline \Fcal_{0,p,d}(E,\Omega)\leq M$, for  all $0<r<R_0$ and all $x\in \Omega$ such that $\mathrm{dist}(x,\partial \Omega)>r$,
		\begin{equation}
			\per (E,B_r(x))\leq C_0(M) r^{d-1}.
		\end{equation}
	\end{lemma}
	
	\begin{proof} The claim follows immediately from the following bound:
		\begin{equation*}
			\Biggl(\frac{\per(E,B_r(x))}{\omega_{d-1} r^{d-1}}-1\Biggr)^{p-d}\lesssim{r^{p-2d}}\overline \Fcal_{0,p,d}(E,B_r(x)),
		\end{equation*}
		which will be an immediate consequence of the following two estimates:
		\begin{align}
			&	\Biggl(\frac{\per(E,B_r(x))}{\omega_{d-1} r^{d-1}}-1\Biggr)^{p-d}\lesssim\fint_{\S^{d-1}}\fint_{(B_r(x))_\theta^\perp}\bigl(\per^{\mathrm{1D}}(E_{z_\theta^\perp}, (B_r(x))_{z_\theta^\perp})-1\bigr)^{p-d}\dz_\theta^\perp\d\theta,\label{eq:stima1}\\
			&\fint_{\S^{d-1}}\fint_{(B_r(x))_\theta^\perp}\bigl(\per^{\mathrm{1D}}(E_{z_\theta^\perp}, (B_r(x))_{z_\theta^\perp})-1\bigr)^{p-d}\dz_\theta^\perp\d\theta\lesssim {r^{p-2d}}\overline \Fcal_{0,p,d}(E,B_r(x)).\notag
		\end{align}
		
		The first estimate~\eqref{eq:stima1} is a direct consequence of the slicing formula~\eqref{eq:igper} and Jensen's inequality applied to the convex function $t\mapsto t^{p-d}$.
		
		The second estimate can be deduced as follows, using the integral geometric formulation~\eqref{eq:f0intgeom}, the convexity of the function $t\mapsto t^{-(p-d-1)}$ and the fact that $(d-1)-(p-d-1)=2d-p$
		\begin{align}
			\overline \Fcal_{0,p,d}(E,B_r(x))&=\int_{\S^{d-1}}\int_{\theta^\perp}\sum_{s\in\partial^*E_{z_\theta^\perp}\cap (B_r(x))_{z_\theta^\perp}}\frac{1}{|s-s^+|^{p-d-1}}\dz_\theta^\perp\d\theta\notag\\
			&\geq\int_{\S^{d-1}}\int_{\theta^\perp}\bigl(\per^{\mathrm{1D}}(E_{z_\theta^\perp},(B_r(x))_{z_\theta^\perp})-1\bigr)\cdot\notag\\
			&\cdot \sum_{s, s^+\in\partial^*E_{z_\theta^\perp}\cap (B_r(x))_{z_\theta^\perp}}\frac{1}{\bigl(\per^{\mathrm{1D}}(E_{z_\theta^\perp},(B_r(x))_{z_\theta^\perp})-1\bigr)}\frac{1}{|s-s^+|^{p-d-1}}\dz_\theta^\perp\d\theta\notag\\
			&\geq \int_{\S^{d-1}}\int_{\theta^\perp}\bigl(\per^{\mathrm{1D}}(E_{z_\theta^\perp},(B_r(x))_{z_\theta^\perp})-1\bigr)^{p-d}r^{-(p-d-1)}\dz_\theta^\perp\d\theta \notag\\
			&\gtrsim r^{2d-p}\fint_{\S^{d-1}}\fint_{(B_r(x))_\theta^\perp}\bigl(\per^{\mathrm{1D}}(E_{z_\theta^\perp},(B_r(x))_{z_\theta^\perp})-1\bigr)^{p-d}\dz_\theta^\perp\d\theta.\label{eq:5.23}
		\end{align}
		
	\end{proof}

	We can now give a proof of Proposition~\ref{prop:brezis}.
	
	\begin{proof}[Proof of Proposition~\ref{prop:brezis}]
		Let $2\rho<R_0$, where $R_0$ is chosen  as in Lemma~\ref{lemma:upper_bound_perimeter}.
			We will show that, for all $z\in\Omega$ such that $\mathrm{dist}(z,\partial\Omega)>2\rho$,
			\begin{equation*}
				\int_{\partial^*E\cap B_\rho(z)}\int_{\partial ^*E\cap B_\rho(z)}\frac{\|\nu_E(x)-\nu_E(y)\|^2}{\|x-y\|^{p-2}}\d\mathcal H^{d-1}(x)\d\mathcal H^{d-1}(y)\lesssim C(M) \overline \Fcal_{0,p,d}(E, B_\rho(z)),
			\end{equation*}
			thus proving~\eqref{eq:brezis}.
		\WithoutLoss, we assume that $2\rho=1$.
		Define, for all $k\geq 0$, and for all $x\in\partial^* E$
		\begin{equation*}
			\tilde A_k(x)=\{y\in\partial^*E: \,2^{-k-1}<\|y-x\|\leq2^{-k}\}.
		\end{equation*}
		Notice that
		\begin{equation}\label{eq:ak_omegak}
			y\in \tilde A_k(x)\quad\Leftrightarrow \quad \frac{y-x}{\|y-x\|}\in\tilde \Omega_k(x),
		\end{equation}
		where $\tilde \Omega_k(x)$ was defined in~\eqref{eq:tildeomegak}.
		Then, using the equality $p-2=(p-d-1)+(d-1)$, the equivalence~\eqref{eq:ak_omegak} and the bound (see~\eqref{eq:bgtrsimnu})
		\begin{equation*}
			\|\nu_E(x)-\nu_E(y)\|\lesssim|\Omega_{k-1}(x)|+|\Omega_{k-1}(y)|,\quad \text{ for $|y-x|\leq 2^{-k}$},
		\end{equation*}
		where $\Omega_k(x)$ was defined in~\eqref{eq:omegak},
		we obtain
		\begin{align*}
			&\int_{\partial^*E\cap B_\rho(z)}\int_{\partial ^*E\cap B_\rho(z)}\frac{\|\nu_E(x)-\nu_E(y)\|^2}{\|x-y\|^{p-2}}\d\mathcal H^{d-1}(y)\d\mathcal H^{d-1}(x)\leq	\notag\\
			&\leq\int_{\partial^*E\cap B_\rho(z)}\sum_{k\geq 0}	\int_{\partial ^*E\cap \tilde A_k(x)\cap B_\rho(x)}\frac{\|\nu_E(x)-\nu_E(y)\|^2}{\|x-y\|^{p-2}}\d\mathcal H^{d-1}(y)\d\mathcal H^{d-1}(x)\notag\\
			&\leq\int_{\partial^*E\cap B_\rho(z)}\sum_{k\geq 0}2^{(k+1)(p-d-1)}\int_{\partial ^*E\cap \tilde A_k(x)\cap B_\rho(z)}\frac{\|\nu_E(x)-\nu_E(y)\|^2}{2^{-(k+1)(d-1)}}\d\mathcal H^{d-1}(y)\d\mathcal H^{d-1}(x)\notag\\
			&\lesssim\int_{\partial^*E\cap B_\rho(z)}\sum_{k\geq 0}2^{(k+1)(p-d-1)}\int_{\partial ^*E\cap\tilde  A_k(x)\cap B_\rho(z)}\frac{|\Omega_{k-1}(x)|^2+|\Omega_{k-1}(y)|^2}{2^{-(k+1)(d-1)}}\d\mathcal H^{d-1}(y)\d\mathcal H^{d-1}(x)\notag\\
			&\lesssim\int_{\partial^*E\cap B_\rho(z)}\sum_{k\geq 0}2^{(k+1)(p-d-1)}\bigl(|\Omega_{k-1}(x)|^2+|\Omega_{k-1}(y)|^2\bigr)\frac{\per (E,B_{2^{-k}}(x))}{2^{-(k+1)(d-1)}}\d\mathcal H^{d-1}(x).
		\end{align*}
		Now we recall  Lemma~\ref{lemma:est_omega_ex} and the fact that for $2^{-k}\leq 1< R_0$  Lemma~\ref{lemma:upper_bound_perimeter} holds, thus getting
		\begin{align*}
			&\int_{\partial^*E\cap B_\rho(z)}\int_{\partial ^*E\cap B_\rho(z)}\frac{\|\nu_E(x)-\nu_E(y)\|^2}{\|x-y\|^{p-2}}\d\mathcal H^{d-1}(y)\d\mathcal H^{d-1}(x)\lesssim	\notag\\
			&\lesssim\int_{\partial^*E\cap B_\rho(z)}\sum_{k\geq 0}2^{(k+1)(p-d-1)}\bigl(|\Omega_{k-1}(x)|^2+\Omega_{k-1}(y)|^2\bigr)\frac{\per (E,B_{2^{-k}}(x))}{2^{-(k+1)(d-1)}}\d\mathcal H^{d-1}(x)\notag\\
			&\lesssim C_0(M)\int_{\partial^*E\cap B_\rho(z)} \bigl(e(x)+e(y)\bigr)\d\mathcal H^{d-1}(x)\notag\\
			&\lesssim	C_0(M)\overline \Fcal_{0,p,d}(E, B_\rho(z) ).
		\end{align*}
	\end{proof}
	
	\subsection{(iii) Regularity and  nonlocal curvature bounds imply flatness}\label{subs:5.2}
	
	In the next lemma, we show that the finiteness of the nonlocal curvature quantity
	\[
	\int_{\partial^*E\cap B_\rho(z)}\int_{\partial ^*E\cap B_\rho(z)}\frac{\|\nu_E(x)-\nu_E(y)\|^2}{\|x-y\|^{p-2}}\d\mathcal H^{d-1}(y)\d\mathcal H^{d-1}(x),
	\] proved in Proposition~\ref{prop:brezis}, together with a regularity assumption on the reduced boundary of $E$, implies when $p\geq d+3$ flatness of the boundary of $E$.
	Together with $[0,L)^d$-periodicity this implies then that $E$ has boundaries given by affine hyperplanes orthogonal to a fixed direction, namely it is a union of stripes.

	\begin{lemma}\label{lemma:d+3}
		Let $d\geq 2$, $p\geq d+3$, $\ell>0$, $\bar R>0$.
		Let $\Omega\subset\R^d$ open and bounded and for all $0<r<\bar R$ define $\Omega_{2r}= \{x\in\Omega:\, \mathrm{dist}(x,\partial\Omega)>2r\}$. Let $E\subset \R^d$ of finite perimeter in $\Omega$ which enjoys the following regularity properties:
		\begin{enumerate}
			\item $\partial^*E\cap \Omega=\partial E\cap\Omega$;
			\item For all $0<r<\bar R$ and for every $x\in\partial E\cap \Omega_{2r}$ the set $\partial E\cap B_r(x)$ is given by the graph of an $\ell$-Lipschitz function defined on a connected open subset of a  $(d-1)$-dimensional affine subspace of $\R^d$.
		\end{enumerate}
		If moreover the following holds
		\begin{equation*}
			\sup_{x\in\Omega_{2r},0<r\leq \bar R}\int_{\partial E\cap B_r(x)}\int_{\partial E\cap B_r(x)}\frac{\|\nu_E(z)-\nu_E(y)\|^2}{\|z-y\|^{p-2}}\d\mathcal H^{d-1}(z)\d\mathcal H^{d-1}(y)<+\infty,
		\end{equation*}
		then $\partial E\cap\Omega$ is given by the  disjoint union of the intersections of finitely many affine hyperplanes of $\R^d$ with $\Omega$.
		
		Moreover, if  $\Omega=[0,L)^d$ and in addition $E$ is $[0,L)^d$-periodic, up to a rigid motion the set $E$ satisfying the above is  of the form
		\begin{equation*}
			E=\widehat E\times\R^{d-1},\quad \widehat E\subset\R,\quad\widehat E\cap [0,L)=\bigcup_{i=1}^{N_0}(s_i,t_i), \quad s_i<t_i<s_{i+1}<t_{i+1}.
		\end{equation*}
	\end{lemma}
	
	\begin{proof}
		The statement of the lemma follows immediately from the following two technical tools.
		The first is the Area Formula applied to the bi-Lipschitz parametrization of  $\partial E\cap B_r(x)$ given by the graph of an $\ell$-Lipschitz function.
		The second is   the following result due to Brezis, Bourgain and Mironescu~\cite{bbm,brezis}: given $D\subset\R^{d-1}$ open and connected, $g\in L^\infty(D)$ such that $g\geq c>0$ on $D$ and  $f:D\to\R^{d}$ measurable such that
		\begin{equation*}
			\int_{D}\int_{D}\frac{\|f(x)-f(y)\|^2}{\|x-y\|^{d+1}}g(x)g(y)\dx\dy<+\infty,
		\end{equation*}
		there exists a constant $c\in\R^{d}$ such that $f=c$ \ae on $D$.
		
		Indeed, if $\phi:D\subset \pi_{d-1}\to\R$ is the $\ell$-Lipschitz function such that
		$\partial E\cap B_r(x)=\{(z',\phi(z')):\,z'\in D\}$, by the Area Formula
		\begin{align*}
			&\int_{\partial E\cap B_r(x)}\int_{\partial E\cap B_r(x)}\frac{\|\nu_E(z)-\nu_E(y)\|^2}{\|z-y\|^{p-2}}\d\mathcal H^{d-1}(z)\d\mathcal H^{d-1}(y)\gtrsim\notag\\
			&\gtrsim	\int_{D}\int_{D}\frac{\|\nu_E((z',\phi(z')))-\nu_E((y',\phi(y')))\|^2}{\|z'-y'\|^{p-2}}J_{\mathrm{graph}\phi}(z')J_{\mathrm{graph}\phi}(y')\,dz'\,dy',
		\end{align*}
		where $J_{\mathrm{graph}\phi}$ is the Jacobian associated to  the bi-Lipschitz map $\mathrm{graph}\phi$. Notice that $J_{\mathrm{graph}\phi}\geq c>0$ due to the fact that the graph of $\phi$ is a bi-Lipschitz map.
		
		Applying then the result of Brezis, Bourgain and Mironescu recalled above to
		$f(z)=\nu_E((z,\phi(z)))$ and $g(z)=J_{\mathrm{graph}\phi}(z)$,   one has  that $\nu_E(x)=\nu$ $\mathcal H^{d-1}\llcorner \partial^*E$-\ae inside $B_r(x)$, $\nu\in\S^{d-1}$.
		Then, by the standard characterization of hyperplanes in geometric measure theory and the fact that the radii of the balls on which such result holds  are uniform with respect to  $x\in \partial E\cap \Omega_{2r}$,  for all $0<r<\bar R$, gives the first statement of the Lemma.
		
		The fact that the connected components of the boundary are all orthogonal to a single  direction, and thus $E$ is a union of stripes, is a consequence of the $[0,L)^d$-periodicity of $E$.
	\end{proof}
	
	\subsection{(ii) Regularity of sets of finite energy}\label{subs:5.3}
	The aim of this section is to give a proof of the regularity Theorem~\ref{thm:regularity}
	for sets of finite energy.   The proof of Theorem~\ref{thm:regularity} goes through a series of preliminary lemmas.
	As classical in regularity theory, we   look for uniform lower and upper density bounds on perimeter and volume of sets of equibounded energy at boundary points  and for  power law decay of the excess.
	
	In the next lemma, we show uniform  lower bounds on perimeter and volume at
	points of the topological boundary of $E$.
	\begin{lemma}\label{lemma:per_low_bound}
		Let $d\geq2$, $p\geq 2d$, $M>0$.
		Then, there exist $C_1,\bar C_1, R_1>0$ such that the following holds.
		For all $\Omega\subset\R^d$ open and bounded, for all $E\subset\R^d$ such that $\overline \Fcal_{0,p,d}(E,\Omega)\leq M$,  for all $0<r<R_1$ and for all $x\in\partial E \cap \Omega_{r}$ (where $\Omega_{r}=\{z\in\Omega:\,\mathrm{dist}(z,\partial \Omega)>r\}$)
		\begin{align}
			\min\bigl\{|E\cap B_r(x)|,|B_r(x)\setminus E|\bigr\}\geq \bar C_1 r^d,\label{eq:vol_low_bound_0}\\
			\per(E, B_r(x))\geq C_1 r^{d-1}.\label{eq:per_low_bound}
		\end{align}
	\end{lemma}
	
	\begin{proof}
		By the isoperimetric inequality, the proof of~\eqref{eq:per_low_bound}  follows immediately from~\eqref{eq:vol_low_bound_0}.

		In order to prove~\eqref{eq:vol_low_bound_0}, let us consider \withoutLoss a point $x\in\partial E$ such that  $	\min\bigl\{|E\cap B_r(x)|, |B_r(x)\setminus E|\bigr\}=|E\cap B_r(x)|$ and
		\begin{equation}\label{eq:density_b}
			|E\cap B_r(x)|=c_1\delta^{d} r^{d},
		\end{equation}
		for some $\delta<1/2$, $r<R_1$ and $R_1$,$c_1$ sufficiently small to be fixed later.

		Then define, for $y\in\partial^*E\cap B_r(x)$, the sets
		\begin{align*}
			\Theta(y)&:=\bigl\{\theta\in\S^{d-1}:\,\langle \nu_E(y),\theta\rangle<0, \,|y_\theta-y_\theta^+|<\dist(y, \partial B_r(x))\bigr\},\notag\\
			V(y)&:=\bigl\{z\in\R^d:\,z=y+s\theta, \theta\in\Theta(y) , s\in(0,|y_\theta-y_\theta^+|)\bigr\}.
		\end{align*}
		Notice that
		\begin{equation}\label{eq:vincl}
			V(y)\subset E\cap B_r(x).
		\end{equation}
		
		Moreover, if $c_1$ in~\eqref{eq:density_b} is sufficiently small, namely
		\begin{equation}\label{eq:c1cond}
			c_1<\frac{\omega_d}{4},
		\end{equation}then for all $0<\bar \omega_1\ll1$
		\begin{equation}\label{eq:thetab}
			|\Theta(y)|\geq\bar  \omega_1\text{ for all $y\in\partial^*E\cap B_{r(1-\delta)}(x)$}.
		\end{equation}
		Indeed, by the properties~\eqref{eq:stimeblowup} of the reduced boundary, for all $\theta\in \S^{d-1}$ such that $\langle\nu_E(y),\theta\rangle<0$ the segment $[y, y_\theta^\perp+y_\theta^+\theta]$ is contained in $E$.
		Thus, if~\eqref{eq:thetab} does not hold,   there exists  $y\in \partial^*E\cap B_{r(1-\delta)}(x)$ and  $|y_\theta-y_\theta^+|\geq\dist(y, \partial B_r(x))\geq \delta r$ for a set of $\theta\in \{\langle \nu_E(y),\theta\rangle<0\} \setminus \Theta(y)$ of measure greater than $\frac12\mathcal H^{d-1}(\S^{d-1})-\bar \omega_1\geq \frac14\mathcal H^{d-1}(\S^{d-1})\geq \omega_d/2$, thus implying by~\eqref{eq:c1cond}
		\begin{align*}
			\bigl |\bigl\{z\in\R^d:\,z=y+s\theta, \theta\in \{\langle \nu_E(y),\theta\rangle<0\}\setminus \Theta(y),  s\leq\delta r\bigr\}\bigr|&\geq|\S^{d-1}\setminus \Theta(y)| (\delta r)^d\notag\\
			&\geq\frac{\omega_d}{2}(\delta r)^d\notag\\
			&\geq 2c_1\delta^{d}r^d,
		\end{align*}
		which, by the fact that  for $y\in\partial^*E\cap B_{r(1-\delta)}(x)$
		\begin{equation*}
			\bigl\{z\in\R^d:\,z=y+s\theta, \theta\in \{\langle \nu_E(y),\theta\rangle<0\}\setminus \Theta(y),  s\leq\delta r\bigr\}\subset E\cap B_r(x),
		\end{equation*} contradicts~\eqref{eq:density_b}.
		
		In particular, by~\eqref{eq:thetab}, one has that
		\begin{equation}
			\label{eq:vyb}
			|V(y)|>0\quad\forall\,y\in\partial^*E\cap B_{r(1-\delta)}(x).
		\end{equation}
		
		Using~\eqref{eq:thetab}, Jensen's inequality for the convex function $t\mapsto t^{-(p-d-1)/d}$, the inequality $|\langle \nu_E(y),\theta\rangle|\leq~1$ and~\eqref{eq:vyb},  we get for some constant $C(\bar\omega_1)$
		
		\begin{align*}
			\overline \Fcal_{0,p,d}(E, B_r(x))&\geq \int_{\partial^*E\cap B_r(x)}\int_{\Theta(y)}\frac{|\langle \nu_E(y), \theta\rangle|}{(\r_\theta(y))^{p-d-1}}\d\theta\d\mathcal H^{d-1}(y)\notag\\
			&\geq C(\bar \omega_1)\int_{\partial^*E\cap B_{r(1-\delta)}(x)\cap \{y:|V(y)|>0\}}\Bigr(\int_{\Theta(y)}(\r_\theta(y))^d\d\theta\Bigl)^{-\frac{p-d-1}{d}}\d\mathcal H^{d-1}(y)\notag\\
			&\geq C(\bar \omega_1)\int_{\partial^*E\cap B_{r(1-\delta)}(x)}|V(y)|^{-\frac{p-d-1}{d}}\d\mathcal H^{d-1}(y).
		\end{align*}
		Now we use the following facts: first, since  $x\in\partial E$, $\partial^*E\cap B_{r(1-\delta)}(x)\neq\emptyset$ for all $r>0$, then as noticed in~\eqref{eq:vincl} $V(y)\subset E\cap B_r(x)$ and finally the isoperimetric inequality holds.
		Thus, we  obtain
		\begin{align}
			\overline \Fcal_{0,p,d}(E, B_r(x))&\geq C(\bar \omega_1)\frac{\per (E, B_{r(1-\delta)}(x))}{|E\cap B_r(x)|^{(p-d-1)/d}}\notag\\
			&\geq C(\bar \omega_1)\frac{\bigl(\min\bigl\{|E\cap B_{r(1-\delta)}(x)|, |B_{r(1-\delta)}(x)\setminus E|\bigr\}\bigr)^{(d-1)/d}}{|E\cap B_r(x)|^{(p-d-1)/d}}.\label{eq:fisoper}
		\end{align}
		Now we claim that, under the assumptions~\eqref{eq:density_b},~\eqref{eq:c1cond} and $\delta <1/2$,   it holds
		\begin{equation}
			\label{eq:emin}
			\min\bigl\{|E\cap B_{r(1-\delta)}(x)|, | B_{r(1-\delta)}(x)\setminus E|\bigr\}=     |E\cap B_{r(1-\delta)}(x)|.
		\end{equation}
		Indeed, let us assume that~\eqref{eq:emin} does not hold and define
		\begin{align*}
			\alpha(r)&:=\frac{|E\cap B_r(x)|}{r^d},\notag\\
			A_{r,\delta}&=B_r(x)\setminus B_{r(1-\delta)}(x).
		\end{align*}
		On the one hand, one has that (since 	$|E\cap B_{r(1-\delta)}(x)|>|B_{r(1-\delta)}(x)\setminus E|$)
		\begin{equation}\label{eq:alphalow}
			\alpha(r(1-\delta))>\frac{\omega_d}{2}.
		\end{equation}
		On the other hand, by assumptions~\eqref{eq:density_b} and~\eqref{eq:c1cond},								\begin{equation*}
			\alpha(r)=c_1\delta^d<\frac{\delta^d\omega_d}{4} \end{equation*}
		and thus since $\delta<1/2$
		\begin{align*}
			\alpha(r(1-\delta))&=\frac{|E\cap B_{r(1-\delta)}(x)|}{(1-\delta)^{d} r^d}
			\leq \frac{|E\cap B_{r}(x)|}{(1-\delta)^{d} r^d}=\frac{\alpha(r)}{(1-\delta)^{d}}\leq \frac{c_1\delta^d}{(1-\delta)^d}\leq\frac{\omega_d}{16} ,
		\end{align*}
		which contradicts~\eqref{eq:alphalow}.
		
		Hence, by~\eqref{eq:emin} and the fact that $p\geq 2d$, the lower bound~\eqref{eq:fisoper} becomes
		\begin{align}
			\overline \Fcal_{0,p,d}(E, B_r(x))&\geq C(\bar \omega_1)\Bigl(\frac{|E\cap B_{r(1-\delta)}(x)|}{|E\cap B_r(x)|}\Bigr)^{(d-1)/d}.\label{eq:fisoper2}
		\end{align}
		
		Let us  now define
		\begin{align*}
			\mu(A)&:=\overline \Fcal_{0,p,d}(E, A),\quad\text{ $A\subset\Omega$ Borel. }
		\end{align*}
		
		Then,~\eqref{eq:fisoper2} can be rewritten as
		\begin{align*}
			\mu(B_r(x))&\geq C(\bar \omega_1)\Bigl(\frac{\alpha(r(1-\delta))}{\alpha(r)}\Bigr)^{(d-1)/d}(1-\delta)^{(d-1)}.
		\end{align*}
		Since $\mu\ll\mathcal H^{d-1}\llcorner(\partial^*E\cap \Omega)$, $\mu(\Omega)\leq M<+\infty$, and by Lemma~\ref{lemma:upper_bound_perimeter} $\mathcal H^{d-1}(\partial^*E\cap B_\rho(x))\leq C_0\rho^{d-1}$ for all $\rho\leq r\leq R_0$ and for all $x\in\partial ^*E\cap \Omega_{\rho}$, one has  that for  every $0<\gamma\ll1$ there exists $0< \bar R(\gamma)<R_0$ such that  for all $r<\bar R(\gamma)$ and for all $x\in \Omega_{r}$ it holds
		\begin{align}\label{eq:mubound}
			\mu(B_r(x))\leq \gamma C(\bar\omega_1).
		\end{align}
		In particular, if $\gamma$ is smaller than a dimensional constant,
		\begin{equation*}
			\alpha(r(1-\delta))\leq\alpha(r)\Bigl(\frac{\gamma}{(1-\delta)^{d-1}}\Bigr)^{d/(d-1)}\leq\frac{1}{2^d}\alpha(r).
		\end{equation*}
		Let us choose $R_1=\bar R(\gamma)$ such that~\eqref{eq:mubound} holds for all $r<\bar R(\gamma)$, where $\gamma$ is such that for all $\delta<1/2$
		\[
		\Bigl(\frac{\gamma}{(1-\delta)^{d-1}}\Bigr)^{d/(d-1)}\leq\frac{1}{2^d}.
		\]
		
		Let us then rename $r_0:=r<R_1$, $\delta_0:=\delta$, $\alpha_0:=\alpha(r_0)$, $r_1:=r_0(1-\delta_0)$, $\alpha_1:=\alpha(r_1)$.
		Observe that, by assumption~\eqref{eq:density_b}
		\[
		\frac{\alpha_0}{c_1}=\delta_0^d.
		\]
		Using the assumption~\eqref{eq:c1cond} on $c_1$ and the fact that $\delta_0<1/2$, we were just able to show that
		\begin{equation}\label{eq:alpha10}
			\alpha_1\leq\frac{1}{2^d}\alpha_0.
		\end{equation}
		Now define iteratively, for $i\in\N$,
		\begin{align}
			\delta_i&= \Bigl(\frac{\alpha_i}{c_1}\Bigr)^{1/d}\notag\\
			r_{i+1}&=r_{i}(1-\delta_{i}), \notag\\
			\alpha_{i+1}&=\alpha(r_{i+1}).
		\end{align}
		Notice that, by~\eqref{eq:alpha10}, $\delta_i\leq\frac12\delta_{i-1}\leq \frac12$ and thus one can reason as before getting the analogue of~\eqref{eq:alpha10} for all $i\in\N$, namely
		\begin{equation*}
			\alpha_{i+1}\leq\frac{1}{2^d}\alpha_i.
		\end{equation*}
		In particular,
		\begin{equation}\label{eq:limai}
			\lim_{i\to\infty}	\alpha_i=0, \lim_{i\to+\infty}\delta_i=0.
		\end{equation}
		We now claim that
		\begin{equation}\label{eq:barr2}
			\exists\,\bar r>0:\quad r_i\geq \bar r>0\quad\text{ for all $i\in\N$.}
		\end{equation}
		Once~\eqref{eq:barr2} is proved, we can easily conclude since by~\eqref{eq:limai} one would have that $\alpha(\bar r)=0$, thus contradicting the fact that $x\in\partial E$ and thus by~\eqref{eq:topbdry} that $\alpha(r)>0$ for all $r$.
		In order to prove~\eqref{eq:barr2}, observe that by definition of $r_i$,  and by the fact that $\delta_i\leq \Bigl(\frac12\Bigr)^{1/d}\delta_{i-1}<\frac12$, one has that there exists $C>0$ such that
		\begin{align*}
			\ln r_i=\ln r_0+\sum_{k=0}^{i-1}\ln(1-\delta_k)\geq\ln r_0+\sum_{i=0}^\infty\delta_i-2\sum_{i=0}^{\infty}\delta_i^2\geq -C>-\infty.
		\end{align*}
		Thus, we reached a contradiction to the assumption~\eqref{eq:density_b} and thus the Lemma is proved by~\eqref{eq:density_b} and~\eqref{eq:c1cond} choosing any \[
		0<\bar C_1<\frac{\omega_d}{4}(1/2)^{d}
		\] and $R_1$ as above.

	\end{proof}
	
	Before stating the next lemma, we recall the definition of  (spherical) excess given in~\eqref{eq:exc}, namely
	\begin{equation*}
		Exc(E,x,r)=\frac{1}{r^{d-1}}\bigl[|D\chi_E|(B_r(x))-|D\chi_E(B_r(x))|\bigr].
	\end{equation*}

	In order to show regularity, we will need power law decay of the excess, with uniform constants which are independent of the point $x$.
	More precisely, we prove the following.

	\begin{lemma}
		\label{lemma:excf}
		Let $d\geq 2$, $p>2d$, $M>0$.
		There exist $R_2,C_2(M)>0$ such that for all $\Omega\subset\R^d$ bounded and open  and $E$ of locally finite perimeter such that $\overline \Fcal_{0,p,d}(E,\Omega)\leq M$, for all $0<r<R_2$ and for all $x\in\Omega_r$,
		\begin{equation}\label{eq:fexc}
			Exc(E,x,r)\leq C_2(M)r^{(p-2d)/\max\{p-d,8d\}}.
		\end{equation}
	\end{lemma}
	
	\begin{proof}
		\WithoutLoss we assume that $x=0$ and we denote by $B_r=B_r(0)$.
		
		We first claim that in general dimension $d$  the following integral geometric formula holds
		\begin{align*}
			r^{d-1}Exc(E,0,r)&=\frac{1}{C_{1,d}}\int_{\S^{d-1}}\int_{\partial^*E\cap B_r}|\langle\nu_E(y), \theta\rangle|\d\mathcal H^{d-1}(y)\d\theta\notag\\
			&-\frac{1}{C_{1,d}}\int_{\S^{d-1}}\Biggl\|\int_{\partial^*E\cap B_r}\langle\nu_E(y), \theta\rangle\d\mathcal H^{d-1}(y)\Biggr\|\d\theta,
		\end{align*}
		where $C_{1,d}$ is the constant defined in~\eqref{eq:Cd}.
		
		One can deduce the above directly from the definition of the excess,  from the identity
		\begin{equation*}
			\|z\|=\frac{1}{C_{1,d}}\int_{\S^{d-1}}|\langle z,\theta\rangle|\d\theta, \quad z\in\R^d
		\end{equation*}
		applied to $z=\int_{\partial^*E\cap B_r}\nu_E(y)\d\mathcal H^{d-1}(y)$ and the formula~\eqref{eq:igper}.
		In particular, by the classical slicing formulas~\eqref{eq:slicing} and~\eqref{eq:slicing_sign} one has that
		\begin{align}
			r^{d-1}Exc(E,0,r)&=\frac{1}{C_{1,d}}\int_{\S^{d-1}}\int_{\theta^\perp}\sum_{s\in\partial^*E_{z_\theta^\perp}\cap (B_r)_{z_\theta^\perp}}1\dz_\theta^\perp\d\theta\notag\\
			&-\frac{1}{C_{1,d}}\int_{\S^{d-1}}\Biggl|\int_{\theta^\perp}\sum_{s\in\partial^*E_{z_\theta^\perp}\cap (B_r)_{z_\theta^\perp}}\mathrm{sign}(\langle\nu_E(z_\theta^\perp+s\theta),\theta\rangle)\dz_\theta^\perp\Biggr|\d\theta.\label{eq:146}
		\end{align}
		
		In order to show~\eqref{eq:fexc}, we define
		\begin{align*}
			\Omega_1(\theta^\perp)=\{z_\theta^\perp\in\theta^\perp:\per^{\mathrm{1D}}(E_{z_\theta^\perp},(B_r)_{z_\theta^\perp})=1\},\notag\\
			\Omega_2(\theta^\perp)=\{z_\theta^\perp\in\theta^\perp:\per^{\mathrm{1D}}(E_{z_\theta^\perp},(B_r)_{z_\theta^\perp})\geq 2\}.
		\end{align*}
		
		Using the triangle inequality we bound~\eqref{eq:146} in the following way
		\begin{align}
			r^{d-1}Exc(E,0,r)&\leq\frac{1}{C_{1,d}}\int_{\S^{d-1}}\Biggl[\int_{\Omega_1(\theta^\perp)}1\dz_\theta^\perp-\Biggl|\int_{\Omega_1(\theta^\perp)}\mathrm{sign}(\langle\nu_E(z_\theta^\perp+z_\theta\theta),\theta\rangle)\dz_\theta^\perp\Biggr|\Biggr]\d\theta\notag\\
			&+\frac{1}{C_{1,d}}\int_{\S^{d-1}}\int_{\Omega_2(\theta^\perp)}2\per^{\mathrm{1D}}(E_{z_\theta^\perp}, (B_r)_{z_\theta^\perp})\dz_\theta^\perp\d\theta.\label{eq:omega12}
		\end{align}
		
		Now observe that, as in~\eqref{eq:5.23},  using the formula~\eqref{eq:f0intgeom}, Jensen's inequality  and the fact that whenever $z_\theta^\perp\in\Omega_2(\theta^\perp)$ then $\exists\,s,s^+\in\partial^*E_{z_\theta^\perp}\cap (B_r)_{z_\theta^\perp}$ and thus  $|s-s^+|\leq r$, one obtains the lower bound
		\begin{align*}
			\overline\Fcal_{0,p,d}(E,B_r)&\gtrsim\int_{\S^{d-1}}\int_{(B_r)_\theta^\perp}\frac{(\per^{\mathrm{1D}}(E_{z_\theta^\perp}, (B_r)_{z_\theta^\perp})-1)^{p-d}}{r^{p-d-1}}\dz_\theta^\perp\d\theta\notag\\
			&\gtrsim r^{2d-p}r^{(1-d)(p-d)}\Biggl(\int_{\S^{d-1}}\int_{(B_r)_\theta^\perp}[\per^{\mathrm{1D}}(E_{z_\theta^\perp}, (B_r)_{z_\theta^\perp})-1]\dz_{\theta}^\perp\d\theta\Biggr)^{p-d}.
		\end{align*}
		Hence, the part of the excess relative to the second term in~\eqref{eq:omega12} can be bounded in the following way:
		\begin{align}
			\frac{1}{r^{d-1}}	\int_{\S^{d-1}}\int_{\Omega_2(\theta^\perp)}\per^{\mathrm{1D}}(E_{z_\theta^\perp}, (B_r)_{z_\theta^\perp})\dz_\theta^\perp\d\theta&\lesssim\frac{1}{r^{d-1}}\int_{\S^{d-1}}\int_{\Omega_2(\theta^\perp)}{\bigl[\per^{\mathrm{1D}}(E_{z_\theta^\perp}, (B_r)_{z_\theta^\perp})-1\bigr]}\dz_\theta^\perp\d\theta\notag\\
			&\lesssim r^{(p-2d)/(p-d)}\overline\Fcal_{0,p,d}(E, B_r)^{1/(p-d)},
		\end{align}
		implying in particular (by the boundedness of the functional on $\Omega$) the decay for the excess given in~\eqref{eq:fexc}.
		
		The estimate of the first term in~\eqref{eq:omega12} is instead more involved, due to the necessity of more precise estimates in the case of cancellations inside the second integral.
		
		We denote by
		\begin{align*}
			I&=\int_{\S^{d-1}}\Biggl[\int_{\Omega_1(\theta^\perp)}1\dz_\theta^\perp-\Biggl|\int_{\Omega_1(\theta^\perp)}\mathrm{sign}(\langle\nu_E(z_\theta^\perp+z_\theta\theta),\theta\rangle)\dz_\theta^\perp\Biggr|\Biggr]\d\theta,\notag\\
			I_\theta&=\int_{\Omega_1(\theta^\perp)}1\dz_\theta^\perp-\Biggl|\int_{\Omega_1(\theta^\perp)}\mathrm{sign}(\langle\nu_E(z_\theta^\perp+z_\theta\theta),\theta\rangle)\dz_\theta^\perp\Biggr|,\quad\theta\in\S^{d-1}
		\end{align*}
		and choose $\bar\theta\in\S^{d-1}$ an angle such that $\bar\theta \in \mathrm{argmax}_\theta I_{\theta}$.
		Such a maximum point always exists as the map $\theta \to I_{\theta}$  is continuous and, by Lemma~\ref{lemma:upper_bound_perimeter},
		\begin{align*}
			I_\theta\lesssim\per(E,B_r)\lesssim r^{d-1}.
		\end{align*}
		Then, decompose $\Omega_1(\bar\theta^\perp)$ as follows:
		\begin{align*}
			\Omega_1(\bar\theta^\perp)&=\Omega_1^+\cup\Omega_1^-,\notag\\
			\Omega_1^+&=\{z_{\bar\theta}^\perp\in\Omega_1(\bar\theta^\perp):\,\langle \nu_E(z_{\bar\theta}^\perp+s(z_{\bar\theta}^\perp)\bar\theta), \bar\theta\rangle>0\},\notag\\
			\Omega_1^-&=\{z_{\bar\theta}^\perp\in\Omega_1(\bar\theta^\perp):\,\langle \nu_E(z_{\bar\theta}^\perp+s(z_{\bar\theta}^\perp)\bar\theta), \bar \theta\rangle<0\}.
		\end{align*}
		whereby $s(z_{\bar \theta}^\perp)$ we denote the point such that
		$\partial^*E_{z_{\bar\theta}^\perp}\cap (B_r)_{z_\theta^\perp}=\{z_{\bar\theta}^\perp+s(z_{\bar\theta}^\perp)\bar \theta\}$.

		Assume \withoutLoss that $|\Omega_1^+|\geq|\Omega_1^-|$ (the other case can be treated analogously).
		Notice that, under this non-restrictive assumption,
		\begin{equation}
			I_{\bar\theta}=2|\Omega_1^-|.
		\end{equation}
		The proof of the Lemma reduces to show that $|\Omega_1^-|\lesssim r^{d-1+\alpha}$, where $\alpha=(p-2d)/(8d)$.
		
		Assume on the contrary that
		\begin{equation}\label{eq:o1-alpha}
			|\Omega_{1}^{-}|\geq Cr^{d-1 + \alpha}
		\end{equation}
		and denote by $\beta = p-2d$.
		
		Notice that
		\begin{align}
			e(z)&:= \int_{\S^{d-1}} \frac{|\Scal{\nu_{E}(z)}{\theta}|}{\r_\theta(z)^{d-1+\beta}} \d\theta\notag\\
			&\geq\int_{\insieme{\theta \in \S^{d-1}:\ \r_\theta(z) < r}}\frac{|\Scal{\nu_E(z)}{\theta}|}{r^{d-1+\beta}}\d\theta\notag\\
			&\geq \hat C\frac{	\Big|\insieme{\theta \in \S^{d-1}:\ \r_\theta(z) < r}\Big|^2}{r^{d-1+\beta}}. \label{eq:etheta}
		\end{align}
		
		Let now $H_{\nu_E(z)}(z)$ be the affine halfspace given by $\{y:\,\Scal{y-z}{\nu_E(z)}<0\}$.
		By the blow up properties of the reduced boundary~\eqref{eq:stimeblowup}, whenever $\Scal{\nu_E(z)}{\theta}>0$, then $(z, z+\r_\theta(z)\theta)\subset \R^d\setminus E$ and whenever $\Scal{\nu_E(z)}{\theta}<0$, then $(z, z+\r_\theta(z)\theta)\subset E$.
		By~\eqref{eq:etheta},  we have that whenever $z\in\partial^*E$ is such that $e(z) < \bar C_1\hat C r^{-d+1-\beta/2}/4$, then  $|\insieme{\theta \in \S^{d-1}:\ \r_\theta(z) < r}|<\bar C_1r^{\beta/4}/4$, where $\bar C_1$ is the constant of Lemma~\ref{lemma:per_low_bound}.
		In particular, by the above consequences of the blow up properties at points of the reduced boundary, there exists a cone with vertex in $z$ and of angles of total measure  $r^{\beta/4}$, that we denote by $\tilde C:=\tilde C(z, \nu_{E}(z), r^{\beta/4})$ such that (see Figure~\ref{fig:3})
		\begin{equation}\label{eq:cone}
			\|\chi_{H_{\nu_E(z)}(z)}-\chi_{E}\|_{L^1(B_r)}\leq \|\chi_{\tilde C}\|_{L^1(B_r)}\leq\frac{\bar C_1}{4} r^{ d + \beta/4}.
		\end{equation}
		
		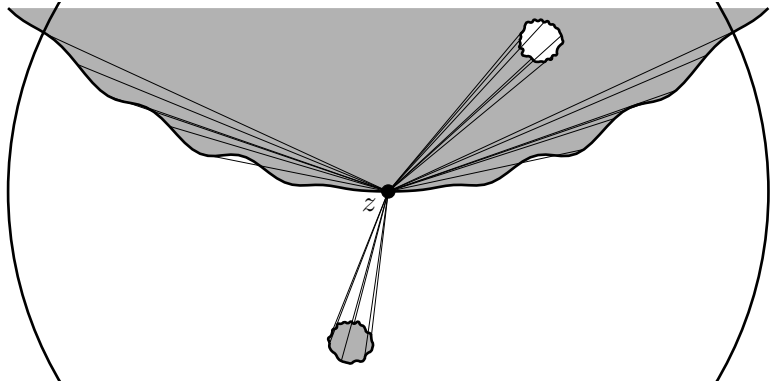
\begin{figure}
			\centering
			\begin{tikzpicture}
				\begin{scope}[scale=2.5]
					\tikzset{arrow/.style={-latex}}
					\clip (-2.1,-1) rectangle (2.1,1);
					\draw[line width=1pt,fill=black!30] (-2,0.97264) -- (-1.9596,0.93536) -- (-1.91919,0.90551) -- (-1.87879,0.88104) -- (-1.83838,0.85858) -- (-1.79798,0.83426) -- (-1.75758,0.80467) -- (-1.71717,0.7678) -- (-1.67677,0.72368) -- (-1.63636,0.67443) -- (-1.59596,0.62388) -- (-1.55556,0.5766) -- (-1.51515,0.53674) -- (-1.47475,0.50693) -- (-1.43434,0.4875) -- (-1.39394,0.47625) -- (-1.35354,0.46898) -- (-1.31313,0.46045) -- (-1.27273,0.44579) -- (-1.23232,0.42177) -- (-1.19192,0.38776) -- (-1.15152,0.34596) -- (-1.11111,0.30089) -- (-1.07071,0.25829) -- (-1.0303,0.22354) -- (-0.9899,0.20023) -- (-0.94949,0.18916) -- (-0.90909,0.18804) -- (-0.86869,0.19216) -- (-0.82828,0.19563) -- (-0.78788,0.19306) -- (-0.74747,0.18108) -- (-0.70707,0.15927) -- (-0.66667,0.1302) -- (-0.62626,0.09862) -- (-0.58586,0.06991) -- (-0.54545,0.0484) -- (-0.50505,0.03605) -- (-0.46465,0.03199) -- (-0.42424,0.03302) -- (-0.38384,0.03497) -- (-0.34343,0.03445) -- (-0.30303,0.03009) -- (-0.26263,0.02286) -- (-0.22222,0.01513) -- (-0.18182,0.00901) -- (-0.14141,0.00505) -- (-0.10101,0.00255) -- (-0.06061,0.00092) -- (-0.0202,0.0001) -- (0.0202,0.0001) -- (0.06061,0.00092) -- (0.10101,0.00255) -- (0.14141,0.00505) -- (0.18182,0.00901) -- (0.22222,0.01513) -- (0.26263,0.02286) -- (0.30303,0.03009) -- (0.34343,0.03445) -- (0.38384,0.03497) -- (0.42424,0.03302) -- (0.46465,0.03199) -- (0.50505,0.03605) -- (0.54545,0.0484) -- (0.58586,0.06991) -- (0.62626,0.09862) -- (0.66667,0.1302) -- (0.70707,0.15927) -- (0.74747,0.18108) -- (0.78788,0.19306) -- (0.82828,0.19563) -- (0.86869,0.19216) -- (0.90909,0.18804) -- (0.94949,0.18916) -- (0.9899,0.20023) -- (1.0303,0.22354) -- (1.07071,0.25829) -- (1.11111,0.30089) -- (1.15152,0.34596) -- (1.19192,0.38776) -- (1.23232,0.42177) -- (1.27273,0.44579) -- (1.31313,0.46045) -- (1.35354,0.46898) -- (1.39394,0.47625) -- (1.43434,0.4875) -- (1.47475,0.50693) -- (1.51515,0.53674) -- (1.55556,0.5766) -- (1.59596,0.62388) -- (1.63636,0.67443) -- (1.67677,0.72368) -- (1.71717,0.7678) -- (1.75758,0.80467) -- (1.79798,0.83426) -- (1.83838,0.85858) -- (1.87879,0.88104) -- (1.91919,0.90551) -- (1.9596,0.93536) -- (2,0.97264);
					\draw[line width=1pt,smooth,tension=1,fill=white] plot coordinates {(0.68897,0.8) (0.69503,0.78294) (0.69051,0.76345) (0.70945,0.75247) (0.70409,0.72793) (0.72147,0.71824) (0.74233,0.71646) (0.75526,0.70571) (0.76853,0.69137) (0.78633,0.68742) (0.80445,0.6896) (0.8229,0.68784) (0.83713,0.7021) (0.85752,0.70037) (0.865,0.72039) (0.88497,0.72472) (0.88733,0.74478) (0.89241,0.76063) (0.90326,0.77455) (0.9185,0.79043) (0.91477,0.80927) (0.91098,0.82735) (0.90234,0.8436) (0.89507,0.86012) (0.87925,0.87021) (0.86903,0.88455) (0.85055,0.88756) (0.83896,0.90274) (0.82001,0.89801) (0.80446,0.91076) (0.78735,0.90416) (0.76794,0.91069) (0.75405,0.89684) (0.73232,0.89806) (0.72022,0.88306) (0.71606,0.86308) (0.70822,0.84817) (0.69675,0.83447) (0.69796,0.81658) (0.68897,0.8)};
					\draw[line width=1pt,smooth,tension=1,fill=black!30] plot coordinates {(-0.3066,-0.8) (-0.31021,-0.81791) (-0.305,-0.83505) (-0.29399,-0.84933) (-0.28002,-0.86013) (-0.27097,-0.87389) (-0.26138,-0.88893) (-0.24701,-0.89908) (-0.23029,-0.90458) (-0.21323,-0.90898) (-0.1955,-0.91162) (-0.17927,-0.90156) (-0.1629,-0.89783) (-0.14128,-0.90171) (-0.12624,-0.89034) (-0.11995,-0.87092) (-0.10851,-0.85785) (-0.09083,-0.84651) (-0.08352,-0.82871) (-0.08065,-0.80964) (-0.0873,-0.7909) (-0.08959,-0.77279) (-0.09388,-0.75479) (-0.11383,-0.74551) (-0.12113,-0.73013) (-0.13027,-0.71459) (-0.14641,-0.70719) (-0.15872,-0.69115) (-0.17981,-0.70109) (-0.19559,-0.69055) (-0.21235,-0.69829) (-0.23091,-0.6933) (-0.2468,-0.70137) (-0.26646,-0.70372) (-0.27263,-0.72439) (-0.2909,-0.73169) (-0.29557,-0.74984) (-0.3022,-0.76588) (-0.31585,-0.78117) (-0.3066,-0.8)};
					\draw[fill=black] (0,0.00003) circle[radius=1pt];
					\draw[opacity=0.8,line width=.1pt] (0,0.00003) -- (0.68897,0.8);
					\draw[opacity=0.8,line width=.1pt] (0,0.00003) -- (0.75526,0.70571);
					\draw[opacity=0.8,line width=.1pt] (0,0.00003) -- (0.865,0.72039);
					\draw[opacity=0.8,line width=.1pt] (0,0.00003) -- (0.91098,0.82735);
					\draw[opacity=0.8,line width=.1pt] (0,0.00003) -- (0.82001,0.89801);
					\draw[opacity=0.8,line width=.1pt] (0,0.00003) -- (0.71606,0.86308);
					\draw[opacity=0.8,line width=.1pt] (0,0.00003) -- (-0.3066,-0.8);
					\draw[opacity=0.8,line width=.1pt] (0,0.00003) -- (-0.24701,-0.89908);
					\draw[opacity=0.8,line width=.1pt] (0,0.00003) -- (-0.12624,-0.89034);
					\draw[opacity=0.8,line width=.1pt] (0,0.00003) -- (-0.08959,-0.77279);
					\draw[opacity=0.8,line width=.1pt] (0,0.00003) -- (-0.17981,-0.70109);
					\draw[opacity=0.8,line width=.1pt] (0,0.00003) -- (-0.2909,-0.73169);
					\draw[line width=.1pt,opacity=0.8,color=black] (0,0.00003) -- (1.0303,0.22354);
					\draw[line width=.1pt,opacity=0.8,color=black] (0,0.00003) -- (1.19192,0.38776);
					\draw[line width=.1pt,opacity=0.8,color=black] (0,0.00003) -- (1.35354,0.46898);
					\draw[line width=.1pt,opacity=0.8,color=black] (0,0.00003) -- (1.51515,0.53674);
					\draw[line width=.1pt,opacity=0.8,color=black] (0,0.00003) -- (1.67677,0.72368);
					\draw[line width=.1pt,opacity=0.8,color=black] (0,0.00003) -- (1.83838,0.85858);
					\draw[line width=.1pt,opacity=0.8,color=black] (0,0.00003) -- (-0.9899,0.20023);
					\draw[line width=.1pt,opacity=0.8,color=black] (0,0.00003) -- (-1.15152,0.34596);
					\draw[line width=.1pt,opacity=0.8,color=black] (0,0.00003) -- (-1.31313,0.46045);
					\draw[line width=.1pt,opacity=0.8,color=black] (0,0.00003) -- (-1.47475,0.50693);
					\draw[line width=.1pt,opacity=0.8,color=black] (0,0.00003) -- (-1.63636,0.67443);
					\draw[line width=.1pt,opacity=0.8,color=black] (0,0.00003) -- (-1.79798,0.83426);
					\draw[line width=1pt] (0,0.00003) circle[radius=2];
					\draw (-0.1,-0.06998) node {$z$};
				\end{scope}
			\end{tikzpicture}
			\caption{By the blow-up properties~\eqref{eq:stimeblowup}, if the set $\{\theta\in\S^{d-1}:\,\r_\theta(z)<r\}$ has small measure, then the set $E\cap B_r(z)$ is close to $H_{\nu_E(z)}(z)\cap B_r(z)$}
			\label{fig:3}
		\end{figure}

		On the other hand, setting  $\Gamma_r = \insieme{z \in \partial ^{*} E \cap  B_r:\, e(z) \geq \bar C_1\hat C r^{-d+1-\beta/2}}$, it is immediate to notice that
		\begin{equation}
			\label{eq:4}
			M\geq \overline\Fcal_{0,p,d}(E, B_r)=\int _{\partial^{*} E \cap B_{r}} e(z)\d\mathcal H^{d-1}(z) \gtrsim |\Gamma_r| r^{-d+1-\beta/2}.
		\end{equation}
		We then consider the following two cases:
		\begin{itemize}
			\item [\textbf{Case 1}] Assume $|\Gamma_r| > |\Omega_{1}^{-}|/2$.
			In this case we have, by~\eqref{eq:o1-alpha} and~\eqref{eq:4} that
			\begin{equation*}
				Cr^{\alpha -\beta/2} \leq |\Omega_1^{-}| r^{-d+1-\beta/2} \lesssim 2|\Gamma_r|  r^{-d+1-\beta/2} \leq 2M.
			\end{equation*}
			Thus, if $\alpha<\beta/2=(p-2d)/2$ this leads to a contradiction for $r$ smaller than some uniform $R\ll1$ depending only on $M$, $\alpha-\beta/2$.
			\item [\textbf{Case 2}] Assume $|\Gamma_r| \leq |\Omega_{1}^{-}|/2$. In particular, there exist $x\in \partial^*E\cap P_{\bar \theta^\perp}^{-1}(\Omega_1^+)\cap B_r\setminus\Gamma_r$ and $y\in  \partial^*E\cap P_{\bar \theta^\perp}^{-1}(\Omega_1^-)\cap B_r\setminus \Gamma_r$, where $P_{\bar \theta^\perp}^{-1}:\R^d\to\bar \theta^\perp$ is the orthogonal projection map on $\bar \theta^\perp$.

			In general, by Lemma~\ref{lemma:epointwisenu} and the definition of $\Gamma_r$, whenever  $x,y \in \partial^{*}E\cap B_r \setminus \Gamma_r$, then
			\begin{equation}\label{eq:nuxy}
				\|\nu_{E}(x) - \nu_{E}(y)\| \lesssim r^{\beta/4}.
			\end{equation} 
			
			Moreover, for  $x,y$ as above such that $x_{\bar\theta^\perp}\in \Omega_{1}^{+}$ and $y_{\bar\theta^\perp} \in \Omega_{1}^{-}$, it holds $\Scal{\nu_{E}(x)}{\bar\theta}>0$ and $\Scal{\nu_{E}(y)}{\bar \theta}<0$, thus combining with~\eqref{eq:nuxy} one has that $\max\{|\Scal{\nu_E(x)}{\bar \theta}|, |\Scal{\nu_E(y)}{\bar \theta}|\} \lesssim r^{\beta/4}$.

			For $x \in \partial^*E\cap B_r\setminus \Gamma_r$, assume now that there exists $y \in \partial^*E\cap B_r$ such that $\dist\bigl(y-x, \partial H_{\nu_E(x)}(x)-x\bigr)> 2 r^{1+\beta/(8d)}$.
			Because of the volume density estimate~\eqref{eq:vol_low_bound_0} of Lemma~\ref{lemma:per_low_bound} we have that there exists $\bar C_1$ such that $|B_{r^{1+\beta/(8d)}}(y) \cap E| > \bar C_1r^{d+\beta/8}$ and $|B_{r^{1+\beta/(8d)}}(y) \setminus E| > \bar C_1r^{d+\beta/8}$, and by the fact that $\dist\bigl(y-x, \partial H_{\nu_E(x)}(x)-x\bigr)> 2 r^{1+\beta/(8d)}$ either $B_{r^{1+\beta/(8d)}}(y)\subset \{z:\,\Scal{\nu_E(x)}{z-x}>0\}$ or $B_{r^{1+\beta/(8d)}}(y)\subset \{z:\,\Scal{\nu_E(x)}{z-x}<0\}$.
			This,  by the blow up properties of the reduced boundary~\eqref{eq:stimeblowup}, would contradict the cone condition~\eqref{eq:cone} for $H_{\nu_E(x)}(x)$ and the cone $\tilde C$ centered at $x$, provided $r$ is sufficiently small (being $\beta/8<\beta/4$).
			Thus, we have that $\sup_{y\in\partial ^{*}E\cap B_r}\dist\bigl( y-x , \partial H_{\nu_{E}(x)}(x)-x\bigr) \leq 2 r^{1+\beta/(8d)}$ and then  $\partial ^{*}E \cap B_{r} \subset (\partial H_{\nu_{E}(x)}(x))_{2r^{1+\beta/(8d)}}$, where  $(\partial H_{\nu_{E}(x)}(x))_{2r^{1+\beta/(8d)}}$ is the $2r^{1+\beta/(8d)}$-neighbourhood of $\partial H_{\nu_{E}(x)}(x)$ defined in Section~\ref{sec:prel}.
			
			To conclude it is sufficient to notice that whenever $x\in \partial^*E\cap P_{\bar \theta^\perp}^{-1}(\Omega_1^+)\cap B_r\setminus\Gamma_r$, $|\Scal{\nu_E(x)}{\bar \theta}|\lesssim r^{\beta/4}$ and thus
			\begin{equation*}
				P_{\bar \theta^\perp}((\partial H_{\nu_{E}(x)}(x))_{r^{1+\beta/(8d)}}) \lesssim r^{d-1+\beta/8},
			\end{equation*}
			and thus
			\begin{equation*}
				r^{d-1+\alpha}|\Omega_{1}^{-}| \leq
				P_{\bar\theta^\perp}((\partial H_{\nu_{E}(x)}(x))_{2r^{1+\beta/(8d)}}) \lesssim r^{d-1+\beta/(8d)},
			\end{equation*}
			which yields a contradiction for $\alpha < \beta/(8d)=(p-2d)/(8d)$.
			
		\end{itemize}

		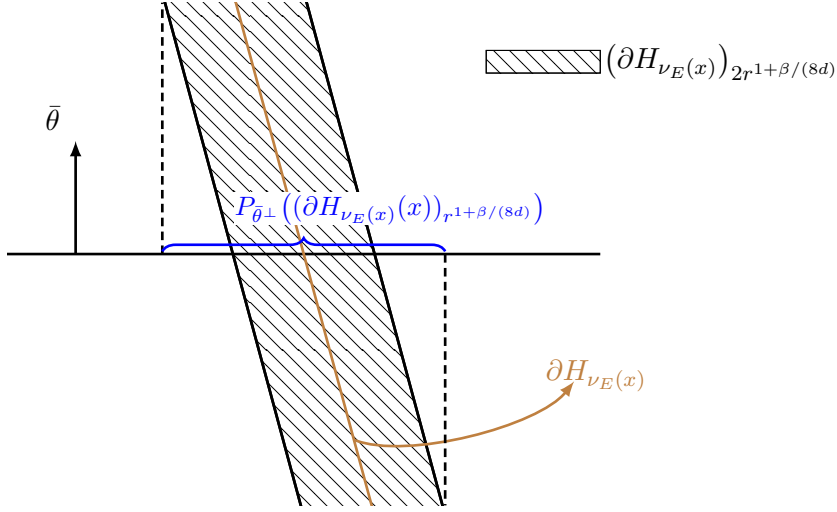
\begin{figure}
			\centering
			\begin{tikzpicture}[scale=3]
				\tikzset{arrow/.style={-latex}}
				\clip (-1.8,-1.11111) rectangle (1.85,1.11111);
				\draw[line width=1pt] (-2,0) -- (0.8,0);
				\begin{scope}[shift={(-.5, 0)},rotate=15]
					\draw[line width=1pt,pattern=north west hatch,hatch distance=6pt, hatch thickness=.3pt,opacity=0.3] (0.3,-5) -- (0.3,5) -- (-0.3,5) -- (-0.3,-5) --cycle;
					\draw[line width=1pt] (0.3,-5) -- (0.3,5) -- (-0.3,5) -- (-0.3,-5) --cycle;
					\draw[line width=1pt,color=brown] (0,-5) -- (0,5);
					\draw (1.1,-0.85) node[fill=white,inner sep=.5pt,text=brown] {$\partial H_{\nu_E(x)}$};
					\draw[arrow,line width=1pt,color=brown] (0,-0.85) .. controls (0.2,-1) and (0.8,-1) .. (1,-0.85);
				\end{scope}
				\draw[line width=1pt,arrow] (-1.5,0) -- (-1.5,0.5);
				\draw (-1.6,0.6) node {$\bar{\theta}$};
				\draw[pattern=north west hatch,hatch distance=6pt, hatch thickness=.3pt,opacity=0.3] (0.3,0.8) rectangle (0.8,0.9);
				\draw (0.3,0.8) rectangle (0.8,0.9);
				\draw (1.33,0.85) node {$\big(\partial H_{\nu_E(x)}\big)_{2r^{1+\beta/(8d)}}$};
				\draw (-0.12,0.2) node[fill=white,text=blue,inner sep=0.01pt] {\small$ P_{\bar \theta^\perp}\big((\partial H_{\nu_{E}(x)}(x))_{r^{1+\beta/(8d)}}\big)$};
				\draw[line width=1pt,dashed,dash pattern=on 3.24pt off 2pt] (-1.12,0) -- (-1.12,1.11111);
				\draw[line width=1pt,dashed,dash pattern=on 3.24pt off 2pt] (0.12,0) -- (0.12,-1.11111);
				\draw[decorate,decoration={brace,amplitude=5pt,mirror,raise=1pt},line width=1pt,color=blue] (0.12,0) -- (-1.12,0);
				
			\end{tikzpicture}
			\caption{A picture for Case 2 of Lemma~\ref{lemma:excf}, at a point $x$ such that $|\Scal{\nu_E(x)}{\bar\theta}|\ll1$.}
			\label{fig:x2}
		\end{figure}

	\end{proof}

	\begin{proof}  [Proof of Theorem~\ref{thm:regularity}: ]
		In the \hyperref[sec:appendix]{Appendix} we give a self-contained proof of the fact that  uniform upper and lower bounds on perimeter and volume, together with uniform  power law decay of the excess as in Lemmas~\ref{lemma:upper_bound_perimeter},~\ref{lemma:per_low_bound} and~\ref{lemma:excf}, imply the Lipschitz regularity of the boundary given in Theorem~\ref{thm:regularity}. In particular, the boundary of $E$ is of class $C^{1,\alpha}$, with $\alpha=\frac{p-2d}{2\max\{p-d,8d\}}$. Though this general strategy is the same used to prove regularity of quasi-minimizers, we provide a self-contained proof of this fact not exploiting directly the quasi-minimality property.
	\end{proof}

	\subsection{(iv) Rigidity: from $d=2$ to arbitrary dimension}\label{subs:5.4}
	
	In dimension $d=2$, whenever $p\geq d+3$ one has that Proposition~\ref{prop:brezis}, Theorem~\ref{thm:regularity} and Lemma~\ref{lemma:d+3} hold. Indeed, in this case $2d=4<5=d+3$, so  $p>2d$ whenever $p\geq d+3$. This gives the rigidity Theorem~\ref{thm:rigidity} in dimension $d=2$.
	
	In general dimension $d\geq 3$, to overcome the fact that $2d\geq d+3$ and thus $p\geq d+3$ does not imply $p>2d$, we exploit another integral geometric formulation for the functional $\overline \Fcal_{0,p,d}$ on two-dimensional affine planes and we recover the rigidity result from two-dimensional rigidity on such planes.
	
	First of all, we prove  an analogue of the integral geometric formulation~\eqref{eq:fintgeom}, by decomposing the functional $\overline \Fcal_{0,p,d}$ along the Grassmanian of two-dimensional linear subspaces in $\R^d$ instead of one dimensional linear subspaces. For the notation, see Section~\ref{sec:prel}.
	
	Let us preliminarily recall some basic facts about slicing of sets of finite perimeter with $k$-dimensional affine planes.
	
	\begin{remark}
		\label{rmk:normalSlice}
		Let $E\subset \R^{d}$ be a set of locally finite perimeter and $\pi_k$ be a $k$-dimensional plane in $\R^{d}$.
		Denote by $P_{\pi_k}$ the orthogonal projection on $\pi_k$. Then for almost any $x^{\perp}_{\pi_k}$ in $\pi_k^{\perp}$, the set $E_{x^{\perp}_{\pi_k}}$ is a set of finite perimeter. Moreover, for $\hausd^{d-k}$-almost  every $x^{\perp}_{\pi_k}$ and for $\hausd^{k}$-almost every $x_{\pi_k} \in \partial ^{*} E_{x^{\perp}_{\pi_k}} $ one has that
		\begin{equation}
			\label{eq:normalSlice}
			\nu_{E_{x^{\perp}_{\pi_k}}}(x_{\pi_k}) = \frac{P_{\pi_k}\big(\nu_{E}(x^{\perp}_{\pi_k}, x_{\pi_k})\big)}{\|P_{\pi_k}\big(\nu_{E}(x^{\perp}_{\pi_k}, x_{\pi_k})\big)\|}.
		\end{equation}
	\end{remark}

	\begin{proposition}\label{prop:slicing}
		Let $\overline\Fcal_{0,p,d}$ be the functional defined in~\eqref{eq:f0p}, $\Omega\subset\R^d$ bounded and open.  One has that
		\begin{align}
			\overline \Fcal_{0,p,d}(E, \Omega)&=\frac{1}{2}\int_{G(2,\R^d)}\int_{\pi_2^\perp}\int_{\S^1_{\pi_2}}\int_{\partial^*E_{x_{\pi_2}^\perp}\cap \Omega}\frac{|\langle\nu_{E_{x_{\pi_2}^\perp}}(x_{\pi_2}),\theta\rangle|}{\r_\theta(x_{\pi_2})^{p-d-1}}\d\mathcal H^1(x_{\pi_2})\d\theta\dx_{\pi_2}^\perp\d\mu_{2,d}(\pi_2)\notag\\
			&=\frac{1}{2}\int_{G(2,\R^d)}\int_{\pi_2^\perp}\overline  F_{0,p,{\pi_2}}^{\mathrm{2D}}(E_{x_{\pi_2}^\perp}, \Omega_{x_{\pi_2}^\perp})\dx_{\pi_2}^\perp\d\mu_{2,d}(\pi_2),
		\end{align}
		where
		\begin{equation}\label{eq:2dformula}
			\overline  F_{0,p,{\pi_2}}^{\mathrm{2D}}(E_{x_{\pi_2}^\perp}, \Omega_{x_{\pi_2}^\perp})=\int_{\S^1_{\pi_2}}\int_{\partial^*E_{x_{\pi_2}^\perp}\cap\Omega_{x_{\pi_2}^\perp}}\frac{|\langle\nu_{E_{x_{\pi_2}^\perp}}(x_{\pi_2}),\theta\rangle|}{\r_\theta(x_{\pi_2})^{p-d-1}}\d\mathcal H^1(x_{\pi_2})\d\theta.
		\end{equation}
		
	\end{proposition}
	
	\begin{proof}
		By the classical slicing formulas for the perimeter~\eqref{eq:slicing}, Fubini Theorem and~\eqref{eq:haar}, one has that
		\begin{align*}
			&\int_{G(2,\R^d)}\int_{\pi_2^\perp}\int_{\S^1_{\pi_2}}\int_{\partial^*E_{x_{\pi_2}^\perp}\cap \Omega_{x_{\pi_2}^\perp}}\frac{|\langle\nu_{E_{x_{\pi_2}^\perp}}(x_{\pi_2}),\theta\rangle|}{\r_\theta(x_{\pi_2})^{p-d-1}}\d\mathcal H^1(x_{\pi_2})\d\theta\dx_{\pi_2}^\perp\d\mu_{2,d}(\pi_2)\notag\\
			&=\int_{G(2,\R^d)}\int_{\pi_2^\perp}\int_{\S^1_{\pi_2}}\int_{(\pi_2)_\theta^\perp}\sum_{x_{\theta}\in\partial^*E_{x_{(\pi_2)_\theta^\perp}}\cap \Omega_{x_{(\pi_2)_\theta^\perp}}}\frac{1}{\r_\theta(x_\theta)^{p-d-1}}\dx_{(\pi_2)_\theta^\perp}\d\theta\dx_{\pi_2}^\perp\d\mu_{2,d}(\pi_2)\notag\\
			&=\int_{G(2,\R^d)}\int_{\S^1_{\pi_2}}\int_{\theta^\perp}\sum_{x_\theta\in\partial^*E_{x_\theta^\perp}\cap \Omega_{x_\theta^\perp}}\frac{1}{\r_\theta(x_\theta)^{p-d-1}}\dx_\theta^\perp\d\theta\d\mu_{2,d}(\pi_2)\notag\\
			&=\int_{\S^{d-1}}\int_{\theta^\perp}\sum_{x_\theta\in\partial^*E_{x_\theta^\perp}\cap \Omega_{x_\theta^\perp}}\frac{1}{\r_\theta(x_\theta)^{p-d-1}}\dx_\theta^\perp\d\theta\notag\\
			&=\int_{\S^{d-1}}\int_{\partial^*E\cap\Omega}\frac{|\langle \nu_E(x),\theta\rangle|}{\r_\theta(x)^{p-d-1}}\d\mathcal H^{d-1}(x)\d\theta.
		\end{align*}
	\end{proof}
	
	As a consequence of the above two-dimensional slicing formula, one has that sets of $\overline\Fcal_{0,p,d}(\cdot, \Omega)$-finite energy  have, inside almost all the intersections of $\Omega$ with two-dimensional affine planes,  flat topological boundary (and $[0,L)^d$-periodic sets of $\overline\Fcal_{0,p,d}(\cdot, [0,L)^d)$-finite energy are, on almost all two-dimensional affine planes,  periodic stripes).
	
	\begin{proposition}\label{prop:nusliced}
		Let $p\geq d+3$, $d\geq 2$, $\Omega \subset  \R^{d}$ an open bounded domain and let $E\subset\R^d$ be a set such that $\overline\Fcal_{0,p,d}(E,\Omega)<+\infty$. Then, for $\mu_{2,d}$-\ae two-dimensional plane $\pi_2\in G(2,\R^d)$  and for $\mathcal H^{d-2}$-\ae $x_{\pi_2}^\perp\in \pi_2^\perp$, the boundary of the two-dimensional slice of $E\cap\Omega$ given by $E_{x_{\pi_2}^\perp}\cap\Omega_{x_{\pi_2}^\perp}$ is given by the disjoint union  of the intersections of finitely many lines in $x_{\pi_2}^\perp+\pi_2$ with $\Omega_{x_{\pi_2}^\perp}$.

		If $\Omega=[0,L)^d$ and $E$ is $[0,L)^d$-periodic, then for $\mu_{2,d}$-\ae two-dimensional plane $\pi_2\in G(2,\R^d)$  and for $\mathcal H^{d-2}$-\ae $x_{\pi_2}^\perp\in \pi_2^\perp$, the boundary of the two-dimensional slice of $E\cap[0,L)^d$ given by $E_{x_{\pi_2}^\perp}\cap[0,L)^d_{x_{\pi_2}^\perp}$ is given by the disjoint union  of the intersections of finitely many parallel lines in $x_{\pi_2}^\perp+\pi_2$ with $[0,L)^d_{x_{\pi_2}^\perp}$.  More precisely,
		for $\mathcal H^{d-2}$-a.e. $x_{\pi_2}^\perp$ there exists $\nu_{x_{\pi_2}^\perp}\in\S^{d-1}$ such that
		\begin{equation}
			\nu_{E_{x_{\pi_2}^\perp}}(y)=\pm\nu_{x_{\pi_2}^\perp}\quad\text{for $\mathcal H^1$-\ae $y\in \partial^*E_{x_{\pi_2}^\perp}$.}
		\end{equation}
	\end{proposition}
	
	\begin{proof}
		Thanks to Proposition~\ref{prop:slicing}, if $\overline{\Fcal}_{0,p,d}(E, \Omega)<+\infty$, then for $\mu_{2,d}$-\ae two-dimensional plane $\pi_2\in G(2,\R^d)$ and for $\mathcal H^{d-2}$-\ae $x_{\pi_2}^\perp\in \pi_2^\perp$ one has that
		\begin{equation}
			\overline  F_{0,p,{\pi_2}}^{\mathrm{2D}}(E_{x_{\pi_2}^\perp}, \Omega_{x_{\pi_2}^\perp})<+\infty.
		\end{equation}
		Recalling the formula~\eqref{eq:2dformula}, this means that if we identify the plane $x_{\pi_2}^\perp+\pi_2$ with $\R^2$, then
		\begin{equation}\label{eq:2dfinite}
			\overline{\Fcal}_{0,p,2}(E_{x_{\pi_2}^\perp}, \Omega_{x_{\pi_2}^\perp})<+\infty.
		\end{equation}
		Now we observe that for $d=2$ it holds $p\geq d+3=5>4=2d$. Hence, Proposition~\ref{prop:brezis}, Theorem~\ref{thm:regularity} and Lemma~\ref{lemma:d+3} can be applied, implying that there exists $\bar R(x_{\pi_2}^\perp)>0$  such that, for all $0<r<\bar R(x_{\pi_2}^\perp)$, $\nu_E$ is constant on balls of radius $r>0$ centered at points of $\partial^*E\cap\Omega_r$, where $\Omega_r=\{x\in\Omega:\,\mathrm{dist}(x,\partial\Omega)>r\}$. In particular,  the boundary of the two-dimensional slice of $E\cap\Omega$ given by $E_{x_{\pi_2}^\perp}\cap\Omega_{x_{\pi_2}^\perp}$ is given by the disjoint union  of the intersections of finitely many lines in $x_{\pi_2}^\perp+\pi_2$ with $\Omega_{x_{\pi_2}^\perp}$.
		
		If $\Omega=[0,L)^d$ and $E$ is $[0,L)^d$-periodic,
		\begin{equation}
			\frac{1}{L^d}\overline\Fcal_{0,p,d}(E, [0,L)^d)=\frac{1}{(kL)^d}	\overline\Fcal_{0,p,d}(E, [-kL,kL)^d) \quad\text{for $k\gg1$,}
		\end{equation}
		thus for $\mathcal H^{d-2}$-a.e. $x_{\pi_2}^\perp$ there exists $\nu_{x_{\pi_2}^\perp}\in\S^{d-1}$ such that
		\begin{equation}\label{eq:557}
			\nu_{E_{x_{\pi_2}^\perp}}(y)=\pm\nu_{x_{\pi_2}^\perp}\quad\text{for $\mathcal H^1$-\ae $y\in \partial^*E_{x_{\pi_2}^\perp}$, }
		\end{equation}
		namely the connected components of the boundary of $E_{x_{\pi_2}^\perp}$ are all flat and parallel (otherwise they would intersect in a sufficiently large cube giving an infinite two-dimensional energy).
	\end{proof}
	
	In order to  complete the proof of the rigidity Theorem~\ref{thm:rigidity},
	the main idea is to show that flatness of the boundary of a set of finite energy inside a.e. two-dimensional slice (\ie Proposition~\ref{prop:nusliced}) implies flatness  of the boundary in $\R^d$.
	
	If $\Omega=[0,L)^d$ and $E$ is $[0,L)^d$-periodic, our goal is to show that there exists $\nu\in\S^{d-1}$ such that $\nu_E(x)=\pm\nu$ for $\mathcal H^{d-1}\llcorner\partial^*E$-a.e. $x$. We proceed by contradiction assuming that there exist Lebesgue points $x,y$ of $\nu_E$ w.r.t. $\mathcal H^{d-1}\llcorner\partial ^*E$ such that $\|\nu_E(x)-\nu_E(y)\|>\eps$ and $\|\nu_E(x)+\nu_E(y)\|>\eps$ for some $\eps>0$. Then, we show that there exists an nonempty open subset of $\pi\in G(2,\R^d)$ and $\delta>0$ such that the following hold:
	\begin{align}
		\Bigl\|\frac{P_{\pi}(\nu_1)}{\|P_{\pi}(\nu_1)\|}-\frac{P_{\pi}(\nu_2)}{\|P_{\pi}(\nu_2)\|}\Bigr\|> c_0(\eps)>0,\quad	\Bigl\|\frac{P_{\pi}(\nu_1)}{\|P_{\pi}(\nu_1)\|}+\frac{P_{\pi}(\nu_2)}{\|P_{\pi}(\nu_2)\|}\Bigr\|> c_0(\eps)>0\label{eq:proj}
	\end{align}
	for all $\nu_1,\nu_2\in\S^{d-1}$ s.t. $\|\nu_1-\nu_E(x)\|<\delta$ and $\|\nu_2-\nu_E(y)\|<\delta$ and moreover
	\begin{align}
		\mathcal H^{d-2}\Bigl(\Bigl\{x_\pi^\perp:\,&\mathcal H^1((\partial^*E\cap\{\|\nu_E(\cdot)-\nu_E(x)\|<\delta\})_{x_\pi^\perp})>0,\notag\\
		&\mathcal H^1((\partial^*E\cap\{\|\nu_E(\cdot)-\nu_E(y)\|<\delta\})_{x_\pi^\perp})>0 \Bigr\}\Bigr)>0\label{eq:2slices1}
	\end{align}
	In particular, recalling that by Remark~\ref{rmk:normalSlice}
	\begin{equation*}
		\nu_{E_{x_\pi^\perp}}(z_\pi)=\frac{P_\pi(\nu_E(x_\pi^\perp+z_\pi))}{\|P_\pi(\nu_E(x_\pi^\perp+z_\pi))\|},
	\end{equation*}
	for a set of two dimensional slices of positive measure parametrized by $(\pi,x_\pi^\perp)$ as above it holds that there exists a set of positive $\mathcal H^1\llcorner\partial^*E_{x_\pi^\perp}$-measure of $z_\pi$, $w_\pi$ such that
	\begin{equation*}
		\|\nu_{E_{x_\pi^\perp}}(z_\pi)-	\nu_{E_{x_\pi^\perp}}(w_\pi)\|>c_0(\eps),\quad 	\|\nu_{E_{x_\pi^\perp}}(z_\pi)+\nu_{E_{x_\pi^\perp}}(w_\pi)\|>c_0(\eps),
	\end{equation*}
	thus contradicting \eqref{eq:557}.
	
	In what follows,
	Lemma~\ref{lemma:pigeqeps} and Lemma~\ref{lemma:measpos} contain respectively the proofs of \eqref{eq:proj} and \eqref{eq:2slices1}.
	
	For a general $\Omega$ instead the proof is more involved, due to the fact that in two dimensions the connected components of the boundary of sets of finite energy are not necessarily parallel. Using \eqref{eq:proj} and \eqref{eq:2slices1}, we will then show that for sets of finite energy
	\begin{enumerate}
		\item\label{itemPoint1} $\nu_E$ is continuous
		\item\label{itemPoint2} $\partial E$ is locally parametrized by a Lipschitz graph,
	\end{enumerate}
	which will imply by Lemma~\ref{lemma:d+3} that $\nu_E$ is locally constant. Points~\ref{itemPoint1} and~\ref{itemPoint2} above will be the content of  Corollary~\ref{cor:continuity} and  Lemma~\ref{lemma:graph}.
	
	We can now state the first two preliminary lemmas.

	\begin{lemma}\label{lemma:pigeqeps}
		Let $\eps>0$. There exist $c_0(\eps), c_1(\eps)>0$ such that $\forall\,\nu_1,\nu_2\in\S^{d-1}$ s.t. $\|\nu_1-\nu_2\|>\eps$ and  $\|\nu_1+\nu_2\|>\eps$, the set
		\begin{align*}
			G(\nu_1,\nu_2,\eps):=\Bigl\{\pi\in G(2,\R^d):\,&\Bigl\|\frac{P_{\pi}(\nu_1)}{\|P_{\pi}(\nu_1)\|}-\frac{P_{\pi}(\nu_2)}{\|P_{\pi}(\nu_2)\|}\Bigr\|> c_0(\eps),\notag\\
			&\Bigl\|\frac{P_{\pi}(\nu_1)}{\|P_{\pi}(\nu_1)\|}+\frac{P_{\pi}(\nu_2)}{\|P_{\pi}(\nu_2)\|}\Bigr\|> c_0(\eps),\notag\\
			&\bigl\|P_\pi(\nu_1)\bigr\|> c_1(\eps),\,\bigl\|P_\pi(\nu_2)\bigr\|> c_1(\eps)\Bigr\}
		\end{align*}
		is nonempty and  for every $v\in\S^{d-1}$ there exists $\pi\in G(\nu_1,\nu_2,\eps)$ such that $v\in\pi$.
		Moreover, the set
		\begin{equation}\label{eq:geps}
			G_\eps:=\bigl\{(\nu_1,\nu_2,\pi):\,\|\nu_1-\nu_2\|>\eps,\,  \|\nu_1+\nu_2\|>\eps,\,\pi\in G(\nu_1,\nu_2,\eps)\bigr\}
		\end{equation}
		is an open subset of $\S^{d-1}\times\S^{d-1}\times G(2,\R^d)$.
	\end{lemma}

	\begin{proof}
		Let $v\in\S^{d-1}$. We will show that $G(\nu_1,\nu_2,\eps)$ contains a plane $\pi$ such that $v\in\pi$, thus it is nonempty. The fact that the set $G_\eps$ is open follows from the continuity of the projection map.
		
		Define
		\begin{align*}
			\bar \nu_i&=\nu_i-\langle\nu_i,v\rangle v, \quad i=1,2\\
			\mathring \nu_i&=\langle\nu_i,v\rangle v, \quad i=1,2.
		\end{align*}
		Then, for $c\in(0,1)$, either
		\begin{equation}
			\label{eq:case11}
			\|\bar \nu_1-\bar \nu_2\|\geq c\eps,
		\end{equation}
		or
		\begin{equation}
			\label{eq:case21}
			\|\mathring \nu_1-\mathring\nu_2\|\geq\sqrt{1-c^2}\eps.
		\end{equation}

		Depending on which of the above two conditions holds, we will choose different  $\pi=\pi(\nu_1,\nu_2,v)$.
		
		Assume first \eqref{eq:case11} holds.
		Let
		\[
		\theta=\frac{\bar\nu_1-\bar\nu_2}{\|\bar\nu_1-\bar\nu_2\|},\quad\pi=\mathrm{span}\{v,\theta\}.
		\]
		Assume that for the set of such $\pi$ the statement of the lemma does not hold. Then $\exists\,\eps,c>0$ s.t. $\forall\,\delta>0$ there exist $\nu_1^\delta, \nu_2^\delta, v^\delta\in\S^{d-1}$ as in \eqref{eq:case11}  and $\pi^\delta$ as above s.t.  either
		\begin{equation}\label{eq:ppiddeltadir}
			\Biggl\|\frac{P_{\pi^\delta}(\nu_1^\delta)}{\|P_{\pi^\delta}(\nu_1^\delta)\|}-\frac{P_{\pi^\delta}(\nu_2^\delta)}{\|P_{\pi^\delta}(\nu_2^\delta)\|}\Biggr\|\leq\delta,
		\end{equation}
		or
		\begin{equation}\label{eq:ppiddeltadir+}
			\Biggl\|\frac{P_{\pi^\delta}(\nu_1^\delta)}{\|P_{\pi^\delta}(\nu_1^\delta)\|}+\frac{P_{\pi^\delta}(\nu_2^\delta)}{\|P_{\pi^\delta}(\nu_2^\delta)\|}\Biggr\|\leq\delta,
		\end{equation}
		or
		\begin{equation}\label{eq:ppidelta1}
			\bigl\|P_{\pi^\delta}(\nu_1^\delta)\bigr\|\leq\delta,
		\end{equation}
		or
		\begin{equation}\label{eq:ppidelta2}
			\bigl\|P_{\pi^\delta}(\nu_2^\delta)\bigr\|\leq\delta.
		\end{equation}
		By compactness, we can extract a subsequence $\nu_1^\delta\to\nu_1$, $\nu_2^\delta\to\nu_2$, $v^\delta\to v$, with $\nu_1,\nu_2,v\in\S^{d-1}$.
		
		Let us first show that neither \eqref{eq:ppidelta1} nor \eqref{eq:ppidelta2} can hold as $\delta\to0$. Indeed, if \eg \eqref{eq:ppidelta1} holds, then
		\begin{equation*}
			0=P_\pi(\nu_1)=\mathring\nu_1+\langle\bar\nu_1,\frac{\bar\nu_1-\bar\nu_2}{\|\bar\nu_1-\bar\nu_2\|}\rangle\theta,
		\end{equation*}
		hence $\nu_1=\bar \nu_1$ and
		\begin{equation*}
			1=\|\bar\nu_1\|^2=\langle\bar\nu_1,\bar\nu_1\rangle=\langle\bar\nu_1,\bar\nu_2\rangle,
		\end{equation*}
		which implies since $\|\bar\nu_2\|\leq1$ that $\bar\nu_2=\bar\nu_1$, thus contradicting \eqref{eq:case11}. The same reasoning gives the validity of condition \eqref{eq:ppi2}. Assume now that \eqref{eq:ppiddeltadir} holds.
		Thanks to the validity of
		\begin{equation}
			\bigl\|P_\pi(\nu_1)\bigr\|\geq c_1(\eps) \label{eq:ppi1}
		\end{equation}
		and
		\begin{equation}
			\bigl\|P_\pi(\nu_2)\bigr\|\geq c_1(\eps)\label{eq:ppi2}
		\end{equation} one can pass to the limit in \eqref{eq:ppiddeltadir} and get
		\begin{equation}\label{eq:ppidir0}
			\Biggl\|\frac{P_{\pi}(\nu_1)}{\|P_{\pi}(\nu_1)\|}-\frac{P_{\pi}(\nu_2)}{\|P_{\pi}(\nu_2)\|}\Biggr\|=0.
		\end{equation}
		Let $w=\frac{P_{\pi}(\nu_1)}{\|P_{\pi}(\nu_1)\|}=\frac{P_{\pi}(\nu_2)}{\|P_{\pi}(\nu_2)\|}$. By \eqref{eq:ppidir0}
		\begin{equation*}
			\nu_1=\lambda w+a_1,\quad\nu_2=\mu w+a_2,\quad a_1,a_2\in\pi^\perp.
		\end{equation*}
		In particular, $\nu_1-\nu_2=(\lambda-\mu) w+(a_1-a_2)$, where both $\nu_1-\nu_2$ and $w$ belong to $\pi$. Hence $a_1=a_2$, and since $\|\nu_1\|=\|\nu_2\|=1$ it holds $\lambda=\mu$ giving $\nu_1=\nu_2$ which again contradicts \eqref{eq:case11}. A similar reasoning can be applied to exclude \eqref{eq:ppiddeltadir+}.  Thus we proved that in case \eqref{eq:case11} holds, then the lemma holds true.
		
		Assume now the validity of \eqref{eq:case21}.
		
		Let us now choose $\pi\in G(2,\R^d)$ as follows. If
		\begin{equation}\label{eq:geq1}
			\|\mathring \nu_1\|>\|\mathring\nu_2\|+\sqrt{1-c^2}\eps,		\end{equation}
		choose
		\begin{equation*}
			\pi=\mathrm{span}\Bigl\{v,\frac{\bar\nu_2}{\|\bar \nu_2\|}\Bigr\},
		\end{equation*}
		if else
		\begin{equation}\label{eq:geq2}
			\|\mathring \nu_2\|>\|\mathring\nu_1\|+\sqrt{1-c^2}\eps,		\end{equation}
		choose
		\begin{equation*}
			\pi=\mathrm{span}\Bigl\{v,\frac{\bar\nu_1}{\|\bar \nu_1\|}\Bigr\}.
		\end{equation*}
		Assume that the statement of the lemma does not hold for such choice of $\pi$. Then, there exist $\nu_1^\delta, \nu_2^\delta, v^\delta\in\S^{d-1}$ as in \eqref{eq:case21} and $\pi^\delta$ as above s.t.  either
		\eqref{eq:ppiddeltadir} or \eqref{eq:ppiddeltadir+} or \eqref{eq:ppidelta1} or \eqref{eq:ppidelta2} hold.
		By compactness, we can extract a subsequence $\nu_1^\delta\to\nu_1$, $\nu_2^\delta\to\nu_2$, $v^\delta\to v$, $\pi^\delta\to\pi$ with $\nu_1,\nu_2,v\in\S^{d-1}$, $\pi\in G(2,\R^d)$.
		We first show that \eqref{eq:ppidelta1} or \eqref{eq:ppidelta2} cannot hold in the limit as $\delta\to0$. Assume w.l.o.g. that \eqref{eq:geq1} holds and then   $\pi=\mathrm{span}\{v,\frac{\bar\nu_2}{\|\bar\nu_2\|}\}$. In particular, $P_\pi(\nu_2)=\nu_2$, hence \eqref{eq:ppidelta2} cannot hold along the chosen subsequence. Moreover, if
		\begin{equation*}
			0=P_\pi(\nu_1)=\mathring\nu_1+\langle\nu_1,\frac{\bar\nu_2}{\|\bar\nu_2\|}\rangle\frac{\bar\nu_2}{\|\bar\nu_2\|},
		\end{equation*}
		one has that $\mathring \nu_1=0$, thus contradicting \eqref{eq:geq1}.
		Thus, \eqref{eq:ppi1} and \eqref{eq:ppi2} hold also for the vectors such that \eqref{eq:case21} holds.
		We now show that also  \eqref{eq:ppiddeltadir}  cannot hold as $\delta\to0$. In case \eqref{eq:geq1} holds (the case \eqref{eq:geq2} can be treated analogously), one has that
		\begin{equation*}
			\frac{P_\pi(\nu_1)}{\|P_\pi(\nu_1)\|}=\nu_2=\frac{P_\pi(\nu_2)}{\|P_\pi(\nu_2)\|}.
		\end{equation*}
		Hence, $\nu_1=\nu_2+a_1$, with $a_1\in\pi^\perp$ and then $\mathring \nu_1=\mathring \nu_2$ which again violates \eqref{eq:geq1}. One can reason in a similar way in order to exclude \eqref{eq:ppiddeltadir+}.

		Thus we concluded the proof of the lemma.
		
	\end{proof}

	\begin{lemma}\label{lemma:measpos}
		Let $x,y\in\partial^*E$ be Lebesgue points of $\nu_E$ w.r.t. $\mathcal H^{d-1}\llcorner\partial^*E$ such that $\nu_E(x)\neq\pm\nu_E(y)$ and let
		\begin{equation}\label{eq:nu1nu2}
			\nu_1=\nu_E(x),\quad\nu_2=\nu_E(y),\quad\eps<\min\{\|\nu_1-\nu_2\|,\|\nu_1+\nu_2\|\},\quad\delta\leq\eps,\quad\rho>0,
		\end{equation}
		\begin{align*}
			\Gamma_{1,\rho,\delta}&=\{z\in\partial^*E\cap B_{\rho}(x):\,\|\nu_E(z)-\nu_1\|<\delta\},\\	\Gamma_{2,\rho,\delta}&=\{z\in\partial^*E\cap B_{\rho}(y):\,\|\nu_E(z)-\nu_2\|<\delta\}.
		\end{align*}
		Then,  the set
		\begin{align}\label{eq:measpos}
			G_{\eps,\rho,\delta}=\Bigl\{\pi\in G(\nu_1,\nu_2,\eps):\,\mathcal H^{d-2}\Bigl(\Bigl\{x_\pi^\perp\in\pi^\perp:\,\mathcal H^{1}((	\Gamma_{1,\rho,\delta})_{x_\pi^\perp})>0,\,\mathcal H^{1}(( 	\Gamma_{2,\rho,\delta})_{x_\pi^\perp})>0\Bigr\}\Bigr)>0 \Bigr\}
		\end{align}
		is a nonempty open subset of $G(2,\R^d)$. Moreover, the set $G_{\eps,\rho,\delta}$ contains every $\pi\in G(\nu_1,\nu_2,\eps)$ such that $v=\frac{x-y}{\|x-y\|}\in \pi$.
	\end{lemma}
	
	\begin{proof}
		Let $v=\frac{x-y}{\|x-y\|}$ and, given $\nu_1,\nu_2$ as in \eqref{eq:nu1nu2}, let $\pi\in G(\nu_1,\nu_2,\eps)$ such that $v\in\pi$ as in Lemma~\ref{lemma:pigeqeps}.  To prove the lemma we first show that (as a consequence of a more quantitative estimate)
		\begin{equation}\label{eq:pimeaspos}
			\mathcal H^{d-2}\Bigl(\Bigl\{x_\pi^\perp\in\pi^\perp:\,\mathcal H^{1}((	\Gamma_{1,\rho,\delta})_{x_\pi^\perp})>0,\,\mathcal H^{1}((	\Gamma_{2,\rho,\delta})_{x_\pi^\perp})>0\Bigr\}\Bigr)>0.
		\end{equation}
		Then we show that there exists $\sigma>0$ such that if $\|\pi'-\pi\|\leq \sigma$ then \eqref{eq:pimeaspos} holds for $\pi'$ as well. This, by the fact that the set $G(\nu_1,\nu_2,\eps)$ is an open subset of $G(2,\R^d)$, concludes the proof of the lemma.
		
		Let $B^2$, $B^{d-2}$ be the unit balls centered at the origin and contained respectively in $\pi$ and $\pi^\perp$. For a set $A\subset\R^d$, $z\in\R^d$ and $\rho>0$, let $A_{z,\rho}=\frac{A-z}{\rho}$. Assume w.l.o.g. that $\pi^\perp=\mathrm{span}\{e_1,\dots,e_{d-2}\}$, $\pi=\mathrm{span}\{e_{d-1},e_d\}$.
		
		By De Giorgi's structure theorem, as $\rho\to0$
		\begin{align*}
			&\chi_{E_{x,\rho}}\to\chi_{H_{\nu_1}}\quad\text{in $L^1(B_{d-2}\times B_2)$},\\
			&\chi_{(E_{x,\rho})_{x_\pi^\perp}}\to\chi_{(H_{\nu_1})_{x_\pi^\perp}}\quad\text{in $L^1(\{x_\pi^\perp\}\times B^2)$, for $\mathcal H^{d-2}$-a.e. $x_\pi^\perp\in B^{d-2}$.}
		\end{align*}
		The same holds at $y$ w.r.t. $\nu_2$.
		
		Moreover, since by definition of $G(\nu_1,\nu_2,\eps)$ one has that
		\[
		\|P_\pi(\nu_1)\|>c_1(\eps),\quad\|P_{\pi}(\nu_2)\|>c_1(\eps),
		\]
		then there exists $\bar c_1(\eps)>0$ such that for \ae $x^{\perp}_{\pi}$ it holds
		\begin{align*}
			\liminf_\rho\per((E_{x,\rho})_{x_\pi^\perp}, x_\pi^\perp+B^2)\geq\per((H_{\nu_1})_{x_\pi^\perp}, x_\pi^\perp+B^2)\geq \bar c_1(\eps)>0,\\
			\liminf_\rho\per((E_{y,\rho})_{x_\pi^\perp}, x_\pi^\perp+B^2)\geq\per((H_{\nu_2})_{x_\pi^\perp}, x_\pi^\perp+B^2)\geq \bar c_1(\eps)>0.
		\end{align*}
		Define
		\begin{equation*}
			f_{x,\rho}(x_\pi^\perp)=\per((E_{x,\rho})_{x_\pi^\perp}, x_\pi^\perp+B^2),\quad f_{y,\rho}(x_\pi^\perp)=\per((E_{y,\rho})_{x_\pi^\perp}, x_\pi^\perp+B^2).
		\end{equation*}
		
		In particular, $\forall\,0<\rho<\rho(\eps,x,y)$,
		\begin{equation}\label{eq:fstar}
			\mathcal H^{d-2}\Bigl(\Bigl\{x_\pi^\perp\in B^{d-2}:\,f_{x,\rho}(x_\pi^\perp)\geq \frac{\bar c_1(\eps)}{2}, f_{y,\rho}(x_\pi^\perp)\geq \frac{\bar c_1(\eps)}{2}\Bigl\}\Bigr)\geq \frac12\mathcal H^{d-2}(B^{d-2}).
		\end{equation}
		Let now
		\begin{equation*}
			g_{x,\rho}(x_\pi^\perp)=\mathcal H^1(((\Gamma_{1,1,\delta})_{x,\rho})_{x_\pi^\perp}\cap x_\pi^\perp+B^2),\quad 	g_{y,\rho}(x_\pi^\perp)=\mathcal H^1(((\Gamma_{2,1,\delta})_{y,\rho})_{x_\pi^\perp}\cap x_\pi^\perp+B^2).
		\end{equation*}
		We would like to substitute in \eqref{eq:fstar} $f_{x,\rho}$ with $g_{x,\rho}$, $f_{y,\rho}$ with $g_{y,\rho}$ and the factor $2$ by the factor $4$. We can do that since $x,y$ are Lebesgue points of $\nu_E$ w.r.t. $\mathcal H^{d-1}\llcorner \partial^*E$, thus as $\rho\to 0$
		\begin{equation*}
			\liminf_\rho	\int_{B^{d-2}}(g_{x,\rho}-f_{x,\rho})\dx_\pi^\perp\geq0,\quad 	\liminf_\rho\int_{B^{d-2}}(g_{y,\rho}-f_{y,\rho})\dx_\pi^\perp\geq0.
		\end{equation*}
		Since $v\in\pi$, $P_{\pi^\perp}(x-y)=0$ and thus such a lower bound translates into
		\begin{align*}
			\mathcal H^{d-2}\Bigl(\Bigl\{x_\pi^\perp\in\pi^\perp:\,\mathcal H^1((\Gamma_{1,\rho,\delta})_{x_\pi^\perp})>\frac{\bar c_1(\eps)}{4}\rho,\,\mathcal H^1((\Gamma_{2,\rho,\delta})_{x_\pi^\perp})>\frac{\bar c_1(\eps)}{4}\rho \Bigr\}\Bigr)\geq \frac14 \rho^{d-2}
		\end{align*}
		for all $0<\rho<\rho(\eps,\delta, x,y)$.
		Choosing now any $\pi'\in G(2,\R^d)$ whose distance from $\pi$ is sufficiently small depending on $\rho(\eps,\delta, x,y)$ we can obtain the bound
		\begin{align*}
			\mathcal H^{d-2}\Bigl(\Bigl\{x_{\pi'}^\perp\in{\pi'}^\perp:\,\mathcal H^1((\Gamma_{1,\rho,\delta})_{x_{\pi'}^\perp})>\frac{\bar c_1(\eps)}{8}\rho,\,\mathcal H^1((\Gamma_{2,\rho,\delta})_{x_{\pi'}^\perp})>\frac{\bar c_1(\eps)}{8}\rho \Bigr\}\Bigr)\geq \frac18 \rho^{d-2}
		\end{align*}
		for any sufficiently small $\rho>0$, thus concluding the proof of the lemma.
	\end{proof}
	
	\begin{proof}[Proof of Theorem~\ref{thm:rigidity} in case $\Omega=[0,L)^d$ and $E$ is $[0,L)^d$-periodic: ]
		
		Assume there exist $x,y$ Lebesgue points of $\nu_E$ w.r.t. $\mathcal H^{d-1}\llcorner \partial^*E$ such that $\nu_E(x)\neq\pm\nu_E(y)$. Let $\nu_1=\nu_E(x)$, $\nu_2=\nu_E(y)$ and  $\eps<\min\{\|\nu_1-\nu_2\|, \|\nu_1+\nu_2\|\}$.
		
		Let $v=\frac{x-y}{\|x-y\|}\in\pi$. By Lemma~\ref{lemma:pigeqeps} there exists  $\pi\in G(\nu_1,\nu_2,\eps)$ such that $v\in\pi$. Moreover since the set $G_\eps$ defined in \eqref{eq:geps} is open, there exist $\delta>0$ and $\sigma>0$ such that $\tilde\nu_1,\tilde \nu_2\in\S^{d-1}$ such that $\|\tilde\nu_1-\nu_1\|\leq\delta$, $\|\tilde\nu_2-\nu_2\|\leq\delta$ and for all $\tilde\pi\in G(2,\R^d)$ such that $\|\pi-\tilde\pi\|\leq\sigma$ then
		\begin{equation*}
			(\tilde\nu_1,\tilde\nu_2,\tilde\pi)\in G_{\eps}.
		\end{equation*}
		Thus $\|\tilde\nu_1-\tilde\nu_2\|>\eps$, $\|\tilde\nu_1+\tilde\nu_2\|>\eps$and $\tilde\pi\in G(\tilde\nu_1,\tilde\nu_2,\eps)$, implying
		\begin{equation}\label{eq:596}
			\Bigl\|\frac{P_{\tilde\pi}(\tilde\nu_1)}{\|P_{\tilde\pi}(\tilde\nu_1)\|}-\frac{P_{\tilde\pi}(\tilde\nu_2)}{\|P_{\tilde\pi}(\tilde\nu_2)\|}\Bigr\|> c_0(\eps)>0,\quad 			\Bigl\|\frac{P_{\tilde\pi}(\tilde\nu_1)}{\|P_{\tilde\pi}(\tilde\nu_1)\|}-\frac{P_{\tilde\pi}(\tilde\nu_2)}{\|P_{\tilde\pi}(\tilde\nu_2)\|}\Bigr\|> c_0(\eps)>0.
		\end{equation}
		By Lemma~\ref{lemma:measpos}, eventually choosing $\sigma>0$ smaller we can ensure that all $\tilde\pi$ such that $\|\tilde\pi-\pi\|\leq\sigma$ belong to the set $G_{\eps,\delta,\rho}$ defined in \eqref{eq:measpos}, for $\delta$ as above. In particular, it holds
		\begin{equation*}
			\mathcal H^{d-2}\Bigl(\Bigl\{x_{\tilde\pi}^\perp\in{\tilde\pi}^\perp:\,\mathcal H^{1}((	\Gamma_{1,\rho,\delta})_{x_{\tilde\pi}^\perp})>0,\,\mathcal H^{1}((	\Gamma_{2,\rho,\delta})_{x_{\tilde\pi}^\perp})>0\Bigr\}\Bigr)>0.
		\end{equation*}

		Hence, there exists a set of positive $\mathcal H^{d-2}$-measure of two-dimensional slices parallel to $\tilde \pi$ containing portions of $\Gamma_{1,\rho,\delta}$ and $\Gamma_{2,\rho,\delta}$, both of  $\mathcal H^1$-positive measure. By Remark~\ref{rmk:normalSlice}, for $\mathcal H^1$-a.e. $(z_1)_{\tilde\pi}\in(\Gamma_{1,\rho,\delta})_{x_{\tilde\pi}^\perp}$ and $\mathcal H^1$-a.e. $(z_2)_{\tilde\pi}\in(\Gamma_{2,\rho,\delta})_{x_{\tilde\pi}^\perp}$ one has that
		\begin{equation}\label{eq:nurem}
			\nu_{E_{x_{\tilde\pi}^\perp}}((z_1)_{\tilde\pi})=\frac{P_{\tilde\pi}(\nu_E(x_{\tilde\pi}^\perp+(z_1)_{\tilde\pi} ))}{\|P_{\tilde\pi}(\nu_E(x_{\tilde\pi}^\perp+(z_1)_{\tilde\pi} ))\|},\quad	\nu_{E_{x_{\tilde\pi}^\perp}}((z_2)_{\tilde\pi})=\frac{P_{\tilde\pi}(\nu_E(x_{\tilde\pi}^\perp+(z_2)_{\tilde\pi} ))}{\|P_{\tilde\pi}(\nu_E(x_{\tilde\pi}^\perp+(z_2)_{\tilde\pi} ))\|}.
		\end{equation}
		Moreover, by definition of $\Gamma_{i,\rho,\delta}$, $i=1,2$,  $\|\nu_E(x_{\tilde\pi}^\perp+(z_1)_{\tilde\pi} )-\nu_1\|<\delta$ and $\|\nu_E(x_{\tilde\pi}^\perp+(z_2)_{\tilde\pi} )-\nu_2\|<\delta$.
		Thus, by \eqref{eq:nurem} and \eqref{eq:596},
		\begin{equation*}
			\bigl\|	\nu_{E_{x_{\tilde\pi}^\perp}}((z_1)_{\tilde\pi})-\nu_{E_{x_{\tilde\pi}^\perp}}((z_2)_{\tilde\pi})\bigr\|>c_0(\eps)>0
		\end{equation*}
		on a set of $(z_1)_{\tilde\pi}$ and $(z_2)_{\tilde\pi}$ of $\mathcal H^1$-positive measure in $\partial^*E_{x_{\tilde\pi}^\perp}$, which being the set of such $\tilde\pi$ open and the two-dimensional slices of $\mathcal H^{d-2}$-positive measure contradicts \eqref{eq:557}.
		
	\end{proof}
	To prove Theorem~\ref{thm:rigidity} in case $\Omega$ is a general bounded open set we need some further preliminary results.
	
	First we state a consequence of Lemma~\ref{lemma:pigeqeps} and Lemma~\ref{lemma:measpos}, given by the continuity of the exterior normal.
	
	We recall the following definition
	\begin{equation*}
		\Omega_r:=\{x\in\Omega:\, \mathrm{dist}(x,\partial\Omega)>r\}, \quad r>0.
	\end{equation*}
	
	\begin{corollary}\label{cor:continuity}
		For all $r>0$, the measure theoretic exterior normal $\nu_E$ is continuous in $\Omega_r$, namely $\forall\,\eps>0$ $\exists\,\rho>0$ such that for all $x,y$ Lebesgue points of $\nu_E$ w.r.t. $\mathcal H^{d-1}\llcorner \partial^*E$ with $\|x-y\|<\rho$ one has that $\|\nu_E(x)-\nu_E(y)\|<\eps$.
	\end{corollary}
	
	\begin{proof}Assume $\exists\,r>0, \eps>0$ such that for all $\rho>0$ there exist $x,y\in\Omega_r$  Lebesgue points of $\nu_E$ w.r.t. $\mathcal H^{d-1}\llcorner \partial^*E$ with $\|x-y\|<\rho$ and $\|\nu_E(x)-\nu_E(y)\|>\eps$. We will show that if $\rho$ is sufficiently small w.r.t. $\eps$ and $r$ this leads to a contradiction.
		
		Reasoning as in the proof of Theorem~\ref{thm:rigidity} for $[0,L)^d$ periodic sets, one can show that there exist an open set of planes $\pi\in G(2,\R^d)$ and $\delta>0$ such that for all $\nu_1,\nu_2\in\S^{d-1}$ such that $\|\nu_E(x)-\nu_1\|<\delta$, $\|\nu_E(y)-\nu_2\|<\delta$, then $\pi\in G(\nu_1,\nu_2,\eps)$ and
		\begin{equation*}
			\mathcal H^{d-2}\Bigl(\Bigl\{x_{\pi}^\perp\in{\pi}^\perp:\,\mathcal H^{1}((	\Gamma_{1,\rho,\delta})_{x_{\pi}^\perp})>0,\,\mathcal H^{1}((	\Gamma_{2,\rho,\delta})_{x_{\pi}^\perp})>0\Bigr\}\Bigr)>0.
		\end{equation*}
		In particular, for $\mathcal H^1$-a.e. $(z_1)_{\pi}\in(\Gamma_{1,\rho,\delta})_{x_{\pi}^\perp}$ and $\mathcal H^1$-a.e. $(z_2)_{\pi}\in(\Gamma_{2,\rho,\delta})_{x_{\pi}^\perp}$ one has that
		\begin{equation*}
			\bigl\|	\nu_{E_{x_{\pi}^\perp}}((z_1)_{\pi})-\nu_{E_{x_{\pi}^\perp}}((z_2)_{\pi})\bigr\|>c_0(\eps)>0.
		\end{equation*}
		Let now $\ell((z_1)_\pi)$, $\ell((z_2)_\pi)\subset x_\pi^\perp+\pi$ be the lines of the boundary of  $E_{x_{\pi}^\perp}$ passing respectively through $(z_1)_\pi$ and $(z_2)_\pi$ and orthogonal respectively to $\nu_{E_{x_{\pi}^\perp}}((z_1)_{\pi})$ and $\nu_{E_{x_{\pi}^\perp}}((z_2)_{\pi})$.  Then, by the above and due to the fact that $\|(z_1)_\pi-(z_2)_\pi)\|\leq3\rho$, one has that  $\ell((z_1)_\pi)\cap \ell((z_2)_\pi)=\{w_\pi\}$ with $\min\{\|w_\pi-(z_1)_{\pi}\|, \,\|w_\pi-(z_2)_{\pi}\|\}\sim\frac{\rho}{c_0(\eps)}$.
		Choosing now $\rho>0$ sufficiently small such that $2 \frac{\rho}{c_0(\eps)}<r$, the two lines of the boundary of $E_{x_\pi^\perp}$ have to meet inside $\Omega_{x_\pi^\perp}$ for a set of positive measure of two-dimensional slices and two-dimensional planes, thus contradicting Proposition~\ref{prop:nusliced}.

	\end{proof}
	
	Given a vector $\nu\in\S^{d-1}$ and $\rho>0$, let
	\[
	C(x,\rho,\nu):=\bigl\{y\in\R^d:\,\|P_{\nu^\perp}(y-x)\|<\rho, \,\langle y-x,\nu\rangle<\rho\bigr\}.
	\]
	
	As a consequence of the previous corollary, for every $\eps>0$, $r>0$ there exists $\rho>0$ such that for all $\nu\in\S^{d-1}$ and $x,y\in\Omega_r$ Lebesgue points of $\nu_E$ w.r.t. $\mathcal H^{d-1}\llcorner\partial^*E$ with $y\in C(x,\rho,\nu)$, then $\|\nu_E(x)-\nu_E(y)\|<\eps$.

	Let us introduce for simplicity of notation the following set
	\[
	\partial^*E_L=\{x\in\partial^*E:\, x \text{ is a Lebesgue point of $\nu_E$ w.r.t. $\mathcal H^{d-1}\llcorner\partial^*E$ }\}.
	\]
	Moreover, given an hyperplane $\nu^\perp$, we denote by $B_{\nu^\perp,\rho}(x_\nu^\perp)$ the $d-1$ dimensional ball of radius $\rho$ and center $x_\nu^\perp$ inside $\nu^\perp$.

	\begin{lemma}\label{lemma:graph}
		For every $r>0$, $\ell>0$, $0<\eps\ll1$ there exists $\rho>0$ such that for all $x\in\partial^*E_L$ and for a.e. $\nu\in\S^{d-1}$ such that $\|\nu-\nu_E(x)\|<\eps$ there exists $f_{x,\rho,\nu}:B_{\nu^\perp, \rho}(x_{\nu}^\perp)\to\R\nu$ such that
		\begin{equation}\label{eq:egraphf}
			\mathcal H^{d-1}\bigl((C(x,\rho,\nu)\cap \partial^*E)\,\triangle\,\mathrm{graph}f_{x,\rho,\nu}\bigr)=0
		\end{equation}
		and moreover $f_{x,\rho,\nu}$ is $\ell$-Lipschitz.
		
	\end{lemma}
	
	\begin{remark}
		The fact that we can parametrize the boundary of sets of finite energy with Lipschitz graphs is necessary in our approach to show flatness of the boundary as in Lemma~\ref{lemma:d+3}. The fact that, as stated in the previous Lemma, the Lipschitz constants of such functions can be chosen to be arbitrarily small is not necessary indeed for our approach, but we include it in the proof for completeness. Therefore, the reader interested in the proof of the rigidity theorem can skip the details related to the arbitrary choice of $\ell$ in what follows.
	\end{remark}
	
	\begin{proof}
		
		Assume w.l.o.g. that $x=0$, fix $\nu\in\S^{d-1}$ 
		and define
		\begin{align}
			A_0(\nu^\perp,\rho)=\{y_\nu^\perp\in B_{\nu^\perp,\rho}(0):\,\#\partial^*E_{y_\nu^\perp}\cap(-\rho,\rho)=0\},\label{eq:a0}\\
			A_1(\nu^\perp,\rho)=\{y_\nu^\perp\in B_{\nu^\perp,\rho}(0):\,\#\partial^*E_{y_\nu^\perp}\cap(-\rho,\rho)=1\},\label{eq:a1}\\
			A_2(\nu^\perp,\rho)=\{y_\nu^\perp\in B_{\nu^\perp,\rho}(0):\,\#\partial^*E_{y_\nu^\perp}\cap(-\rho,\rho)\geq2\}.\label{eq:a2}
		\end{align}

		We will show that  there exists $\rho>0$ (dependent only on $\ell$ and $\eps$) such that
		\begin{equation}\label{eq:a22}
			\mathcal H^{d-1}(A_2(\nu^\perp,\rho))=0\quad\text{ for a.e. $\nu\in\S^{d-1}$ s.t. $\langle\nu_E(0),\nu\rangle>\ell\|P_{\nu_E(0)^\perp}(\nu)\|$},
		\end{equation}
		\begin{equation}\label{eq:a02}
			\mathcal H^{d-1}(A_0(\nu^\perp,\rho))=0\quad\text{for a.e. $\nu\in\S^{d-1}$ s.t. $\|\nu-\nu_E(0)\|<\eps$}.
		\end{equation}
		In particular, for  a.e. $\nu\in\S^{d-1}$ s.t. $\|\nu-\nu_E(0)\|<\eps$, this implies immediately the existence of a set of full measure $\tilde A(\nu^\perp,\rho)\subset B_{\nu^\perp,\rho}(0)$ and of an $\ell$-Lipschitz  function $f_{x,\nu,\rho}:\tilde A(\nu^\perp,\rho)\to\R\nu$ such that \eqref{eq:egraphf} holds.
		Indeed, since the set of $\nu\in\S^{d-1}$ for which both \eqref{eq:a22} and \eqref{eq:a02}  hold is, up to a null set,  open,  locally uniformly in $\Omega_r$, the boundary of a set of finite energy $E$ can be parametrized by a Lipschitz graph. The additional fact that the Lipschitz constant $\ell$ of such a graph can be chosen to be arbitrarily small provided we restrict to sufficiently small neighborhoods of size $\rho$ (not necessary to the proof of the rigidity Theorem) can be obtained by the fact that the estimate \eqref{eq:a22} holds for a larger set of directions outside the $\ell$-cone with base plane $\nu_E(0)^\perp$.
		
		Let us first show that there exists $\rho>0$ such that \eqref{eq:a22} holds. Let $\eps\ll1$ such that
		\begin{equation}\label{eq:elleps}
			\ell(1-\eps)-\eps>0
		\end{equation}
		and let $\rho>0$ such that for all $y\in\partial^*E_L\cap C(0,\rho,\nu)$ one has that $\|\nu_E(0)-\nu_E(y)\|<\eps$.
		Then, for all such $y$ and $\nu$ as above one has that
		\begin{equation}\label{eq:sprodell}
			\langle\nu,\nu_E(y)\rangle\geq \ell\|P_{\nu_E(x)^\perp}(\nu)\|(1-\eps)-\eps\|P_{\nu_E(x)^\perp}(\nu)\|\geq0.
		\end{equation}
		Now observe that for a.e. $y_\nu^\perp$ such that $\#\partial^*E_{y_\nu^\perp}\cap (-\rho,\rho)\geq 2$ then there exists $y=y_{\nu}^\perp+s\nu\in\partial^*E_L$, $s\in(-\rho,\rho)$ such that $\langle\nu_E(y),\nu\rangle<0$, thus contradicting \eqref{eq:sprodell}.

		Let us then show that there exists $\rho>0$ such that  \eqref{eq:a02} holds.
		Assume instead that  $\mathcal H^{d-1}(A_0(\nu^\perp,\rho))>0$ for a set of $\nu\in\S^{d-1}$ of positive measure as above.
		Given $\nu$ belonging to this set, let $\theta\in\S^{d-1}\cap\nu^\perp$ be such that the plane $\pi(\nu,\theta)=\mathrm{span}\{\nu,\theta\}$ satisfies
		\begin{equation}\label{eq:2dslice}
			\overline{\mathcal F}_{0,p,\pi(\nu,\theta)}^{\mathrm{2D}} (E_{x_{\pi(\theta,\nu)}^\perp},\Omega_{x_{\pi(\theta,\nu)}^\perp})<+\infty,\quad \text{for a.e. $x_{\pi(\theta,\nu)}^\perp\in\pi(\nu,\theta)^\perp$.}
		\end{equation}
		By Proposition~\ref{prop:nusliced}, a.e. pair $(\nu,\theta)$ satisfies the above property for a set of finite energy $E$.
		In particular, for every two-dimensional slice as in \eqref{eq:2dslice} one has that $\partial^*E_{x_{\pi(\theta,\nu)}^\perp}\cap\Omega_{x_{\pi(\theta,\nu)}^\perp}$ is flat and composed of finitely many non-intersecting lines.
		Observe that, since $\mathcal H^{d-1}(A_0(\nu^\perp,\rho))>0$, for a.e. $\theta$ as above
		\[
		\mathcal H^{d-2}(P_{\pi(\nu,\theta)^\perp}(A_0(\nu^\perp,\rho)))>0.
		\]
		Hence, for all $x_{\pi(\theta,\nu)}^\perp\in P_{\pi(\nu,\theta)^\perp}(A_0(\nu^\perp,\rho))$, one has that
		\[
		\mathcal H^1((A_0(\nu^\perp,\rho))_{x_{\pi(\theta,\nu)}^\perp})>0,
		\]
		which means that the set
		\begin{equation*}
			\hat A_0(E_{x_{\pi(\theta,\nu)}^\perp}):=\{s\in\R:\, \|x_{\pi(\theta,\nu)}^\perp+s\theta\|<\rho,\, \#\partial^*E_{x_{\pi(\theta,\nu)}^\perp+s\theta}\cap(-\rho,\rho)=0\}
		\end{equation*}
		has positive $\mathcal H^1$-measure.
		In particular, the lines constituting the boundary of $E_{x_{\pi(\theta,\nu)}^\perp}$ cannot cross the set $	\hat A_0(E_{x_{\pi(\theta,\nu)}^\perp})\times(-\rho,\rho)\nu\subset (C(0,\rho,\nu))_{x_{\pi(\theta,\nu)}^\perp}$.
		Since by continuity of the normal and by the choice of $C(0,\rho,\nu)$ one has that the normals $\nu_i$ to such lines $\ell_i$ satisfy $\|\nu_i-\nu\|\leq\eps$, then by an easy geometric argument one can deduce that
		\begin{equation*}
			\mathrm{dist}\bigl(\ell_i, (x_{\pi(\theta,\nu)}^\perp+(-\rho,\rho)\theta)\times \{\pm\rho\nu\}\bigr)\leq \rho\eps.
		\end{equation*}
		In particular, choosing $C(0,\rho-\rho\eps,\nu)$ instead of $C(0,\rho,\nu)$ one has that
		\begin{equation*}
			\mathcal H^{d-1}(A_0(\nu^\perp,\rho-\rho\eps)\cap B_{\nu^\perp,\rho-\rho\eps}(0))=\mathcal H^{d-1}(B_{\nu^\perp,\rho-\rho\eps}(0)),
		\end{equation*}
		which is against the assumption $0\in\partial^*E_L$.
	\end{proof}

	Then we can finally prove Theorem~\ref{thm:rigidity} for a general $\Omega$.
	
	\begin{proof}[Proof of Theorem~\ref{thm:rigidity} for a general bounded open set $\Omega$: ] Thanks to the fact that, by Lemma~\ref{lemma:graph}, for every   $r>0$ one has that  locally uniformly in $\Omega_r$ the boundary of a set of finite energy $E$ can be parametrized by a Lipschitz graph, one can apply Lemma~\ref{lemma:d+3} to the estimate of Proposition~\ref{prop:brezis} and conclude the proof of the Theorem.
	\end{proof}

	\section{Proof of Theorem~\ref{thm:main}}
	\label{sec:positive_tau}
	
	The goal of this section is to prove Theorem~\ref{thm:main}, thus showing that for sufficiently small $\tau>0$ minimizers of $\Fcal_{\tau,p,d}$ are periodic stripes.
	
	By Theorem~\ref{thm:rigidity} and recalling the definition of $h^*_L$ given in~\eqref{eq:hL}, we know that the following holds true.
	
	\begin{corollary}\label{cor:gammaconv}  Let $L>0$, $p\geq d+3$. Then, for every $\sigma>0$ there exists $\tau(\sigma)>0$ such that for all $0<\tau<\tau(\sigma)$ any  $[0,L)^d$-periodic global minimizer $E_\tau$ of $\Fcal_{\tau,p,d}(\cdot, [0,L)^d)$ satisfies
		\begin{equation}\label{eq:l1small}
			\|\chi_{E_\tau}-\chi_{S_\theta}\|_{L^1[0,L)^d}\leq \sigma,
		\end{equation}
		for some $\theta\in \S^{d-1}$ and some  $[0,L)^d$-periodic set $S_\theta$ made of stripes with boundaries orthogonal to $\theta\in\S^{d-1}$, and of constant distance one from the other given by $h^*_L>0$.
	\end{corollary}
	
	In order to prove Theorem~\ref{thm:main}, we will show that among sets  $E_\tau$ satisfying~\eqref{eq:l1small} as above for sufficiently small $\sigma$, there are periodic stripes which have lower energy.
	
	We will need a series of preliminary lemmas.

	Let us now introduce some preliminary notation.

	For any set  $E$ of locally  finite perimeter, $\theta\in\S^{d-1}$ and $\delta>0$, define
	\begin{equation*}
		\begin{split}
			\etautheta(E_{x^{\perp}_{\theta}}, x_{\theta})&:= \begin{cases}
				\frac{1}{\max\{\tau, |x^{+}_{\theta} - x_{\theta}|^{p-d-1}\}}, & \text{if } |x_{\theta} - x_{\theta}^{+}| < \delta, \\
				0, & \text{otherwise}.
			\end{cases}\\
			\etau (E, x) &:= \int_{\S^{d-1}}  |\Scal{\nu_{E}(x)}{\theta}|\etautheta(E_{x^{\perp}_{\theta}}, x_{\theta})\d\theta.
		\end{split}
	\end{equation*}
	
	Notice that for sufficiently small $\tau$, by Proposition~\ref{prop:1dbound},  it holds $\max\{0, r_{\tau}(E_{x_{\theta}^{\perp}}, x_{\theta})\}\gtrsim \etautheta(E_{x^{\perp}_{\theta}}, x_{\theta})$.
	Moreover, if $\delta_{1} < \delta_{2}$, $e_{\tau, \delta_{1}, \theta}<  e_{\tau, \delta_{2}, \theta} $.
	
	The proof of the next lemma is very similar to the proof of Lemma~\ref{lemma:epointwisenu} for $\tau=0$.

	\begin{lemma}
		\label{lemma:small_density}
		There exists a dimensional constant $C_3>0$ such that the following holds. For every set $E$ of locally finite perimeter, $x\in\partial^*E$, $\alpha, r>0$, $r<\delta$  and $0<\tau<r$, whenever
		$|\big(E \Delta H_{\nu_{E}(x)}(x)\big)\cap B_r(x)| >\alpha r^{d}$ it holds $\etau(E, x) > C_3\alpha^{2}/r^{p-d-1}$.
		
	\end{lemma}
	
	\begin{proof}
		
		Define the set
		\[
		\Omega(x,r)=\{\theta\in\S^{d-1}: \,|x_\theta-x_\theta^+|<r\}.
		\]
		By the blow up properties~\eqref{eq:stimeblowup} at points of the reduced boundary and the fact that $|\big(E \Delta H_{\nu_{E}(x)}(x)\big)\cap B_r(x)| >\alpha r^{d}$, one has that
		$|\Omega(x,r)|\geq\alpha$. In particular, there exist $c_1,c_2>0$ such that
		the set $S_\alpha(x)=\{\theta\in\S^{d-1}:\,|\Scal{\nu_E(x)}{\theta}|\geq c_1\alpha\}$ satisfies $|\Omega(x,r)\cap S_\alpha(x)|\geq c_2\alpha$.
		Hence,
		
		\begin{align*}
			e_{\tau,\delta}(E,x)&\geq e_{\tau,r}(E,x)\gtrsim\int_{\Omega(x,r)\cap S_\alpha(x)}\frac{\alpha}{r^{p-d-1}}\d\theta\gtrsim \frac{\alpha^2}{r^{p-d-1}},
		\end{align*}
		thus proving the desired claim.
		
	\end{proof}
	
	Below we consider sets $E$  which are close to a halfspace in a given rectangle and list some geometric/measure theoretic conditions at points $x\in\partial^*E$ implying that the function $e_{\tau,\delta}(E,x)$ is large (see Figure~\ref{fig:x1}).
	
	Given $L,\delta>0$ we denote by $R_{L, \delta} := [0, L)^{d-1}\times (-\delta, \delta)$, by  $H_{L, \delta} := [0, L)^{d-1}\times (-\delta, 0)$.
	
	\begin{rmk}
		\label{rmk:after_small_density_lemma}
		Let $\delta>0$, $L>0$, $\sigma>0$, $\tau<\min\{\delta,\sigma^{1/d}\}$. Denote by $R_{L, \delta} := [0, L)^{d-1}\times (-\delta, \delta)$, by  $H_{L, \delta} := [0, L)^{d-1}\times (-\delta, 0)$ and by $\sigma = \big|(E \Delta H_{L, \delta}) \cap  R_{L, \delta}\big|$. Let $x,y\in\partial^*E\cap R_{L,\delta}$.
		Then the following hold:
		\begin{enumerate}[label= (\roman*)]
			\item
			If $\sigma <|x_d|^{d}/4$ then $\etau(E,x) > c/|x_d|^{p-d-1}$.
			\item
			If $\|\nu_{E}(x) - e_{d}\|  > \frac{\sigma}{4\delta^{d}}$,
			then $\etau(E, x) \geq \|\nu_{E}(x) - e_{d}\|/\delta^{p-d-1}$.
			\item If $\|x-y\|<\delta$, then $\etau(x) + \etau(y) \gtrsim \frac{\|\nu_{E}(x) - \nu_{E}(y)\|^{2}}{\delta^{p-d-1}}$.
		\end{enumerate}
		To prove the first two claims it is sufficient to notice that in the ball $B_{r}(x)$ with $r = |x_d|$ in the first case and $r=\delta$ in the second case we have that, if $x\in \{\chi_{H_{L,\delta}}=0\}$, $|E \cap H_{ \nu_{E}(x)}(x) \cap  B_{r}(x)| < 1/4 r^{d}$ and if  $x\in \{\chi_{H_{L,\delta}}=1\}$, $|(\R^d\setminus E) \cap (\R^d\setminus H_{ \nu_{E}(x)}(x)) \cap  B_{r}(x)| < 1/4 r^{d}$.
		Thus, by using Lemma~\ref{lemma:small_density}, we have the desired claim.
		The last statement follows from directly from Lemma~\ref{lemma:small_density}.
	\end{rmk}
	
	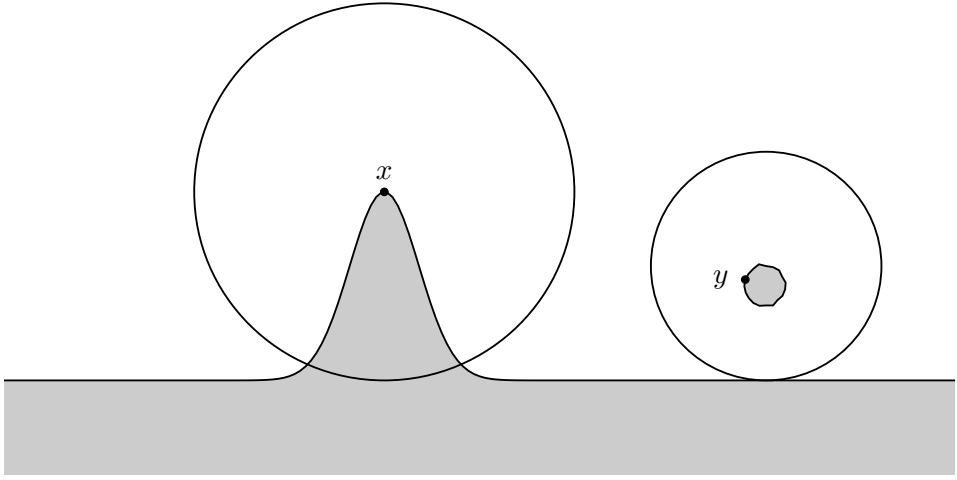
\begin{figure}
		\centering
		\begin{tikzpicture}[scale=2.5]
			\clip (-2,-0.5) rectangle (3,2.05);
			\tikzset{arrow/.style={-latex}}
			\draw[line width=.7pt,fill=black!20!white,smooth] (-100,-100) -- (-100,0) -- (-10,0) -- (-9.97139,0) -- (-9.94278,0) -- (-9.91416,0) -- (-9.88555,0) -- (-9.85694,0) -- (-9.82833,0) -- (-9.79971,0) -- (-9.7711,0) -- (-9.74249,0) -- (-9.71388,0) -- (-9.68526,0) -- (-9.65665,0) -- (-9.62804,0) -- (-9.59943,0) -- (-9.57082,0) -- (-9.5422,0) -- (-9.51359,0) -- (-9.48498,0) -- (-9.45637,0) -- (-9.42775,0) -- (-9.39914,0) -- (-9.37053,0) -- (-9.34192,0) -- (-9.3133,0) -- (-9.28469,0) -- (-9.25608,0) -- (-9.22747,0) -- (-9.19886,0) -- (-9.17024,0) -- (-9.14163,0) -- (-9.11302,0) -- (-9.08441,0) -- (-9.05579,0) -- (-9.02718,0) -- (-8.99857,0) -- (-8.96996,0) -- (-8.94134,0) -- (-8.91273,0) -- (-8.88412,0) -- (-8.85551,0) -- (-8.8269,0) -- (-8.79828,0) -- (-8.76967,0) -- (-8.74106,0) -- (-8.71245,0) -- (-8.68383,0) -- (-8.65522,0) -- (-8.62661,0) -- (-8.598,0) -- (-8.56938,0) -- (-8.54077,0) -- (-8.51216,0) -- (-8.48355,0) -- (-8.45494,0) -- (-8.42632,0) -- (-8.39771,0) -- (-8.3691,0) -- (-8.34049,0) -- (-8.31187,0) -- (-8.28326,0) -- (-8.25465,0) -- (-8.22604,0) -- (-8.19742,0) -- (-8.16881,0) -- (-8.1402,0) -- (-8.11159,0) -- (-8.08298,0) -- (-8.05436,0) -- (-8.02575,0) -- (-7.99714,0) -- (-7.96853,0) -- (-7.93991,0) -- (-7.9113,0) -- (-7.88269,0) -- (-7.85408,0) -- (-7.82546,0) -- (-7.79685,0) -- (-7.76824,0) -- (-7.73963,0) -- (-7.71102,0) -- (-7.6824,0) -- (-7.65379,0) -- (-7.62518,0) -- (-7.59657,0) -- (-7.56795,0) -- (-7.53934,0) -- (-7.51073,0) -- (-7.48212,0) -- (-7.45351,0) -- (-7.42489,0) -- (-7.39628,0) -- (-7.36767,0) -- (-7.33906,0) -- (-7.31044,0) -- (-7.28183,0) -- (-7.25322,0) -- (-7.22461,0) -- (-7.19599,0) -- (-7.16738,0) -- (-7.13877,0) -- (-7.11016,0) -- (-7.08155,0) -- (-7.05293,0) -- (-7.02432,0) -- (-6.99571,0) -- (-6.9671,0) -- (-6.93848,0) -- (-6.90987,0) -- (-6.88126,0) -- (-6.85265,0) -- (-6.82403,0) -- (-6.79542,0) -- (-6.76681,0) -- (-6.7382,0) -- (-6.70959,0) -- (-6.68097,0) -- (-6.65236,0) -- (-6.62375,0) -- (-6.59514,0) -- (-6.56652,0) -- (-6.53791,0) -- (-6.5093,0) -- (-6.48069,0) -- (-6.45207,0) -- (-6.42346,0) -- (-6.39485,0) -- (-6.36624,0) -- (-6.33763,0) -- (-6.30901,0) -- (-6.2804,0) -- (-6.25179,0) -- (-6.22318,0) -- (-6.19456,0) -- (-6.16595,0) -- (-6.13734,0) -- (-6.10873,0) -- (-6.08011,0) -- (-6.0515,0) -- (-6.02289,0) -- (-5.99428,0) -- (-5.96567,0) -- (-5.93705,0) -- (-5.90844,0) -- (-5.87983,0) -- (-5.85122,0) -- (-5.8226,0) -- (-5.79399,0) -- (-5.76538,0) -- (-5.73677,0) -- (-5.70815,0) -- (-5.67954,0) -- (-5.65093,0) -- (-5.62232,0) -- (-5.59371,0) -- (-5.56509,0) -- (-5.53648,0) -- (-5.50787,0) -- (-5.47926,0) -- (-5.45064,0) -- (-5.42203,0) -- (-5.39342,0) -- (-5.36481,0) -- (-5.33619,0) -- (-5.30758,0) -- (-5.27897,0) -- (-5.25036,0) -- (-5.22175,0) -- (-5.19313,0) -- (-5.16452,0) -- (-5.13591,0) -- (-5.1073,0) -- (-5.07868,0) -- (-5.05007,0) -- (-5.02146,0) -- (-4.99285,0) -- (-4.96423,0) -- (-4.93562,0) -- (-4.90701,0) -- (-4.8784,0) -- (-4.84979,0) -- (-4.82117,0) -- (-4.79256,0) -- (-4.76395,0) -- (-4.73534,0) -- (-4.70672,0) -- (-4.67811,0) -- (-4.6495,0) -- (-4.62089,0) -- (-4.59227,0) -- (-4.56366,0) -- (-4.53505,0) -- (-4.50644,0) -- (-4.47783,0) -- (-4.44921,0) -- (-4.4206,0) -- (-4.39199,0) -- (-4.36338,0) -- (-4.33476,0) -- (-4.30615,0) -- (-4.27754,0) -- (-4.24893,0) -- (-4.22031,0) -- (-4.1917,0) -- (-4.16309,0) -- (-4.13448,0) -- (-4.10587,0) -- (-4.07725,0) -- (-4.04864,0) -- (-4.02003,0) -- (-3.99142,0) -- (-3.9628,0) -- (-3.93419,0) -- (-3.90558,0) -- (-3.87697,0) -- (-3.84835,0) -- (-3.81974,0) -- (-3.79113,0) -- (-3.76252,0) -- (-3.73391,0) -- (-3.70529,0) -- (-3.67668,0) -- (-3.64807,0) -- (-3.61946,0) -- (-3.59084,0) -- (-3.56223,0) -- (-3.53362,0) -- (-3.50501,0) -- (-3.47639,0) -- (-3.44778,0) -- (-3.41917,0) -- (-3.39056,0) -- (-3.36195,0) -- (-3.33333,0) -- (-3.30472,0) -- (-3.27611,0) -- (-3.2475,0) -- (-3.21888,0) -- (-3.19027,0) -- (-3.16166,0) -- (-3.13305,0) -- (-3.10443,0) -- (-3.07582,0) -- (-3.04721,0) -- (-3.0186,0) -- (-2.98999,0) -- (-2.96137,0) -- (-2.93276,0) -- (-2.90415,0) -- (-2.87554,0) -- (-2.84692,0) -- (-2.81831,0) -- (-2.7897,0) -- (-2.76109,0) -- (-2.73247,0) -- (-2.70386,0) -- (-2.67525,0) -- (-2.64664,0) -- (-2.61803,0) -- (-2.58941,0) -- (-2.5608,0) -- (-2.53219,0) -- (-2.50358,0) -- (-2.47496,0) -- (-2.44635,0) -- (-2.41774,0) -- (-2.38913,0) -- (-2.36052,0) -- (-2.3319,0) -- (-2.30329,0) -- (-2.27468,0) -- (-2.24607,0) -- (-2.21745,0) -- (-2.18884,0) -- (-2.16023,0) -- (-2.13162,0) -- (-2.103,0) -- (-2.07439,0) -- (-2.04578,0) -- (-2.01717,0) -- (-1.98856,0) -- (-1.95994,0) -- (-1.93133,0) -- (-1.90272,0) -- (-1.87411,0) -- (-1.84549,0) -- (-1.81688,0) -- (-1.78827,0) -- (-1.75966,0) -- (-1.73104,0) -- (-1.70243,0) -- (-1.67382,0) -- (-1.64521,0) -- (-1.6166,0) -- (-1.58798,0) -- (-1.55937,0) -- (-1.53076,0) -- (-1.50215,0) -- (-1.47353,0) -- (-1.44492,0) -- (-1.41631,0) -- (-1.3877,0) -- (-1.35908,0) -- (-1.33047,0) -- (-1.30186,0) -- (-1.27325,0) -- (-1.24464,0) -- (-1.21602,0) -- (-1.18741,0) -- (-1.1588,0) -- (-1.13019,0) -- (-1.10157,0) -- (-1.07296,0) -- (-1.04435,0) -- (-1.01574,0) -- (-0.98712,0) -- (-0.95851,0) -- (-0.9299,0) -- (-0.90129,0) -- (-0.87268,0.00001) -- (-0.84406,0.00001) -- (-0.81545,0.00003) -- (-0.78684,0.00006) -- (-0.75823,0.00013) -- (-0.72961,0.00026) -- (-0.701,0.00049) -- (-0.67239,0.00092) -- (-0.64378,0.00168) -- (-0.61516,0.00297) -- (-0.58655,0.0051) -- (-0.55794,0.00851) -- (-0.52933,0.01382) -- (-0.50072,0.02185) -- (-0.4721,0.03361) -- (-0.44349,0.05034) -- (-0.41488,0.07343) -- (-0.38627,0.10432) -- (-0.35765,0.14441) -- (-0.32904,0.19481) -- (-0.30043,0.25614) -- (-0.27182,0.32833) -- (-0.2432,0.41036) -- (-0.21459,0.50014) -- (-0.18598,0.59451) -- (-0.15737,0.6893) -- (-0.12876,0.77962) -- (-0.10014,0.86025) -- (-0.07153,0.9261) -- (-0.04292,0.97275) -- (-0.01431,0.99693) -- (0.01431,0.99693) -- (0.04292,0.97275) -- (0.07153,0.9261) -- (0.10014,0.86025) -- (0.12876,0.77962) -- (0.15737,0.6893) -- (0.18598,0.59451) -- (0.21459,0.50014) -- (0.2432,0.41036) -- (0.27182,0.32833) -- (0.30043,0.25614) -- (0.32904,0.19481) -- (0.35765,0.14441) -- (0.38627,0.10432) -- (0.41488,0.07343) -- (0.44349,0.05034) -- (0.4721,0.03361) -- (0.50072,0.02185) -- (0.52933,0.01382) -- (0.55794,0.00851) -- (0.58655,0.0051) -- (0.61516,0.00297) -- (0.64378,0.00168) -- (0.67239,0.00092) -- (0.701,0.00049) -- (0.72961,0.00026) -- (0.75823,0.00013) -- (0.78684,0.00006) -- (0.81545,0.00003) -- (0.84406,0.00001) -- (0.87268,0.00001) -- (0.90129,0) -- (0.9299,0) -- (0.95851,0) -- (0.98712,0) -- (1.01574,0) -- (1.04435,0) -- (1.07296,0) -- (1.10157,0) -- (1.13019,0) -- (1.1588,0) -- (1.18741,0) -- (1.21602,0) -- (1.24464,0) -- (1.27325,0) -- (1.30186,0) -- (1.33047,0) -- (1.35908,0) -- (1.3877,0) -- (1.41631,0) -- (1.44492,0) -- (1.47353,0) -- (1.50215,0) -- (1.53076,0) -- (1.55937,0) -- (1.58798,0) -- (1.6166,0) -- (1.64521,0) -- (1.67382,0) -- (1.70243,0) -- (1.73104,0) -- (1.75966,0) -- (1.78827,0) -- (1.81688,0) -- (1.84549,0) -- (1.87411,0) -- (1.90272,0) -- (1.93133,0) -- (1.95994,0) -- (1.98856,0) -- (2.01717,0) -- (2.04578,0) -- (2.07439,0) -- (2.103,0) -- (2.13162,0) -- (2.16023,0) -- (2.18884,0) -- (2.21745,0) -- (2.24607,0) -- (2.27468,0) -- (2.30329,0) -- (2.3319,0) -- (2.36052,0) -- (2.38913,0) -- (2.41774,0) -- (2.44635,0) -- (2.47496,0) -- (2.50358,0) -- (2.53219,0) -- (2.5608,0) -- (2.58941,0) -- (2.61803,0) -- (2.64664,0) -- (2.67525,0) -- (2.70386,0) -- (2.73247,0) -- (2.76109,0) -- (2.7897,0) -- (2.81831,0) -- (2.84692,0) -- (2.87554,0) -- (2.90415,0) -- (2.93276,0) -- (2.96137,0) -- (2.98999,0) -- (3.0186,0) -- (3.04721,0) -- (3.07582,0) -- (3.10443,0) -- (3.13305,0) -- (3.16166,0) -- (3.19027,0) -- (3.21888,0) -- (3.2475,0) -- (3.27611,0) -- (3.30472,0) -- (3.33333,0) -- (3.36195,0) -- (3.39056,0) -- (3.41917,0) -- (3.44778,0) -- (3.47639,0) -- (3.50501,0) -- (3.53362,0) -- (3.56223,0) -- (3.59084,0) -- (3.61946,0) -- (3.64807,0) -- (3.67668,0) -- (3.70529,0) -- (3.73391,0) -- (3.76252,0) -- (3.79113,0) -- (3.81974,0) -- (3.84835,0) -- (3.87697,0) -- (3.90558,0) -- (3.93419,0) -- (3.9628,0) -- (3.99142,0) -- (4.02003,0) -- (4.04864,0) -- (4.07725,0) -- (4.10587,0) -- (4.13448,0) -- (4.16309,0) -- (4.1917,0) -- (4.22031,0) -- (4.24893,0) -- (4.27754,0) -- (4.30615,0) -- (4.33476,0) -- (4.36338,0) -- (4.39199,0) -- (4.4206,0) -- (4.44921,0) -- (4.47783,0) -- (4.50644,0) -- (4.53505,0) -- (4.56366,0) -- (4.59227,0) -- (4.62089,0) -- (4.6495,0) -- (4.67811,0) -- (4.70672,0) -- (4.73534,0) -- (4.76395,0) -- (4.79256,0) -- (4.82117,0) -- (4.84979,0) -- (4.8784,0) -- (4.90701,0) -- (4.93562,0) -- (4.96423,0) -- (4.99285,0) -- (5.02146,0) -- (5.05007,0) -- (5.07868,0) -- (5.1073,0) -- (5.13591,0) -- (5.16452,0) -- (5.19313,0) -- (5.22175,0) -- (5.25036,0) -- (5.27897,0) -- (5.30758,0) -- (5.33619,0) -- (5.36481,0) -- (5.39342,0) -- (5.42203,0) -- (5.45064,0) -- (5.47926,0) -- (5.50787,0) -- (5.53648,0) -- (5.56509,0) -- (5.59371,0) -- (5.62232,0) -- (5.65093,0) -- (5.67954,0) -- (5.70815,0) -- (5.73677,0) -- (5.76538,0) -- (5.79399,0) -- (5.8226,0) -- (5.85122,0) -- (5.87983,0) -- (5.90844,0) -- (5.93705,0) -- (5.96567,0) -- (5.99428,0) -- (6.02289,0) -- (6.0515,0) -- (6.08011,0) -- (6.10873,0) -- (6.13734,0) -- (6.16595,0) -- (6.19456,0) -- (6.22318,0) -- (6.25179,0) -- (6.2804,0) -- (6.30901,0) -- (6.33763,0) -- (6.36624,0) -- (6.39485,0) -- (6.42346,0) -- (6.45207,0) -- (6.48069,0) -- (6.5093,0) -- (6.53791,0) -- (6.56652,0) -- (6.59514,0) -- (6.62375,0) -- (6.65236,0) -- (6.68097,0) -- (6.70959,0) -- (6.7382,0) -- (6.76681,0) -- (6.79542,0) -- (6.82403,0) -- (6.85265,0) -- (6.88126,0) -- (6.90987,0) -- (6.93848,0) -- (6.9671,0) -- (6.99571,0) -- (7.02432,0) -- (7.05293,0) -- (7.08155,0) -- (7.11016,0) -- (7.13877,0) -- (7.16738,0) -- (7.19599,0) -- (7.22461,0) -- (7.25322,0) -- (7.28183,0) -- (7.31044,0) -- (7.33906,0) -- (7.36767,0) -- (7.39628,0) -- (7.42489,0) -- (7.45351,0) -- (7.48212,0) -- (7.51073,0) -- (7.53934,0) -- (7.56795,0) -- (7.59657,0) -- (7.62518,0) -- (7.65379,0) -- (7.6824,0) -- (7.71102,0) -- (7.73963,0) -- (7.76824,0) -- (7.79685,0) -- (7.82546,0) -- (7.85408,0) -- (7.88269,0) -- (7.9113,0) -- (7.93991,0) -- (7.96853,0) -- (7.99714,0) -- (8.02575,0) -- (8.05436,0) -- (8.08298,0) -- (8.11159,0) -- (8.1402,0) -- (8.16881,0) -- (8.19742,0) -- (8.22604,0) -- (8.25465,0) -- (8.28326,0) -- (8.31187,0) -- (8.34049,0) -- (8.3691,0) -- (8.39771,0) -- (8.42632,0) -- (8.45494,0) -- (8.48355,0) -- (8.51216,0) -- (8.54077,0) -- (8.56938,0) -- (8.598,0) -- (8.62661,0) -- (8.65522,0) -- (8.68383,0) -- (8.71245,0) -- (8.74106,0) -- (8.76967,0) -- (8.79828,0) -- (8.8269,0) -- (8.85551,0) -- (8.88412,0) -- (8.91273,0) -- (8.94134,0) -- (8.96996,0) -- (8.99857,0) -- (9.02718,0) -- (9.05579,0) -- (9.08441,0) -- (9.11302,0) -- (9.14163,0) -- (9.17024,0) -- (9.19886,0) -- (9.22747,0) -- (9.25608,0) -- (9.28469,0) -- (9.3133,0) -- (9.34192,0) -- (9.37053,0) -- (9.39914,0) -- (9.42775,0) -- (9.45637,0) -- (9.48498,0) -- (9.51359,0) -- (9.5422,0) -- (9.57082,0) -- (9.59943,0) -- (9.62804,0) -- (9.65665,0) -- (9.68526,0) -- (9.71388,0) -- (9.74249,0) -- (9.7711,0) -- (9.79971,0) -- (9.82833,0) -- (9.85694,0) -- (9.88555,0) -- (9.91416,0) -- (9.94278,0) -- (9.97139,0) -- (10,0) -- (100,0) -- (100,-100);
			\draw[line width=.7pt] (0,1) circle[radius=1.0];
			\draw[fill] (0,1) circle[radius=0.02];
			\draw (0,1.1) node {$x$};
			\draw[smooth,fill=black!20!white,line width=.7pt] (1.89378,0.5) -- (1.90113,0.46606) -- (1.91788,0.43608) -- (1.94156,0.41055) -- (1.97356,0.39561) -- (2.00848,0.39767) -- (2.04497,0.39748) -- (2.06839,0.42571) -- (2.09474,0.44873) -- (2.10829,0.48193) -- (2.11143,0.51859) -- (2.09287,0.55026) -- (2.07713,0.58379) -- (2.04426,0.6009) -- (2.00882,0.6064) -- (1.97066,0.61585) -- (1.94043,0.59117) -- (1.91609,0.56531) -- (1.89897,0.53468) --cycle;
			\draw[line width=.7pt] (2.00882,0.6064) circle[radius=0.6063959419710523];
			\draw[fill] (1.89897,0.53468) circle[radius=0.02];
			\draw (1.77,0.53468) node {$y$};
		\end{tikzpicture}
		\caption{
			In the above figure $x$ corresponds to \textit{(i)} in Remark~\ref{rmk:after_small_density_lemma}, and $y$ corresponds to \textit{(ii)} in Remark~\ref{rmk:after_small_density_lemma}.
		}
		\label{fig:x1}
	\end{figure}

	The inequality of the following lemma will be used in the proof of Proposition~\ref{prop:stability_bound}.
	
	\begin{lemma}
		\label{lemma:mean_control}
		There exists a constant $C_4>0$ such that the following holds. Let $\rho>0$, $L>0$ and  ${Q}_\rho(x)$ the $d$-dimensional cube of centre $x$ and side length $\rho>0$, for all $x\in[0,L)^d$.
		Let $f \in L^{2}_{\loc}(\R^{d}; \R^{d})$ be a $[0,L]^{d}$-periodic function. Assume that $\int_{[0,L)^d}f \dx=0$.
		Then,
		\begin{equation}
			\label{eq:o25}
			\frac{C_4}{\rho^{2}}\int_{[0,L)^d} \fint_{{Q}_\rho(x)}\Big\|f(y) - \fint_{{Q}_\rho(x)} f(z)\dz \Big\|^{2} \dy \dx \geq
			\int_{[0,L)^d} \|f(x)\|^{2} \dx.
		\end{equation}
	\end{lemma}

	\begin{proof}

		Assume that the statement of the lemma is false. Thus, there exist $\rho_{n} \downarrow 0$ and $f_{n}\in L^{2}_{\loc}(\R^{d}; \R^{d})$ a sequence of $[0,L]^{d}$-periodic functions such that
		\begin{equation}\label{eq:1overn}
			\frac{1}{\rho_{n}^{2}}\int_{[0,L)^d} \fint_{{Q}_{\rho_{n}}(x)}\Big\|f_{n}(y) - \fint_{{Q}_{\rho_{n}}(x)} f_{n}(z)\dz \Big\|^{2} \dy \dx \leq \frac{1}{n}
			\int_{[0,L)^d} \|f_{n}(x)\|^{2} \dx.
		\end{equation}
		
		\WithoutLoss, we can assume that $\int_{[0,L)^d}\|f_{n}(x)\|^{2} \dx= 1$. We can also assume that $f_{n} \to f_{0}$ weakly in $L^{2}([0,L)^d;\R^d)$.
		Let $\varepsilon>0$ and let $\varphi_{\varepsilon}$ be a convolution kernel. Then $f_{n}*\varphi_{\varepsilon} \to f_{0}* \varphi_{\varepsilon}$ strongly in $L^{2}([0,L)^d;\R^d)$.
		In particular
		\begin{equation}\label{eq:strongc}
			\int_{[0,L)^d}\|f_{n} * \varphi_{\varepsilon} (x)\|^{2}\dx \to \int_{([0,L)^d)} \|f_{0}*\varphi_{\varepsilon}(x)\|^{2}\dx.
		\end{equation}
		
		We claim that
		\begin{align}\label{eq:fstarineq}
			\int_{[0,L)^d} \fint_{{Q}_{\rho_{n}}(x)}&\Big\|f_{n}*\varphi_\varepsilon(y) - \fint_{{Q}_{\rho_{n}}(x)} f_{n}*\varphi_\varepsilon(z)\dz \Big\|^{2} \dy \dx \leq\notag\\
			&\leq \int_{[0,L)^d} \fint_{{Q}_{\rho_{n}}(x)}\Big\|f_{n}(y) - \fint_{{Q}_{\rho_{n}}(x)} f_{n}(z)\dz \Big\|^{2} \dy \dx.
		\end{align}
		To prove~\eqref{eq:fstarineq}, denote first by $\psi_{\rho} = \frac{1}{\rho^{d}}\chi_{Q_{\rho}(0)}$.
		For every function $h \in L^{1}_{\loc}$ one has that $\fint_{Q_{\rho}(x)} h(y) \dy= h*\psi_{\rho}(x)$.
		By  Jensen inequality and  Fubini Theorem, we have that
		\begin{align}
			\int_{[0,L)^{d}}\fint_{Q_{\rho_n}(x)} &\Big\|f_{n}*\varphi_{\varepsilon}(y) - \fint_{Q_{\rho_n}(x)}f_{n}*\varphi_{\varepsilon}(z) \dz\Big\|^{2}\dy\dx=\notag\\
			&=	\int_{[0,L)^{d}} \fint_{Q_{\rho_n}(x)} \Big\|f_{n}*\varphi_{\varepsilon}(y) - f_{n}*\varphi_{\varepsilon}*\psi_{\rho_n}(x) \Big\|^{2}\dy\dx\notag\\
			&=	\int_{[0,L)^{d}} \fint_{Q_{\rho_n}(x)}\Bigl\|\int_{\R^d}\Bigl(f_n(y-z)-f_n*\psi_{\rho_n}(x-z)\Bigr)\varphi_\varepsilon(z)\dz\Bigr\|^2\dy\dx\notag\\
			&\leq
			\int_{[0,L)^{d}}\int_{\R^d}\Biggl[\fint_{Q_{\rho_n}(x)} \Big\|(f_{n}(y-z) - f_{n}*\psi_{\rho_n}(x-z) \Big\|^{2}\dy\Biggr]\varphi_{\varepsilon}(z) \dz\dx\notag\\
			& =	\int_{[0,L)^{d}}\int_{\R^d} h_n(x-z)\varphi_\varepsilon (z)\dz\dx,
		\end{align}
		where
		\begin{equation}
			h_n(x):=\fint_{Q_{\rho_n}(0)}\|f_n(x+t)-f_n*\psi_{\rho_n}(x)\|^2\dt=\fint_{{Q}_{\rho_{n}}(x)}\Big\|f_{n}(y) - \fint_{{Q}_{\rho_{n}}(x)} f_{n}(z)\dz \Big\|^{2} \dy.
		\end{equation}
		By $[0,L)^d$-periodicity,
		\begin{equation}
			\int_{[0,L)^{d}}\int_{\R^d} h_n(x-z)\varphi_\varepsilon (z)\dz\dx=	\int_{[0,L)^{d}}h_n(x)\dx,
		\end{equation}
		thus the claim~\eqref{eq:fstarineq} is proved.
		Thanks to~\eqref{eq:fstarineq} and~\eqref{eq:strongc}, eventually convolving with a kernel $\varphi_{\varepsilon}$ we can in addition assume that $f_{n} \to f_{0}$ strongly  in $L^{2}$ and $f_{n}$ and $f_{0}$ are uniformly $C^{2}$.
		Under these regularity conditions  it is not difficult to see
		\begin{equation*}
			\frac{1}{\rho_{n}^{2}} \int_{[0,L)^d}  \fint_{Q_{\rho_{n}}(x)}\Bigl\|f_{n}(y)-\fint_{Q_{\rho_{n}}(x)} f_{n}(z)\dz\Bigr\|^2\dy \to \int_{[0,L)^d} \|\nabla f_{0}(x)\|^{2}\dx,
		\end{equation*}
		thus by using Poincaré inequality,~\eqref{eq:1overn} and the fact that $\|f_0\|_{L^2([0,L)^d;\R^d)}=1$,  we have the desired contradiction.
	\end{proof}

	\begin{proposition}
		\label{prop:stability_bound}
		Let $L,\delta, C, M,M_1>0$. Then, there exist $\sigma_0,\tau_2>0$ such that for every $0<\sigma<\sigma_0$, $0<\tau<\tau_2$, for every $[0,L)^d$-periodic set $E$ of locally finite perimeter with $\Fcal_{\tau,p,d}(E,[0,L)^d) < M$ and  $\sigma = |E \Delta H_{L,\delta}|$ it holds
		\begin{equation}
			\label{eq:34}
			\int_{\partial^{*} E \cap R_{L,\delta}}   \etau(E, x) \d\hausd^{d-1} (x)\geq M_1  \bigg( \per(E; R_{L,\delta}) - \Big\| \int_{\partial^{*}E \cap R_{L,\delta }} \nu_{E} (x)\d\hausd^{d-1}(x) \Big\|  \bigg).
		\end{equation}
	\end{proposition}
	
	\begin{proof}

		From the uniform bound $\Fcal_{\tau, p,d}(E,[0,L)^d) < M$, by Corollary~\ref{cor:fboundsp} we have that for every $\tilde \tau>0$  there exists a constant $\tilde C$ such that $\per(E; [0,L)^d) < \tilde C$ whenever $0<\tau <\tilde \tau$.
		By $[0,L)^d$-periodicity,
		\begin{equation}\label{eq:nued}
			\int_{\partial^{*}E \cap R_{L, \delta}}   \nu_{E}(x)  \d\hausd^{d-1}(x) = e_{d} \Big\|
			\int_{\partial^{*}E \cap R_{L, \delta}}   \nu_{E}(x)
			\d\hausd^{d-1}(x)
			\Big\|
		\end{equation}
		Thus
		\begin{equation*}
			\Big\|
			\int_{\partial^{*}E \cap R_{L,\delta}} \nu_{E}(x) \d\hausd^{d-1}(x)
			\Big\| =
			\Big|
			\int_{\partial^{*}E \cap R_{L,\delta}} \Scal{\nu_{E}(x)}{e_{d}} \d\hausd^{d-1}(x)
			\Big|
		\end{equation*}
		For simplicity of notation let us denote by $\nu_{d}(x) := e_{d}\Scal{\nu_{E}(x)}{e_{d}}$ and by $\nu^{\perp}_{d}(x) = \nu_{E}(x)- \nu_{d}(x)$

		Let now fix $\varepsilon>0$.  The range of admissible values for  $\varepsilon$ will be apparent from the proof, but it is helpful to anticipate that $\varepsilon$ will be chosen in such a way that  $\tau_2 < \varepsilon \ll \delta$ and $\sigma_{0} \ll \varepsilon^{d}$.
		Moreover, denote by
		\begin{equation*}
			\begin{split}
				A_{1} &:= \insieme{x \in \partial^{*}E:\ \|\nu_{E}(x) - e_{d}\| > \delta},\\
				A_{2} &:= \insieme{x =(x^{\perp}_{d}, x_{d}) \in \partial^{*}E:\  |x_{d}| > \varepsilon}
			\end{split}
		\end{equation*}
		and by
		\begin{equation*}
			\begin{split}
				\Omega_{0} &:= \insieme{x^{\perp}_{d} \in [0, L]^{d-1}: \per^{\mathrm{1D}}(E_{x^{\perp}_{d}}, (-\delta,\delta)) = 0} \qquad \Omega_{0}^{\delta} := \Omega_{0} \times (-\delta, \delta)\\
				\Omega_{1} &:= \insieme{x^{\perp}_{d} \in [0, L]^{d-1}: \per^{\mathrm{1D}}(E_{x^{\perp}_{d}},(-\delta,\delta)) = 1} \qquad \Omega_{1}^{\delta} := \Omega_{1} \times (-\delta, \delta)\\
				\Omega_{2} &:= \insieme{x^{\perp}_{d} \in [0, L]^{d-1}: \per^{\mathrm{1D}}(E_{x^{\perp}_{d}},(-\delta,\delta)) \geq 2} \qquad \Omega_{2}^{\delta} := \Omega_{2} \times (-\delta, \delta)\\
			\end{split}
		\end{equation*}

		In particular, we have
		\begin{align}
			&\int_{\partial^{*}E \cap R_{L,\delta}} \|\nu_{E}(x)\|\d\hausd^{d-1}(x) - \Big\|
			\int_{\partial^{*}E \cap R_{L,\delta}} \nu_{d} (x)\d\hausd^{d-1}(x)
			\Big\| \leq \notag\\
			&
			\leq\int_{(A_{1} \cup A_{2})^{C} \cap R_{L,\delta }}\|\nu_{E}(x)\|\d\hausd^{d-1}(x)
			-\Big\|
			\int_{(A_{1} \cup A_{2})^{C} \cap R_{L,\delta }} \nu_{d}(x) \d\hausd^{d-1}(x)
			\Big\|\notag\\
			&+ 2 \hausd^{d-1}((A_{1} \cup A_{2}) \cap R_{L,\delta})
		\end{align}

		Using Remark~\ref{rmk:after_small_density_lemma}, we have that for every $x \in (A_{1} \cup A_{2})\cap  R_{L,\delta}$, it holds $\etau(E, x)\gtrsim 1/\varepsilon^{p-d-1}$.
		Thus,
		\begin{equation}\label{eq:ehausd}
			\int_{\partial^{*}E\cap R_{L,\delta}} \etau (E, x) \d\hausd^{d-1}(x) > \frac{c}{\varepsilon^{p-d-1}} \hausd^{d-1}((A_{1}\cup A_{2})\cap R_{L,\delta})
		\end{equation}
		for some constant $c$.
		In particular from the above  if $\hausd^{d-1}((A_{1}\cup  A_{2})\cap R_{L,\delta}) > \frac{\varepsilon^{p-d-1}M}{c}\per(E; R_{L,\delta})$, then~\eqref{eq:34} is trivially satisfied. Thus, we can assume \withoutloss that
		\begin{equation}
			\label{eq:38}
			\hausd^{d-1}((A_{1}\cup  A_{2})\cap R_{L,\delta}) < \frac{\varepsilon^{p-d-1}M}{c} \per(E; R_{L,\delta })
		\end{equation}
		and given the uniform bound on $\per(E; R_{L,\delta})$ we can assume \withoutloss that $\hausd^{d-1}((A_{1}\cup A_{2})\cap R_{L,\delta}) \ll 1$.

		Moreover,
		\begin{align}
			\int_{\partial^{*}E\cap R_{L,\delta}} \|\nu^{\perp}_{d}(x)\| \d\hausd^{d-1}(x)&=
			\int_{A_{1}\cap R_{L,\delta}} \|\nu^{\perp}_{d}(x)\|  \d\hausd^{d-1}(x)+  \int_{A_{1}^{C}\cap R_{L,\delta}} \|\nu^{\perp}_{d}(x)\| \d\hausd^{d-1}(x) \notag\\
			&\leq
			\int_{A_{1}\cap R_{L,\delta}} \|\nu_{E}(x)\|  \d\hausd^{d-1}(x)+  \delta \hausd^{d-1}(A_{1}^{C} \cap R_{L,\delta})
			\notag\\ & \lesssim \hausd^{d-1}(A_{1}\cap R_{L,\delta }) + \delta \hausd^{d-1}(A^{C}_{1}\cap R_{L,\delta}) \\
			& \lesssim \big(
			\varepsilon^{p-d-1}M+ \delta
			\big)\per(E; R_{L,\delta}) \ll 1
		\end{align}
		for $\varepsilon, \delta$ sufficiently small.
		
		Using the triangle inequality we have that
		\begin{align}
			\int_{\partial^{*}E \cap \Omega_{1}^{\delta}}\|\nu_{E}(x)\|\d\hausd^{d-1}(x)&+
			\int_{\partial^{*}E \cap \Omega_{2}^{\delta}}\|\nu_{E}(x)\|\d\hausd^{d-1}(x)\notag\\
			&
			-
			\Big\|\int_{\partial^{*}E \cap \Omega_{1}^{\delta}}\nu_{E}(x)\d\hausd^{d-1}(x)+
			\int_{\partial^{*}E \cap \Omega_{2}^{\delta}}\nu_{E}(x)\d\hausd^{d-1}(x)
			\Big\|\notag\\
			&\leq
			\int_{\partial^{*}E \cap \Omega_{1}^{\delta}}\|\nu_{E}(x)\|\d\hausd^{d-1}(x)
			-\Big\|
			\int_{\partial^{*}E \cap \Omega_{1}^{\delta}}\nu_{d}(x)\d\hausd^{d-1}(x)
			\Big\|\notag
			\\
			& +2\int_{\partial^{*}E \cap \Omega_{2}^{\delta}}\|\nu_{E}(x)\|\d\hausd^{d-1}(x).
			\label{eq:stimaproppart1}
		\end{align}
		
		We now show that
		
		\begin{equation}
			\label{eq:36}
			\int_{\partial^{*}E \cap \Omega_{2}^{\delta}} \etau(E, x) \d\mathcal H^{d-1}(x)\gtrsim \frac{1}{\varepsilon^{p-d-1}} \int_{\partial^{*} E\cap \Omega^{\delta}_{2} } \|\nu_{E}(x)\|\d\hausd^{d-1}(x).
		\end{equation}

		Notice that whenever  $x^{\perp}_{d} \in \Omega_{2}$ and $s, s^{+} \in \partial  E_{x^{\perp}_{d}}$, then
		\begin{equation*}
			\sgn(
			\Scal{\nu_{E}(x^{\perp}_{d},s)}
			{e_{d}}
			)
			\neq
			\sgn(
			\Scal{\nu_{E}(x^{\perp}_{d},s^{+})}
			{e_{d}}
			)
		\end{equation*}
		Thus, either $(x^{\perp}_{d}, s)\in A_{1}$ or $(x^{\perp}_{d}, s^{+})\in A_{1}$.
		In particular, we have that either $e_{\tau,\delta,e_d}(E_{x^{\perp}_{d}}, s) > c/\varepsilon^{p-d-1}$ or $e_{\tau,\delta,e_d}(E_{x^{\perp}_{d}}, s^+)  > c/\varepsilon^{p-d-1}$.
		Hence, we have that
		\begin{align}
			\int_{\partial^{*}E \cap \Omega_{2}^{\delta}} \etau(E, x) \d\hausd^{d-1}(x)
			&\geq
			\int_{\partial^{*}E \cap \Omega_{2}^{\delta}} |\Scal{\nu_{E}(x)}{e_{d}}|\etau(E, x) \d\hausd^{d-1}(x)\notag\\
			&
			\gtrsim
			\int_{\Omega_{2}} \sum_{s \in \partial E_{x^{\perp}_{d}}\cap(-\delta,\delta)} e_{\tau,\delta,e_d}(E_{x^{\perp}_{d}}, s) \dx^{\perp}_{d} \notag\\
			& \gtrsim   \frac{1}{\varepsilon^{p-d-1}}\int_{\Omega_{2}} \frac{\per^{\mathrm{1D}}(E_{x^{\perp}_{d}},(-\delta,\delta)) -1}{2} \dx^{\perp}_{d} \notag\\
			&
			\gtrsim               \frac{1}{\varepsilon^{p-d-1}} \int_{\Omega_{2}}
			\per^{\mathrm{1D}}(E_{x^{\perp}_{d}},(-\delta,\delta)) \dx^{\perp}_{d}. \label{eq:35}
		\end{align}
		
		On the other hand,
		\begin{equation}\label{eq:nueomega2}
			\begin{split}
				\int_{\partial^{*} E \cap  \Omega^{\delta}_{2}} \|\nu_{E}(x)\| \d\hausd^{d-1}(x)  =
				\int_{\partial^{*} E \cap  \Omega^{\delta}_{2} \cap A_{1}}\|\nu_{E}(x)\|\d\hausd^{d-1}(x)
				+     \int_{\partial^{*} E \cap  \Omega^{\delta}_{2} \cap A^{C}_{1}} \|\nu_{E}(x)\| \d\hausd^{d-1}(x) .
			\end{split}
		\end{equation}
		Given that in $A_{1}^{C}$ we have that $\|\nu_{E}(x)- e_{d}\|<\delta$, we have that
		\begin{equation}\label{eq:nueper}
			\int_{\Omega^{\delta}_{2} \cap \partial^{*} E \cap A^{C}_{1}} \|\nu_{E}(x)\|\d\hausd^{d-1}(x)   \leq 2
			\int_{\Omega^{\delta}_{2} \cap  \partial^{*} E \cap A^{C}_{1}} \|\nu_{d}(x)\|\d\hausd^{d-1}(x)   \leq \int_{\Omega_{2}} \per^{\mathrm{1D}}(E_{x^{\perp}_{d}}, (-\delta,\delta)) \dx^{\perp}_{d}.
		\end{equation}
		Thus, combining~\eqref{eq:nueomega2} with~\eqref{eq:38},~\eqref{eq:nueper} and ~\eqref{eq:35}, we have that~\eqref{eq:36} holds.
		
		Thus, by~\eqref{eq:stimaproppart1} and~\eqref{eq:36}, the statement of the lemma is proved provided we show that
		\begin{equation}
			\int_{\partial^{*}E \cap \Omega_{1}^{\delta}}\|\nu_{E}(x)\|\d\hausd^{d-1}(x)
			-\Big\|
			\int_{\partial^{*}E \cap \Omega_{1}^{\delta}}\nu_{d}(x)\d\hausd^{d-1}(x)
			\Big\|\lesssim \varepsilon^{p-d-1} \int_{\partial^{*}E \cap R_{L,\delta}} \etau(E, x)\d\hausd^{d-1}(x)\label{eq:final}.
		\end{equation}
		
		To this aim, let us define
		\begin{equation*}
			\begin{split}
				\Omega^{-}  := \insieme{x^{\perp}_{d} \in \Omega_{1}:\ \Scal{\nu_{E}(x)}{e_{d}} \leq 0 \text{ for } x = (x^{\perp}_{d}, x_{d}) \in \partial^{*} E\cap R_{L,\delta}}
			\end{split}
		\end{equation*}
		and   $\tilde \Omega :=P_{e_d^\perp}(A_{1} \cup A_{2})$. Given that $\tilde \Omega\supset  \Omega^{-}$ and
		using the triangular inequality we have that
		\begin{align}
			\int_{\partial^{*}E \cap \Omega_{1}^{\delta}}\|\nu_{E}(x)\|\d\hausd^{d-1}(x)
			&-\Big\|
			\int_{\partial^{*}E \cap \Omega_{1}^{\delta}}\nu_{d}(x)\d\hausd^{d-1}(x)
			\Big\|\notag\\& \leq
			\int_{\Omega_{1}\setminus \tilde\Omega} \sqrt{1 + \|\nu^{\perp}_{d}(x)\|^2/\|\nu_{d}(x)\|^{2}}  \dx_d^\perp- \Big| \int_{\Omega_{1} \setminus \tilde\Omega} \sgn(\Scal{\nu_{d}(x)}{e_d})\dx_d^\perp\Big|
			+ 2|\tilde\Omega|
			\notag\\ &\leq
			\int_{\Omega_{1}\setminus \tilde\Omega} \sqrt{1 + \|\nu^{\perp}_{d}(x)\|^2/\|\nu_{d}(x)\|^{2}}  \dx_d^\perp- |\Omega_{1} \setminus \tilde \Omega|
			+ 2|\tilde\Omega|
			\notag \\
			&\leq
			\int_{\Omega_{1}\setminus \tilde \Omega} \|\nu^{\perp}_{d}(x)\|^2/\|\nu_{d}(x)\|^{2}\dx_d^\perp
			+ 2|\tilde \Omega|,
			\label{eq:stimaproppart2}
		\end{align}
		where in the last inequality we used the fact that $\sqrt{1+z^2}-1\leq z^2$.
		
		Putting together estimates~\eqref{eq:stimaproppart1} and~\eqref{eq:stimaproppart2} we have
		\begin{align}
			&\int_{\partial^{*}E \cap R_{L,\delta}} \|\nu_{E}(x)\|\d\hausd^{d-1}(x) - \Big\|
			\int_{\partial^{*}E \cap R_{L,\delta}} \nu_{d} (x)\d\hausd^{d-1}(x)
			\Big\|\leq\notag\\
			&\leq 2 \int_{(\Omega_{1}\setminus \tilde\Omega)\times (-\delta, \delta)} \|\nu^{\perp}_{d}(x)\|^2/\|\nu_{d}(x)\|^{2}\d\hausd^{d-1}(x) + 2 |\tilde\Omega| + 2 \int_{\partial^{*}E \cap \Omega_{2}^{\delta}} \|\nu_{E}(x)\|\d\hausd^{d-1}(x)\notag\\
			&\leq 2 \int_{(\Omega_{1}\setminus \tilde\Omega)\times (-\delta, \delta)} \|\nu^{\perp}_{d}(x)\|^2/\|\nu_{d}(x)\|^{2}\d\hausd^{d-1}(x) + \varepsilon^{p-d-1} \int_{\partial^{*}E \cap R_{L,\delta}} \etau(E, x)\d\hausd^{d-1}(x).\label{eq:nuperpnud}
		\end{align}

		To estimate the first term in~\eqref{eq:nuperpnud}, define
		\begin{equation*}
			h(x^{\perp}_{d}) := \begin{cases}
				\nu^{\perp}_{d}(x^{\perp}_{d}, x_{d})/\|\nu_{d}(x^{\perp}_{d}, x_{d})\|, & \text{if } x^{\perp}_{d} \in \Omega_{1}\setminus \tilde \Omega \\
				0, & \text{otherwise}.
			\end{cases}
		\end{equation*}
		Notice that from the slicing formula~\eqref{eq:slicing} and from the fact that $\Omega_1\setminus\tilde\Omega\subset P_{e_d^\perp}(A_2^C)$, for every Borel set $A \subset  \Omega_{1}\setminus \tilde\Omega$ one has that
		\begin{equation*}
			\int_{A} h(x^{\perp}_{d})\dx_d^\perp = \int_{\partial^*E\cap (A \times  (-\varepsilon, \varepsilon))} \nu^{\perp}_{d}(x)\d\hausd^{d-1}(x).
		\end{equation*}
		Indeed, in $A_2^C$ it holds $|x_d|<\varepsilon$.

		For $\gamma>\varepsilon$, $z_{d}^\perp\in[0,L)^{d-1}$, let $Q_\gamma^\perp(z_d^\perp)$ be the $(d-1)$-dimensional cube of side length $\gamma$ and center $z_{d}^\perp$.
		Using Lemma~\ref{lemma:mean_control} and the fact that $p-d-1\geq2$, one has that for $\varepsilon<\gamma\ll1$
		\begin{equation}\label{eq:lemmamean}
			\frac{1}{\gamma^{p-d-1}} \int_{[0,L]^{d-1}} \fint_{Q^{\perp}_{\gamma}(z^{\perp}_{d})}
			\Big\|h(y^{\perp}_{d}) - \fint_{Q^{\perp}_{\gamma}(z^{\perp}_{d})} h(x^{\perp}_{d})\dx_d^\perp \Big\|^{2} \dy_d^\perp \gtrsim \int_{[0,L]^{d-1}} \|h(x^{\perp}_{d}) -\tilde h\|^{2} \dx_d^\perp
		\end{equation}
		with $\tilde h= \int_{[0,L]^{d-1}} h(x^{\perp}_{d}) \dx^{\perp}_{d}$.
		
		Notice that
		\begin{align}\label{eq:hest}
			\int_{(\Omega_{1}\setminus \tilde\Omega)\times (-\delta, \delta)} \|\nu^{\perp}_{d}(x)\|^2/\|\nu_{d}(x)\|^{2}\d\hausd^{d-1}(x) \lesssim \int_{[0,L]^{d-1}} \|h(x^{\perp}_{d}) -\tilde h\|^{2} \dx_d^\perp+L^{d-1}\|\tilde h\|^2.
		\end{align}
		
		Let us estimate the first term in   	\eqref{eq:hest}. Letting $N = \lfloor \log_{2}(\delta/\varepsilon) \rfloor -1$, $\gamma_{i} = 2^{i}\varepsilon$,  $i=1,\dots,N$, and recalling~\eqref{eq:lemmamean}, it holds
		\begin{align}
			N	\int_{[0,L]^{d-1}} &\|h(x^{\perp}_{d}) -\tilde h\|^{2} \dx_d^\perp\lesssim\notag\\
			&\lesssim \sum_{i=1}^N\frac{1}{\gamma_i^{p-d-1}}\int_{[0,L)^{d-1}} \fint_{Q^\perp_{\gamma_i}(z_d^\perp)} \Big\|h(x^{\perp}_{d}) -\fint_{Q^\perp_{\gamma_i}(z_d^\perp)} h   (y_d^\perp)\dy_d^\perp\Big\|^{2}\dx^{\perp}_{d}\dz_d^\perp.\label{eq:621}
		\end{align}
		
		Let $Q_{\gamma}(z_d^\perp)$ be the $d$-dimensional cube given by $Q_\gamma^\perp(z_d^\perp)\times (0,\gamma)$. Using the fact that $\gamma_i>\varepsilon>x_d$ for all $(x_d^\perp, x_d)\in \Omega_1\setminus \tilde \Omega$, the definition of $h$ and Jensen inequality, we have that
		
		\begin{align}
			\fint_{Q^\perp_{\gamma_i}(z_d^\perp)} &\Big\|h(x^{\perp}_{d}) -\fint_{Q^\perp_{\gamma_i}(z_d^\perp)} h   (y_d^\perp)\dy_d^\perp\Big\|^{2}\dx^{\perp}_{d}\lesssim\notag\\&\lesssim \fint_{\partial^{*}E \cap Q_{\gamma_i} (z_d^\perp)}  \Bigl\|\nu^{\perp}_{d}(x)- \fint_{\partial^{*}E \cap Q_{\gamma_i} (z_d^\perp)}  \nu^{\perp}_{d}(y)\d\hausd^{d-1}(y)\Bigr\|^{2}\d\hausd^{d-1}(x)\notag\\
			&\lesssim
			\fint_{\partial^{*}E \cap Q_{\gamma_i}(z_d^\perp)}  \fint_{\partial^{*}E \cap Q_{\gamma_i}(z_d^\perp)}\|\nu^{\perp}_{d}(x)-\nu^{\perp}_{d}(y)\|^{2}\d\hausd^{d-1}(x)\d\hausd^{d-1}(y)\notag\\
			&\lesssim
			\fint_{\partial^{*}E \cap Q_{\gamma_i}(z_d^\perp)}  \fint_{\partial^{*}E \cap Q_{\gamma_i}(z_d^\perp)}\|\nu_E(x)-\nu_E(y)\|^{2}\d\hausd^{d-1}(x)\d\hausd^{d-1}(y).\label{eq:622}
		\end{align}
		
		Thus, recalling~\eqref{eq:621}, we would like to estimate the following quantity
		\begin{equation}
			\int_{[0,L]^{d-1}} \sum_{i=1}^{N}\frac{1}{(\gamma_{i}^{d-1})^{2}}\iint_{Q_{\gamma_{i}}(z^{\perp}_{d})\cap\partial^*E}  \frac{\|\nu_{E}(x) -\nu_{E}(y)\|^{2}}{\gamma_{i}^{p-d-1}} \d\hausd^{d-1}(x)\d\hausd^{d-1}(y)\dz_d^\perp.
		\end{equation}
		To this aim, we will use the following facts:
		\begin{enumerate}
			\item \label{item:fact1} For any $x,y$ such that $\|x-y\|<\gamma_i$ and $\max\{|x_d|,|y_d|\}<\gamma_i$
			\begin{equation}
				\gamma_i^{d-1}\lesssim\bigl|\{z_d^\perp:\,x,y\in Q_\gamma(z_d^\perp)\}\bigr|\lesssim \gamma_i^{d-1}.
			\end{equation}
			\item \label{item:fact2} By (iii) in Remark~\ref{rmk:after_small_density_lemma}, and the fact that  the largest term in a geometric series bounds the sum,
			\begin{equation}
				\etau(x) + \etau(y) \gtrsim \sum_{i=1}^N\frac{\|\nu_{E}(x) - \nu_{E}(y)\|^{2}}{\gamma_i^{p-d-1}} \chi_{(\|x-y\|,+\infty)}(\gamma_i)\quad\text{for $\tau<\varepsilon$}.
			\end{equation}
			\item  \label{item:fact3} There exists $\tilde \tau>0$ such that for every $0<\tau \leq \tilde\tau$ and for almost every $x^{\perp}_{d} \in [0,L]^{d-1}$ we have that
			\begin{equation}\label{eq:hausd}
				\hausd^{d-1}(\partial^{*} E \cap  Q_{\gamma_{i}}(x^{\perp}_{d})) \leq 2 \gamma^{d-1}_{i}.
			\end{equation}
			Indeed, if~\eqref{eq:hausd} were false then there would exist $\tau_{n}\downarrow 0$ and $\sigma_{n}\downarrow 0$, $C_n\to+\infty$  and a set of $x_d^\perp$ of positive measure such that $	\hausd^{d-1}(\partial^{*} E_n \cap  Q_{\gamma_{i}}(x^{\perp}_{d})) >2 \gamma^{d-1}_{i}$,  $\chi_{E_{n}} \to \chi_{H_{L,\delta}}$  in $L^{1}$ and $\sup_n\int_{\partial^*E_n\cap R_{L,\delta}} \etau(E_{n}, x) \d\mathcal H^{d-1}(x)< C < +\infty$.
			By the rigidity result of Theorem~\ref{thm:rigidity} we have that $\per(E_{n}; R_{L,\delta })\to \per(H_{L,\delta}; R_{L,\delta})$. Moreover, for any $\gamma_i$ we have that $\per(E_{n}; Q_{\gamma_{i}}(z^{\perp}_{d})) \to \per(H_{L,\delta}; Q_{\gamma_{i}}(z^{\perp}_{d}))=\gamma_i^{d-1}$, thus getting a contradiction.

			As a consequence of~\eqref{eq:hausd},
			\begin{align}
				&	\iint_{\{x,y\in\partial^*E\cap R_{L,\delta}:\, \|x-y\| < \gamma_{i},\,\max\{\|x_d\|,\|y_d\|\}<\gamma_i\}} \etau(E, x) + \etau(E, y) \d\hausd^{d-1}(x)\d\hausd^{d-1}(y)=\notag\\
				&= 2 \iint_{\{x,y\in\partial^*E\cap R_{L,\delta}: \,\|x-y\| < \gamma_{i},\,\max\{\|x_d\|,\|y_d\|\}<\gamma_i\}} \etau(E, x)\d\hausd^{d-1}(x)\d\hausd^{d-1}(y)  \notag\\
				&\leq 4\gamma_{i}^{d-1} \int _{\partial^*E\cap R_{L,\delta}}\etau(E,x )\d\hausd^{d-1}(x).
			\end{align}
		\end{enumerate}

		Using Fubini Theorem and the above facts \ref{item:fact1}--\ref{item:fact2}, we have that
		\begin{align}
			&\int_{[0,L]^{d-1}}\sum_{i}^{N} \frac{1}{(\gamma_{i}^{d-1})^{2}}\iint_{Q_{\gamma_{i}}(z^{\perp}_{d})\cap\partial^*E}  \frac{\|\nu_{E}(x) -\nu_{E}(y)\|^{2}}{\gamma_{i}^{p-d-1}} \d\hausd^{d-1}(x)\d\hausd^{d-1}(y)\dz_d^\perp=\notag\\
			&=\sum_{i=1}^{N}
			\frac{1}{(\gamma_{i}^{d-1})^{2}}
			\iint_{\partial^*E}\int_{[0,L]^{d-1}}
			\chi_{Q_{\gamma_{i}}(z^{\perp}_{d})}(x)
			\chi_{Q_{\gamma_{i}}(z^{\perp}_{d})}(y)
			\frac{|\nu_{E}(x)- \nu_{E}(y)|^{2}}{\gamma_{i}^{p-d-1}}\dz_d^\perp\d\hausd^{d-1}(x)\d\hausd^{d-1}(y)\notag
			\\
			&\overset{\text{1.}}{\lesssim }
			\sum_{i=1}^{N}
			\frac{1}{\gamma_{i}^{d-1}}
			\iint_{\partial^*E\cap\{x,y\in R_{L,\delta}:\,\|x-y\| \leq \gamma_{i}, \,\max\{\|x_d\|, \|y_d\|\}<\gamma_i\}}
			\frac{\|\nu_{E}(x)- \nu_{E}(y)\|^{2}}{\gamma_{i}^{p-d-1}}\d\hausd^{d-1}(x)\d\hausd^{d-1}(y)\notag\\
			&\lesssim
			\sum_{j}\frac{1}{\gamma_{j}^{d-1}}\iint_{\partial^*E\cap\{x,y\in R_{L,\delta}:\, \gamma_{j}<\|x-y\| \leq 2\gamma_{j},\,\max\{\|x_d\|, \|y_d\|\}<2\gamma_j\}}\sum_{i>j} \frac{\|\nu_{E}(x)- \nu_{E}(y)\|^{2}}{\gamma_{i}^{p-d-1}}\d\hausd^{d-1}(x)\d\hausd^{d-1}(y)\notag
			\\&\overset{\text{2.}}{\lesssim }
			\sum_{j}\frac{1}{\gamma_{j}^{d-1}}\iint_{\partial^*E\cap\{x,y\in R_{L,\delta}:\, \gamma_{j}<\|x-y\| \leq 2\gamma_{j},\,\max\{\|x_d\|, \|y_d\|\}<2\gamma_j\}}\etau(E, x) + \etau(E, y)\d\hausd^{d-1}(x)\d\hausd^{d-1}(y).\label{eq:nuj}
		\end{align}
		Finally using fact~\ref{item:fact3}, we have that we can bound \eqref{eq:nuj} from above, up to a multiplicative constant,  by
		\begin{equation}
			\label{eq:eq623}
			\int_{\partial^*E\cap R_{L,\delta}}\etau(E, x)\d\hausd^{d-1}(x).
		\end{equation}

		Combining~\eqref{eq:621},~\eqref{eq:622} and~\eqref{eq:nuj}, we have that
		\begin{align}
			N	\int_{[0,L]^{d-1}} &\|h(x^{\perp}_{d}) -\tilde h\|^{2} \dx_d^\perp\lesssim
			\int_{\partial^*E\cap R_{L,\delta}}\etau(E, x)\d\hausd^{d-1}(x).
		\end{align}
		
		Recall now the second term in~\eqref{eq:hest} and the definition of $\tilde h$.
		
		Given that, by~\eqref{eq:nued}, $\int_{\partial^*E\cap R_{L,\delta}}\nu_{d}^\perp(x)\d\hausd^{d-1}(x) = 0$,  and that~\eqref{eq:ehausd} holds, we have that
		\begin{align}
			\Big\|\int_{[0,L]^{d-1}} h(z_d^\perp)\dz_d^\perp\Big\|^{2}& =
			\Big\|\int_{\partial^*E\cap R_{L,\delta}\setminus (A_1\cup A_2)} \nu^{\perp}_{d}(x)\d\hausd^{d-1}(x)\Big\|^{2} \notag\\
			&=\Big\|\int_{\partial^*E\cap R_{L,\delta}\cap (A_1\cup A_2)} \nu^{\perp}_{d}(x)\d\hausd^{d-1}(x)\Big\|^{2}\notag\\
			&  \leq \bigl(\hausd^{d-1}(A_{1}\cup A_{2}) \bigr)^2\notag\\
			&\lesssim \varepsilon^{p-d-1} \int_{\partial^*E\cap R_{L,\delta}}\etau(E, x)\d\hausd^{d-1}(x).
		\end{align}
		
		To conclude the proof it is sufficient to take $\varepsilon$ such that {$N \gtrsim M_1$ and $\varepsilon^{-(p-d-1)}\gtrsim M_1$.}
		
	\end{proof}
	
	\begin{remark}\label{rmk:stripesotherdirections}
		Notice that the statement and the proof of the above Proposition~\ref{prop:stability_bound} are invariant under rotations of an angle $\theta$ such that stripes with boundary orthogonal to $\theta$ are $[0,L)^d$-periodic. Thus, the assumption that the reference rectangle has height parallel to $e_d$ is not a restriction.
	\end{remark}

	We can now give a proof of Theorem~\ref{thm:main}.
	
	\begin{proof}[Proof of Theorem~\ref{thm:main}: ]
		Let $L>0$, $M_1\gg1$ and $\sigma,\delta,\varepsilon,\tau>0$ sufficiently small to be fixed later. By the rigidity estimate and $\Gamma$-convergence result (see Corollary~\ref{cor:gammaconv}), we know that if $0<\tau<\tau(\sigma)$, then minimizers $E_\tau$ of $\Fcal_{\tau,p,d}(\cdot, [0,L)^d)$ satisfy
		\begin{equation}
			\bigl\|\chi_{E_\tau}-\chi_{S_\theta}\bigr\|_{L^1([0,L)^d)}\leq\sigma,
		\end{equation}
		where $\theta\in\S^{d-1}$ and $S_\theta$ is a $[0,L)^d$ periodic set made of stripes with boundaries orthogonal to $\theta\in\S^{d-1}$ and of constant distance one from the other given by $h^*_L>0$.
		Let $E=E_{\tau}$ be a minimizer of $\Fcal_{\tau, L}$  and fix $\delta, \varepsilon >0$. Thanks also to Remark~\ref{rmk:stripesotherdirections}, we can assume \withoutLoss that $\theta=e_d$ and denote by $S=S_{e_d}$. In particular, up to a translation in direction $e_d$,
		\[
		\partial S\cap [0,L)^d=\bigcup_{i=0}^{L/h^*_L}[0,L)^{d-1}\times \{ih^*_L\},
		\]
		where $h^*_L$ was introduced in Section~\ref{sec:1D}.
		Let us denote by $R_{i}=[0,L)^{d-1}\times [ih^*_L-\delta/2,ih^*_L+\delta/2]$ and $\tilde R_i=[0,L)^{d-1}\times [ih^*_L-\delta,ih^*_L+\delta]$, for $i=0,\dots,L/h^*_L$.
		
		Then, define the sets (see Figure~\ref{fig:2})
		\begin{align*}
			\mathcal B_{1} &:= \insieme{x\in \partial^{*}E:\ x \not \in \bigcup_{i} R_{i}}\\
			\mathcal B_{2, i} &:= \insieme{(x, \theta)\in \partial^{*}E \times  \S^{d-1}:\ x \in R_{i} \text{ and }x_\theta^\perp+x_{\theta}^{+}\theta \in \tilde R_{i}}
		\end{align*}
		and  for any $\theta \in \S^{d-1}$ and $x^{\perp}_{\theta} \in \theta^{\perp}$ let
		\begin{align*}
			\mathcal B_{1, \theta, x^{\perp}_{\theta}}& := \{s\in\partial^*E_{x_\theta^\perp}:\, x^{\perp}_{\theta} + s \theta \in \mathcal{B}_{1}\} \\
			\mathcal B_{2, i, \theta} &:= \insieme{ x\in\partial^*E: (x, \theta) \in \mathcal{B}_{2,i}}\\
			\mathcal B_{2, i, \theta, x^{\perp}_{\theta}} &:= \insieme{s\in\partial^*E_{x_\theta^\perp}:\, x^{\perp}_{\theta} + s \theta \in \mathcal{B}_{2, i, \theta}}.
		\end{align*}
		
		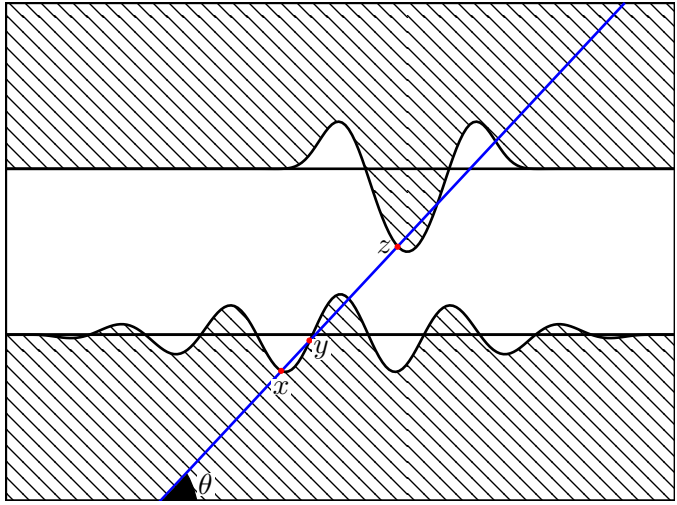
\begin{figure}
			\centering
			\begin{tikzpicture}[scale=2.2]
				\clip (0,1) rectangle (4,4);
				\draw[color=black!20!white,smooth,pattern=north west hatch,hatch distance=6pt, hatch thickness=.5pt] (-2,2) -- (0,2) -- (0.0201,2) -- (0.0402,2) -- (0.0603,2) -- (0.0804,2) -- (0.1005,2) -- (0.1206,2.00003) -- (0.1407,2.00007) -- (0.1608,2.00005) -- (0.1809,1.99976) -- (0.20101,1.99903) -- (0.22111,1.99771) -- (0.24121,1.99579) -- (0.26131,1.99333) -- (0.28141,1.99051) -- (0.30151,1.98755) -- (0.32161,1.9847) -- (0.34171,1.98225) -- (0.36181,1.98046) -- (0.38191,1.97959) -- (0.40201,1.97983) -- (0.42211,1.98136) -- (0.44221,1.98425) -- (0.46231,1.98852) -- (0.48241,1.99412) -- (0.50251,2.00091) -- (0.52261,2.00867) -- (0.54271,2.01714) -- (0.56281,2.02595) -- (0.58291,2.03473) -- (0.60302,2.04305) -- (0.62312,2.05047) -- (0.64322,2.05656) -- (0.66332,2.0609) -- (0.68342,2.06314) -- (0.70352,2.06298) -- (0.72362,2.06022) -- (0.74372,2.05474) -- (0.76382,2.04654) -- (0.78392,2.03574) -- (0.80402,2.0226) -- (0.82412,2.00747) -- (0.84422,1.99084) -- (0.86432,1.97327) -- (0.88442,1.95543) -- (0.90452,1.93803) -- (0.92462,1.92182) -- (0.94472,1.90752) -- (0.96482,1.89585) -- (0.98492,1.88745) -- (1.00503,1.88287) -- (1.02513,1.88252) -- (1.04523,1.88668) -- (1.06533,1.89545) -- (1.08543,1.90873) -- (1.10553,1.92626) -- (1.12563,1.94758) -- (1.14573,1.97204) -- (1.16583,1.99885) -- (1.18593,2.02706) -- (1.20603,2.05564) -- (1.22613,2.0835) -- (1.24623,2.1095) -- (1.26633,2.13255) -- (1.28643,2.15163) -- (1.30653,2.16582) -- (1.32663,2.17438) -- (1.34673,2.17673) -- (1.36683,2.17256) -- (1.38693,2.16178) -- (1.40704,2.14456) -- (1.42714,2.12134) -- (1.44724,2.09281) -- (1.46734,2.05989) -- (1.48744,2.02371) -- (1.50754,1.98555) -- (1.52764,1.94681) -- (1.54774,1.90896) -- (1.56784,1.87348) -- (1.58794,1.84176) -- (1.60804,1.81512) -- (1.62814,1.79468) -- (1.64824,1.78135) -- (1.66834,1.77581) -- (1.68844,1.7784) -- (1.70854,1.78919) -- (1.72864,1.8079) -- (1.74874,1.83397) -- (1.76884,1.86652) -- (1.78894,1.90441) -- (1.80905,1.94628) -- (1.82915,1.99061) -- (1.84925,2.03574) -- (1.86935,2.08) -- (1.88945,2.12169) -- (1.90955,2.15925) -- (1.92965,2.19123) -- (1.94975,2.2164) -- (1.96985,2.23379) -- (1.98995,2.24272) -- (2.01005,2.24272) -- (2.03015,2.23379) -- (2.05025,2.2164) -- (2.07035,2.19123) -- (2.09045,2.15925) -- (2.11055,2.12169) -- (2.13065,2.08) -- (2.15075,2.03574) -- (2.17085,1.99061) -- (2.19095,1.94628) -- (2.21106,1.90441) -- (2.23116,1.86652) -- (2.25126,1.83397) -- (2.27136,1.8079) -- (2.29146,1.78919) -- (2.31156,1.7784) -- (2.33166,1.77581) -- (2.35176,1.78135) -- (2.37186,1.79468) -- (2.39196,1.81512) -- (2.41206,1.84176) -- (2.43216,1.87348) -- (2.45226,1.90896) -- (2.47236,1.94681) -- (2.49246,1.98555) -- (2.51256,2.02371) -- (2.53266,2.05989) -- (2.55276,2.09281) -- (2.57286,2.12134) -- (2.59296,2.14456) -- (2.61307,2.16178) -- (2.63317,2.17256) -- (2.65327,2.17673) -- (2.67337,2.17438) -- (2.69347,2.16582) -- (2.71357,2.15163) -- (2.73367,2.13255) -- (2.75377,2.1095) -- (2.77387,2.0835) -- (2.79397,2.05564) -- (2.81407,2.02706) -- (2.83417,1.99885) -- (2.85427,1.97204) -- (2.87437,1.94758) -- (2.89447,1.92626) -- (2.91457,1.90873) -- (2.93467,1.89545) -- (2.95477,1.88668) -- (2.97487,1.88252) -- (2.99497,1.88287) -- (3.01508,1.88745) -- (3.03518,1.89585) -- (3.05528,1.90752) -- (3.07538,1.92182) -- (3.09548,1.93803) -- (3.11558,1.95543) -- (3.13568,1.97327) -- (3.15578,1.99084) -- (3.17588,2.00747) -- (3.19598,2.0226) -- (3.21608,2.03574) -- (3.23618,2.04654) -- (3.25628,2.05474) -- (3.27638,2.06022) -- (3.29648,2.06298) -- (3.31658,2.06314) -- (3.33668,2.0609) -- (3.35678,2.05656) -- (3.37688,2.05047) -- (3.39698,2.04305) -- (3.41709,2.03473) -- (3.43719,2.02595) -- (3.45729,2.01714) -- (3.47739,2.00867) -- (3.49749,2.00091) -- (3.51759,1.99412) -- (3.53769,1.98852) -- (3.55779,1.98425) -- (3.57789,1.98136) -- (3.59799,1.97983) -- (3.61809,1.97959) -- (3.63819,1.98046) -- (3.65829,1.98225) -- (3.67839,1.9847) -- (3.69849,1.98755) -- (3.71859,1.99051) -- (3.73869,1.99333) -- (3.75879,1.99579) -- (3.77889,1.99771) -- (3.79899,1.99903) -- (3.8191,1.99976) -- (3.8392,2.00005) -- (3.8593,2.00007) -- (3.8794,2.00003) -- (3.8995,2) -- (3.9196,2) -- (3.9397,2) -- (3.9598,2) -- (3.9799,2) -- (4,2) -- (8,2) -- (8,-2) -- (-2,-2) --cycle;
				\draw[smooth,line width=1pt] (-2,2) -- (0,2) -- (0.0201,2) -- (0.0402,2) -- (0.0603,2) -- (0.0804,2) -- (0.1005,2) -- (0.1206,2.00003) -- (0.1407,2.00007) -- (0.1608,2.00005) -- (0.1809,1.99976) -- (0.20101,1.99903) -- (0.22111,1.99771) -- (0.24121,1.99579) -- (0.26131,1.99333) -- (0.28141,1.99051) -- (0.30151,1.98755) -- (0.32161,1.9847) -- (0.34171,1.98225) -- (0.36181,1.98046) -- (0.38191,1.97959) -- (0.40201,1.97983) -- (0.42211,1.98136) -- (0.44221,1.98425) -- (0.46231,1.98852) -- (0.48241,1.99412) -- (0.50251,2.00091) -- (0.52261,2.00867) -- (0.54271,2.01714) -- (0.56281,2.02595) -- (0.58291,2.03473) -- (0.60302,2.04305) -- (0.62312,2.05047) -- (0.64322,2.05656) -- (0.66332,2.0609) -- (0.68342,2.06314) -- (0.70352,2.06298) -- (0.72362,2.06022) -- (0.74372,2.05474) -- (0.76382,2.04654) -- (0.78392,2.03574) -- (0.80402,2.0226) -- (0.82412,2.00747) -- (0.84422,1.99084) -- (0.86432,1.97327) -- (0.88442,1.95543) -- (0.90452,1.93803) -- (0.92462,1.92182) -- (0.94472,1.90752) -- (0.96482,1.89585) -- (0.98492,1.88745) -- (1.00503,1.88287) -- (1.02513,1.88252) -- (1.04523,1.88668) -- (1.06533,1.89545) -- (1.08543,1.90873) -- (1.10553,1.92626) -- (1.12563,1.94758) -- (1.14573,1.97204) -- (1.16583,1.99885) -- (1.18593,2.02706) -- (1.20603,2.05564) -- (1.22613,2.0835) -- (1.24623,2.1095) -- (1.26633,2.13255) -- (1.28643,2.15163) -- (1.30653,2.16582) -- (1.32663,2.17438) -- (1.34673,2.17673) -- (1.36683,2.17256) -- (1.38693,2.16178) -- (1.40704,2.14456) -- (1.42714,2.12134) -- (1.44724,2.09281) -- (1.46734,2.05989) -- (1.48744,2.02371) -- (1.50754,1.98555) -- (1.52764,1.94681) -- (1.54774,1.90896) -- (1.56784,1.87348) -- (1.58794,1.84176) -- (1.60804,1.81512) -- (1.62814,1.79468) -- (1.64824,1.78135) -- (1.66834,1.77581) -- (1.68844,1.7784) -- (1.70854,1.78919) -- (1.72864,1.8079) -- (1.74874,1.83397) -- (1.76884,1.86652) -- (1.78894,1.90441) -- (1.80905,1.94628) -- (1.82915,1.99061) -- (1.84925,2.03574) -- (1.86935,2.08) -- (1.88945,2.12169) -- (1.90955,2.15925) -- (1.92965,2.19123) -- (1.94975,2.2164) -- (1.96985,2.23379) -- (1.98995,2.24272) -- (2.01005,2.24272) -- (2.03015,2.23379) -- (2.05025,2.2164) -- (2.07035,2.19123) -- (2.09045,2.15925) -- (2.11055,2.12169) -- (2.13065,2.08) -- (2.15075,2.03574) -- (2.17085,1.99061) -- (2.19095,1.94628) -- (2.21106,1.90441) -- (2.23116,1.86652) -- (2.25126,1.83397) -- (2.27136,1.8079) -- (2.29146,1.78919) -- (2.31156,1.7784) -- (2.33166,1.77581) -- (2.35176,1.78135) -- (2.37186,1.79468) -- (2.39196,1.81512) -- (2.41206,1.84176) -- (2.43216,1.87348) -- (2.45226,1.90896) -- (2.47236,1.94681) -- (2.49246,1.98555) -- (2.51256,2.02371) -- (2.53266,2.05989) -- (2.55276,2.09281) -- (2.57286,2.12134) -- (2.59296,2.14456) -- (2.61307,2.16178) -- (2.63317,2.17256) -- (2.65327,2.17673) -- (2.67337,2.17438) -- (2.69347,2.16582) -- (2.71357,2.15163) -- (2.73367,2.13255) -- (2.75377,2.1095) -- (2.77387,2.0835) -- (2.79397,2.05564) -- (2.81407,2.02706) -- (2.83417,1.99885) -- (2.85427,1.97204) -- (2.87437,1.94758) -- (2.89447,1.92626) -- (2.91457,1.90873) -- (2.93467,1.89545) -- (2.95477,1.88668) -- (2.97487,1.88252) -- (2.99497,1.88287) -- (3.01508,1.88745) -- (3.03518,1.89585) -- (3.05528,1.90752) -- (3.07538,1.92182) -- (3.09548,1.93803) -- (3.11558,1.95543) -- (3.13568,1.97327) -- (3.15578,1.99084) -- (3.17588,2.00747) -- (3.19598,2.0226) -- (3.21608,2.03574) -- (3.23618,2.04654) -- (3.25628,2.05474) -- (3.27638,2.06022) -- (3.29648,2.06298) -- (3.31658,2.06314) -- (3.33668,2.0609) -- (3.35678,2.05656) -- (3.37688,2.05047) -- (3.39698,2.04305) -- (3.41709,2.03473) -- (3.43719,2.02595) -- (3.45729,2.01714) -- (3.47739,2.00867) -- (3.49749,2.00091) -- (3.51759,1.99412) -- (3.53769,1.98852) -- (3.55779,1.98425) -- (3.57789,1.98136) -- (3.59799,1.97983) -- (3.61809,1.97959) -- (3.63819,1.98046) -- (3.65829,1.98225) -- (3.67839,1.9847) -- (3.69849,1.98755) -- (3.71859,1.99051) -- (3.73869,1.99333) -- (3.75879,1.99579) -- (3.77889,1.99771) -- (3.79899,1.99903) -- (3.8191,1.99976) -- (3.8392,2.00005) -- (3.8593,2.00007) -- (3.8794,2.00003) -- (3.8995,2) -- (3.9196,2) -- (3.9397,2) -- (3.9598,2) -- (3.9799,2) -- (4,2) -- (8,2) -- (8,-2) -- (-2,-2) --cycle;
				\draw[color=black!20!white,smooth,pattern=north west hatch,hatch distance=6pt, hatch thickness=.5pt] (-2,3) -- (0,3) -- (0.0201,3) -- (0.0402,3) -- (0.0603,3) -- (0.0804,3) -- (0.1005,3) -- (0.1206,3) -- (0.1407,3) -- (0.1608,3) -- (0.1809,3) -- (0.20101,3) -- (0.22111,3) -- (0.24121,3) -- (0.26131,3) -- (0.28141,3) -- (0.30151,3) -- (0.32161,3) -- (0.34171,3) -- (0.36181,3) -- (0.38191,3) -- (0.40201,3) -- (0.42211,3) -- (0.44221,3) -- (0.46231,3) -- (0.48241,3) -- (0.50251,3) -- (0.52261,3) -- (0.54271,3) -- (0.56281,3) -- (0.58291,3) -- (0.60302,3) -- (0.62312,3) -- (0.64322,3) -- (0.66332,3) -- (0.68342,3) -- (0.70352,3) -- (0.72362,3) -- (0.74372,3) -- (0.76382,3) -- (0.78392,3) -- (0.80402,3) -- (0.82412,3) -- (0.84422,3) -- (0.86432,3) -- (0.88442,3) -- (0.90452,3) -- (0.92462,3) -- (0.94472,3) -- (0.96482,3) -- (0.98492,3) -- (1.00503,3) -- (1.02513,3) -- (1.04523,3) -- (1.06533,3) -- (1.08543,3) -- (1.10553,3) -- (1.12563,3) -- (1.14573,3) -- (1.16583,3) -- (1.18593,3) -- (1.20603,3) -- (1.22613,3) -- (1.24623,3) -- (1.26633,3) -- (1.28643,3) -- (1.30653,3) -- (1.32663,3) -- (1.34673,3) -- (1.36683,3) -- (1.38693,3) -- (1.40704,3) -- (1.42714,3) -- (1.44724,3) -- (1.46734,2.99999) -- (1.48744,2.99997) -- (1.50754,2.99995) -- (1.52764,2.9999) -- (1.54774,2.99981) -- (1.56784,2.9997) -- (1.58794,2.99957) -- (1.60804,2.99947) -- (1.62814,2.99952) -- (1.64824,2.99994) -- (1.66834,3.00104) -- (1.68844,3.0033) -- (1.70854,3.00733) -- (1.72864,3.0138) -- (1.74874,3.02342) -- (1.76884,3.03683) -- (1.78894,3.05441) -- (1.80905,3.07626) -- (1.82915,3.10204) -- (1.84925,3.13094) -- (1.86935,3.16171) -- (1.88945,3.19271) -- (1.90955,3.22202) -- (1.92965,3.24763) -- (1.94975,3.26762) -- (1.96985,3.28027) -- (1.98995,3.28423) -- (2.01005,3.27864) -- (2.03015,3.26313) -- (2.05025,3.23783) -- (2.07035,3.20335) -- (2.09045,3.1607) -- (2.11055,3.11118) -- (2.13065,3.05632) -- (2.15075,2.99776) -- (2.17085,2.93715) -- (2.19095,2.87611) -- (2.21106,2.81617) -- (2.23116,2.75871) -- (2.25126,2.70495) -- (2.27136,2.65595) -- (2.29146,2.61259) -- (2.31156,2.5756) -- (2.33166,2.54553) -- (2.35176,2.52283) -- (2.37186,2.5078) -- (2.39196,2.50064) -- (2.41206,2.50144) -- (2.43216,2.51018) -- (2.45226,2.52676) -- (2.47236,2.55097) -- (2.49246,2.58246) -- (2.51256,2.62078) -- (2.53266,2.66533) -- (2.55276,2.71535) -- (2.57286,2.76994) -- (2.59296,2.828) -- (2.61307,2.88828) -- (2.63317,2.94935) -- (2.65327,3.00969) -- (2.67337,3.06764) -- (2.69347,3.12156) -- (2.71357,3.16982) -- (2.73367,3.21094) -- (2.75377,3.24365) -- (2.77387,3.26703) -- (2.79397,3.28055) -- (2.81407,3.28418) -- (2.83417,3.27839) -- (2.85427,3.26416) -- (2.87437,3.2429) -- (2.89447,3.21639) -- (2.91457,3.18658) -- (2.93467,3.15548) -- (2.95477,3.12496) -- (2.97487,3.0966) -- (2.99497,3.07156) -- (3.01508,3.05055) -- (3.03518,3.03382) -- (3.05528,3.02122) -- (3.07538,3.01227) -- (3.09548,3.00635) -- (3.11558,3.00273) -- (3.13568,3.00074) -- (3.15578,2.99981) -- (3.17588,2.99949) -- (3.19598,2.99948) -- (3.21608,2.99959) -- (3.23618,2.99973) -- (3.25628,2.99983) -- (3.27638,2.99991) -- (3.29648,2.99995) -- (3.31658,2.99998) -- (3.33668,2.99999) -- (3.35678,3) -- (3.37688,3) -- (3.39698,3) -- (3.41709,3) -- (3.43719,3) -- (3.45729,3) -- (3.47739,3) -- (3.49749,3) -- (3.51759,3) -- (3.53769,3) -- (3.55779,3) -- (3.57789,3) -- (3.59799,3) -- (3.61809,3) -- (3.63819,3) -- (3.65829,3) -- (3.67839,3) -- (3.69849,3) -- (3.71859,3) -- (3.73869,3) -- (3.75879,3) -- (3.77889,3) -- (3.79899,3) -- (3.8191,3) -- (3.8392,3) -- (3.8593,3) -- (3.8794,3) -- (3.8995,3) -- (3.9196,3) -- (3.9397,3) -- (3.9598,3) -- (3.9799,3) -- (4,3) -- (8,2) -- (8,12) -- (-2,12) --cycle;
				
				\draw[smooth,line width=1pt] (-2,3) -- (0,3) -- (0.0201,3) -- (0.0402,3) -- (0.0603,3) -- (0.0804,3) -- (0.1005,3) -- (0.1206,3) -- (0.1407,3) -- (0.1608,3) -- (0.1809,3) -- (0.20101,3) -- (0.22111,3) -- (0.24121,3) -- (0.26131,3) -- (0.28141,3) -- (0.30151,3) -- (0.32161,3) -- (0.34171,3) -- (0.36181,3) -- (0.38191,3) -- (0.40201,3) -- (0.42211,3) -- (0.44221,3) -- (0.46231,3) -- (0.48241,3) -- (0.50251,3) -- (0.52261,3) -- (0.54271,3) -- (0.56281,3) -- (0.58291,3) -- (0.60302,3) -- (0.62312,3) -- (0.64322,3) -- (0.66332,3) -- (0.68342,3) -- (0.70352,3) -- (0.72362,3) -- (0.74372,3) -- (0.76382,3) -- (0.78392,3) -- (0.80402,3) -- (0.82412,3) -- (0.84422,3) -- (0.86432,3) -- (0.88442,3) -- (0.90452,3) -- (0.92462,3) -- (0.94472,3) -- (0.96482,3) -- (0.98492,3) -- (1.00503,3) -- (1.02513,3) -- (1.04523,3) -- (1.06533,3) -- (1.08543,3) -- (1.10553,3) -- (1.12563,3) -- (1.14573,3) -- (1.16583,3) -- (1.18593,3) -- (1.20603,3) -- (1.22613,3) -- (1.24623,3) -- (1.26633,3) -- (1.28643,3) -- (1.30653,3) -- (1.32663,3) -- (1.34673,3) -- (1.36683,3) -- (1.38693,3) -- (1.40704,3) -- (1.42714,3) -- (1.44724,3) -- (1.46734,2.99999) -- (1.48744,2.99997) -- (1.50754,2.99995) -- (1.52764,2.9999) -- (1.54774,2.99981) -- (1.56784,2.9997) -- (1.58794,2.99957) -- (1.60804,2.99947) -- (1.62814,2.99952) -- (1.64824,2.99994) -- (1.66834,3.00104) -- (1.68844,3.0033) -- (1.70854,3.00733) -- (1.72864,3.0138) -- (1.74874,3.02342) -- (1.76884,3.03683) -- (1.78894,3.05441) -- (1.80905,3.07626) -- (1.82915,3.10204) -- (1.84925,3.13094) -- (1.86935,3.16171) -- (1.88945,3.19271) -- (1.90955,3.22202) -- (1.92965,3.24763) -- (1.94975,3.26762) -- (1.96985,3.28027) -- (1.98995,3.28423) -- (2.01005,3.27864) -- (2.03015,3.26313) -- (2.05025,3.23783) -- (2.07035,3.20335) -- (2.09045,3.1607) -- (2.11055,3.11118) -- (2.13065,3.05632) -- (2.15075,2.99776) -- (2.17085,2.93715) -- (2.19095,2.87611) -- (2.21106,2.81617) -- (2.23116,2.75871) -- (2.25126,2.70495) -- (2.27136,2.65595) -- (2.29146,2.61259) -- (2.31156,2.5756) -- (2.33166,2.54553) -- (2.35176,2.52283) -- (2.37186,2.5078) -- (2.39196,2.50064) -- (2.41206,2.50144) -- (2.43216,2.51018) -- (2.45226,2.52676) -- (2.47236,2.55097) -- (2.49246,2.58246) -- (2.51256,2.62078) -- (2.53266,2.66533) -- (2.55276,2.71535) -- (2.57286,2.76994) -- (2.59296,2.828) -- (2.61307,2.88828) -- (2.63317,2.94935) -- (2.65327,3.00969) -- (2.67337,3.06764) -- (2.69347,3.12156) -- (2.71357,3.16982) -- (2.73367,3.21094) -- (2.75377,3.24365) -- (2.77387,3.26703) -- (2.79397,3.28055) -- (2.81407,3.28418) -- (2.83417,3.27839) -- (2.85427,3.26416) -- (2.87437,3.2429) -- (2.89447,3.21639) -- (2.91457,3.18658) -- (2.93467,3.15548) -- (2.95477,3.12496) -- (2.97487,3.0966) -- (2.99497,3.07156) -- (3.01508,3.05055) -- (3.03518,3.03382) -- (3.05528,3.02122) -- (3.07538,3.01227) -- (3.09548,3.00635) -- (3.11558,3.00273) -- (3.13568,3.00074) -- (3.15578,2.99981) -- (3.17588,2.99949) -- (3.19598,2.99948) -- (3.21608,2.99959) -- (3.23618,2.99973) -- (3.25628,2.99983) -- (3.27638,2.99991) -- (3.29648,2.99995) -- (3.31658,2.99998) -- (3.33668,2.99999) -- (3.35678,3) -- (3.37688,3) -- (3.39698,3) -- (3.41709,3) -- (3.43719,3) -- (3.45729,3) -- (3.47739,3) -- (3.49749,3) -- (3.51759,3) -- (3.53769,3) -- (3.55779,3) -- (3.57789,3) -- (3.59799,3) -- (3.61809,3) -- (3.63819,3) -- (3.65829,3) -- (3.67839,3) -- (3.69849,3) -- (3.71859,3) -- (3.73869,3) -- (3.75879,3) -- (3.77889,3) -- (3.79899,3) -- (3.8191,3) -- (3.8392,3) -- (3.8593,3) -- (3.8794,3) -- (3.8995,3) -- (3.9196,3) -- (3.9397,3) -- (3.9598,3) -- (3.9799,3) -- (4,3) -- (8,2) -- (8,12) -- (-2,12) --cycle;
				\draw[color=black,line width=1pt] (0,0) rectangle (4,4);
				\draw[line width=1pt] (0,1) -- (4,1);
				\draw[line width=1pt] (0,2) -- (4,2);
				\draw[line width=1pt] (0,2) -- (4,2);
				\draw[line width=1pt] (0,3) -- (4,3);
				\draw[line width=1pt,color=blue] (0,0) -- (3.7,4);
				\draw ;
				\draw (1.646,1.682) node[fill=white,inner sep=.5pt] {$x$};
				\draw[fill,color=red] (1.646,1.782) circle[radius=0.015];
				\draw[fill,color=red] (1.815,1.965) circle[radius=0.015];
				\draw (1.885,1.91) node[fill=white,inner sep=.5pt] {$y$};
				\draw (2.262,2.53) node {$z$};
				\draw[fill,color=red] (2.342,2.53) circle[radius=0.015];
				\path[fill=black,opacity=0.4] (0.934,1) -- (1.082,1.16) arc[start angle=30,end angle=0,radius=5mm] --cycle;
				\path[line width=.5pt,draw] (1.082,1.16) arc[start angle=30,end angle=0,radius=5mm];
				\draw (1.194,1.1) node[fill=white,inner sep=.5pt] {$\theta$};
			\end{tikzpicture}
			\caption{In the above example we have that $(x, \theta) \in \mathcal{B}_{2,i}$, $(y, -\theta) \in \mathcal{B}_{2,i}$ and $z \in \mathcal{B}_{1}$.}
			\label{fig:2}
		\end{figure}
		
		Using the integral geometric  formula~\eqref{eq:forintgeom} we have
		\begin{align}
			L^{d}\Fcal_{\tau, p,d}(E,[0,L)^d) &=
			\int_{\S^{d-1}} \int_{\theta^{\perp}} \sum_{s \in \partial^*E_{x^{\perp}_{\theta}} \cap  ([0,L)^d)_{x_\theta^\perp}} r_{\tau}(E_{x^{\perp}_{\theta}}, s)\dx_\theta^\perp\d\theta
			\notag\\
			&=
			\int_{\S^{d-1}} \int_{\theta^{\perp}} \sum_{s \in \partial^*E_{x^{\perp}_{\theta} }\setminus \big(\mathcal B_{1, \theta,x^{\perp}_{\theta}} \cup \bigcup_{i} \mathcal B_{2, i, \theta, x^{\perp}_{\theta}}\big) \cap ([0,L)^d)_{x_\theta^\perp} } r_{\tau}(E_{x^{\perp}_{\theta}}, s)\dx_\theta^\perp\d\theta\notag\\
			&+ \int_{\S^{d-1}} \int_{\theta^{\perp}} \sum_{s\in\mathcal B_{1,\theta,x_\theta^\perp} \cup \bigcup_{i} \mathcal B_{2, i, \theta, x^{\perp}_{\theta}} \cap ([0,L)^d)_{x_\theta^\perp}} r_{\tau }(E_{x^{\perp}_{\theta}}, s)\dx_\theta^\perp\d\theta.\label{eq:fdec}
		\end{align}
		
		We claim that, provided, $\delta, \sigma$ and $\tau$ are sufficiently small,
		\begin{align}
			\int_{\S^{d-1}} \int_{\theta^{\perp}} &\sum_{s\in\mathcal B_{1,\theta,x_\theta^\perp} \cup \bigcup_{i} \mathcal B_{2, i, \theta, x^{\perp}_{\theta}} \cap ([0,L)^d)_{x_\theta^\perp}} r_{\tau }(E_{x^{\perp}_{\theta}}, s)\dx_\theta^\perp\d\theta\geq\notag\\
			& \geq M \int_{\S^{d-1}} \int_{\theta^{\perp}} \#\Big\{s \in \big(\mathcal B_{1,\theta, x^{\perp}_{\theta}} \cup \bigcup_{i} \mathcal B_{2, i, \theta, x^{\perp}_{\theta}}\big) \cap  ([0,L)^d)_{x_\theta^\perp}\Big\}\dx_\theta^\perp\d\theta.\label{eq:b1b2m}
		\end{align}
		
		Indeed, on the one hand, if $\delta$ is such that $\delta^{-(p-d-1)}\gtrsim M_1$, $\sigma<\frac14(\delta/2)^d$ and $\tau<\tau(\sigma)$ as in Corollary~\ref{cor:gammaconv}, one has that whenever $x\in\mathcal B_1$,  $\sigma<\frac14(\mathrm{dist}(x,\partial S))^d$ thus one can apply  point (i) of Remark~\ref{rmk:after_small_density_lemma} to get
		
		\begin{align}
			\int_{\S^{d-1}} \int_{\theta^{\perp}} &\sum_{s\in\mathcal B_{1,\theta,x_\theta^\perp} \cap ([0,L)^d)_{x_\theta^\perp}} r_{\tau }(E_{x^{\perp}_{\theta}}, s)\dx_\theta^\perp\d\theta \geq M_1 \int_{\S^{d-1}} \int_{\theta^{\perp}} \#\Big\{s \in \mathcal B_{1,\theta, x^{\perp}_{\theta}}  \cap  ([0,L)^d)_{x_\theta^\perp}\Big\}\dx_\theta^\perp\d\theta.\label{eq:b1m}
		\end{align}

		Notice that, by Proposition~\ref{prop:stability_bound}, whenever $\sigma<sigma_0$ and $\tau<\tau_2$ are sufficiently small,
		\begin{align}
			\int_{\partial^{*} E \cap R_i}   e_{\tau,\delta/2}(E, x) \d\hausd^{d-1} (x)&\geq M_1 \bigg( \per(E; R_i) - \Big| \int_{\partial^{*}E \cap R_i} \nu_{E} (x)\d\hausd^{d-1}(x) \Big|  \bigg). \label{eq:prop67}
		\end{align}
		Since
		\[
		\insieme{x\in\partial^*E\cap R_i:\,\etautheta(E_{x_\theta^\perp},x_\theta)>0}\subset \mathcal B_{2,i,\theta},
		\]
		the bound~\eqref{eq:prop67} gives
		\begin{align}
			\int_{\S^{d-1}}\int_{\mathcal B_{2,i,\theta}}  &|\Scal{\nu_{E}(x)}{\theta}|\etautheta(E_{x^{\perp}_{\theta}}, x_{\theta})\d\mathcal H^{d-1}(x)\d\theta\geq\notag\\
			&\geq\int_{\S^{d-1}}\int_{\partial^*E\cap R_i}|\Scal{\nu_{E}(x)}{\theta}|\etautheta(E_{x^{\perp}_{\theta}}, x_{\theta})\d\hausd^{d-1}(x)\d\theta\notag\\
			&=	\int_{\partial^{*} E \cap R_i}   e_{\tau,\delta/2}(E, x) \d\hausd^{d-1} (x)\notag\\
			&\geq M_1  \bigg( \per(E; R_i) - \Big\| \int_{\partial^{*}E \cap R_i} \nu_{E} (x)\d\hausd^{d-1}(x) \Big\| \notag\\
			&\geq
			M_1
			\int_{\S^{d-1}}\int_{\theta^{\perp}}\Big(\per^{\mathrm{1D}}(E_{x^{\perp}_{\theta}}, (R_i)_{x_\theta^\perp})) - \mathrm{mod}_{2}\big(\per^{\mathrm{1D}}(E_{x^{\perp}_{\theta}}, (R_i)_{x_\theta^\perp})\big) \
			\Big)\dx_\theta^\perp\d\theta\notag\\
			&\geq 2M_1
			\int_{\S^{d-1}}\int_{\theta^{\perp}}
			\#\mathcal B_{2,i,\theta,x^{\perp}_{\theta}} \dx_\theta^\perp\d\theta.
		\end{align}
		Hence, since by Proposition~\ref{prop:1dbound}
		\[
		r_\tau(E_{x_\theta^\perp},s)\geq -\gamma_0+\gamma_1 e_{\tau,\delta,\theta}(E_{x_\theta^\perp},s),
		\]
		and $M_1\gg1$,
		\begin{align}
			\int_{\S^{d-1}} \int_{\theta^{\perp}} &\sum_{s\in\bigcup_i\mathcal B_{2,i,\theta,x_\theta^\perp} \cap ([0,L)^d)_{x_\theta^\perp}} r_{\tau }(E_{x^{\perp}_{\theta}}, s)\dx_\theta^\perp\d\theta \geq M_1 \int_{\S^{d-1}} \int_{\theta^{\perp}} \#\Big\{s \in \bigcup_i\mathcal B_{2,i,\theta, x^{\perp}_{\theta}}  \cap  ([0,L)^d)_{x_\theta^\perp}\Big\}\dx_\theta^\perp\d\theta,\label{eq:b2m}
		\end{align}
		and thus~\eqref{eq:b1b2m} is proved.
		
		In particular,~\eqref{eq:fdec} becomes
		
		\begin{align}
			\Fcal_{\tau,p,d}(E,[0,L)^d)&\geq \frac{1}{L^d}\Biggl[
			\int_{\S^{d-1}} \int_{\theta^{\perp}} \sum_{s \in \partial^*E_{x^{\perp}_{\theta} }\setminus \big(\mathcal B_{1,\theta, x^{\perp}_{\theta}} \cup \bigcup_{i} \mathcal B_{2, i, \theta, x^{\perp}_{\theta}}\big) \cap ([0,L)^d)_{x_\theta^\perp}} r_{\tau}(E_{x^{\perp}_{\theta}}, s)\dx_\theta^\perp\d\theta\notag\\
			&+ M_1 \int_{\S^{d-1}} \int_{\theta^{\perp}} \#\Big\{s \in \big(\mathcal B_{1,\theta, x^{\perp}_{\theta}} \cup \bigcup_{i} \mathcal B_{2, i, \theta, x^{\perp}_{\theta}}\big) \cap  ([0,L)^d)_{x_\theta^\perp}\Big\}\dx_\theta^\perp\d\theta\Biggr].\label{eq:fineqb1b2}
		\end{align}
		
		Now notice that, due to the  $[0,L)^d$-periodicity of $E$,  both the \rhs and \lhs of~\eqref{eq:fineqb1b2} are invariant if we substitute $[0,L)^d$ with $[-kL,kL)^d$,    for any $k\in\N$.
		
		Let us fix $\ell\gg\eta_0$, $0<\varepsilon\ll1$ and define the set
		\begin{align}
			\mathcal G_{k,\ell,\varepsilon}:=\{(\theta,x_\theta^\perp)\in \S^{d-1}\times ([-kL,kL)^d)_\theta^\perp:\,|([-kL,kL)^d)_{x_\theta^\perp}|\geq\ell, \,|\Scal{\nu_S}{\theta}|\geq \varepsilon\}.
		\end{align}
		It is immediate to see that
		\begin{equation}
			\int_{\S^{d-1}}\int_{([-kL,kL)^d)_\theta^\perp}\bigl[1-\chi_{\mathcal G_{k,\ell,\varepsilon}}(\theta,x_\theta^\perp)\bigr]\dx_\theta^\perp\d\theta\lesssim \ell^{d-1}+(kL)^{d-1}\varepsilon.
		\end{equation}
		
		Let now $(\theta,x_\theta^\perp)\in\mathcal G_{k,\ell}$.
		By using Lemma~\ref{lemma:1d_replacement} with $E=E_{x_\theta^\perp}$, $A_0\cup A_1=\mathcal{B}_{1,\theta,x_\theta^\perp} \cup \bigcup_i \mathcal{B}_{2, i,\theta, x^{\perp}_{\theta}}$, and $F_{x_\theta^\perp}$ such that $\partial^*F_{x_\theta^\perp}\cap ([0,L)^d)_{x_\theta^\perp}=\partial^*E_{x_\theta^\perp}\setminus(A_0\cup A_1)$,  we have
		\begin{align*}
			\sum_{s \in \partial E_{x^{\perp}_{\theta}} \cap ([-kL,kL)^d)_{x_\theta^\perp} \setminus \Bigl(\mathcal{B}_{1,\theta,x_\theta^\perp} \cup \bigcup_i \mathcal{B}_{2, i,{\theta},x^{\perp}_{\theta}}\Bigr)}  r_{\tau}(E_{x^{\perp}_{\theta}}, s)&\geq -
			M_0\#\Bigl\{\Bigl(\mathcal{B}_{1,\theta,x_\theta^\perp} \cup \bigcup_i \mathcal{B}_{2, i,\theta, x^{\perp}_{\theta}}\Bigr) \cap ([-kL,kL)^d)_{x_\theta^\perp} \Bigr\}\\
			& + \sum_{s \in \partial F_{x^{\perp}_{\theta}} \cap ([-kL,kL)^d)_{x_\theta^\perp} }  r_{\tau}(F_{x^{\perp}_{\theta}}, s).
		\end{align*}
		
		We now want to apply Lemma~\ref{lemma:1D-optimization} to $F=F_{x_\theta^\perp}$, $I=((0,L)^d)_{x_\theta^\perp}$. Given $\{k_1,\dots,k_m\}=\partial^*F_{x_\theta^\perp}\cap I$, notice that $\forall\,i=1,\dots,m$ there exists $s_i\in\partial S_{x_\theta^\perp}$ such that $|s_i-k_i|<\frac{\delta}{2\cos\gamma}\ll h^*_L/\cos\gamma$, with $\gamma=\Scal{\theta}{e_d}$ and thus $|k_i-k_j|>\eta_0$. We then choose $k_0,k_{m+1},k_{m+2}$ in the following way. Let
		\begin{align*}
			k_0&=\sup\big\{s\in\partial S_{x_\theta^\perp}\setminus I:\,s<k_1-\delta/\cos\gamma\big\}\notag\\
			\tilde k_{m+1}&=\inf\big\{s\in\partial S_{x_\theta^\perp}:\,s>k_m+\delta/\cos\gamma\big\}\\
			\tilde k_{m+2}&=\inf\big\{s\in\partial S_{x_\theta^\perp}:\,s>\tilde k_{m+1}\big\}\\
			\tilde k_{m+3}&=\inf\big\{s\in\partial S_{x_\theta^\perp}:\,s>\tilde k_{m+2}\big\}.
		\end{align*}
		We now distinguish four cases:
		\begin{align}
			&\#\partial ^*F_{x_\theta^\perp}\cap I\in2\N \quad\text{and}\quad\#\partial S_{x_\theta^\perp}\cap[k_0,\tilde k_{m+1}]\in 2\N\label{eq:Caso1}\\
			&\#\partial ^*F_{x_\theta^\perp}\cap I\in2\N \quad\text{and}\quad\#\partial S_{x_\theta^\perp}\cap[k_0,\tilde k_{m+1}]\in 2\N+1\label{eq:Caso2}\\
			&\#\partial ^*F_{x_\theta^\perp}\cap I\in2\N +1\quad\text{and}\quad\#\partial S_{x_\theta^\perp}\cap[k_0,\tilde k_{m+1}]\in 2\N+1\label{eq:Caso3}\\
			&\#\partial ^*F_{x_\theta^\perp}\cap I\in2\N+1 \quad\text{and}\quad\#\partial S_{x_\theta^\perp}\cap[k_0,\tilde k_{m+1}]\in 2\N\label{eq:Caso4}.
		\end{align}
		If~\eqref{eq:Caso1} holds, then choose
		\[
		k_{m+1}=\tilde k_{m+1},\quad k_{m+2}=\tilde k_{m+2};
		\]
		if~\eqref{eq:Caso2} holds, set
		\[
		k_{m+1}=\tilde k_{m+1},\quad k_{m+2}=\tilde k_{m+3};
		\]
		if~\eqref{eq:Caso3} holds, let
		\[
		k_{m+1}=\tilde k_{m+1};
		\]
		and finally if~\eqref{eq:Caso4} holds, then choose
		\[
		k_{m+1}=\tilde k_{m+2}.
		\]
		
		In this way, on $\tilde I=[k_0,k_{\mathrm{max}}]$ as in Lemma~\ref{lemma:1D-optimization}, both the extension $\tilde F$ of $F=F_{x_\theta^\perp}$ and the simple periodic set  $S_{x_\theta^\perp}$ have an odd number of boundary points, thus they are both $|\tilde I|$-periodic.
		
		Define also $\tilde S_{x_\theta^\perp}$ as the simple periodic $|\tilde I_{x_\theta^\perp}|$-periodic set on $x_\theta^\perp+\R\theta$ with the same number of boundary points of $\tilde F$ on $\tilde I$, as in Lemma~\ref{lemma:1D-optimization}.
		
		By  Lemma~\ref{lemma:1D-optimization},
		\begin{align*}
			\sum_{s \in \partial F_{x^{\perp}_{\theta}} \cap ([-kL,kL)^d)_{x_\theta^\perp} }  r_{\tau}(F_{x^{\perp}_{\theta}}, s)&\geq \sum_{s\in\partial^*\tilde F\cap \tilde I_{x_\theta^\perp}}r_{\tau}(\tilde F, s)-C\notag\\
			&\geq  \sum_{s\in\partial^*\tilde S_{x_\theta^\perp}\cap \tilde I_{x_\theta^\perp}}r_{\tau}(\tilde S_{x^{\perp}_{\theta}}, s)-C\notag\\
			&=|\tilde I_{x_\theta^\perp}|\overline F^{\mathrm{1D}}_{\tau,p,d}(\tilde S_{x_\theta^\perp}, \tilde I_{x_\theta^\perp})-C.
		\end{align*}
		
		Moreover, since $\per(\tilde S_{x_\theta^\perp},\tilde I_{x_\theta^\perp})\leq \per( S_{x_\theta^\perp},\tilde I_{x_\theta^\perp})$, the distance $\tilde h_{x_\theta^\perp}$ between boundary points in $\tilde S_{x_\theta^\perp}$ is greater than or equal to the distance $h_{x_\theta^\perp}$  between boundary points in $S_{x_\theta^\perp}$, which is in turn greater than or equal to $h^*_L/\cos\gamma\geq h^*_L$. Then notice that for $h>h^*_L$ the energy of simple periodic sets with boundaries at distance $h$ is a strictly increasing function of $h$. This can be seen by the classical  explicit computation of such an energy for $\tau=0$  (see \eg~\cite{gr,dr_arma}), which is equal to $h^{-1}+ch^{-(p-d)}$, and the fact that, as seen in Proposition~\ref{prop:almost_reflection_positivity},  for boundary points at distance greater than some given constant $\eta_0$ and $\tau$ sufficiently small,  the energies $\Fcal_{\tau,p,d}$ and $\Fcal_{0,p,d}$ coincide on simple periodic sets.
		
		Thus,
		\begin{align}
			|\tilde I_{x_\theta^\perp}|\overline F^{\mathrm{1D}}_{\tau,p,d}(\tilde S_{x_\theta^\perp}, \tilde I_{x_\theta^\perp})\geq 	|\tilde I_{x_\theta^\perp}|\overline F^{\mathrm{1D}}_{\tau,p,d}( S_{x_\theta^\perp}, \tilde I_{x_\theta^\perp})
		\end{align}
		where equality holds if and only if $\per^{\mathrm{1D}}(\tilde S_{x_\theta^\perp}, \tilde I_{x_\theta^\perp})= \per^{\mathrm{1D}}(S_{x_\theta^\perp}, \tilde I_{x_\theta^\perp})$ and thus the two stripes coincide.

		From the above (denoting by $\mathcal B_1$ and $\mathcal B_{2,i}$ the corresponding sets on $[-kL,kL)^{d-1}$ instead of $[0,L)^d$), and recalling that from Proposition~\ref{prop:1dbound} $r_\tau(E,s)\geq-\gamma_0$,  one has that
		\begin{align}
			\Fcal_{\tau,p,d}(E,[-kL,kL)^d)&\geq \frac{1}{(kL)^d}\Biggl[
			\int_{\S^{d-1}} \int_{\theta^{\perp}} \sum_{s \in \partial^*E_{x^{\perp}_{\theta} }\setminus \big(\mathcal B_{1,\theta, x^{\perp}_{\theta}} \cup \bigcup_{i} \mathcal B_{2, i, \theta, x^{\perp}_{\theta}}\big) \cap ([-kL,kL)^d)_{x_\theta^\perp}} r_{\tau}(E_{x^{\perp}_{\theta}}, s)\dx_\theta^\perp\d\theta\notag\\
			&+ M_1 \int_{\S^{d-1}} \int_{\theta^{\perp}} \#\Big\{s \in \big(\mathcal B_{1,\theta, x^{\perp}_{\theta}} \cup \bigcup_{i} \mathcal B_{2, i, \theta, x^{\perp}_{\theta}}\big) \cap  ([-kL,kL)^d)_{x_\theta^\perp}\Big\}\dx_\theta^\perp\d\theta\Biggr]\notag\\
			&\geq -\frac{c\gamma_0\ell^{d-1}+c\gamma_0\varepsilon(kL)^{d-1}}{(kL)^d}\notag\\
			&+\int_{\mathcal G_{k,\ell,\varepsilon}}\frac{(M_1-M_0)}{(kL)^d}\#\Bigl\{\Bigl(\mathcal{B}_{1,\theta,x_\theta^\perp} \cup \bigcup_i \mathcal{B}_{2, i,\theta, x^{\perp}_{\theta}}\Bigr) \cap ([-kL,kL)^d)_{x_\theta^\perp} \Bigr\}\dx_\theta^\perp\d\theta\notag\\
			&+\frac{1}{(kL)^d}\int_{\mathcal G_{k,\ell,\varepsilon}}| \tilde I_{x_\theta^\perp}|\overline F^{\mathrm{1D}}_{\tau,p,d}(S_{x_\theta^\perp}, \tilde I_{x_\theta^\perp})\dx_\theta^\perp\d\theta\notag\\
			&-\frac{C}{kL}.
		\end{align}
		
		Sending $k\to+\infty$ and then $\varepsilon\to0$, since $	\Fcal_{\tau,p,d}(E,[-kL,kL)^d)=	\Fcal_{\tau,p,d}(E,[0,L)^d)$ and $|\mathcal G_{k,\ell, \varepsilon}|/(kL)^d\to |\{\theta:\,|\Scal{\nu_S}{\theta}|\geq\varepsilon\}$ as $k\to+\infty$, one obtains
		\begin{align}
			\Fcal_{\tau,p,d}(E,[0,L)^d)\geq \Fcal_{\tau,p,d}(S,[0,L)^d),
		\end{align}
		with equality if and only if $\mathcal{B}_{1,\theta,x_\theta^\perp} \cup \bigcup_i \mathcal{B}_{2, i,\theta, x^{\perp}_{\theta}}=\emptyset$ and $E_{x_{\theta}^\perp}=S_{x_\theta^\perp}$ for all $\theta,x_\theta^\perp$, namely $E=S$.

	\end{proof}

	\appendix
	\section*{Appendix}
  \label{sec:appendix}
	
  \renewcommand{\theequation}{A.\arabic{equation}}
	Given Lemmas~\ref{lemma:upper_bound_perimeter},~\ref{lemma:per_low_bound} and~\ref{lemma:excf}, the proof of  the Lipschitz regularity of the boundary as in Theorem~\ref{thm:regularity} follows from uniform upper and lower bounds on perimeter and volume and power law decay of the excess, similarly to the classical De Giorgi's regularity proof for quasi-minimizers.
	
	We give here a self-contained proof of this fact, without exploiting directly any quasi-minimality property.

	In the following,  the  upper and lower bounds for the family of sets we consider are assumed to hold locally with uniform constants in the whole $\R^d$. Thus, also the regularity results we prove will hold on the whole $\R^d$. However, since the proofs depend only on the local behaviour of the sets, this extends trivially to sets $E$ whose upper and lower bounds hold on the balls contained inside any given open and bounded set $\Omega$ (as is the case for the sets of $\overline \Fcal_{0,p,d}(\cdot,\Omega)$-equibounded energy).   The choice to work on the whole $\R^d$ is done for simplicity of notation and in order to facilitate the reader.
	
	\begin{proposition} \label{lemma:regularity} Let $d\geq 2$, $\mathscr{F}$ a class of  subsets of $\R^d$ of locally finite perimeter for which there exist constants  $C_0,C_1,C_2,\bar C_1>0$ and $\alpha,R_3>0$  such that for every $E\in\mathscr{F}$, for every $x\in\partial ^*E$ and for every $0<r<R_3$ the following holds:
		\begin{align}
			\per(E,B_r(x))&\leq C_0 r^{d-1}\label{eq:per1}\\
			\per (E,B_r(x))&\geq C_1 r^{d-1}\label{eq:per2}\\
			\min\bigl\{|E\cap B_r(x)|,| B_r(x)\setminus E|\bigr\}&\geq\bar C_1 r^d\label{eq:vol1}\\
			Exc(E,x,r)&\leq C_2r^{\alpha}\label{eq:exc1}
		\end{align}
		Then, there exist $0< R_4<R_3$, $0<c_0<c_1<1$ such that the following holds. For every $x\in\partial ^*E$, $0<r< R_4$ let
		\begin{align}
			\theta(x,r)=\frac{\nu(x,r)}{|\nu(x,r)|},\quad \nu(x,r)=\frac{\int_{\partial^* E\cap B_r(x)}\nu_E(y)\d\mathcal H^{d-1}(y)}{\int_{\partial^* E\cap B_r(x)}|\nu_E|(y)\d\mathcal H^{d-1}(y)},\label{eq:nuxr}\end{align}
		and let
		\begin{align}
			Cyl(x, cr,\theta(x,r)):=\{z\in\R^d:\,\|z_{\theta(x,r)}^\perp-x_{\theta(x,r)}^\perp\|\leq cr, \,|z_{\theta(x,r)}-x_{\theta(x,r)}|\leq cr\},
		\end{align}
		where $c$ is such that $	Cyl(x, cr,\theta(x,r))\subset B_r(x)$, $c_0<c<c_1$.
		Then there exists an affine halfspace $\tilde H$  with exterior normal $\nu_{\tilde H}=\theta(x,r)$ such that
		{ \begin{equation}
				\label{eq:l1exc}
				\Biggl(\fint_{	Cyl(x, cr,\theta(x,r))}|\chi_E(z)-\chi_{\tilde H}(z)|\dz\Biggr)^2\lesssim Exc (E, x,r)^{1/2}.
		\end{equation} }
		
	\end{proposition}
	
	\begin{proof}
		\textbf{Step 0: Preliminary excess estimates}
		Let $\nu(x,r)$ and $\theta=\theta(x,r)\in\S^{d-1}$ be as in~\eqref{eq:nuxr}.
		We also introduce the orthogonal decomposition
		\[
		\nu_E(y)=\nu_\theta(y)\theta+\nu_{\theta}^\perp(y), \quad\nu_\theta\in\R
		\]
		and define
		\begin{align*}
			\per_\theta(E,B_r(x))=\int_{\partial^* E\cap B_r(x)}|\nu_\theta(y)|\d\mathcal H^{d-1}(y),\\
			\per_\theta^\perp(E,B_r(x))=\int_{\partial ^*E\cap B_r(x)}\|\nu_\theta^\perp(y)\|\d\mathcal H^{d-1}(y).
		\end{align*}

		The goal of this step is to show that
		\begin{equation}
			\label{eq:pthetaperp}
			\per_\theta^\perp(E, B_r(x))\lesssim r^{d-1}(Exc(E,x,r))^{1/2}\lesssim r^{d-1+\alpha/2}.
		\end{equation}
		
		The excess in the ball $B_r(x)$ can be rewritten as follows
		\begin{align}
			r^{d-1}	Exc(E,x,r)=\int_{\partial ^*E\cap B_r(x)}\sqrt{\nu_\theta^2(y)+\|\nu_\theta^\perp(y)\|^2}\d\mathcal H^{d-1}(y)-\int_{\partial^* E\cap B_r(x)}\nu_\theta(y)\d\mathcal H^{d-1}(y).\label{eq:excineq1}
		\end{align}

		By Jensen's inequality applied to the convex function $[0,+\infty)\times[0,+\infty)\ni(x,y)\mapsto \sqrt{x^2+y^2}$,~\eqref{eq:excineq1} gives
		\begin{align*}
			r^{d-1}	Exc(E,x,r)\geq \sqrt{\per_\theta(E,B_r(x))^2+\per_\theta^\perp(E, B_r(x))^2}-\per_\theta(E, B_r(x)).
		\end{align*}
		Now notice that, by definition of excess and the estimates~\eqref{eq:per2} and~\eqref{eq:exc1},
		\begin{equation*}
			\|\nu(x,r)\|-1\lesssim r^\alpha.
		\end{equation*} This fact, together with the definition of $\theta$, $\nu_\theta$, implies that for $0<r\leq R_4<R_3$ and $x\in\partial^*E$,
		
		\begin{align}
			\per_\theta(E, B_r(x))&\geq\Biggl|\int_{\partial ^*E\cap B_r(x)}\nu_\theta(y)\d\mathcal H^{d-1}(y)\Biggr|\notag\\
			&\geq\Biggl \|\int_{\partial ^*E\cap B_r(x)}\nu_E(y)\d\mathcal H^{d-1}(y)\Biggr\|\notag\\
			&\geq \|\nu(x,r)\|\per(E, B_r(x))\notag\\
			&\gtrsim r^{d-1}.\label{eq:pertheta}
		\end{align}
		Hence, expanding the function $y\mapsto\sqrt{1+y^2}-1$ up to second order  and  using~\eqref{eq:pertheta} and the fact that $\per_\theta^\perp(E,B_r(x))\leq \per(E, B_r(x))\leq C_0 r^{d-1}$, one obtains
		\begin{align*}
			r^{d-1}	Exc(E,x,r)&\geq \sqrt{\per_\theta(E,B_r(x))^2+\per_\theta^\perp(E, B_r(x))^2}-\per_\theta(E, B_r(x))\notag\\
			&=\per_\theta(E, B_r(x))\Biggl(\sqrt{ 1+\frac{\per_\theta^\perp(E, B_r(x))^2}{\per_\theta(E,B_r(x))^2}}-1\Biggr)\notag\\
			&\gtrsim r^{d-1}\frac{\per_\theta^\perp(E, B_r(x))^2}{\per_\theta(E,B_r(x))^2}.
		\end{align*}
		In particular, by~\eqref{eq:per1} and~\eqref{eq:exc1}, the estimate~\eqref{eq:pthetaperp} holds.

		\textbf{Step 1: Estimate of the excess for polyhedral sets}
		
		Assume from now on that $E$ is polyhedral and that $\mathcal H^{d-1}(\{z\in\partial^*E\cap B_r(x):\,\nu_\theta(z)=0\})=0$. Indeed, if we show that~\eqref{eq:l1exc} holds for polyhedral sets with sides whose normal is not orthogonal to $\theta$, then it holds by approximation for the sets of the given family $\mathscr{F}$.
		
		Define $\hat B:=(B_r(x))_\theta^\perp$ and let $N(z_\theta^\perp):=\per^{\mathrm{1D}}(E_{z_\theta^\perp}, (B_r(x))_{z_\theta^\perp})$, $f(z_\theta^\perp+s\theta)=\frac{|\nu_\theta^\perp(z_\theta^\perp+s\theta)|}{|\nu_\theta(z_\theta^\perp+s\theta)|}$,
		$\Omega_0=\{z_\theta^\perp\in\hat B: \,N(z_\theta^\perp)=0\}$.
		
		One has that
		\begin{align*}
			r^{d-1}Exc(E, x,r)	&=\int_{\partial ^*E\cap B_r(x)}\sqrt{\nu_\theta^2(y)+\|\nu_\theta^\perp(y)\|^2}\d\mathcal H^{d-1}(y)-\int_{\partial ^*E\cap B_r(x)}\nu_\theta(y)\d\mathcal H^{d-1}(y)	\notag\\
			&=\int_{\hat B}\Biggl[\sum_{\bigl\{s:\,z_\theta^\perp+s\theta\in\partial^* E_{x_\theta^\perp}\cap (B_r(x))_{z_\theta^\perp}\bigr\}}\sqrt{1+f^2(z_\theta^\perp+s\theta)}\notag\\
			&-\sum_{\bigl\{s:\,z_\theta^\perp+s\theta\in\partial ^*E_{x_\theta^\perp}\cap (B_r(x))_{z_\theta^\perp}\bigr\}}\mathrm{sign}(\nu_\theta(z_\theta^\perp+s\theta))\Biggr]\dz_\theta^\perp\notag\\
			&\geq \int_{\hat B\setminus \Omega_0}\Biggl[\Biggl(\sum_{\bigl\{s:\,z_\theta^\perp+s\theta\in\partial^* E_{x_\theta^\perp}\cap (B_r(x))_{z_\theta^\perp}\bigr\}}\sqrt{1+f^2(z_\theta^\perp+s\theta)}\Biggr)-1\Biggr]\dz_\theta^\perp\notag\\
			&=\int_{\hat B\setminus \Omega_0}\Biggl[\Biggl(\sum_{i=1}^{N(z_\theta^\perp)}\sqrt{1+\|\nabla h_i(z_\theta^\perp)\|^2}\Biggr)-1\Biggr]\dz_\theta^\perp,
		\end{align*}
		where $h_i:A_i\subset \hat B\setminus \Omega_0\to\R$ parametrizes the leaves of the boundary of $E$ in an open set $A_i$, and whose gradient is locally constant by the fact that $E$ is polyhedral.
		Defining $h(z_\theta^\perp)=\sum_{i=1}^{N(z_\theta^\perp)}h_i(z_\theta^\perp)$ and applying Jensen's inequality to the convex function $y\mapsto\sqrt{1+\|y\|^2}$, one has that
		\begin{align*}
			r^{d-1}Exc(E, x,r)	&\geq \int_{\hat B\setminus \Omega_0}\Bigl[\sqrt{N(z_\theta^\perp)^2+\|\nabla h(z_\theta^\perp)\|^2}-1\Bigr]\dz_\theta^\perp.
		\end{align*}
		Notice that, given a cylinder $Cyl(x,cr,\theta)$ as in the statement of the lemma, setting $\hat B_{cr}=\{z_\theta^\perp:\, \|z_\theta^\perp-x_\theta^\perp\|\leq cr\}$, $\tilde N(z_\theta^\perp):=\per^{\mathrm{1D}}(E_{z_\theta^\perp}, (x_\theta-cr,x_\theta+cr))$,
		$\tilde \Omega_0=\{z_\theta^\perp\in\hat B_{cr}: \,\tilde N(z_\theta^\perp)=0\}$, one has that
		\begin{align}\label{eq:exccyl}
			r^{d-1}Exc(E, x,r))	&\geq \int_{\hat B_{cr}\setminus \tilde \Omega_0}\Bigl[\sqrt{\tilde N(z_\theta^\perp)^2+\|\nabla h(z_\theta^\perp)\|^2}-1\Bigr]\dz_\theta^\perp.
		\end{align}
		
		\textbf{Step 2: Closeness of $\partial E$ to a single  graph inside the cylinder}
		
		We first decompose $\hat B_{cr}$ as follows: $\hat B_{cr}=\tilde\Omega \cup\tilde\Omega_0\cup\tilde\Omega_1\cup\tilde\Omega_2$, where
		$\tilde\Omega_0=\mathrm{int}\{z_\theta^\perp\in\hat B_{cr}: \,\tilde N(z_\theta^\perp)=0\}$, $\tilde\Omega_1=\mathrm{int}\{z_\theta^\perp\in\hat B_{cr}:\, \tilde N(z_\theta^\perp)=1\}$, 		$\tilde\Omega_2=\mathrm{int}\{z_\theta^\perp\in\hat B_{cr}:\, \tilde N(z_\theta^\perp)\geq2\}$. By the fact that  $E$ is polyhedral, $\mathcal H^{d-1}(\tilde\Omega)=0$ and $\tilde\Omega_0,\tilde\Omega_1,\tilde\Omega_2$ are sets of finite perimeter in $\hat B_{cr}$.
		
		Moreover, let us  decompose $\tilde \Omega _1$ as follows:
		\begin{align}
			\tilde\Omega_1&=\tilde\Omega_1^+\cup\tilde \Omega_1^-,\notag\\
			\tilde \Omega_1^+&=\{z_\theta^\perp\in\tilde\Omega_1:\,|E_{z_\theta^\perp}\cap [x_\theta-cr, h(z_\theta^\perp)]|=|E_{z_\theta^\perp}\cap [x_\theta-cr,x_\theta+cr ]|\},\notag\\
			\tilde \Omega_1^-&=\{z_\theta^\perp\in\tilde\Omega_1:\,|E_{z_\theta^\perp}\cap [h(z_\theta^\perp), x_\theta+cr]|=|E_{z_\theta^\perp}\cap [x_\theta-cr,x_\theta+cr ]|\}.\label{eq:omega1+}
		\end{align}
		The goal of this step is to show the following:\,there exist $\tilde C>0$, $R_4>0$ such that for all $E\in\mathscr{F}$, for all $x\in\partial ^*E$, either
		\begin{equation}\label{eq:bom+}
			|\hat B_{cr}\setminus \tilde \Omega_1^+|\leq\tilde  Cr^{d-1}Exc(E,x,r)^{1/2},			\quad\forall\,0<r< R_4
		\end{equation}
		or
		\begin{equation}\label{eq:bom-}
			|\hat B_{cr}\setminus \tilde \Omega_1^-|\leq\tilde  Cr^{d-1}Exc(E,x,r)^{1/2},			\quad\forall\,0<r<R_4	.
		\end{equation}

		The estimates~\eqref{eq:bom+} or~\eqref{eq:bom-} will be an immediate consequence of the following facts: there exists $\tilde C>0,  0<R_4\ll1$ s.t. $\forall\,x\in\partial E$, $0<r<R_4$
		\begin{align}
			|\tilde \Omega_2|&\lesssim	\tilde Cr^{d-1}Exc(E,x,r)\label{eq:omega2}\\
			|\tilde \Omega_0|&\lesssim \tilde Cr^{d-1}Exc(E, x,r)^{1/2}\label{eq:omega0}\\
			\min\{|\tilde \Omega_1^+|,|\tilde\Omega_1^-|\}&\lesssim\tilde C r^{d-1}(Exc(E, x,r))^{1/2}\label{eq:omega+}.
		\end{align}
		
		\textbf{Proof of~\eqref{eq:omega2}: }
		
		The  upper bound~\eqref{eq:omega2} holds by~\eqref{eq:exccyl} and the bound~\eqref{eq:exc1}.
		
		\textbf{Proof of~\eqref{eq:omega0}: }
		
		We now aim at proving~\eqref{eq:omega0}.
		
		Let
		\begin{align*}
			\tilde\Omega_0&=\tilde\Omega_0^+\cup\tilde\Omega_0^-,\notag\\
			\tilde\Omega_0^+&=\{z_\theta^\perp\in \tilde \Omega_0:\,|E_{z_\theta^\perp}\cap[x_\theta-cr,x_\theta+cr]|=2cr\},\notag\\
			\tilde\Omega_0^-&=\{z_\theta^\perp\in \tilde \Omega_0:\,|E_{z_\theta^\perp}\cap[x_\theta-cr,x_\theta+cr]|=0\},
		\end{align*}
		and assume \withoutLoss
		that $|\tilde \Omega_0^+|\geq |\tilde \Omega_0^-|$.
		
		For every $t\in[x_\theta-cr,x_\theta+cr]$, let $E_t=E\cap Cyl(x, cr, \theta)\cap \{z_\theta=t\}$. Then, by~\eqref{eq:exc1},~\eqref{eq:pthetaperp} and disintegration in direction $\theta$, one has that
		\begin{align*}
			r^{d-1+\alpha/2}\gtrsim r^{d-1}Exc(E,x,r)^{1/2}&\gtrsim\per_\theta^\perp(E, Cyl(x,cr,\theta))\notag\\
			&=\int_{x_\theta-cr}^{x_\theta+cr}\per(E_t, \hat B_{cr}+t\theta)\dt.
		\end{align*}
		In particular, letting
		\begin{equation*}
			A:=\{t\in[x_\theta-cr,x_\theta+cr]:\,\per (E_t, \hat B_{cr}+t\theta)\leq C r^{d-2} Exc(E, x,r)^{1/2}\},
		\end{equation*}
		one has that
		\begin{equation}\label{eq:a>c}
			|A|\gtrsim cr\Bigl(1-\frac{1}{C}\Bigr).
		\end{equation}
		
		On the other hand, notice that by definition
		\begin{equation}\label{eq:o0et}
			|\tilde \Omega_0^+|\leq |E_t|,\quad \text{ for \ae $t\in [x_\theta-cr,x_\theta+cr]$}.
		\end{equation}
		Since, by the isoperimetric inequality,
		\begin{align*}
			\min\bigl\{|E_t|, |(\hat B_{cr}+t\theta)\setminus E_t|\bigr\}^{\frac{d-2}{d-1}}&\lesssim	\per(E_t, \hat B_{cr}+t\theta), \quad \text{ for \ae $t\in[x_\theta-cr, x_\theta+cr]$, $d\geq 3$,}\notag\\
			\min\bigl\{|E_t|, |(\hat B_{cr}+t\theta)\setminus E_t|\bigr\}&\lesssim r	\per(E_t, \hat B_{cr}+t\theta), \quad \text{ for \ae $t\in[x_\theta-cr, x_\theta+cr]$, $d=2$,}
		\end{align*}
		in order to show~\eqref{eq:omega0} it is sufficient to show that
		\begin{equation}\label{eq:a>0}
			\bigl|\bigl\{t\in\,A: \min\bigl\{|E_t|, |(\hat B_{cr}+t\theta)\setminus E_t|\bigr\}=|E_t| \bigr\}\bigr|>0.
		\end{equation}
		Indeed, if~\eqref{eq:a>0} holds, choosing $t\in A$ one has that
		\begin{align*}
			|\tilde\Omega_0|&\leq2|\tilde\Omega_0^+|\leq 2|E_t|\lesssim\per(E_t, \hat B_{cr}+t\theta)^{\frac{d-1}{d-2}}\lesssim r^{d-1}Exc(E,x,r)^{1/2}\quad\text{ if $d\geq3$,}\notag\\
			|\tilde\Omega_0|&\leq2|\tilde\Omega_0^+|\leq 2|E_t|\lesssim r^{d-1}\per(E_t, \hat B_{cr}+t\theta)\lesssim r^{d-1}Exc(E,x,r)^{1/2}\quad\text{ if $d=2$.}
		\end{align*}
		
		Let us assume then that the converse to~\eqref{eq:a>0} holds. Then, using also~\eqref{eq:a>c}, one has that if $d\geq3$
		\begin{align*}
			\frac{1}{|A|}\int_A|(\hat B_{cr}+t\theta)\setminus E_t|^{\frac{d-2}{d-1}}\dt&\leq \frac{1}{|A|}\int_A\per (E_t, \hat B_{cr}+t\theta)\dt\notag\\
			&\lesssim \frac{r^{d-1}Exc(E,x,r)^{1/2}}{|A|}\lesssim r^{d-2}Exc(E,x,r)^{1/2}.
		\end{align*}
		As a consequence,
		\begin{align}
			|Cyl(x,cr,\theta)\setminus E|&\leq (cr)^{d-1}|[x_{\theta}-cr, x_\theta+cr]\setminus A|+\int_A|(\hat B_{cr}+t\theta)\setminus E_t|\dt\notag\\
			&\lesssim\frac{1}{C}r^d+\int _A\per(E_t, \hat B_{cr}+t\theta)^{\frac{d-1}{d-2}}\dt\notag\\
			&\lesssim\frac{1}{C}r^d+\int _A r^{d-1+\bar\alpha}\notag\\
			&\lesssim \Bigl[\frac{1}{C}+CExc(E,x,r)^{1/2}\Bigr]r^{d}.\label{eq:d3}
		\end{align}
		If $d=2$, one analogously obtains
		\begin{align}
			|Cyl(x,cr,\theta)\setminus E|&\lesssim \frac{1}{C}r^d+Cr^{d-1}Exc(E,x,r)^{1/2}. \label{eq:d2}
		\end{align}
		Both~\eqref{eq:d3} and~\eqref{eq:d2}, by choosing $C$ sufficiently large  and then $r$ sufficiently small, contradict the density estimate~\eqref{eq:vol1}. Thus,~\eqref{eq:omega0} is proved.
		
		\textbf{Proof of~\eqref{eq:omega+}: }
		
		Assume by contradiction that for some $x\in\partial ^*E$, $r$ sufficiently small,
		\begin{equation}\label{eq:omega+false}
			\min\bigl\{|\tilde \Omega_1^+|,|\tilde\Omega_1^-|\bigr\}\geq Cr^{d-1}Exc(E,x,r)^{1/2},
		\end{equation}
		where $C$ is sufficiently large, to be fixed later. \WithoutLoss, we can assume that
		\begin{equation}
			\label{eq:-false}
			\min\bigl\{|\tilde \Omega_1^+|,|\tilde\Omega_1^-|\bigr\}=|\tilde\Omega_1^-|.
		\end{equation}
		
		Fix $0<\varepsilon\ll1$ and decompose further $\tilde \Omega_1^\pm$ as follows
		\begin{align*}
			\tilde\Omega_1^+&=\tilde\Omega_{1,\varepsilon}^+\cup\tilde\Omega_{1,1-\varepsilon}^+\cup\tilde\Omega_{1,(\varepsilon,1-\varepsilon)}^+,\notag\\
			\tilde\Omega_1^-&=\tilde\Omega_{1,\varepsilon}^-\cup\tilde\Omega_{1,1-\varepsilon}^-\cup\tilde\Omega_{1,(\varepsilon,1-\varepsilon)}^-,\notag\\
			\tilde\Omega_{1,\varepsilon}^\pm&:=\{z_\theta^\perp\in\tilde\Omega_1^\pm:\,|h(z_\theta^\perp)-(x_\theta-cr)|\leq\varepsilon cr\},\notag\\
			\tilde\Omega_{1,1-\varepsilon}^\pm&:=\{z_\theta^\perp\in\tilde\Omega_1^\pm:\,|h(z_\theta^\perp)-(x_\theta+cr)|\leq\varepsilon cr\},\notag\\
			\tilde\Omega_{1,\varepsilon}^\pm&:=\{z_\theta^\perp\in\tilde\Omega_1^\pm:\,\mathrm{dist}(h(z_\theta^\perp),\{x_\theta\pm cr\})\geq\varepsilon cr\}.
		\end{align*}
		One has that, by~\eqref{eq:omega+false} and~\eqref{eq:-false},
		\begin{equation}\label{eq:max1}
			\max\bigl\{|\tilde\Omega_{1,\varepsilon}^-|,|\tilde\Omega_{1,1-\varepsilon}^-|,|\tilde\Omega_{1,(\varepsilon,1-\varepsilon)}^-|\bigr\}\geq \frac{C}{3}r^{d-1}Exc(E,x,r)^{1/2}.
		\end{equation}
		On the other hand, since by~\eqref{eq:omega2} and~\eqref{eq:omega0} $|\tilde\Omega_1|\gtrsim r^{d-1}$,
		\begin{equation}\label{eq:max2}
			\max\bigl\{|\tilde\Omega_{1,\varepsilon}^+|,|\tilde\Omega_{1,1-\varepsilon}^+|,|\tilde\Omega_{1,(\varepsilon,1-\varepsilon)}^+|\bigr\}\geq c r^{d-1}
		\end{equation}
		for some geometric constant $c$ independent of $x,r$.
		
		Notice that
		\begin{align}
			&z_\theta^\perp+[x_\theta-cr+\varepsilon cr, x_\theta+cr]\theta\subset \R^d\setminus E,&\text{ for $z_\theta^\perp\in \tilde\Omega_{1,\varepsilon}^+$},\label{eq:aeps}\\
			&z_\theta^\perp+[x_\theta-cr, x_\theta+cr-\varepsilon cr]\theta\subset \R^d\setminus E,&\text{ for $z_\theta^\perp\in\tilde\Omega_{1,1-\varepsilon}^-$},\label{eq:aeps2}\\
			&z_\theta^\perp+[x_\theta-cr, x_\theta+cr-\varepsilon cr]\theta\subset E,\quad &\text{ for $z_\theta^\perp\in \tilde\Omega_{1,1-\varepsilon}^+$},\label{eq:beps}\\
			&z_\theta^\perp+[x_\theta-cr+\varepsilon cr, x_\theta+cr]\theta\subset E,\quad & \text{ for $z_\theta^\perp\in \tilde\Omega_{1,\varepsilon}^-$},\label{eq:beps2}\\
			&	z_\theta^\perp+[x_\theta-cr, x_\theta-cr+\varepsilon cr]\theta\subset E,\quad &\text{ for $z_\theta^\perp\in \tilde\Omega_{1,(\varepsilon,1-\varepsilon)}^+$},\label{eq:ceps+}\\
			&z_\theta^\perp+[x_\theta-cr, x_\theta-cr+\varepsilon cr]\theta\subset \R^d\setminus E,\quad& \text{ for $z_\theta^\perp\in \tilde\Omega_{1,(\varepsilon,1-\varepsilon)}^-$},\label{eq:ceps-l}\\
			&	z_\theta^\perp+[x_\theta+cr-\varepsilon cr, x_\theta+cr]\theta\subset \R^d\setminus E,\quad &\text{ for $z_\theta^\perp\in \tilde\Omega_{1,(\varepsilon,1-\varepsilon)}^+$},\label{eq:ceps+u}\\
			&	z_\theta^\perp+[x_\theta+cr-\varepsilon cr, x_\theta+cr]\theta\subset E,\quad &\text{ for $z_\theta^\perp\in \tilde\Omega_{1,(\varepsilon,1-\varepsilon)}^-$},\label{eq:ceps-}
		\end{align}
		For simplicity of notation, let \[\omega^-:=\max\{|\tilde\Omega_{1,\varepsilon}^-|,|\tilde\Omega_{1,1-\varepsilon}^-|,|\tilde\Omega_{1,(\varepsilon,1-\varepsilon)}^-|\},\quad \omega^+:=\max\{|\tilde\Omega_{1,\varepsilon}^+|,|\tilde\Omega_{1,1-\varepsilon}^+|,|\tilde\Omega_{1,(\varepsilon,1-\varepsilon)}^+|\}.
		\]
		
		Now we want to show that if $C$ is sufficiently large (depending only on the uniform  constants appearing in the density, perimeter and excess bounds in the statement of the lemma), we find a contradiction in the following cases:
		\begin{align}
			&	\omega^-=|\tilde\Omega_{1,\varepsilon}^-|,\quad 	\omega^+=|\tilde\Omega_{1,(\varepsilon,1-\varepsilon)}^+|,\label{eq:model1}\\
			&	\omega^-=|\tilde\Omega_{1,\varepsilon}^-|,\quad 	\omega^+=|\tilde\Omega_{1,\varepsilon}^+|,\label{eq:model0}\\
			&	\omega^-=|\tilde\Omega_{1,(\varepsilon,1-\varepsilon)}^-|,\quad 	\omega^+=|\tilde\Omega_{1,\varepsilon}^+|,\\
			&		\omega^-=|\tilde\Omega_{1,(\varepsilon,1-\varepsilon)}^-|,\quad 	\omega^+=|\tilde\Omega_{1,(\varepsilon,1-\varepsilon)}^+|,\\
			&	\omega^-=|\tilde\Omega_{1,(\varepsilon,1-\varepsilon)}^-|,\quad 	\omega^+=|\tilde\Omega_{1,1-\varepsilon}^+|,\\
			&	\omega^-=|\tilde\Omega_{1,1-\varepsilon}^-|,\quad 	\omega^+=|\tilde\Omega_{1,(\varepsilon,1-\varepsilon)}^+|,\label{eq:model6}\\
			&	\omega^-=|\tilde\Omega_{1,1-\varepsilon}^-|,\quad 	\omega^+=|\tilde\Omega_{1,1-\varepsilon}^+|.\label{eq:model7}
		\end{align}
		Since the proof is analogous in all the above cases, we consider only one of them, \eg~\eqref{eq:model1}.
		
		If~\eqref{eq:model1} holds, then recalling~\eqref{eq:beps2} and~\eqref{eq:ceps-}, for $t\in [x_\theta+cr-\varepsilon cr, x_\theta+cr]$ one has that $|E_t|\geq |\tilde\Omega_{1,\varepsilon}^-|=\omega^-$ and $|(\hat B_{cr}+t\theta)\setminus E_t| \geq |\tilde\Omega_{1,(\varepsilon,1-\varepsilon)}^+|=\omega^+$. Thus, by~\eqref{eq:max1} and~\eqref{eq:max2},
		\begin{align*}
			\Bigl(\frac{C}{3}\Bigr)\varepsilon c r^{d-1}Exc(E,x,r)^{1/2}\lesssim	&\int_{[x_\theta+cr-\varepsilon cr, x_\theta+cr]}\min\{|E_t|,|(\hat B_{cr}+t\theta)\setminus E_t| \}^{\frac{d-2}{d-1}}\dt\lesssim \notag\\
			&\lesssim	\int_{[x_\theta+cr-\varepsilon cr, x_\theta+cr]}\per(E_t,\hat B_{cr}+t\theta )\dt\notag\\
			&\lesssim \per_\theta^\perp (E, Cyl(x,cr,\theta))\notag\\
			&\lesssim\bar C r^{d-1}Exc(E,x,r)^{1/2},
		\end{align*}
		that for $\varepsilon$ fixed and  $C$ sufficiently large leads to a contradiction.
		
		The only two cases which are not included in~\eqref{eq:model0}-\eqref{eq:model7} are the following
		\begin{align}
			&	\omega^-=|\tilde\Omega_{1,\varepsilon}^-|,\quad 	\omega^+=|\tilde\Omega_{1,1-\varepsilon}^+|,\label{eq:caso1}\\
			&	\omega^-=|\tilde\Omega_{1,1-\varepsilon}^-|,\quad 	\omega^+=|\tilde\Omega_{1,\varepsilon}^+|\label{eq:caso2}
		\end{align}
		together with respectively
		\begin{align}
			\max\bigl\{|\tilde \Omega^-_{1,1-\varepsilon}|,|\tilde \Omega^-_{1,(\varepsilon,1-\varepsilon)}|,|\tilde \Omega^+_{1,\varepsilon}|, |\tilde \Omega^+_{1,(\varepsilon,1-\varepsilon)}|\bigr\}&\leq \frac{C}{3}r^{d-1}Exc(E,x,r)^{1/2}\label{eq:max1omega}\\
			\max\bigl\{|\tilde \Omega^-_{1,\varepsilon}|,|\tilde \Omega^-_{1,(\varepsilon,1-\varepsilon)}|,|\tilde \Omega^+_{1,1-\varepsilon}|, |\tilde \Omega^+_{1,(\varepsilon,1-\varepsilon)}|\bigr\}&\leq \frac{C}{3}r^{d-1}Exc(E,x,r)^{1/2}\label{eq:max2omega}.
		\end{align}
		
		Notice that, by~\eqref{eq:aeps}-\eqref{eq:beps2} and by the density estimate~\eqref{eq:vol1},
		\begin{align}
			|\tilde\Omega_{1,\varepsilon}^-|+|\tilde\Omega_{1,1-\varepsilon}^+|&\lesssim \frac{| E\cap Cyl(x,cr,\theta)|}{(1-\varepsilon)cr}\leq  (cr)^{d-1}\frac{(1-\bar C_1)}{1-\varepsilon},\label{eq:stob1}\\
			|\tilde\Omega_{1,1-\varepsilon}^-|+|\tilde\Omega_{1,\varepsilon}^+|&\lesssim \frac{| Cyl(x,cr,\theta)\setminus E|}{(1-\varepsilon)cr}\leq (cr)^{d-1}\frac{(1-\bar C_1)}{1-\varepsilon}.\label{eq:stob2}
		\end{align}
		Choose $\varepsilon>0$ s.t. $(1-\bar C_1)(1-\varepsilon)<1-2\delta$ for some small $\delta>0$. Choosing then $r>0$ sufficiently small, due to~\eqref{eq:max1omega},~\eqref{eq:max2omega}, the bounds~\eqref{eq:exc1},~\eqref{eq:omega2} and~\eqref{eq:omega0} it holds
		\begin{align}
			|\tilde\Omega_{1,\varepsilon}^-|+|\tilde\Omega_{1,1-\varepsilon}^+|&\geq  (cr)^{d-1}(1-\delta)>(cr)^{d-1}(1-2\delta)\label{eq:stoc1}\\
			|\tilde\Omega_{1,1-\varepsilon}^-|+|\tilde\Omega_{1,\varepsilon}^+|&\geq (cr)^{d-1}(1-\delta)>(cr)^{d-1}(1-2\delta),\label{eq:stoc2}
		\end{align}
		thus contradicting~\eqref{eq:stob1} and~\eqref{eq:stob2}.

		\textbf{Step 3: Lipschitz extension of $u$ on $\hat B_{cr}$.}
		
		By step 1 (assuming that~\eqref{eq:bom+} holds), one has that
		\begin{equation}\label{eq:exc3}
			r^{d-1}Exc(E, x,r)\geq \int_{\tilde\Omega_1^+}\Bigl[\sqrt{1+\|\nabla h(z_\theta^\perp)\|^2}-1\Bigr]dz_\theta^\perp,
		\end{equation}
		where $\tilde \Omega_1^+$ was defined in~\eqref{eq:omega1+}.
		
		The aim of this step is to show that there exists a Lipschitz function $\tilde v$ on $\hat B_{cr}$ with Lipschitz constant independent of $x,r$ such that $|\{z_\theta^\perp\in\tilde\Omega_1^+:\, \tilde v(z_\theta^\perp)\neq h(z_\theta^\perp)\}|\lesssim r^{d-1}Exc(E,x,r)^{1/2}$ and
		
		\begin{align}\label{eq:exclip}
			Exc(E,x,r)^{1/2}\gtrsim \int_{\hat B_{cr}}\Bigl[\sqrt{1+|\nabla \tilde v(z_\theta^\perp)|^2}-1\Bigr]\dz_\theta^\perp.
		\end{align}

		Define first  the BV function
		\begin{equation}\label{eq:v}
			v:\hat B_{cr}\to\R,\quad v(z_\theta^\perp)=|E_{z_\theta^\perp}\cap[x_\theta-cr, x_\theta+cr]|,
		\end{equation}
		and notice that
		\begin{equation*}
			h=v\quad\text{ \ae on }\tilde\Omega_1^+.
		\end{equation*}
		
		Given $\lambda>1$, define now the set
		\begin{align*}
			A_\lambda:=\Biggl\{z_\theta^\perp\in\tilde\Omega_1^+:\,\sup_{\varepsilon>0, B_\varepsilon(y_\theta^\perp)\ni z_\theta^\perp, B_\varepsilon(y_\theta^\perp)\subset \hat B_{cr}}\frac{|Dv|(B_\varepsilon(y_\theta^\perp))}{\varepsilon^{d-1}}\leq \lambda\Biggr\}.
		\end{align*}
		By classical estimates on maximal functions, $A_\lambda\subset \tilde\Omega_1^+$ is a closed set and
		\begin{equation}\label{eq:omegamenoa}
			|\tilde\Omega_1^+\setminus A_\lambda|\leq\frac{|Dv|(\hat B_{cr})}{\lambda}
		\end{equation}
		\textbf{Clam 1: } It holds
		\begin{equation}\label{eq:dvper}
			|Dv|(\hat B_{cr})\leq \per_\theta^\perp(E, Cyl(x,cr,\theta)),
		\end{equation}
		and thus by~\eqref{eq:omegamenoa} and~\eqref{eq:pthetaperp}
		\begin{equation}\label{eq:omegaa}
			|\tilde\Omega_1^+\setminus A_\lambda|\leq \frac{ r^{d-1}(Exc(E, x,r))^{1/2}}{\lambda}\lesssim\frac{1}{\lambda}r^{d-1+\alpha/2}.
		\end{equation}
		Let us now prove Clam 1. Given a monotone increasing sequence of functions $\phi_n\in C^1([x_\theta-cr, x_\theta+cr];[0,1])$, $\phi_n\uparrow 1$, one has that
		\begin{align*}
			|Dv|(\hat B_{cr})&=\sup_{\psi\in C^1_c(\hat B_{cr}, \R^d), |\psi|\leq1}\Biggl|\int_{\hat B_{cr}}\psi (z_\theta^\perp)\cdot\mathrm{d}Dv(z_\theta^\perp) \Biggr|\notag\\
			&=\sup_{\psi\in C^1_c(\hat B_{cr}, \R^d), |\psi|\leq1}\Biggl|\int_{\hat B_{cr}}\div_{z_\theta^\perp}\psi(z_\theta^\perp)v(z_\theta^\perp)\dz_\theta^\perp\Biggr|\notag\\
			&=\sup_{\psi\in C^1_c(\hat B_{cr}, \R^d), |\psi|\leq1}\Biggl|\int_{\hat B_{cr}}\div_{z_\theta^\perp}\psi(z_\theta^\perp)\Biggl(\sup_n\int\phi_n(z_\theta)\chi_{E_{z_\theta^\perp}}(z_\theta)\dz_\theta\Biggr)\dz_\theta^\perp\Biggr|\notag\\
			&\leq\sup_{\psi\in C^1_c(\hat B_{cr}, \R^d), |\psi|\leq1}\lim_n\Biggl|\int_{\hat B_{cr}}\int_{[x_\theta-cr, x_\theta+cr]}\div_{z_\theta^\perp}\psi(z_\theta^\perp)\phi_n(z_\theta)\chi_{E_{z_\theta^\perp}}(z_\theta)\dz_\theta\dz_\theta^\perp\Biggr|\notag\\
			&\leq \sup_{\psi\in C^1_c(\hat B_{cr}, \R^d), |\psi|\leq1, \phi\in C^1([x_\theta-cr, x_\theta+cr]), |\phi|\leq 1}\Biggl|\int_{\hat B_{cr}}\int_{[x_\theta-cr, x_\theta+cr]}\div_{z_\theta^\perp}\psi(z_\theta^\perp)\phi(z_\theta)\chi_{E}(z)\dz_\theta\dz_\theta^\perp\Biggr|\notag\\
			&= \sup_{\psi\in C^1_c(\hat B_{cr}, \R^d), |\psi|\leq1, \phi\in C^1_c([x_\theta-cr, x_\theta+cr]), |\phi|\leq 1}\Biggl|\int_{\partial^*E\cap Cyl(x,cr,\theta)}(\psi\phi)(y)\cdot\nu_\theta^\perp(y)\d\mathcal H^{d-1}(y)\Biggr|\notag\\
			&\leq \sup_{T\in C^1_c(Cyl(x,cr,\theta);\R^d), |T|\leq1}\Biggl|\int_{\partial^*E\cap Cyl(x,cr,\theta)}T\cdot\nu_\theta^\perp(y)\d\mathcal H^{d-1}(y)\Biggr|\notag\\
			&=\int_{\partial^*E\cap Cyl(x,cr,\theta)}|\nu_\theta^\perp(y)|\d\mathcal H^{d-1}(y)\notag\\
			&=\per_\theta^\perp(E, Cyl(x,cr,\theta)).
		\end{align*}
		
		\textbf{Claim 2: }The function  $v$ has a representative $\tilde v$ which is Lipschitz with constant of order $\lambda$ on $A_\lambda$. To this aim, for $r_0>0$ and $k\in\N$, we define the sets
		\begin{equation*}
			\hat B_{cr,k}:=\{z_\theta^\perp\in\hat B_{cr}:\,\mathrm{dist}(z_\theta^\perp,\partial\hat B_{cr})\geq r_0 2^{-k}\}, \quad B_k(z_\theta^\perp):=B_{r_0 2^{-k}}(z_\theta^\perp),
		\end{equation*}
		and the functions
		\begin{equation*}
			v_k:\hat B_{cr,k}\to\R, \quad v_k(z_\theta^\perp)=\fint_{B_{k}(z_\theta^\perp)}v(y_\theta^\perp)\dy_\theta^\perp.
		\end{equation*}
		We first claim that $\{v_k\chi_{A_\lambda}\}$ is a Cauchy sequence with respect to uniform convergence on open sets $W\subset\subset\hat B_{cr}$.
		Given $W\subset\subset\hat B_{cr}$, and $k\in\N$ such that $v_m$ is well-defined on $W$ for all $m\geq k$, for all $z_\theta^\perp\in W\cap A_\lambda$, $m\geq k$ one has that
		\begin{align}
			|v_m(z_\theta^\perp)-v_{m+1}(z_\theta^\perp)|&=\Biggl|\fint_{B_m(z_\theta^\perp)}v(y_\theta^\perp)\dy_\theta^\perp-\fint_{B_{m+1}(z_\theta^\perp)}v(y_\theta^\perp)\dy_\theta^\perp\Biggr|\notag\\
			&\leq\frac{1}{r_{0} 2^{-(m+1)}}\int_{B_{m+1}(z_\theta^\perp)}\Biggl|v(y_\theta^\perp)-\fint_{B_m(z_\theta^\perp)}v(w_\theta^\perp)\dw_\theta^\perp\Biggr|\dy_\theta^\perp\notag\\
			&\leq \frac{r_0 2^{-m}}{r_{0} 2^{-(m+1)}}\fint_{B_{m}(z_\theta^\perp)}\Biggl|v(y_\theta^\perp)-\fint_{B_m(z_\theta^\perp)}v(w_\theta^\perp)\dw_\theta^\perp\Biggr|\dy_\theta^\perp\notag\\
			&\leq 2\frac{|Dv|(B_m(z_\theta^\perp))}{(r_0 2^{-m})^{d-2}},\label{eq:245}
		\end{align}
		where in the last equality we used Poincar\'e inequality in BV\@.
		
		Using now the fact that $z_\theta^\perp\in A_\lambda$,
		\begin{align}\label{eq:vm}
			|v_m(z_\theta^\perp)-v_{m+1}(z_\theta^\perp)|&\leq 2\frac{|Dv|(B_m(z_\theta^\perp))}{(r_0 2^{-m})^{d-2}}\leq 2\lambda r_0 2^{-m},
		\end{align}
		thus proving that $\{v_m\}$ is Cauchy in $C^0(W\cap A_\lambda)$.
		
		Define then in $\hat B_{cr}$ the function (representative of $v$)
		\begin{equation*}
			\tilde v(z_\theta^\perp)=\lim_k v_k.
		\end{equation*}
		We now show that $\tilde v$ is Lipschitz on $A_\lambda$  with Lipschitz constant of order $\lambda$.
		
		Let $y_\theta^\perp,z_\theta^\perp\in A_\lambda$, $|y_\theta^\perp-z_\theta^\perp|\in [r_0 2^{-k-1}, r_0 2^{-k}]$ for some $r_0>0$, $k\in\N$ and such that $B_{k-1}(y_\theta^\perp)\subset \hat B_{cr}$. One has that
		\begin{align}
			|\tilde v(y_\theta^\perp)-\tilde v(z_\theta^\perp)|&\leq\Biggl|\tilde v(y_\theta^\perp)-\fint_{B_{k-1}(y_\theta^\perp)}v(w_\theta^\perp)\dw_\theta^\perp\Biggr|\notag\\
			&+\Biggl|\tilde v(z_\theta^\perp)-\fint_{B_{k}(z_\theta^\perp)}v(w_\theta^\perp)\dw_\theta^\perp\Biggr|\notag\\
			&+\Biggl|\fint_{B_{k-1}(y_\theta^\perp)}v(w_\theta^\perp)\dw_\theta^\perp-\fint_{B_{k}(z_\theta^\perp)}v(w_\theta^\perp)\dw_\theta^\perp\Biggr|.\label{eq:lipmedia}
		\end{align}
		By the estimate~\eqref{eq:vm}, one has that
		\begin{align*}
			\Biggl|\tilde v(y_\theta^\perp)-\fint_{B_{k-1}(y_\theta^\perp)}v(w_\theta^\perp)\dw_\theta^\perp\Biggr|&=\Biggl|\tilde v(y_\theta^\perp)-v_{k-1}(y_\theta^\perp)\Biggr|\leq\lambda r_{0} 2^{-k+1}\leq 4\lambda|y_\theta^\perp-z_\theta^\perp|
		\end{align*}
		and an analogous estimate holds for the second term in~\eqref{eq:lipmedia}.
		The last term in~\eqref{eq:lipmedia} can be estimated as follows:
		\begin{align}\Biggl|\fint_{B_{k-1}(y_\theta^\perp)}v(w_\theta^\perp)\dw_\theta^\perp-\fint_{B_{k}(z_\theta^\perp)}v(w_\theta^\perp)\dw_\theta^\perp\Biggr|&\leq2\fint_{B_{k-1}(y_\theta^\perp)}\Biggl|v(w_\theta^\perp)-\fint_{B_{k-1}(y_\theta^\perp)}v(t_\theta^\perp)\dt_\theta^\perp\Biggr|\dw_\theta^\perp\notag\\
			&\leq 2\frac{|Dv|(B_{k-1}(y_\theta^\perp))}{(r_o2^{-k+1})^{d-2}}\notag\\
			&\leq 2\lambda r_0 2^{-k+1}\notag\\
			&\leq 8\lambda|y_\theta^\perp-z_\theta^\perp|,\label{eq:250}
		\end{align}
		where in the first inequality we reasoned as in~\eqref{eq:245}.
		Thus, putting together~\eqref{eq:245},~\eqref{eq:lipmedia} and~\eqref{eq:250},  $\tilde v$ is Lipschitz on $A_\lambda$ with Lipschitz constant smaller than $16\lambda$.
		
		By Whitney's Theorem we can now extend $\tilde v$ to a $16\lambda$-Lipschitz function $\tilde v$ of $\hat B_{cr}$.
		
		Finally, recalling the bounds~\eqref{eq:exc3},~\eqref{eq:omegaa},~\eqref{eq:omega2},~\eqref{eq:omega0},~\eqref{eq:exc1} and the fact that $h=\tilde v$ \ae on $A_\lambda$, the estimate~\eqref{eq:exclip} holds.  Indeed,
		
		\begin{align*}
			r^{d-1}Exc(E,x,r)&\gtrsim\int_{\tilde\Omega_1^+}\Bigl[\sqrt{1+\|\nabla h(z_\theta^\perp)\|^2}-1\Bigr]\dz_\theta^\perp\notag\\
			&\gtrsim \int_{A_\lambda}\Bigl[\sqrt{1+\|\nabla \tilde v(z_\theta^\perp)\|^2}-1\Bigr]\dz_\theta^\perp\notag\\
			&\gtrsim \int_{\hat B_{cr}}\Bigl[\sqrt{1+\|\nabla \tilde v(z_\theta^\perp)\|^2}-1\Bigr]\dz_\theta^\perp-\sqrt{1+(16\lambda)^2}(|\tilde \Omega_0|+|\tilde\Omega_2|+|\tilde \Omega_1^-|+ |\tilde \Omega_1^+\setminus A_\lambda|)\notag\\
			&\gtrsim \int_{\hat B_{cr}}\Bigl[\sqrt{1+\|\nabla \tilde v(z_\theta^\perp)\|^2}-1\Bigr]\dz_\theta^\perp-C\lambda r^{d-1}Exc(E,x,r)^{1/2}. \notag
		\end{align*}
		
		\textbf{Step 4: Conclusion}
		Applying to~\eqref{eq:exclip} Jensen's and Poincar\'e inequality, is $r$ is sufficiently small so that $Exc(E,x,r)\ll1$, one has that
		\begin{align*}
			r^{d-1}&Exc(E, x,r)^{1/2}\gtrsim\int_{\hat B_{cr}}\Bigl[\sqrt{1+\|\nabla \tilde v(z_\theta^\perp)\|^2}-1\Bigr]\dz_\theta^\perp\notag\\
			&\gtrsim r^{d-1}\Biggl[\sqrt{1+\Biggl(\fint_{\hat B_{cr}}\|\nabla \tilde v(w_\theta^\perp)\|\dw_\theta^\perp\Biggr)^2}-1\Biggr]\notag\\
			&\gtrsim r^{d-1}\min\Biggl\{\fint_{\hat B_{cr}}\|\nabla \tilde v(w_\theta^\perp)\|\dw_\theta^\perp, \Biggl(\fint_{\hat B_{cr}}\|\nabla \tilde v(w_\theta^\perp)\|\dw_\theta^\perp\Biggr)^2\Biggr \}\notag\\
			&\gtrsim r^{d-1}\min\Biggl\{\frac1r\fint_{\hat B_{cr}}\Biggl|\tilde v(y_\theta^\perp)-\fint_{\hat B_{cr}}\tilde v(w_\theta^\perp)\dw_\theta^\perp\Biggr|\dy_\theta^\perp, \frac{1}{r^2}\Biggl(\fint_{\hat B_{cr}}\Biggl|\tilde v(y_\theta^\perp)-\fint_{\hat B_{cr}}\tilde v(w_\theta^\perp)\dw_\theta^\perp\Biggr|\dy_\theta^\perp\Biggr)^2\Biggr\}.
		\end{align*}
		In particular, by the bound~\eqref{eq:exc1}, if $0<r<\bar R\ll1$ then
		\begin{equation*}
			Exc(E,x,r)^{1/2}\gtrsim \frac{1}{r^2}\Biggl(\fint_{\hat B_{cr}}\Biggl|\tilde v(y_\theta^\perp)-\fint_{\hat B_{cr}}\tilde v(w_\theta^\perp)\dw_\theta^\perp\Biggr|\dy_\theta^\perp\Biggr)^2.
		\end{equation*}
		Now observe that, since $|\hat B_{cr}\setminus A_\lambda|\lesssim r^{d-1}Exc(E,x,r)^{1/2}$ and since $|\tilde v|\lesssim r$, then
		\begin{align*}
			\frac{1}{r^2|\hat B_{cr}|^2}\Biggl(\int_{\hat B_{cr}\setminus A_\lambda }\Biggl|\tilde v(y_\theta^\perp)-\fint_{\hat B_{cr}}\tilde v(w_\theta^\perp)\dw_\theta^\perp\Biggr|\dy_\theta^\perp\Biggr)^2\lesssim Exc(E, x,r)\lesssim r^\alpha.
		\end{align*}
		On the other hand, for \ae $y_\theta^\perp\in A_\lambda$, $\tilde v(y_\theta^\perp)=v(y_\theta^\perp)=|E_{y_\theta^\perp}\cap [x_\theta-cr, x_\theta+cr]|$. Thus, defining $\tilde h=\fint_{\hat B_{cr}}\tilde v(w_\theta^\perp)\dw_\theta^\perp$ and  $\tilde H$ as the halfspace $\tilde H= \{z:\,z_\theta\leq \tilde h\}$, it holds
		\begin{align*}
			\frac{1}{r^2|\hat B_{cr}|^2}&\Biggl(\int_{A_\lambda }\Biggl|\tilde v(y_\theta^\perp)-\fint_{\hat B_{cr}}\tilde v(w_\theta^\perp)\dw_\theta^\perp\Biggr|\dy_\theta^\perp\Biggr)^2\sim\notag\\
			&\sim 	\frac{1}{|\hat B_{cr}|^2}	\Biggl(\int_{A_\lambda }\Biggl|\fint_{[x_\theta-cr, x_\theta+cr]}\chi_{E_{y_\theta^\perp}}(y_\theta)-\chi_{\tilde H}(y_\theta)\dy_\theta\Biggr|\dy_\theta^\perp\Biggr)^2.
		\end{align*}
		Now we use the fact that $A_\lambda\subset \tilde \Omega_1^+$ and thus for all $y_\theta^\perp\in A_\lambda$
		\begin{equation*}
			\Biggl|\fint_{[x_\theta-cr, x_\theta+cr]}\chi_{E_{y_\theta^\perp}}(y_\theta)-\chi_{\tilde H}(y_\theta)\dy_\theta\Biggr|=\fint_{[x_\theta-cr, x_\theta+cr]}\Biggl|\chi_{E_{y_\theta^\perp}}(y_\theta)-\chi_{\tilde H}(y_\theta)\Biggr|\dy_\theta,
		\end{equation*}
		thus getting
		\begin{align*}
			Exc(E, x,r)^{1/2}\gtrsim \Biggl(\fint_{Cyl(x, cr, \theta)}|\chi_E(z)-\chi_{\tilde H}(z)|\dz\Biggr)^2
		\end{align*}
		as desired.
		
		As a consequence of the above proposition, one can control the uniform $L^1$ distance of $E$ from the halfspace orthogonal to the measure theoretic exterior normal. More precisely, one has the following

		\begin{corollary}\label{cor:eh}
			Let $\mathscr{F}$ be a family of sets of locally finite perimeter as in Proposition~\ref{lemma:regularity}.
			
			For every $\varepsilon>0$ there exists $R_5(\varepsilon)>0$ such that for all $E\in\mathscr{F}$, for all $x\in\partial^*E$ and for all $0<r<R_5(\varepsilon)$
			\begin{equation}\label{eq:hnucyl}
				\fint_{Cyl(x,r,\nu_E(x))}\Bigl|\chi_E(z)-\chi_{H_{\nu_E(x)}}(z)\Bigr|\dz\leq \varepsilon,
			\end{equation}
			where $Cyl(x,r,\nu_E(x))=\big\{z\in\R^d:\,\|z_{\nu_E(x)}^\perp-x_{\nu_E(x)}^\perp\|<r, |z_{\nu_E(x)}-x_{\nu_E(x)}|<r\big\}$.
		\end{corollary}
		\begin{proof}
			First we prove that there exists $\bar R_5(\varepsilon)>0$ such that for all $E\in\mathscr{F}$, for all $x\in\partial^*E$ and for all $0<r<\bar R_5(\varepsilon)$
			\begin{equation}\label{eq:tildeh1}
				\mathrm{dist}(x,\tilde H)\leq \varepsilon r,
			\end{equation}
			where $\tilde H$ is the affine halfspace of Proposition~\ref{lemma:regularity}.
			Indeed, assume $\mathrm{dist}(x,\tilde H)\geq \varepsilon r$. Then, $B_{\varepsilon r}(x)\subset \R^d\setminus \tilde H$, hence by Proposition~\ref{lemma:regularity} and the uniform bound on the excess~\eqref{eq:fexc},
			\begin{align}|E\cap B_{\varepsilon r}(x)|\leq \|\chi_E-\chi_{\tilde H}\|_{L^1(Cyl(x,cr,\theta))}\lesssim r^{d}Exc(E,x,r)^{1/4}.\label{eq:cor1e}
			\end{align}
			On the other hand, by the volume density estimate~\eqref{eq:vol1}
			\begin{align}
				|E\cap B_{\varepsilon r}(x)|&\geq \bar C_1(\varepsilon r)^d,\label{eq:cor2e}
			\end{align}
			hence by~\eqref{eq:cor1e} and~\eqref{eq:cor2e}
			\begin{equation*}
				\bar C_1\varepsilon^d r^d \lesssim r^{d}Exc(E,x,r)^{1/4},
			\end{equation*}
			that for $r<\bar R_5(\varepsilon)\ll1$  leads to a contradiction.
			
			Let now $\tilde H(x)$ be the affine halfspace with boundary parallel to $\partial\tilde H$ (\ie orthogonal to $\theta(x,r)$) and passing through the point $x$. By~\eqref{eq:l1exc} and~\eqref{eq:tildeh1}, one has that for $r<\bar R_2(\varepsilon)\ll1$
			\begin{equation}\label{eq:hcyl}
				\fint_{Cyl(x,r,\theta(x,r))}\Bigl|\chi_E(z)-\chi_{\tilde H(x)}(z)\Bigr|\dz\leq \varepsilon.
			\end{equation}
			Now recall (see \eg~\cite{Tam}) that for the uniform lower bound on the perimeter~\eqref{eq:per2} and the uniform power law decay of the excess~\eqref{eq:exc1}, for every $\varepsilon>0$ there exists $\tilde R_5(\varepsilon)$ such that for every $E\in\mathscr{F}$, for every $x\in\partial^* E$ and for all $0<r<\tilde R_5(\varepsilon)$ it holds
			\begin{equation*}
				\|\theta(x,r)-\nu_E(x)\|\leq \varepsilon.
			\end{equation*}
			Hence, up to reducing further $\bar R_5(\varepsilon),\tilde R_5(\varepsilon)$ into $0<R_5(\varepsilon)\ll1$, we can substitute $H_{\nu_E(x)}(x)$ to $\tilde H(x)$ and $\nu_E(x)$ to $\theta(x,r)$ in~\eqref{eq:hcyl}, thus getting~\eqref{eq:hnucyl}.
		\end{proof}
		
		The above uniform  control of the $L^1$-distance of the set $E$ to the halfspace determined by the measure theoretic  exterior  normal  gives Lipschitz regularity of $\partial^* E$. More precisely, let now  $0<\ell<1$, $x\in\partial^*E$ and define the cone  $K_\ell(x,\nu_E(x))=\{z\in\R^d:\,\Scal{z-x}{\nu_E(x)}<\ell\}$.  One has the following
		
		\begin{lemma}\label{lemma:lipcone}
			Let $\mathscr{F}$ be a family of sets as in Proposition~\ref{lemma:regularity}.
			
			For every $0<\ell\ll1$, there exists $R_\ell>0$ s.t. $\forall\,E\in\mathscr{F}$, $\forall\,x\in\partial^*E$, $\forall\,0<r<R_\ell$ it holds
			\begin{equation}\label{eq:lemmalip}
				\partial^*  E\cap B_r(x)\subset B_r(x)\cap (x+K_\ell(x,\nu_E(x))).
			\end{equation}
		\end{lemma}
		\begin{proof}
			Let $0<\varepsilon=\varepsilon(\ell)\ll1$ sufficiently small to be fixed later. Let $ R_5(\varepsilon)$ as in Corollary~\ref{cor:eh} and let $R_\ell$ such that $R_\ell(1+\ell)< R_5(\varepsilon)$.
			
			Let now $0<r<R_\ell$ and assume $\exists\,y\in\partial^*E\cap (B_r(x)\setminus (x+K_\ell(x,\nu_E(x))))$. In particular,
			\begin{equation}\label{eq:bsubs}
				B_{\|y-x\|\ell}(y)\subset B_{\|y-x\|(1+\ell)}(x)\cap \R^d\setminus H_{\nu_E(x)}(x).
			\end{equation}
			Applying Corollary~\ref{cor:eh} on $B_{\|y-x\|\ell}(y)$, since $\|y-x\|\ell\leq R_\ell \ell<R_5(\varepsilon)$,
			\begin{equation}\label{eq:conec1}
				|E\cap 		B_{\|y-x\|\ell}(y)|\geq\Biggl(\frac12-\varepsilon\Biggr)(\|y-x\|\ell)^d.
			\end{equation}
			Applying the inclusion~\eqref{eq:bsubs} and Corollary~\ref{cor:eh} on 		$B_{\|y-x\|(1+\ell)}(x)$, since $\|y-x\|(1+\ell)\leq R_\ell(1+\ell)<R_5(\varepsilon)$,
			\begin{align}
				|E\cap 		B_{\|y-x\|\ell}(y)|&\leq |E\cap B_{\|y-x\|(1+\ell)}(x)\cap \R^d\setminus H_{\nu_E(x)}(x)|\notag\\
				&\leq\|\chi_E-\chi_{H_{\nu_E(x)}(x)}\|_{L^1(B_{\|y-x\|(1+\ell)}(x))}\notag\\
				&\leq \varepsilon \|y-x\|^d(1+\ell)^d\label{eq:conec2}
			\end{align}
			Hence,~\eqref{eq:conec1} and~\eqref{eq:conec2} lead to a contradiction provided $\varepsilon=\varepsilon(\ell)\ll1$.
		\end{proof}
		
		Moreover, as a consequence of Proposition~\ref{lemma:regularity} and Corollary~\ref{cor:eh}, the following holds
		\begin{corollary}
			\label{cor:cont}
			Let $\mathscr{F}$ be a family of sets as in Proposition~\ref{lemma:regularity}. Then, there exists $R_6>0$,  $\tilde \beta:(0,R_6)\to(0,+\infty)$, $\lim_{s\downarrow 0}\tilde \beta(s)=0$ such that for all $E\in\mathscr{F}$ and  for all $x,y\in\partial^*E$ with $\|x-y\|<R_6$ it holds
			\begin{equation*}
				\|\nu_E(x)-\nu_E(y)\|\leq \tilde \beta (\|x-y\|).
			\end{equation*}
		\end{corollary}
		We omit the proof of the corollary since it is a direct consequence of the flatness condition of Corollary~\ref{cor:eh}.

		If $\partial E=\partial ^*E$,  Lemma~\ref{lemma:lipcone} and Corollary~\ref{cor:cont} immediately imply that $\partial E$ is locally given by a single Lipschitz graph, on balls of uniform radii with respect to  $z\in\partial E$, $E\in\mathscr{F}$.  Indeed, one has the following
		\begin{corollary}\label{cor:lip}
			Let $\mathscr{F}$ be a family of sets as in Lemma~\ref{lemma:regularity}. Then, for all $\ell>0$ there exists $\hat R_\ell>0$ such that for all $E\in\mathscr{F}$, for all $0<r<\hat R_\ell $, $B_r(z)\subset\R^{d}$
			\begin{equation*}
				\partial ^*E\cap B_r(z)\subset y+K_\ell(x,\nu_E(x)),\quad \forall\,x,y\in\partial^* E\cap B_r(z).
			\end{equation*}
			In particular, for any $x\in \partial^* E\cap B_r(z)$ there exists a Lipschitz function $\phi_{r,z, x}:\Omega\subset(B_r(z))_{\nu_E(x)}^\perp\to \R$   such that $\partial^* E\cap B_r(z)$ is the intersection of $B_r(z)$ with the graph of  $\phi_{r,z, x}$.
			
			Moreover, if $\partial^*E=\partial E$ and above we choose $x=z\in\partial E$, one has that $\Omega\supset (B_{c_\ell r}(z))_{\nu_E(x)}^\perp$ for some $c_\ell>0$ depending only on  $\ell$.
		\end{corollary}
		
		Let now
		\[
		\mathscr{F}=\{E\subset\R^d: \,\text{ $E$ is of locally finite perimeter in $\R^d$  and $\overline\Fcal_{0,p,d}(E,\Omega)\leq M$}\}.
		\]
		Thanks to Lemmas~\ref{lemma:per_low_bound} and~\ref{lemma:excf}, $\partial^*E=\partial E$ for all $E\in\mathscr{F}$ (see \eg~\cite{Tam}).
		
		Moreover, thanks to Lemmas~\ref{lemma:upper_bound_perimeter},~\ref{lemma:per_low_bound},~\ref{lemma:excf}, the class $\mathscr{F}$ satisfies the assumptions of Proposition~\ref{lemma:regularity} with $R_3<\min\{R_0,R_1,R_2\}$ and $\alpha$ as in Lemma~\ref{lemma:excf}. Thus, the regularity results contained in this Appendix can be used to conclude the proof of Theorem~\ref{thm:regularity} at the end of Section~\ref{subs:5.4}.
		
	\end{proof}
	\printbibliographyMio{}
\end{document}